\documentclass[a4paper, titlepage,oneside,12pt]{book}
\usepackage{setspace}
\usepackage{enumitem}
\usepackage{lscape}
\usepackage{amsmath, amsfonts, amssymb, graphicx, textcomp, amsthm}
\usepackage{tikz}
\usepackage{appendix}
\usepackage{doc} %To interpret the command \BibTex.
\usepackage{index}
\usepackage[tc]{titlepic}
\usepackage{fancyhdr}
\usepackage[nokeyprefix]{refstyle}
\usepackage{spverbatim}
\usepackage{listings}
\usepackage[nottoc]{tocbibind}
\usepackage{changepage}
\usepackage[left=5cm, bottom=5cm]{geometry}
\usepackage[labelsep=period]{caption}
\numberwithin{equation}{section}
\newcommand{\itemEq}[1]{%
        \begingroup%
        \setlength{\abovedisplayskip}{0pt}%
        \setlength{\belowdisplayskip}{0pt}%
        \parbox[c]{\linewidth}{\begin{flalign}#1&&\end{flalign}}%
        \endgroup}

\newcommand\abstractname{Abstract}  %%% here
\makeatletter
\if@titlepage
  \newenvironment{abstract}{%
      \null\vfil
      \@beginparpenalty\@lowpenalty
      \begin{center}%
        \bfseries \abstractname
        \@endparpenalty\@M
      \end{center}}%
     {\par\vfil\null%\endtitlepage
}
\fi
\makeatother

\newcommand\dedicationname{}  %%% here
\makeatletter
\if@titlepage
  \newenvironment{dedication}{%
      \null\vfil
      \@beginparpenalty\@lowpenalty
      \begin{center}%
        \bfseries \dedicationname
        \@endparpenalty\@M
      \end{center}}%
     {\par\vfil\null%\endtitlepage
}
\fi
\makeatother

\newcommand\Acknowledgementsname{Acknowledgements}  %%% here
\makeatletter
\if@titlepage
  \newenvironment{acknowledgements}{%
      \null\vfil
      \@beginparpenalty\@lowpenalty
      \begin{center}%
        \bfseries \Acknowledgementsname
        \@endparpenalty\@M
      \end{center}}%
     {\par\vfil\null%\endtitlepage
}
\fi
\makeatother

\makeatletter
\g@addto@macro\bfseries{\boldmath}
\makeatother

\newtheorem{Theorem}{Theorem}
\newtheorem{Lemma}{Lemma}
\newtheorem{Definition}{Definition}
\newtheorem{Corollary}{Corollary}
\newtheorem{Example}{Example}

\makeindex
\ifpdf
    \pdfcatalog { /PageMode (/UseOutlines)
                  /OpenAction (fitbh)  }
\fi

\makeatletter
\newcommand{\mathleft}{\@fleqntrue\@mathmargin\parindent}
\makeatother

\begin{document}
\begin{titlepage}
\begin{center}
 {\huge\bfseries On the Enumeration of Circulant Graphs of Prime-Power Order:\\ the case of $p^3$\\}
 % ----------------------------------------------------------------
 \vspace{1.5cm}
 {\Large\bfseries Victoria Gatt}\\[5pt]
  % ----------------------------------------------------------------
 \vspace{1cm}
{Thesis  submitted to} \\[5pt]
\emph{{University of Malta}}\\[2cm]
{in partial fulfilment for the award of the degree
 of} \\[2cm]
\textsc{\Large{{Master of Science}}} \\[5pt]
{in Mathematics} \vspace{0.4cm} \\[1cm]
% {By}\\[5pt] {\Large \sc {Me}}
 \vfill
 % ----------------------------------------------------------------
\includegraphics[width=0.8\textwidth]{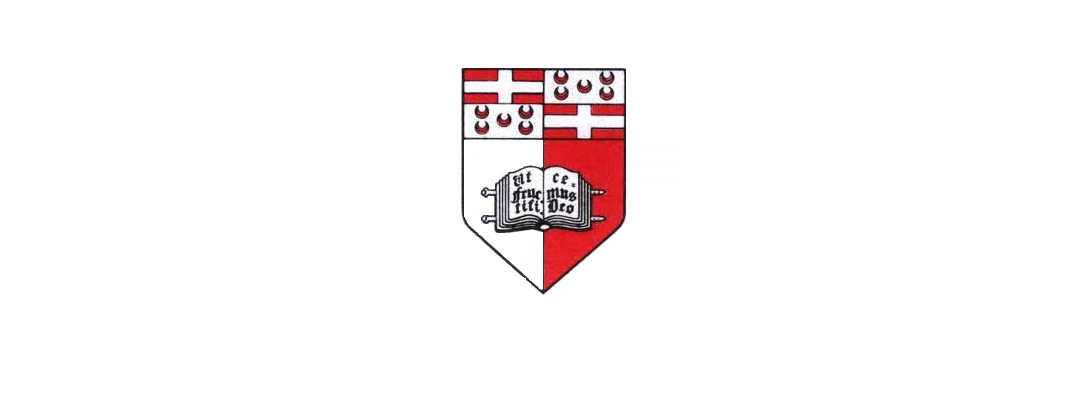}\\[5pt]
{Department of Mathematics}\\[5pt]
{University of Malta}\\[5pt]
 \vfill
{September 2013}
\end{center}
\end{titlepage}

\setcounter{secnumdepth}{3}
\setcounter{tocdepth}{3}
%\frontmatter
\pagenumbering{roman}

\bibliographystyle{plain}
\thispagestyle{empty}
\begin{dedication}
\begin{center}
\Large{To My Family and Friends}
\end{center}
\end{dedication}
\newpage

\thispagestyle{empty}
\begin{acknowledgements}
I would like to express my deepest gratitude to my tutor Prof. Josef Lauri, for his continuous guidance and encouragement during the course of this research work. His critical analysis of my work helped me to improve on my effort to program and write mathematics and his constant patience and support have helped me continue to strive for a solution when programming errors developed and problems seemed unsolvable.

\vspace{5mm}
A word of gratitude also goes to Mikhail Klin and Valery Liskovets for all their help, especially to Professor Klin for passing on a list of Schur rings which they obtained and which we required, to be able to consider a different technique in our work. I would also like to thank Matan Ziv-Av for many helpful hints with GAP.

\vspace{5mm}
Last but certainly not least, I would like to thank all my family and friends for their support throughout these two years, especially my parents who even tried to understand some of the content in this dissertation!
\end{acknowledgements}
\newpage

\thispagestyle{empty}
\vspace{3cm}
\begin{figure}[!hb]
\centering
\includegraphics[width=1.5\linewidth]{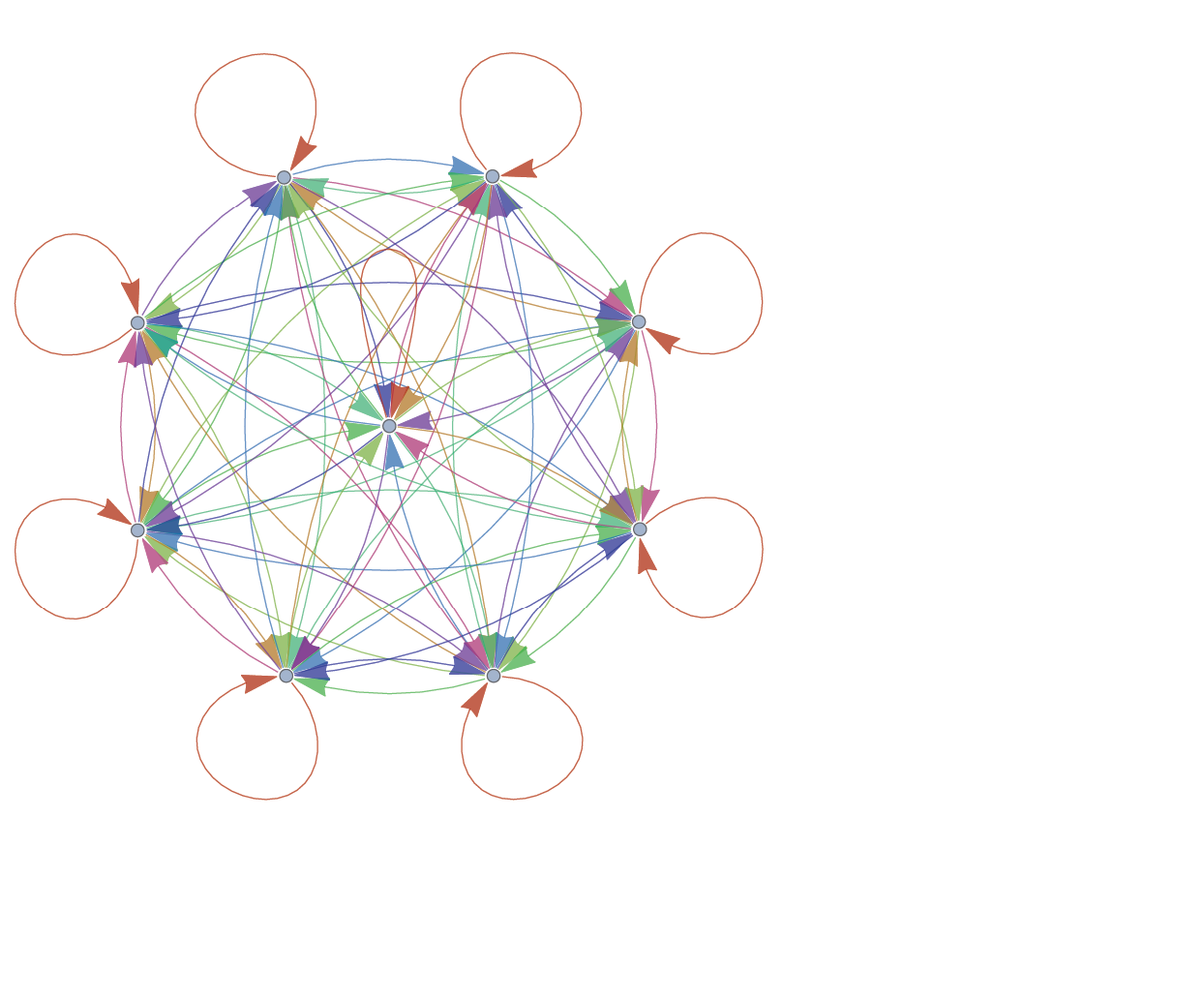}
\caption{A graphical representation of the regular action of $\mathbb{Z}_9$ on itself. (Created by Professor Joseph Muscat using Mathematica).}
\label{PermutationGroups}
\end{figure}
\newpage

\pagestyle{fancy}
\fancyhead[LO,LE]{\slshape \leftmark}
\rhead{}
\fancyfoot[C]{\thepage}
\chead{}
\begin{abstract}
A well-known problem in Algebraic Combinatorics, is the enumeration of circulant graphs. The failure of Adam's Conjecture for such graphs with order containing a repeated prime, led researchers to investigate the problem using two different methods, namely the multiplier method and the structural method. The former makes use of isomorphism theorems whereas the latter involves Schur rings. Both these methods have already been used to count the number of non-isomorphic circulants of order $p^2$. This research focuses on the extension of these two methods to enumerate circulants of order $p^3$, in particular for $p=3$ and $p=5$, through the use of the computer package GAP.
\end{abstract}

\tableofcontents
\newpage
\listoftables
%\newpage
\listoffigures
\mainmatter
\pagebreak
\pagenumbering{arabic}
\fancyhead[R]{\thepage}
\fancyfoot[C]{}

\chapter{Preliminaries and basic definitions}
\section{Permutation Groups}
The following definitions are all standard and can be found in \cite{LP2000} and \cite{KLP96}. The ring of integers is denoted by $\mathbb{Z}$. The additive cyclic group of order $\textit{n}$, where $\textit{n}$ is a positive integer, is defined on the set $\mathbb{Z}_n:=\lbrace0,1,2,...,n-1\rbrace$, where the group operation is addition modulo $n$. $\mathbb{Z}'_n$ will denote $\mathbb{Z}_n - \lbrace{0}\rbrace$. The multiplicative abelian group $\mathbb{Z}^\ast_n$, is defined on the set of numbers in $\mathbb{Z}_n$ which are relatively prime to $\textit{n}$ and where now the group operation is multiplication modulo $n$. Therefore $\vert\mathbb{Z}^\ast_n\vert = \phi(n)$ where $\phi(n)$ is the Euler phi-function. This group is referred to as the prime residue class group (modulo n) \cite{LP2000}. Unless otherwise specified, all arithmetic operations in this thesis will be regarded as modulo $\textit{n}$.

Given a finite set $X$, a subgroup $G$ of the symmetric group $S_{X}$ (that is, the group of all permutations on $X$), will be denoted by $(G, X)$. A regular cyclic permutation group of order (and degree) $n$, is usually denoted by $Z(n)$. This group will usually be considered to be generated by the permutation $(0,1,2,...,n-1)$. Up to similarity of permutation groups, this is in fact $(\mathbb{Z}_n,\mathbb{Z}_n)$, that is, the regular presentation of $\mathbb{Z}_n$.

If a group $G$ acts on a set $X$, then for $x\in X$, the \emph{orbit}\index{orbit} of $x$ under $G$ is the set of all images $x^g$ of $x$ under all the permutations $g$ in $G$, that is
\[Orb(x)=\lbrace x^g:g \in G \rbrace.\]
The orbits form a partition of $X$. If $(G, X)$ is a permutation group, and $x,y \in X$ such that $y=x^g$ for some $g\in G$, then $x$ and $y$ belong to the same orbit of $(G, X)$. We note that the action of a permutation on a set is either denoted by $x^g$ or by $g(x)$.

\vspace{5 mm}
The \emph{direct sum} \index{direct sum} of two groups $G$ and $H$ which act on disjoint sets $U$ and $V$, denoted as follows,

\[(G,U)\oplus(H,V):=(G \times H, U\dot{\cup}V)\]
means that the direct product of the two groups, acts on the disjoint union of the two sets. This action is defined by the rule: for any $(g,h) \in {G \times H}, u\in U$ and $v\in V$,

\[u^{(g,h)}:=u^{g},  v^{(g,h)}:=v^{h}\]

The \emph{direct product} \index{direct product} of the two groups $G$ and $H$ acting on the direct product of two sets is denoted by

\[(G,U)\otimes(H,V):=(G \times H, U \times V)\]
and its action is defined as follows: for any $(g,h) \in G \times H,  u\in U$ and $v\in V$,

\[(u,v)^{(g,h)}=(u^{g},v^{h})\]
Given a group $G$ and two of its actions on disjoint sets $U$ and $V$, given by $(G,U)$ and $(G,V)$ respectively, the \emph{join} \index{join} is denoted by:

\[(G,U) \dot{\vee} (G,V):=(G,U\dot{\cup}V)\]
This combined action is defined by the following rule: for any $g\in G$ and $w\in U \dot{\cup} V$,

\[ w^{g}: = \left\{ \begin{array}{ll}
u^{g} & \mbox{for $w = u\in U$}\\
 v^{g}& \mbox{for $w = v\in V$}\end{array} \right. \]

\vspace {5mm}
The \emph{semi-direct product} \index{semi-direct product} of the group $\mathbb{Z}_n$ with some subgroup $H \leq \mathbb{Z}^\ast_{n}$ will be denoted by $\mathbb{Z}_{n}\rtimes H$; it consists of all permutations of the form $\mathbb{Z}_n \rightarrow \mathbb{Z}_n:z\mapsto az+b$ for $z, b \in \mathbb{Z}_n$ and $a \in H$.

\section{Graphs}
\subsection{Directed and Undirected Graphs}
We shall adopt a terminology for graphs which might not be the one most generally used, but which agrees with that in \cite{KLP2003} in order to aid comparison of our results with those found in these works. The most general type of graph $\Gamma=\Gamma(V,E)$ which we shall consider, denotes
the graph having a non-empty, finite, vertex set $V=V(\Gamma)$ and a set $E=E(\Gamma)$, whose elements are ordered pairs of distinct vertices called arcs. If a graph happens to contain the two arcs $(u,v)$ and $(v,u)$, we say that it contains the edge $\lbrace u,v \rbrace$. Therefore a graph can, in general, contain both edges and arcs. In the literature, such graphs are sometimes called ``mixed graphs". In order to emphasize this, we generally refer to a ``graph'' as a \emph{directed graph}; \index{directed graph} for us these two terms are therefore synonymous. When the graph contains only edges, that is, every arc $(u,v)$ is accompanied by the arc $(v,u)$, we say that it is an \emph{undirected graph}. \index{undirected graph} Therefore an undirected graph is a particular type of graph, whereas a graph can be undirected or directed, the latter implying possible existence of both arcs and edges. Both directed and undirected graphs are considered to be without loops or multiple arcs/edges, that is, edges $\lbrace u,u\rbrace$ consisting of a pair of repeated vertices are not allowed and the same arc cannot appear more than once in $E$. When we define Cayley graphs below, the reason for this way of classifying our structures as ``undirected/directed/mixed graphs'', versus ``undirected graphs'' as a special category, will be more clear. Finally, the out-degree (in-degree) of a vertex $v$, is the number of arcs $(v,x)$, (respectively $(x,v)$) in $\Gamma$. Graph theoretical terms not defined here, may be found in \cite{LS2003}\cite{W96}.

\vspace{5 mm}
Two graphs $\Gamma$ and $\Gamma'$ are said to be \emph{isomorphic} \index{isomorphic graphs} i.e. $\Gamma \cong \Gamma'$, if there exists a bijective function $\textit{f}$ between the vertex set of $\Gamma$ and that of $\Gamma$' such that for any two vertices $\textit{u}$ and $\textit{v}$ in $\Gamma$, $(u,v)$ is an arc in $\Gamma$ if and only if $(f(u),f(v))$ is an arc in $\Gamma'$. In the case when $\Gamma=\Gamma'$, the permutation $\textit{f}$ is called an automorphism of $\Gamma$ \index{automorphism} or a graph automorphism.\index{graph automorphism} By definition, such an automorphism is considered an adjacency preserving permutation of the vertex set V($\Gamma$).  All such permutations $\textit{f}$ form the automorphism group Aut($\Gamma$).

The adjacency matrix $A$ \index{adjacency matrix} of a graph $\Gamma=(V,E)$, is a matrix in which the entries are 1 when given any two vertices $v_i, v_j \in V$, $(v_i,v_j)\in E$ and 0 otherwise.

%A graph with $n$ vertices, i.e. of order $n$, is said to be an \textit{n-graph}. The set of vertices of an $n$-graph is usually taken to be $\mathbb{Z}_n$ for definiteness. When all the vertices in a graph have the same degree, the graph is called \textit{regular}. The complete $n$-graph is the only regular  undirected graph of valency $n-1$. It is an undirected graph of order $n$ having all possible edges between its vertices and is symmetric, as a result it can be regarded as an undirected graph. An $n$-graph with no arcs is termed \textit{null graph}.

\subsection{Cayley Graphs and Circulants}

A \emph{Cayley graph} $Cay(G,X)$, \index{Cayley Graph} where $G$ is a finite group and $X$ is a generating set for $\Gamma$ such that $1\notin X$, is a graph with vertices being elements of $G$ and two vertices $u$ and $v$ connected by an arc $(u,v)$ if and only if there is some generator $x$ in $X$ such that $v=ux$. The condition that $X$ is a generating set of $\Gamma$ ensures that any two vertices $u$ and $v$ are joined by a directed path, that is, $\Gamma$ is strongly connected \cite{LS2003}. This set is often referred to as the \emph{connecting set}\index{connecting set} of the Cayley graph. The fact that $1\notin X$, prevents the Cayley graph from having arcs of the form $(u,u)$. In order to construct an undirected Cayley graph, that is, one in which, for any pair of vertices $u$ and $v$, either both or none of the arcs $(u,v)$ and $(v,u)$ exist, an additional condition on $X$ is required, namely that for any generator $x$ in $X$, $x^{-1}$ must also be in $X$. Therefore we see that, if $X$ does not necessarily have the property that ``$x\in X \Rightarrow x^{-1} \in X$'', then the Cayley graph is what we have termed to be a ``graph'' or a ``directed graph'', which can contain both arcs and edges (it could happen that for some $x\in X$, $x^{-1} \in X$, although not for all $x\in X$). If, however, $X$ does have this property, then the Cayley graph is an ``undirected graph''. These two possibilities for $X$ therefore correspond to our two terms ``graph" and ``undirected graph".

A \emph{circulant graph} (or circulant)\index{Circulant} is a graph $\Gamma$ on $\mathbb{Z}_n$ which remains unchanged when carrying out the cyclic permutation (0,1,2,...,$\textit{n}$-1), that is, a graph which has an automorphism that cyclically permutes all the vertices. This means that if ($\textit{u,v}$) is an arc of $\Gamma$, then $(\textit{u}+1,\textit{v}+1)$ is also an arc of $\Gamma$. Equivalently, any graph with a circulant adjacency matrix, that is an adjacency matrix whose rows are obtained by successively shifting the previous row by one position to the right,
$$\left( \begin{array}{ccccc}
a_0 & a_1 &a_2&...&a_{n-1}\\
a_{n-1}&a_0&a_1&...&a_{n-2}\\
.&&&&.\\
.&&&&.\\
.&&&&.\\
a_1&a_2&a_3&...&a_0\end{array} \right),$$
is a circulant graph.

Circulants are always Cayley graphs of cyclic groups, that is, $\Gamma$ is a circulant of order $n$  if and only if $\Gamma = Cay(\mathbb{Z}_n,X)$. A circulant $\Gamma$ has the connection set $X:=\lbrace v\in \mathbb{Z}'_n \vert (0,v) \in E(\Gamma) \rbrace$ that is, the set of all vertices adjacent from the vertex 0. Since a Cayley graph is specified completely by determining the vertices which are adjacent to a single vertex, the connection set $\textit{X}$ of a circulant $\Gamma$, completely specifies the circulant. In fact $E(\Gamma):=\lbrace (u,v)\vert u,v \in \mathbb{Z}_n,v-u\in X \rbrace$. We may therefore write, $\Gamma=\Gamma(\mathbb{Z}_n,X)$ or $\Gamma=\Gamma(X)$ for short. Note that a circulant is a regular graph of outdegree $\vert X \vert $.

\section{The Cycle Index and P\'{o}lya's Theorem}
As previously mentioned, when a group $G$ acts on a set $U$, the orbits partition $U$. Therefore in order to count the number of non-equivalent elements of $U$ (where two elements are considered equivalent if one may be mapped into another by a permutation in $G$), it suffices to count the number of orbits. This may be achieved by using Burnside's Lemma which states the following:
\begin{Theorem}
Let $G$ be a permutation group acting on a set $U$. For $g\in G$ let $\psi (g)$ denote the number of points of $U$ fixed by $g$. Then the number of orbits of $G$ is equal to $\frac{1}{|G|}\sum_{g \in G}\psi(g)$.
\end{Theorem}

A generalisation of Burnside's Lemma, namely P\'{o}lya's theorem, has enabled enumeration of directed and undirected graphs \cite{HP1973}. In his paper, Turner \cite{Turner67}, also made use of P\'{o}lya's theorem to determine the number of non-isomorphic, undirected circulants of order $p$, for prime $p$. This method was later used to enumerate a larger collection of circulants, such as circulant tournaments and self-complementary directed and undirected circulant graphs \cite{Mishna98}. P\'{o}lya's enumeration theorem will also be the main enumerative tool in this dissertation.

Before proceeding to the statement of this theorem, let us first define the cycle index of a group $G$: The \emph{cycle index} of a group $G$ acting on a finite set $U$ is given by the polynomial of degree $n=\left|U\right|$ in indeterminates $x_{1},x_{2},...,x_{n}$ as
\[I_{(G,U)}(x_1,x_2,...,x_n):=\frac{1}{\left|G\right|}\sum_{g\in G}\prod_{i=1}^{n}x_{i}^{a_{i}(g)}\]
where $a_{i}(g)=a_{i}(g,U)$ stands for the number of disjoint cycles of length $i$ in $g$. We sometimes denote this by $I_{(G,U)}$ for short.

The regular group $Z(n)=(\mathbb{Z}_n,\mathbb{Z}_n)$ has cycle index
%\[I_{Z(n)}=\frac{1}{n}\sum_{i=1}^n x_{[i,n]/i}^{(i,n)}\]
% $(i,n)$ denotes the greatest common divisor of $i$ and $n$ and $[i,n]$ denotes the least common multiple. From this it follows \cite{LP96} that
\[I_{Z(n)}=\frac{1}{n}\sum_{r|n} \phi(r) x_{r}^{\frac {n}{r}},\]
since for each divisor $r$ of $n$, the cyclic group $\mathbb{Z}_n$ of order $n$, has $\phi(r)$ elements of order $r$. Moreover, as a permutation in the regular action of $\mathbb{Z}_n$, an element of order $r$ has $\frac{n}{r}$ cycles of length $r$ \cite{LW2001}.

We shall see that for enumeration of circulants, the action we shall be requiring involves $\mathbb{Z}^\ast_n$ acting on $\mathbb{Z}'_n$ or its subsets, as appropriate. We will therefore need to find the cycle index of variations of the action $(\mathbb{Z}^\ast_n,\mathbb{Z}'_n)$. For the case when $n=p$, where $p$ is a prime number, $\mathbb{Z}^\ast_p\cong\mathbb{Z}'_p$. Therefore $(\mathbb{Z}^\ast_p,\mathbb{Z}'_p)$ becomes the regular action of $\mathbb{Z}^\ast_p$ on itself. The cycle index in this case is given by:

\[I_{\mathbb{Z}^\ast_p}=\mathcal{I}_{p-1}(x)=\frac{1}{p-1}\sum_{r|p-1} \phi(r) x_{r}^{\frac {p-1}{r}}\]
where the sum is taken over all divisors of $p-1$ \cite{KLP2003}.

For an example when $n$ is not prime, consider the action of $\mathbb{Z}^\ast_9$ acting on $\mathbb{Z}'_9$ where $\mathbb{Z}^{\ast}_9=\lbrace 1,2,4,5,7,8 \rbrace$ and $\mathbb{Z}'_9 =\lbrace1,2,3,4,5,6,7,8\rbrace$. The identity permutation would give the action (1)(2)(3)(4)(5)(6)(7)(8), that is eight cycles of length one. Suppose now that we act on $\mathbb{Z}'_9$ with $2 \in \mathbb{Z}^\ast_9$. This means that we need to multiply all elements of $\mathbb{Z}'_9$ by 2, giving the result additively modulo 9. This gives the action $(1\hspace{2mm} 2\hspace{2mm} 4\hspace{2mm} 8\hspace{2mm} 7\hspace{2mm} 5)(3\hspace{2mm} 6)$ which corresponds to $x_{6}^{1}x_{2}^{1}$.
Similarly, acting on $\mathbb{Z}'_9$ with $4 \in \mathbb{Z}^\ast_9$, means that we now need to multiply all elements of $\mathbb{Z}'_9$ by 4. This gives the action $(1\hspace{2mm} 4\hspace{2mm} 7)(2\hspace{2mm} 8\hspace{2mm} 5)(3)(6)$ corresponding to the monomial $x_{3}^{2}x_{1}^{2}$. Repeating this method for all the elements of $\mathbb{Z}^{\ast}_n$, we obtain the following cycle index:

\[I_{(\mathbb{Z}^\ast_9,\mathbb{Z}'_9)}=\frac{1}{6}(x_{1}^{8}+2x_{6}^{1}x_{2}^{1}+2x_{3}^{2}x_{1}^{2}+x_{2}^{4})\]

Recall also that if $G$ acts on $A$ and $B^A$ is the set of functions from $A$ to $B$, then the induced action of $G$ on $B^A$ is defined by $gf=f\circ g$ for all $g\in G$, $f\in B^A$.

\begin{Theorem}[{P\'{o}lya's Enumeration Theorem} \cite{LW2001}]
Let $A$ and $B$ be finite sets and let $G$ act on $A$. Denote by $c_k(G)$ the number of permutations in $G$ that have exactly $k$ cycles in their cycle decomposition on $A$. Then the number of orbits of $G$ on the set $B^A$ of all mappings $f:A \rightarrow B$ is

\[\frac{1}{|G|}\sum^\infty_{k=1}c_k(G)|B|^k\]
\end{Theorem}
From this theorem and the definition of cycle index, we have that the number of orbits of $G$ on $B^A$ is

\[I_{(G,U)}(b,b,...,b):=\frac{1}{\left|G\right|}\sum_{g\in G}b^{a_1(g)+a_2(g)+...+a_n(g)}\]
with $b:=|B|$ \cite{LW2001}.

\section{The Group Ring of Cyclic Groups}
The group ring $\langle\mathbb{Z}\lbrack\mathbb{Z}_{n}\rbrack;+, \cdot\rangle$ of $\mathbb{Z}_{n}$ over $\mathbb{Z}$, consists of the set of all linear formal combinations of elements of $\mathbb{Z}_{n}$ with integral coefficients, that is of all formal sums $\sum_{h\in\mathbb{Z}_{n}}\alpha_{h}\underline{h}$ with $\alpha_{h}\in\mathbb{Z}, h\in\mathbb{Z}_{n}$, together with addition

\[\sum_{h\in\mathbb{Z}_{n}}\alpha_{h}\underline{h}+\sum_{h\in\mathbb{Z}_{n}}\beta_{h}\underline{h}:=\sum_{h\in\mathbb{Z}_{n}}(\alpha_{h}+\beta_{h})\underline{h}\]
and formal multiplication
\[\left(\sum_{h\in\mathbb{Z}_{n}}\alpha_{h}\underline{h}\right)\cdot\left(\sum_{k\in\mathbb{Z}_{n}}\beta_{k}\underline{k}\right):=\sum_{h,k\in\mathbb{Z}_{n}}\alpha_{h}\beta_{k}\left(\underline{h+k}\right)=\sum_{h\in\mathbb{Z}_{n}}\left(\sum_{k\in\mathbb{Z}_{n}}\alpha_{h-k}\beta_{k}\right)\underline{h}.\]

One must note that the sum  $\sum_{h\in\mathbb{Z}_{n}}\alpha_{h}\underline{h}$ is simply a different notation for a sequence of integers $\alpha_{h}$ i.e. for
\[(\alpha_{0},\alpha_{1},\alpha_{2},...\alpha_{n-1})\]
and that $\alpha_{h}$ should not be multiplied by $\underline{h}$ and summed up. In addition, other rings or fields such as real or complex numbers, may be used instead of $\mathbb{Z}$.
$\mathbb{Z}\lbrack\mathbb{Z}_{n}\rbrack$ is also a $\mathbb{Z}$-module with scalar multiplication.
\[\alpha\left(\sum_{h\in\mathbb{Z}_{n}}\alpha_{h}\underline{h}\right):=\sum_{h\in\mathbb{Z}_{n}}\left(\alpha\alpha_{h}\right)\underline{h}\]
for $\alpha\in\mathbb{Z}$.

The $\mathbb{Z}$-submodule of $\mathbb{Z}\lbrack\mathbb{Z}_{n}\rbrack$ generated by elements $\lambda_{1},...,\lambda_{r} \in \mathbb{Z}\lbrack\mathbb{Z}_{n}\rbrack$ will be denoted by
\[\langle \lambda_{1},...,\lambda_{r}\rangle.\]
Therefore the $\mathbb{Z}$-submodule $\langle \lambda_{1},...,\lambda_{r}\rangle$, consists of all linear combinations of $\lambda_{1},...,\lambda_{r}$ and their products.

Assume $T\subseteq\mathbb{Z}_{n}, T=\lbrace t_{1},t_{2},...,t_{r} \rbrace$. Elements of the form
\[\underline{T}:=\sum_{h\in T}\underline{h}\]
are called \emph{simple quantities}\index{simple quantities} of $\mathbb{Z}\lbrack \mathbb{Z}_{n} \rbrack$. One can consider $\underline{T}$ as the formal sum $\sum_{h\in\mathbb{Z}_{n}}\alpha_{h}\underline{h}$ with $\alpha_{h}=1$ if and only if $h \in T$ and $\alpha_{h}=0$ otherwise, that is, a simple quantity is a list in which every entry has multiplicity 1. For $T=\lbrace t_{1},t_{2},...,t_{r} \rbrace$ we use the notation
\[\underline{t_{1},...,t_{r}}\]
instead of $\lbrace t_{1},...,t_{r}\rbrace$.

For example, suppose n=7 and

\[\alpha=3\hspace{0.7mm}\underline{0}+2\hspace{0.7mm}\underline{1}+2\hspace{0.7mm}\underline{2}+3\hspace{0.7mm}\underline{3}+3\hspace{0.7mm}\underline{4}+2 \hspace{0.7mm}\underline{5}+3\hspace{0.7mm}\underline{6}.\]

This can be seen as a list of elements of $\mathbb{Z}_{7}$ with multiple entries, with the elements 1,2,5 appearing twice and 0, 3, 4, 6 appearing three times. This may be expressed as two different sublists, namely $\lbrace 0,3,4,6 \rbrace$ and $\lbrace 1,2,5 \rbrace$, where the former is used three times and the latter twice:

\[\alpha=3\hspace{0.7mm}\underline{0,3,4,6 }+2\hspace{0.7mm}\underline{1,2,5}.\]

Therefore the elements of $\mathbb{Z}(\mathbb{Z}_{n})$ may be seen as lists of elements of $\mathbb{Z}_{n}$ with multiple entries.

The \emph{Schur-Hadamard product} \index{Schur-Hadamard product} $\circ$ of elements of $\mathbb{Z}\lbrack \mathbb{Z}_{n} \rbrack$ is defined as follows

\[\left(\sum_{h\in\mathbb{Z}_{n}}\alpha_{h}\underline{h}\right)\circ\left(\sum_{h\in\mathbb{Z}_{n}}\beta_{h}\underline{h}\right):=\sum_{h\in\mathbb{Z}_{n}}(\alpha_{h}\beta_{h})\underline{h}\]
where, for $T, T' \subseteq \mathbb{Z}_{n} $ we have $\underline{T} \circ \underline{T'}=\underline{T \cap T'}.$ The transpose of $\lambda=\sum_{h\in\mathbb{Z}_{n}}\alpha_{h}\underline{h}\in \mathbb{Z}\lbrack\mathbb{Z}_{n}\rbrack$ is given by

\[\lambda^{t}:=\sum_{h\in\mathbb{Z}_{n}}\alpha_{h}(\underline{-h})\]

The following example describes some of the operations described above.
Let $n=4$ and suppose we have

\[\sigma= 3\hspace{0.7mm}\underline{0}+ 2\hspace{0.7mm}\underline{1}+3\hspace{0.7mm}\underline{2},\hspace{0.5cm} \tau=\underline{1}+\underline{2}-3\hspace{0.7mm}\underline{3}\]

Then,
\begin{equation*}
\begin{split}
2\tau&=2\hspace{0.7mm}\underline{1}+2\hspace{0.7mm}\underline{2}-6\hspace{0.7mm}\underline{3}\\
\sigma \circ \tau&=2\hspace{0.7mm}\underline{1}+3\hspace{0.7mm}\underline{2}\\
\sigma + \tau&=3\hspace{0.7mm}\underline{0}+3\hspace{0.7mm}\underline{1}+4\hspace{0.7mm}\underline{2}-3\hspace{0.7mm}\underline{3}\\
\sigma\cdot\tau&=3\hspace{0.7mm}\underline{1}+3\hspace{0.7mm}\underline{2}-9\hspace{0.7mm}\underline{3}+2\hspace{0.7mm}\underline{2}+2\hspace{0.7mm}\underline{3}-6\hspace{0.7mm}\underline{0}+3\hspace{0.7mm}\underline{3}+3\hspace{0.7mm}\underline{0}-9\hspace{0.7mm}\underline{1}
\end{split}
\end{equation*}

\subsection{Schur Rings of Cyclic Groups and Circulants}
A subring $\mathcal{S}$ of a group ring $\mathbb{Z}(\mathbb{Z}_n)$ is called a \emph{Schur ring} \index{Schur Ring} $\mathfrak{S}$ or $\mathcal{S}$-ring over $\mathbb{Z}_n$, of rank $r$ if the following conditions hold:
\begin{itemize}
\item $\mathcal{S}$ is closed under addition and multiplication with elements from $\mathbb{Z}$ (i.e. $\mathcal{S}$ is a $\mathbb{Z}$-module);
\item Simple quantities $\underline{T}_{0},\underline{T}_{1},...,\underline{T}_{r-1}$ exist in $\mathcal{S}$ such that every element $\sigma \in \mathcal{S}$ has a unique representation;
\[\sigma=\sum_{i=0}^{r-1}\sigma_{i}\underline{T}_{i}\]
\item $\underline{T}_{0}=\underline{0}$, $\sum_{i=0}^{r-1}\underline{T}_{i}=\underline{\mathbb{Z}}_n$, that is, $\lbrace T_{0},T_{1},...,T_{r-1}\rbrace$ is a partition of $\mathbb{Z}_n$;
\item For every $i\in \lbrace 0,1,2,...,r-1 \rbrace$ there exists a $j \in \lbrace 0,1,2,...,r-1 \rbrace$ such that $\underline{T}_{j}=\underline{-T}_{i} (=\underline{\lbrace n-x : x \in T_{i}} \rbrace )$ (therefore, $\underline{T_{i}}^{t}=\underline{T_j}$);
\item For $i,j \in \lbrace 1,...,r \rbrace$, there exist non-negative integers $p_{ij}^{k}$ called structure constants, such that
\[\underline{T}_{i}\cdot\underline{T}_{j}=\sum_{k=1}^{r} p_{ij}^{k}\underline{T}_{k}\]
\end{itemize}

One may note that for an $\mathcal{S}$-ring $\mathfrak{S}$, the final condition is equivalent to saying that $\mathfrak{S}$ is closed under the $\cdot$ operation.

The simple quantities $\underline{T}_{0},\underline{T}_{1},...,\underline{T}_{r-1}$ form a standard basis for $\mathfrak{S}$ and their corresponding sets $T_{i}$ are \emph{basic sets} \index{basic sets} of the $\mathcal{S}$-ring. The circulant graphs $\Gamma_i=\Gamma(T_i)$, where $0\leq i \leq r-1$, are called \emph{basic circulant graphs} \index{basic circulant graphs}of $\mathfrak{S}$ \cite{KRRT91}. The following notation, will denote the relation of an $\mathcal{S}$-ring to its basic sets

\[\mathfrak{S}=\langle\underline{T}_{0},\underline{T}_{1},...\underline{T}_{r-1}\rangle.\]
$\mathbb{Z}(\mathbb{Z}_n)$ and $\langle\underline{0},\underline{\mathbb{Z}_n-\lbrace 0 \rbrace}\rangle$ are both Schur rings over $\mathbb{Z}_n$. In fact both of these are called trivial Schur rings.

A permutation $g:\mathbb{Z}_n\rightarrow\mathbb{Z}_n$ is called an automorphism of an $\mathcal{S}$-ring $\mathfrak{S}$, if it is an automorphism of every graph $\Gamma_{i}$. Equivalently, the intersection of the automorphism groups of the basic circulant graphs of an $\mathcal{S}$-ring $\mathfrak{S}=\langle \underline{T}_0,\underline{T}_1,...,\underline{T}_{r-1} \rangle$, gives the automorphism group of the $\mathcal{S}$-ring.

\begin{equation}
Aut \mathfrak{S}:=\bigcap_{i=0}^{r-1}Aut \Gamma_{i}
\end{equation}

Let $(G,\mathbb{Z}_n)$ be a permutation group containing the cyclic group $(\mathbb{Z}_n,\mathbb{Z}_n)$. Consider the stabilizer $G_{0}$ and its orbits $(G_{0},\mathbb{Z}_n)=\lbrace T_{0},...,T_{r-1}\rbrace$, where again $T_{0}:=\lbrace0\rbrace$. Then

\[\mathfrak{S}(G,\mathbb{Z}_n):=\langle \underline{T_{0}},...,\underline{T_{r-1}}\rangle\]
is called the \emph{transitivity module} of $(G,\mathbb{Z}_n)$. Schur \cite{KLP96}, has shown that $\mathfrak{S}(G,\mathbb{Z}_n)$ is an $\mathcal{S}$-ring. When an $\mathcal{S}$-ring $\mathfrak{S}$ is the transitivity module of some overgroup $(G,\mathbb{Z}_n)$ of $(\mathbb{Z}_n,\mathbb{Z}_n)$,then $\mathfrak{S}$ is called \emph{Schurian} \index{Schurian}.

For the circulants $\Gamma_{i}:=\Gamma(\mathbb{Z}_n,T_{i})$, a graph theoretical interpretation may be given to the structure constants $p_{ij}^{k}$ of a Schur ring $\mathfrak{S}$. For $(u,v)\in E(\Gamma_{k})$,

\[p_{ij}^k=|\lbrace w| (u,w) \in E(\Gamma_{i})\wedge (w,v) \in E(\Gamma_{j})\rbrace|\]
gives, for any arc $(u,v)$ in $\Gamma_{k}$, the number of ways we can go from $u$ to $v$  by first taking an arc $(u,w)$ in $\Gamma_{i}$ and then an arc $(w,v)$ in $\Gamma_{j}$. As a result of S-ring properties, this number does not depend on the choice of $(u,v)$ in $\Gamma_{k}$.

For example, let us consider $\mathbb{Z}_8$. One Schur ring of $\mathbb{Z}_8$ is the following \cite{KLP96}

\[\langle \underline{0},\underline{1,5}, \underline{3,7}, \underline{2,6}, \underline{4}\rangle.\]
Therefore we have $T_0=\lbrace 0 \rbrace$, $T_1=\lbrace 1,5 \rbrace$, $T_2=\lbrace 3,7 \rbrace$, $T_3=\lbrace 2,6 \rbrace$, and $T_4=\lbrace 4 \rbrace$. The circulants $\Gamma_{i}:=\Gamma(\mathbb{Z}_8,T_{i})$ are shown in figure~\ref{Fig2e}. In $\Gamma_3$ and $\Gamma_4$, the edges are in fact pairs of arcs.

\begin{figure}[!h]
\centering
\includegraphics[width=0.75\linewidth]{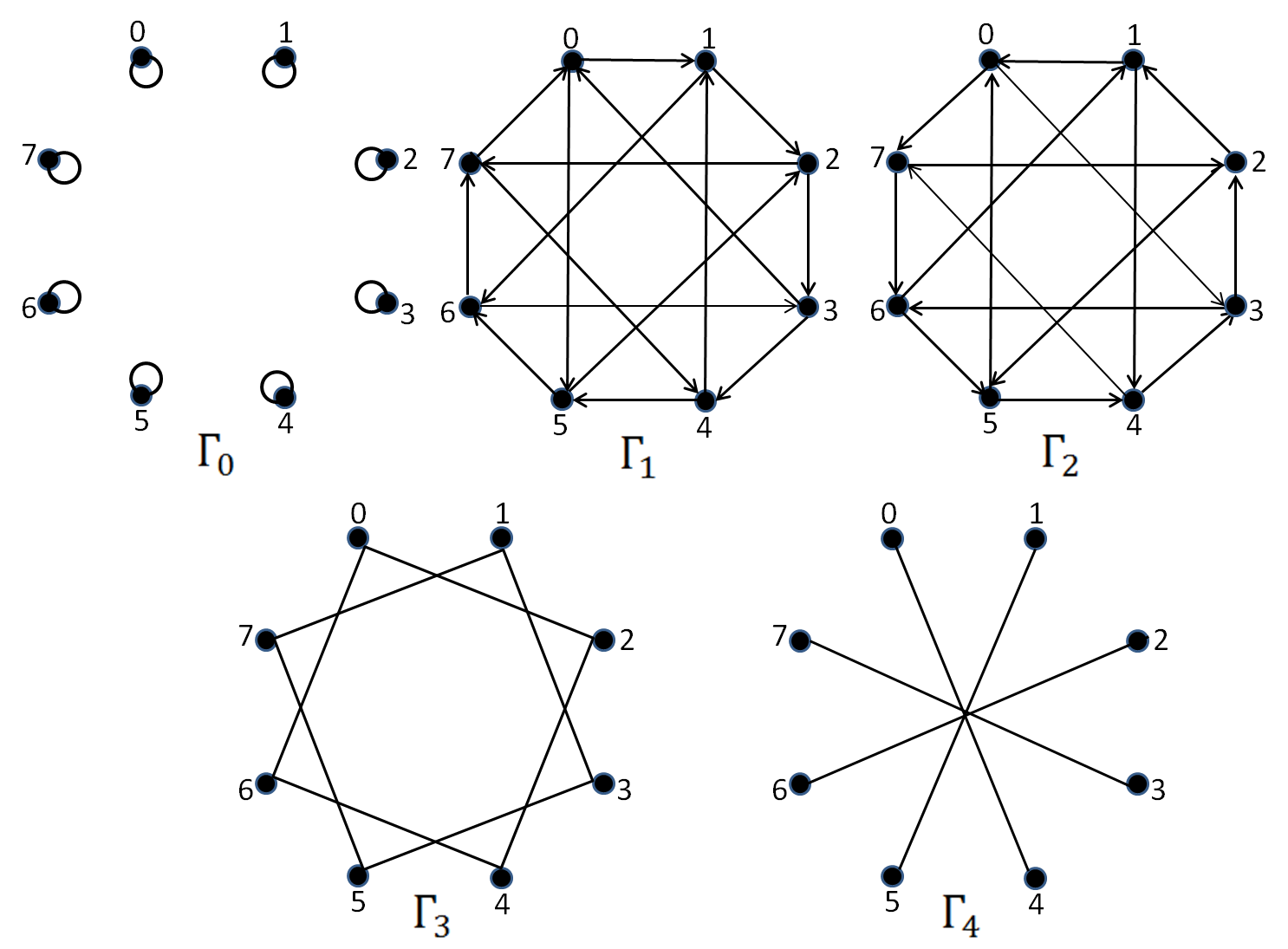}
\caption{Basic Circulant Graphs}
\label{Fig2e}
\end{figure}
In order to determine $p^1_{23}$ for example, we first take an arc in $\Gamma_1$ say $(1,6)$, and see in how many ways we can go from 1 to 6, by first choosing an arc in $\Gamma_2$ and then an arc in $\Gamma_3$. This may be done in two ways. We may either take the arc $(1,4)$ in $\Gamma_2$ and then the arc $(4,6)$ in $\Gamma_3$, or start off with the arc $(1,0)$ in $\Gamma_2$ and then choose the arc $(0,6)$ in $\Gamma_3$. We may verify that if any other arc is chosen in $\Gamma_1$, instead of $(1,6)$, we still get two possible ways of going from one vertex to another by first going through an arc in $\Gamma_2$ and then an arc in $\Gamma_3$. Therefore $p^1_{23}=2$.

Let us now consider $p^4_{21}$. We first choose an arc in $\Gamma_4$, say $(0,4)$, then determine the number of ways we can go from 0 to 4, by first picking an arc in $\Gamma_2$ and then an arc in $\Gamma_1$. Once again, there are two possible ways to do this, namely either choosing $(0,3)$ in $\Gamma_2$ and then $(3,4)$ in $\Gamma_1$, or by choosing $(0,7)$ in $\Gamma_2$, and then $(7,4)$ in $\Gamma_1$. Therefore $p^4_{21}=2$. One may also verify that for example $p^3_{24}=0$, since for any arc $(u,v)$ in $\Gamma_3$, there is no way we can go from $u$ to $v$ by first taking an arc in $\Gamma_2$ and then an arc in $\Gamma_4$.

If $G$ and $H$ are two permutation groups acting on the sets $U$ and $V$ respectively, where $|U|=n$ and $|V|=m$, the \emph{wreath product} $G\wr H$, \index{wreath product} is the group of permutations which acts on the pairs $(u,v) \in U \times V$. This means wreath products are permutation groups which act on sets of ordered pairs. It follows that a wreath product is a subgroup of the symmetric group, in this case $S_{nm}$. The permutations are represented by $f=\lbrack g, h(1),h(2),...h(n)\rbrack$ for $g \in G$ and $h(u) \in H$ for every $u \in U$ \cite{Garten77}. These act on the pairs $(u,v) \in U \times V$ by the following rule:

\[(u,v)^{f}=(u,v)^{\lbrack g, h(1),h(2),...h(n)\rbrack}=(u^{g},v^{h(u)})\]
Every n-tuple in the set of all n-tuples of elements of $H$, may be considered to be an element of the direct product of $n$ copies of $H$, $H_1\times H_2\times...\times H_n$. Since we have $|H|^n$ possible choices of $n$-tuples of elements of $H$ and $|G|$ possible choices of $g \in G$, the order of the wreath product $G\wr H$ is $\vert G \vert \cdot \vert H \vert ^{n}$ where $n=|U|$.

Let us consider the following example in order to understand this definition better. Let $U=\lbrace 1,2,3 \rbrace$ and $V=\lbrace 1,2\rbrace$. Therefore $|U|=3$, $|V|=2$ and $|U \times V|=6$. Suppose $G=\lbrace (1),(123),(132),(13)(2),(12)(3),(32)(1)\rbrace$, a permutation group of order 6, acting on $U$ and the permutation group $H=\lbrace (1), (12) \rbrace$ of order 2, acting on $V$. Since $|G|=6$ and $|H|=2$, then we have $|G \wr H|=6 \cdot 2^3=48$. This means the wreath product consists of 48 elements, of the possible 6! in $S_{6}$.

The ordered pairs $(u,v)$ are as follows:
\begin{equation*}
a_1=(1,1),\hspace{2mm}a_2=(1,2),\hspace{2mm}a_3=(2,1),\hspace{2mm}a_4=(2,2),\hspace{2mm}a_5=(3,1),\hspace{2mm}a_6=(3,2).
\end{equation*}

Now since $U=\lbrace 1,2,3 \rbrace$, the permutations of $G \wr H$ are represented by $f=\lbrack g, h(1),h(2),h(3)\rbrack$, where $h(1),h(2),h(3)$ are all possible triples of elements of $H$. Therefore we obtain the following permutations:

\begin{quote}
\begin{tabular}{llllll}
$f_1$&=&$\lbrack(1),\hspace{2mm} (1),(1),(1)\rbrack$&$f_{25}$&=&$\lbrack(13)(2),\hspace{2mm} (1),(1),(1)\rbrack$\\
$f_2$&=&$\lbrack(1),\hspace{2mm} (1),(1),(1 2)\rbrack$&$f_{26}$&=&$\lbrack(13)(2),\hspace{2mm} (1),(1),(1 2)\rbrack$\\
$f_3$&=&$\lbrack(1),\hspace{2mm} (1),(1 2),(1)\rbrack$&$f_{27}$&=&$\lbrack(13)(2),\hspace{2mm} (1),(1 2),(1)\rbrack$\\
$f_4$&=&$\lbrack(1),\hspace{2mm} (1 2),(1),(1)\rbrack$&$f_{28}$&=&$\lbrack(13)(2),\hspace{2mm} (1 2),(1),(1)\rbrack$\\
$f_5$&=&$\lbrack(1),\hspace{2mm} (1 2),(1 2),(1)\rbrack$&$f_{29}$&=&$\lbrack(13)(2),\hspace{2mm} (1 2),(1 2),(1)\rbrack$\\
$f_6$&=&$\lbrack(1),\hspace{2mm} (1 2),(1),(1 2)\rbrack$&$f_{30}$&=&$\lbrack(13)(2),\hspace{2mm} (1 2),(1),(1 2)\rbrack$\\
$f_7$&=&$\lbrack(1),\hspace{2mm} (1),(1 2),(1 2)\rbrack$&$f_{31}$&=&$\lbrack(13)(2),\hspace{2mm} (1),(1 2),(1 2)\rbrack$\\
$f_8$&=&$\lbrack(1),\hspace{2mm} (1 2),(1 2),(1 2)\rbrack$&$f_{32}$&=&$\lbrack(13)(2),\hspace{2mm} (1 2),(1 2),(1 2)\rbrack$\\
$f_9$&=&$\lbrack(123),\hspace{2mm} (1),(1),(1)\rbrack$&$f_{33}$&=&$\lbrack(12)(3),\hspace{2mm} (1),(1),(1)\rbrack$\\
$f_{10}$&=&$\lbrack(123),\hspace{2mm} (1),(1),(1 2)\rbrack$&$f_{34}$&=&$\lbrack(12)(3),\hspace{2mm} (1),(1),(1 2)\rbrack$\\
$f_{11}$&=&$\lbrack(123),\hspace{2mm} (1),(1 2),(1)\rbrack$&$f_{35}$&=&$\lbrack(12)(3),\hspace{2mm} (1),(1 2),(1)\rbrack$\\
$f_{12}$&=&$\lbrack(123),\hspace{2mm} (1 2),(1),(1)\rbrack$&$f_{36}$&=&$\lbrack(12)(3),\hspace{2mm} (1 2),(1),(1)\rbrack$\\
$f_{13}$&=&$\lbrack(123),\hspace{2mm} (1 2),(1 2),(1)\rbrack$&$f_{37}$&=&$\lbrack(12)(3),\hspace{2mm} (1 2),(1 2),(1)\rbrack$\\
$f_{14}$&=&$\lbrack(123),\hspace{2mm} (1 2),(1),(1 2)\rbrack$&$f_{38}$&=&$\lbrack(12)(3),\hspace{2mm} (1 2),(1),(1 2)\rbrack$\\
$f_{15}$&=&$\lbrack(123),\hspace{2mm} (1),(1 2),(1 2)\rbrack$&$f_{39}$&=&$\lbrack(12)(3),\hspace{2mm} (1),(1 2),(1 2)\rbrack$\\
$f_{16}$&=&$\lbrack(123),\hspace{2mm} (1 2),(1 2),(1 2)\rbrack$&$f_{40}$&=&$\lbrack(12)(3),\hspace{2mm} (1 2),(1 2),(1 2)\rbrack$\\
$f_{17}$&=&$\lbrack(132),\hspace{2mm} (1),(1),(1)\rbrack$&$f_{41}$&=&$\lbrack(32)(1),\hspace{2mm} (1),(1),(1)\rbrack$\\
$f_{18}$&=&$\lbrack(132),\hspace{2mm} (1),(1),(1 2)\rbrack$&$f_{42}$&=&$\lbrack(32)(1),\hspace{2mm} (1),(1),(1 2)\rbrack$\\
$f_{19}$&=&$\lbrack(132),\hspace{2mm} (1),(1 2),(1)\rbrack$&$f_{43}$&=&$\lbrack(32)(1),\hspace{2mm} (1),(1 2),(1)\rbrack$\\
$f_{20}$&=&$\lbrack(132),\hspace{2mm} (1 2),(1),(1)\rbrack$&$f_{44}$&=&$\lbrack(32)(1),\hspace{2mm} (1 2),(1),(1)\rbrack$\\
$f_{21}$&=&$\lbrack(132),\hspace{2mm} (1 2),(1 2),(1)\rbrack$&$f_{45}$&=&$\lbrack(32)(1),\hspace{2mm} (1 2),(1 2),(1)\rbrack$\\
$f_{22}$&=&$\lbrack(132),\hspace{2mm} (1 2),(1),(1 2)\rbrack$&$f_{46}$&=&$\lbrack(32)(1),\hspace{2mm} (1 2),(1),(1 2)\rbrack$\\
$f_{23}$&=&$\lbrack(132),\hspace{2mm} (1),(1 2),(1 2)\rbrack$&$f_{47}$&=&$\lbrack(32)(1),\hspace{2mm} (1),(1 2),(1 2)\rbrack$\\
$f_{24}$&=&$\lbrack(132),\hspace{2mm} (1 2),(1 2),(1 2)\rbrack$&$f_{48}$&=&$\lbrack(32)(1),\hspace{2mm} (1 2),(1 2),(1 2)\rbrack$\\
\end{tabular}
\end{quote}

Now suppose we require $a_2^{f_4}$

\begin{equation*}
\begin{split}
(1,2)^{f_4}&=(1,2)^{\lbrack 1,\hspace{2mm} (1\hspace{2mm}2),(1),(1)\rbrack}\\
&=(1^{(1)},2^{h(1)})\\
&=(1^{(1)},2^{(1\hspace{2mm} 2)})\\
&=(1,1)
\end{split}
\end{equation*}
The first component of the ordered pair, that is 1, is left unchanged by $g$, however, the second component of the ordered pair, that is 2, is acted upon by the $h(1)$ component of $f_4$ that is (1\hspace{2mm} 2), in which 2 is mapped into 1. Therefore $(1,2)$ is mapped into $(1,1)$.
Let us now consider $a_5^{f_{15}}$:

\begin{equation*}
\begin{split}
(3,1)^{f_{15}}&=(3,1)^{\lbrack (1\hspace{2mm}2\hspace{2mm}3),\hspace{2mm} (1),(1\hspace{2mm}2),(1\hspace{2mm}2)\rbrack}\\
&=(3^{(1\hspace{2mm}2\hspace{2mm}3)},1^{h(3)})\\
&=(3^{(1\hspace{2mm}2\hspace{2mm}3)},1^{(1\hspace{2mm} 2)})\\
&=(1,2)
\end{split}
\end{equation*}
In this case, the first component of the ordered pair, that is 3, is acted upon by $(1\hspace{2mm}2\hspace{2mm}3)$, where 3 is mapped into 1. The second component, that is 1, is acted upon by the $h(3)$ component of $f_4$ that is (1\hspace{2mm} 2), in which 1 is mapped into 2. As a result we have $(3,1)$ being mapped to $(1,2)$.\\
Similarly in $a_6^{f_{21}}$ we have the following
\begin{equation*}
\begin{split}
(3,2)^{f_{21}}&=(3,2)^{\lbrack (1\hspace{2mm}3\hspace{2mm}2),\hspace{2mm} (1\hspace{2mm}2),(1\hspace{2mm}2),(1)\rbrack}\\
&=(3^{(1\hspace{2mm}3\hspace{2mm}2)},2^{h(3)})\\
&=(3^{(1\hspace{2mm}3\hspace{2mm}2)},2^{(1)})\\
&=(2,2)
\end{split}
\end{equation*}
in which the first component of the ordered pair is mapped into 2, while the second component is left unchanged by the action of the $h(3)$ component of $f_{21}$, which leaves 2 fixed. Therefore $(3,2)$ is mapped to $(2,2)$.

The 48 elements of the wreath product $G\wr H$, may be expressed as permutations of $a_1, a_2,a_3,a_4,a_5,a_6$. For example we have:

\begin{equation*}
\begin{split}
f_1=&I\\
f_2=&(a_1)(a_2)(a_3)(a_4)(a_5a_6)\\
f_3=&(a_1)(a_2)(a_3a_4)(a_5)(a_6)\\
\end{split}
\end{equation*}
All the other mappings may be obtained in a similar manner.

%The wreath product $G\wr H$ \index{wreath product} of two permutation groups $G$ and $H$, acting on sets $U$ and $V$  respectively, is the permutation group of order $\vert G \vert \cdot \vert H \vert ^{|U|}$ which acts on $U \times V$. This is defined as follows \cite{KLP96}: $G \wr H$ consists of all permutations, $f=\lbrack g, h(u) \rbrack$ (with $g\in G$ and $h(u) \in H$ for every $u \in U$), acting on the pairs $(u,v) \in U \times V$ by the rule
%\[(u,v)^{f}=(u^{g},v^{h(u)}).\]

%A good example of the wreath product is given by the automorphism group of a disconnected graph on $k$ components isomorphic to $G$. In this case, the automorphism group is given by $S_{k} \wr Aut(G)$, having size $k! \vert Aut(G) \vert ^{k}$ and every automorphism is of the form $(\sigma, s_{1},s_{2},...s_{k})$ where $\sigma \in S_{k}$ and $s_{1},s_{2},...s_{k} \in Aut(G)$

%\vspace{5 mm}
%\begin{figure}[!h]
%\centering
%\includegraphics[width=0.5\linewidth]{Fig1.png}
%\caption{Wreath products}
%\end{figure}
\chapter{\'{A}d\'{a}m's Conjecture}
\section{The Conjecture}
\begin{Lemma}\label{L:P*}
For an arbitrary $n$, let $X$ and $X'$ be two connection sets in $\mathbb{Z}'_n$ such that
\begin{equation}\tag{M1}
mX=X'
\end{equation}
for some integer $m$ prime to $n$, where $mX:=\lbrace mv | v\in X \rbrace$. Then the $n$-circulants $\Gamma (\mathbb{Z}_n,X)$ and $\Gamma (\mathbb{Z}_n,X')$ are isomorphic.
\end{Lemma}

\begin{proof}\cite{KLP96}
The mapping $\alpha_{m}:v\mapsto mv$, $\forall v \in \mathbb{Z}_n$, is a bijection from $\mathbb{Z}_n$ onto itself since $m\in\mathbb{Z}^\ast_n$ is invertible. Let $(u,v)$ be an edge of $\Gamma(\mathbb{Z}_n,X)$, then by definition, $v-u \in X$.
Now $(u,v)^{\alpha_{m}}=(mu,mv)$ is an edge of $\Gamma(\mathbb{Z}_n,X')$ since $mv-mu=m(v-u) \in mX=X'$. Hence $\Gamma(\mathbb{Z}_n,mX)\cong \Gamma (\mathbb{Z}_n,X)$, that is $\Gamma(\mathbb{Z}_n,X')\cong \Gamma (\mathbb{Z}_n,X)$.
\end{proof}
The condition $(M_{1})$ simply represents the induced action of the multiplicative group $\mathbb{Z}^\ast_n$ on the connection sets, which are subsets of $\mathbb{Z}'_{n}$. Two connection sets which satisfy this condition, are said to be \emph{equivalent} or \emph{one-multiplier equivalent} with respect to the multiplier $m$. %Moreover, their corresponding circulants are called Cayley isomorphic (CI) \index{Cayley Isomorphic}. A circulant graph $\Gamma$ is said to meet the CI property, or is a CI circulant, if it is Cayley isomorphic with every circulant isomorphic to it. This means that the connection sets of any circulant $\Gamma$ and any circulant $\Gamma'$ isomorphic to it, are one-multiplier equivalent.
The condition given by $(M_{1})$ is sometimes referred to as \textit{\'{A}d\'{a}m's condition}, due to the fact that A.\'{A}d\'{a}m conjectured the opposite assertion, that is:
\begin{quote}
\textbf{A}(n): The connection sets of isomorphic circulant graphs on $n$ vertices are one-multiplier equivalent.
\end{quote}
In other words, \'{A}d\'{a}m conjectured that given two connection sets $X$ and $X'$, of isomorphic circulant n-graphs, then there exists some integer $m$, prime to $n$, such that $mX=X'$.

\subsection{Validity of the Conjecture}
The conjecture \textbf{A}(n) has been shown to be true for the case when $n=p$, where $p$ is prime.

\begin{Theorem}[\cite{ET70}, \cite{FIK90}]
Let $p$ be a prime number. Two cyclic graphs $\Gamma(\mathbb{Z}_p,X)$ and $\Gamma(\mathbb{Z}_p,X')$ are isomorphic if and only if $X'=mX$ for some $m\in \mathbb{Z}^\ast_p$.
\end{Theorem}
This was proved by Turner using spectral techniques in \cite{ET70}. In \cite{FIK90}, the authors make use of a different technique, in particular, one involving Schur rings, in order to prove this result. This technique was also used later by M. Muzychuk in 1995, in order to prove that the conjecture is also valid when $n$ is square-free \cite{Muzy95} (except for $n=4$ and for undirected circulants having $n=9$, for which the conjecture is still valid \cite{KLP96}).

\begin{Theorem}[\cite{Muzy95}]
\label{thm:theorem4a}
If n is a square-free number, then two circulant graphs are isomorphic if and only if they are one multiplier equivalent.
\end{Theorem}

In \cite{KLP96} the authors observe that Theorem~\ref{thm:theorem4a} means that for both directed and undirected graphs, the conjecture \textbf{A}(n) is valid for $n=\varepsilon p_{1}p_{2}...p_{k}$, where $p_{1},p_{2},...p_{k}$ are pairwise distinct odd prime numbers and $\varepsilon \in \lbrace 1,2,4\rbrace$. They also state that the conjecture is false for all other numbers greater than 18. P\'{a}lfy also proved that \'{A}d\'{a}m's conjecture holds when $(n,\phi(n))=1$, by considering the isomorphism problem of Cayley structures over the cyclic group $C_n$ (see \cite{Muzy95}).

\subsection{Falsity of the Conjecture for $n$ non-prime}
Over the years it has been shown that in general this conjecture is false. In 1970, B.Elspas and J. Turner \cite{ET70} constructed a counterexample on 8 vertices.

Consider the connection sets $S=\lbrace 1,2,5 \rbrace$ and $T=\lbrace 1,5,6\rbrace$. The integers relatively prime to 8 are 1,3,5 and 7 (mod 8). One may note however, that $mS\neq T$ for $m=1,3,5$ or 7. Therefore the connection sets are not equivalent. However the graphs $G_1$ and $G_2$ with these connection sets are isomorphic \cite{ET70}.  A vertex $v_i \in G_1$ may be mapped onto a vertex $v_j \in G_2$ where

\[j=4\left[\frac{i+1}{2}\right]+i, \hspace{6mm} i=0,1,...,7\]
which gives an isomorphism.
\vspace{5mm}
\begin{figure}[!h]
\centering
\includegraphics[width=1.0\linewidth]{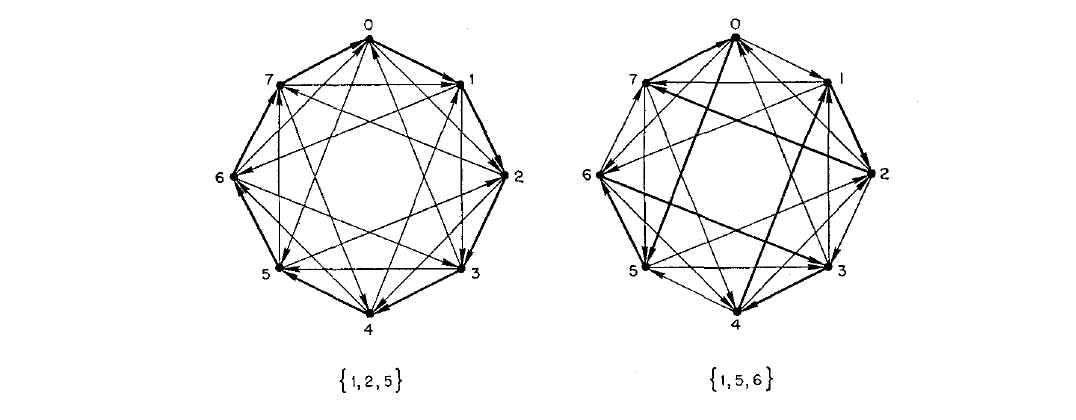}
\caption{Isomorphic graphs with inequivalent connection sets}
\label{isomorphicgraphs}
\end{figure}

In 1979, Alspach and Parsons \cite{AP79} showed that in general the conjecture fails if $n$ is divisible by 8 or by a square of an odd prime. V.N. Egorov and A.I. Markov also showed this independently (see \cite{MKP2001}).

Let us now consider some examples when $n$ is divisible by a square of an odd prime, in particular $n=9$ and $n=25$, in order to show that \'{A}d\'{a}m's conjecture fails in this case. We can see that for the directed circulants $\Gamma=\Gamma(\mathbb{Z}_{9},X)$ and $\Gamma'=\Gamma(\mathbb{Z}_{9},X')$   with $X=\lbrace 1,3,4,7\rbrace$ and $X'=\lbrace 1,6,4,7\rbrace$, no multiplier transforms $X$ into $X'$. However, under the isomorphism (0)(1)(2)(3,6)(4,7)(5,8), these two circulants are isomorphic. Similarly, Adam's conjecture is not valid for the directed circulant graphs $\Gamma(\mathbb{Z}_{25},Y)$ and $\Gamma(\mathbb{Z}_{25},Y')$ where the connection sets, $Y$ and $Y'$ are given by $\lbrace 1,4,6,9,11,14,16,19,21,24,5,20 \rbrace$ and $\lbrace 1,4,6,9,11,14,16,19,21,24,10,15 \rbrace$ respectively.

As a result of Lemma 1, we have the following important enumeration result in the case when $A(n)$ holds.

\begin{Theorem}
If the conjecture \textbf{A}(n) is valid for a given order n and a given set of n-circulants, then the number of non-isomorphic circulant digraphs under consideration is equal to the number of orbits of the group $\mathbb{Z}^\ast_n$ in its induced multiplicative action on the subsets of $\mathbb{Z}_n$ not containing 0.
\end{Theorem}

This means that one can apply the P\'olya enumeration theorem and substitute the appropriate variables in order to count circulants. The easiest case when this can be done, occurs when the order of the circulant is prime. This was done by Turner \cite{Turner67}, and we review these results next.

\section{Enumeration of Circulants of Prime Order}
\subsection{The Directed Case}
As previously stated, in the case when $n$ is prime, $n=p$, we simply need to consider the regular action of ${Z}^\ast_p$ on itself since $\mathbb{Z}^\ast_p\cong\mathbb{Z}'_p$. In the case of directed circulant graphs, we do not have any restrictions on the connection sets, since all subsets of ${Z}'_p$ can be considered as connection sets. Therefore the cycle index may be determined using $\mathcal{I}_{p-1}(x)$. In order to obtain the generating function in terms of $t$, for such graphs, one must substitute $x_{r}:=1+t^{r}$, in $\mathcal{I}_{p-1}(x)$, where $r=1,2,...$, denotes the size of the connection set, that is, the out-degree of the vertices \cite{LP96}. The total number of nonequivalent sets is then obtained by substituting $t:=1$. In other words, one may simply substitute $x_{r} :=2$ for all $r$ in the cycle index, in order to obtain the number of (directed) circulants of prime order.

\vspace{3mm}
For example consider the case $n=13$. We determine the cycle index $I_{\mathbb{Z}^\ast_{13}}(x)$ of the group $\mathbb{Z}^\ast_{13}$ in its regular action and then substitute accordingly.
This means we have the action $(\mathbb{Z} ^\ast_{13},\mathbb{Z}^\ast_{13})$ where $\mathbb{Z}^\ast_{13}=\lbrace 1,2,3,4,5,6,7,8,9,10,11,12 \rbrace$. Table \ref{table:Table1} demonstrates the multiplicative action of each element and the corresponding cycle structure.

\begin{table}[ht]
\caption{The Cycle Structure for the Directed Case when n=13}
\begin{quote}
\centering
\begin{tabular}{|c|c|c|}
\hline
Action of&Action&Cycle Structure\\
\hline
1&(1)(2)(3)(4)(5)(6)(7)(8)(8)(10)(11)(12)&$x_{1}^{12}$\\
2&(1,2,4,8,3,6,12,11,9,5,10,7)&$x_{12}^{1}$\\
3&(1,3,9)(2,6,5)(4,12,10)(7,8,11)&$x_{3}^{4}$\\
4&(1,4,3,12,9,10)(2,8,6,11,5,7)&$x_{6}^{2}$\\
5&(1,5,12,8)(2,10,11,3)(4,7,9,6)&$x_{4}^{3}$\\
6&(1,6,10,8,9,2,12,7,3,5,4,11)&$x_{12}^{1}$\\
7&(1,7,10,5,9,11,12,6,3,8,4,2)&$x_{12}^{1}$\\
8&(1,8,12,5)(2,3,11,10)(4,6,9,7)&$x_{4}^{3}$\\
9&(1,9,3)(2,5,6)(4,10,12)(7,11,8)&$x_{3}^{4}$\\
10&(1,10,9,12,3,4)(2,7,5,11,6,8)&$x_{6}^{2}$\\
11&(1,11,4,5,3,7,12,2,9,8,10,6)&$x_{12}^{1}$\\
12&(1,12)(2,11)(3,10)(4,9)(5,8)(6,7)&$x_{2}^{6}$\\
\hline
\end{tabular}
\label{table:Table1}
\end{quote}
\end{table}

Therefore the cycle index is given by:
\[I_{\mathbb{Z}^\ast_{13}}(x)=\frac{1}{12}(x_{1}^{12}+x_{2}^{6}+2x_{3}^{4}+2x_{4}^{3}+2x_{6}^{2}+4x_{12})\]

Substituting $x_r:=1+t^r$ gives the following generating function:
\begin{equation*}
\begin{split}
\frac{1}{12}(12t^{12}&+12t^{11}+72t^{10}+228t^{9}+516t^{8}+792t^7+960t^6\\
&+792t^5+516t^4+228t^3+72t^2+12t+12)\\
\end{split}
\end{equation*}
\begin{equation*}
=t^{12}+t^{11}+6t^{10}+19t^{9}+43t^{8}+66t^7+80t^6+66t^5+43t^4+19t^3+6t^2+t+1
\end{equation*}

Letting $t=1$ in the generating function, or equivalently $x_i=2$ in the cycle index, we obtain 352 non-isomorphic directed circulants on 13 vertices. Not only can we enumerate the circulants, but the generating function also provides information on the number of edges/arcs each circulant has. The coefficient of $t^k$ gives the number of circulants of out-degree $k$ for directed circulants. For instance the term $6t^{10}$ indicates that 6 of the 352 non-isomorphic directed circulants have out-degree 10.

\subsection{The Undirected Case and Turner's Trick}

Since undirected circulants only contain edges, we must now treat the set receiving the action in a different manner to that described in the directed case, since we now need to exclude connection sets which are not closed under inversion. The set on which the action is taking place will now be made up of inverse pairs of elements, that is, pairs $\lbrace x, -x\rbrace$ for all $x\in \mathbb{Z}'_p$.
In order to demonstrate this let us consider the case $n=13$ once again. In this case our action is now $(\mathbb{Z}^\ast_{13},\lbrace \lbrace 1,12\rbrace,\lbrace 2,11\rbrace,\lbrace 3,10\rbrace,\lbrace 4,9\rbrace,\lbrace 5,8\rbrace,\lbrace 6,7\rbrace \rbrace )$.

In order to simplify this, each pair may be considered as a ``block", that is, if we let $a=\lbrace 1,12\rbrace$, $b=\lbrace 2,11\rbrace$, $c=\lbrace 3,10\rbrace$, $d=\lbrace 4,9\rbrace$, $e=\lbrace 5,8\rbrace$ and $f=\lbrace 6,7\rbrace$ we have the action $(\mathbb{Z}^\ast_{13},\lbrace a,b,c,d,e,f\rbrace)$. The cycle structure may be determined as shown in table \ref{table:Table2}.
\begin{table}[ht]
\caption{The Cycle Structure for the Undirected Case when $n=13$}
\begin{quote}
\centering
\begin{tabular}{|c|c|c|}
\hline
Action of&Action&Cycle Structure\\
\hline
1&$(a)(b)(c)(d)(e)(f)$&$x_{1}^{6}$\\
2&$(a\hspace{2mm}b\hspace{2mm}d\hspace{2mm}e\hspace{2mm}c\hspace{2mm}f)$&$x_{6}^{1}$\\
3&$(a\hspace{2mm}c\hspace{2mm}d)(b\hspace{2mm}f\hspace{2mm}e)$&$x_{3}^{2}$\\
4&$(a\hspace{2mm}d\hspace{2mm}c)(b\hspace{2mm}e\hspace{2mm}f)$&$x_{3}^{2}$\\
5&$(a\hspace{2mm}e)(b\hspace{2mm}c)(d\hspace{2mm}f)$&$x_{2}^{3}$\\
6&$(a\hspace{2mm}f\hspace{2mm}c\hspace{2mm}e\hspace{2mm}d\hspace{2mm}b)$&$x_{6}^{1}$\\
7&$(a\hspace{2mm}f\hspace{2mm}c\hspace{2mm}e\hspace{2mm}d\hspace{2mm}b)$&$x_{6}^{1}$\\
8&$(a\hspace{2mm}e\hspace{2mm})(b\hspace{2mm}c)(d\hspace{2mm}f)$&$x_{2}^{3}$\\
9&$(a\hspace{2mm}d\hspace{2mm}c)(b\hspace{2mm}e\hspace{2mm}f)$&$x_{3}^{2}$\\
10&$(a\hspace{2mm}c\hspace{2mm}d)(b\hspace{2mm}f\hspace{2mm}e)$&$x_{3}^{2}$\\
11&$(a\hspace{2mm}b\hspace{2mm}d\hspace{2mm}e\hspace{2mm}c\hspace{2mm}f)$&$x_{6}^{1}$\\
12&$(a)(b)(c)(d)(e)(f)$&$x_{1}^{6}$\\
\hline
\end{tabular}
\label{table:Table2}
\end{quote}
\end{table}

Therefore the cycle index is given by:
\begin{equation*}
\begin{split}
I_{\mathbb{Z}^\ast_{13}}(x)&=\frac{1}{12}(2x_{1}^{6}+2x_{2}^{3}+4x_{3}^{2}+4x_{6})\\
&=\frac{1}{6}(x_{1}^{6}+x_{2}^{3}+2x_{3}^{2}+2x_{6})
\end{split}
\end{equation*}
and substituting $x_{r} :=2$ for all $r$ we obtain 14 non-isomorphic undirected circulants. In order to obtain the generating function in this case, the substitution $x_{r}:=1+t^{2r}$ for all $r$ is required. This will give the generating function
\begin{equation*}
t^{12}+t^{10}+3t^8+4t^6+3t^4+t^2+1.
\end{equation*}
In the case of undirected circulants, the coefficient of $t^k$ in the generating function, gives the number of circulants of degree $k$. Therefore, for example, the number of non-isomorphic undirected circulants on 13 vertices, with degree 8 is 3.

We had to substitute $x_r:=1+t^{2r}$ in this case because now, each block gave an in-degree and an out-degree, therefore a connection set made up of $r$ blocks gives an undirected circulant of degree $2r$.

In his paper Turner \cite{Turner67}, explains this in the following way:
Since the connecting set $S$ is paired by inversion, we can take this to be $\lbrace s\in S:s\leq \frac{p-1}{2}\rbrace$. As a result we are acting on half of the possible edges, which for the case $n=13$ is $\lbrace 1,2,3,4,5,6\rbrace$. This turns out to be $\mathbb{Z}'_{\frac{p-1}{2}}=\mathbb{Z}'_6$. The group effecting the action can also be considered to be $\frac{\mathbb{Z}'_p}{\lbrace 1,-1\rbrace}$ because the multiplicative action of $m$ is equivalent to that of $-m$. This turns out to be the cyclic group of order 6 and the action is therefore the regular action of this group. The cycle index polynomial in this case may therefore be given by $\mathcal{I}_{\frac{p-1}{2}}(x)$ that is

\[\mathcal{I}_{\frac{p-1}{2}}(x)=\frac{2}{p-1}\sum_{r|\frac{p-1}{2}} \phi(r) x_{r}^{\frac {p-1}{2r}}\]
which therefore gives a general formula for the cycle index for undirected circulants of order $n=p$, $p$ prime.

In the case $n=13$, $r=1,2,3,6$, and $\phi(1)=1$, $\phi(2)=1$, $\phi(3)=2$ (since g.c.d.(1,3)=1, g.c.d.(2,3)=1 and g.c.d.(3,3)=2) and $\phi(6)=2$. By substituting these values in $\mathcal{I}_{\frac{p-1}{2}}(x)$, one can verify that the result is the same as that obtained previously.

Since \'{A}d\'{a}m's conjecture is false in general, a more profound isomorphism criterion is required to enumerate circulants whose order $n$ is non-square-free $n$.

\chapter{Multiplier Theorem for $\MakeLowercase {p^{k}}$}
In order to count non-isomorphic structures using P\'{o}lya enumeration techniques, one needs a criterion for isomorphism. As we have seen, one multiplier equivalence of connection sets is not a valid criterion for isomorphism of circulant graphs in general. Finding an isomorphism criterion for general non-square-free circulants seems to be a hopelessly difficult task, but, using the theory of Schur-rings, considerable progress has been registered for circulants of order $p^k$, $p$ prime. In this chapter we shall look at the special case of $p^2$.
\section{The Case $n=p^2$}
The isomorphism criterion which is to follow, will require us to partition the elements of the connecting sets into \emph{layers}\index{layers}. This is done in the following way: We will first consider the set $\mathbb{Z}'_{p^2}$ and divide the elements into two layers, namely $Y_0$ and $Y_1$, where $Y_0$ will contain those elements which do not have $p$ as a factor and $Y_1$ will contain those elements which do have $p$ as a factor. A connecting set $X$ is then given by

\[X=X_{(0)}\dot{\cup} X_{(1)}\]
where $X_{(0)}=X \cap Y_0$ and $X_{(1)}=X \cap Y_1$. In other words, the layers $X_{(0)}$ and $X_{(1)}$ of the connecting set, are the intersections of the connecting set $X$ with the respective layers $Y_0$ and $Y_1$ in $\mathbb{Z}'_{p^2}$. The layer $X_{(0)}$ is a subset of $\mathbb{Z}^\ast_{p^2}$ while the layer $X_{(1)}$ is a subset of $p\mathbb{Z}^\ast_p$. In addition, when these layers are acted upon (multiplicatively) by elements of $\mathbb{Z}^\ast_{n}$, where in this case $n=p^2$, these layers are invariant.

Let us consider the case when $p=5$. In $\mathbb{Z}'_{25}$ we have the following two layers:
\begin{equation*}
\begin{split}
Y_0&=\lbrace 1,2,3,4,6,7,8,9,11,12,13,14,16,17,18,19,21,22,23,24 \rbrace\\
Y_1&=\lbrace 5,10,15,20 \rbrace\\
\end{split}
\end{equation*}
Now suppose our connecting set is $X=\lbrace 1,3,5,10,24 \rbrace$. Then in this case we have
\begin{equation*}
\begin{split}
X_{(0)}&=\lbrace 1,3,5,10,24 \rbrace\cap Y_0=\lbrace 1,3,24\rbrace\\
X_{(1)}&=\lbrace 1,3,5,10,24 \rbrace\cap Y_1=\lbrace 5,10\rbrace.\\
\end{split}
\end{equation*}
One may note that $X_{(0)}\subseteq\mathbb{Z}^\ast_{25}$ and $X_{(1)}\subseteq5\mathbb{Z}^\ast_5$.

Such partitioning of the elements may be generalised to $n=p^{k}$ for any $k$, for which $k$ layers arise, in which case

\[X=X_{(0)}\dot{\cup}X_{(1)}\dot{\cup}...\dot{\cup}X_{(k-1)}\]
where the $i$-layer is $X_{(i)}=X\cap p^i\mathbb{Z}^\ast_{p^{k-i}}$ \cite{LP96}.

Klin, Liskovets and P\"{o}schel\cite{KLP96} presented the following isomorphism criterion for circulants of order $p^2$.
\begin{Theorem}[\cite{KLP96}]
\label{thm:theorem2}
Two circulant graphs $\Gamma(\mathbb{Z}_n,X)$ and $\Gamma'=\Gamma(\mathbb{Z}_n,X')$ with $n=p^2$ vertices, are isomorphic if and only if their respective layers are multiplicatively equivalent, i.e.

\begin{equation}\tag{$M_{2}$}
X'_{(0)}=m_{0}X_{(0)},   X'_{(1)}=m_{1}X_{(1)},
\end{equation}
for a pair of multipliers $m_{0}, m_{1} \in \mathbb{Z}^\ast_{p^{2}}$. Moreover, in the above, one must have
\begin{equation}\tag{E}
m_{0}=m_{1}
\end{equation}
whenever
\begin{equation}\tag{R}
(1+p)X_{(0)}\not= X_{(0)}
\end{equation}
\end{Theorem}

In order to illustrate this Theorem consider two directed circulants of order 9 with the connection sets $X=\lbrace 1,3,4,7 \rbrace$ and $X'=\lbrace 1,6,4,7 \rbrace$. We have already seen this example in the previous chapter and we saw that no multiplier transforms $X$ into $X'$. In this case we have, $X_{(0)}=\lbrace 1,4,7 \rbrace =X'_{(0)}, X_{(1)}=\lbrace 3 \rbrace$ and $X'_{(1)}=\lbrace 6 \rbrace$. In addition, condition (R) does not hold since $(1+3)X_{(0)}=X_{(0)}$. Therefore by Theorem~\ref{thm:theorem2}, if two non-equal multipliers acting on the layers of the connection sets transform one into the other, then the two circulants are isomorphic. In this case, $m_{0}=1$ and $m_{1}=2$ in ($M_{2}$), may be used to give the required equivalence of the layers, so that $\Gamma (\mathbb{Z}_9,X')\cong \Gamma(\mathbb{Z}_9,X)$.

Now if we consider $X=\lbrace 1,3\rbrace$ and $X'=\lbrace 1,6\rbrace$ for the case $n=9$, again the multipliers $m_{0}=1$ and $m_{1}=2$ give layers which are multiplicatively equivalent. However, condition (R) is satisfied, since $(1+3)X_{(0)}\not= X_{(0)}$. Therefore we cannot have unequal multipliers in this case. Since no multiplier $m$ gives the equality $mX'=X$, the circulants $\Gamma'=\Gamma (\mathbb{Z}_9,X')$ and $\Gamma=\Gamma(\mathbb{Z}_9,X)$ are not isomorphic.

\begin{figure}[!h]
\centering
\includegraphics[width=0.9\linewidth]{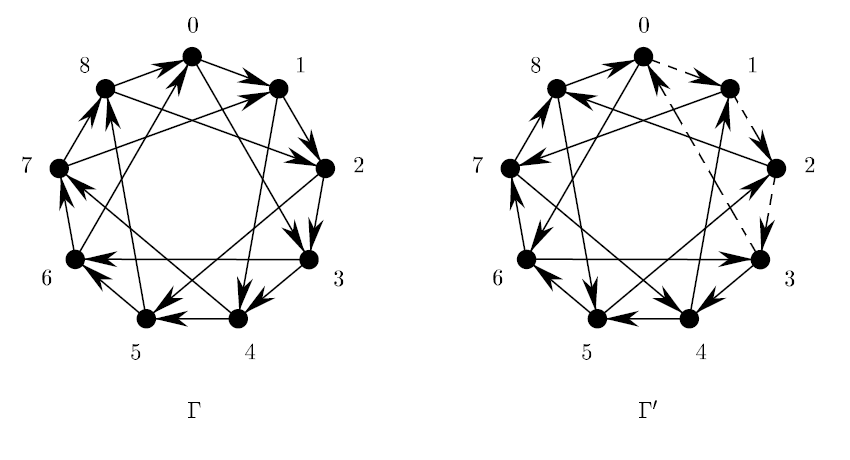}
\caption{Non-isomorphic directed circulants on 9 vertices}
\label{Fig6}
\end{figure}

As seen in Figure~\ref{Fig6}, $\Gamma'$ contains directed 4-cycles such as (0,1,2,3), whereas $\Gamma$ does not.

\section{Application to the Enumeration of Circulants of order $\MakeLowercase{p^2}$}
In this section, we shall show how Theorem~\ref{thm:theorem2} has been used to count circulants of order $p^2$ for some values of the prime $p$. In practice, it is easier to count orbits under invariance conditions
\begin{equation}\tag{$\neg R$}
(1+p)X_{(0)}=X_{(0)},
\end{equation}
that is, when the restrictions have an equality, rather than under the non-invariance condition $(R)$. Therefore, when the problem under consideration includes the non-invariance condition $(R)$, this is changed to the invariance condition $(\neg R)$ and then the result is subtracted from the total amount. The procedure for counting under the non-invariance relation $(R)$, would therefore be as follows:
\begin{itemize}
\item[(i)]Determine the total count, (for the $p^2$ case this will be the number of orbits under the action $(\mathbb{Z}^\ast_{p^2},\mathbb{Z}'_{p^2}))$,
\item[(ii)]Determine the count under the action assuming the invariance relation $(\neg R)$ ,
\item[(iii)]Subtract the result of (ii) from (i).
\end{itemize}

Let us consider the case when $n=9$. In this case we have
\begin{equation*}
\begin{split}
\mathbb{Z}^\ast_{9}&=\lbrace 1,2,4,5,7,8\rbrace \mbox{ and }\\
\mathbb{Z}'_{9}&=\lbrace 1,2,3,4,5,6,7,8\rbrace
\end{split}
\end{equation*}
that is, the connection set $X$ is a subset of $\mathbb{Z}'_{9}$  and the multipliers $m_0$ and $m_1$ come from $\mathbb{Z}^\ast_9$.
Let $Y_0$ and $Y_1$ be the two layers of $\mathbb{Z}'_{9}$, where $Y_0$ contains all elements of $\mathbb{Z}'_9$ which do not have a factor of 3 and $Y_1$ contains those elements which do have a factor of 3. Therefore we have $\mathbb{Z}'_9=Y_0 \dot{\cup}Y_1$ where
\begin{equation*}
\begin{split}
Y_0&=\lbrace 1,2,4,5,7,8\rbrace\\
Y_{1}&=\lbrace3,6\rbrace
\end{split}
\end{equation*}
The connecting sets are then given by the layers
\begin{equation*}
\begin{split}
X_{(0)}&=X\cap Y_0\\
X_{(1)}&=X\cap Y_1
\end{split}
\end{equation*}

Now two circulants may be isomorphic either under one multiplier, that is, $m_0=m_1$, or two distinct multipliers, that is, when $m_0\neq m_1$. From Theorem ~\ref{thm:theorem2}, we have that the non-invariance condition $(R)$, holds only when the multipliers are equal. Therefore in order to count those circulants which are equivalent via two different multipliers, we need to consider the invariance relation $(\neg R)$ given by $4X_{(0)}=X_{(0)}$. One must note however, that this relation may still hold when the multipliers are equal.

When enumerating under this invariance condition, the set $X_{(0)}$ must be taken from whole subsets of $Y_0$ which are invariant under $4Y_0=Y_0$ such that $4X_{(0)}=X_{(0)}$. This will give the partition of $Y_0$ as $Y^\ast_0$. In this case $Y^\ast_0=\lbrace \lbrace 1,4,7 \rbrace, \lbrace 2,8,5 \rbrace \rbrace$. Therefore under the condition $4X_{(0)}=X_{(0)}$, the set $X_{(0)}$ must be a union of these parts. Therefore the multiplicative action is on the sets $\lbrace 1,4,7 \rbrace$ and $\lbrace 2,8,5\rbrace$, not on the single elements of $\mathbb{Z}'_9$. In other words, $X_{(0)}$ must either contain all of the set $\lbrace 1,4,7 \rbrace$, or none of it and similarly all of the set $\lbrace 2,8,5 \rbrace$ or none of it. One may note that these sets are invariant under multiplication by elements of $\mathbb{Z}^\ast_9$, and therefore make the action well-defined.

In order to count the number of non-isomorphic directed circulants on 9 vertices, we will divide the counting problem into two subproblems $A_1(9)$ and $A_2(9)$, in which
\begin{equation*}
\begin{split}
A_1(9):&\mbox{ Counts all those circulants which are distinct under the non-invariance }\\
&\mbox{condition } (R)\mbox{ and } m_0=m_1\\
A_2(9):&\mbox{ Counts all those circulants which are distinct under the invariance }\\
&\mbox{condition } (\neg R)\mbox{ with no restriction on the multipliers. }
\end{split}
\end{equation*}
Let us consider $A_1(9)$. Since in this case we have the non-invariance condition $(R)$, we shall make use of the counting procedure described previously. Let
\begin{equation*}
\begin{split}
A_{11}& \mbox{ count all the circulants which are distinct under } m_0=m_1, \mbox{ that is, when }\\
&\mbox{ \'{A}d\'{a}m's condition holds and }\\
A_{12}& \mbox{ count the number of circulants which are distinct under } (\neg R) \mbox{ and having }\\
&m_0=m_1\\
\end{split}
\end{equation*}
Then
\[A_1(9)=A_{11}-A_{12}.\]

Let us determine $A_{11}$. Here we need to count all circulants assuming \'{A}d\'{a}m's conjecture holds, that is the number of one-multiplier equivalent directed circulants. This is the number of orbits under the action $(\mathbb{Z}^\ast_{9},\mathbb{Z'}_{9})$.

\begin{quote}
\centering
\begin{tabular}{|c|c|c|}
\hline
Action of&Action&Cycle Structure\\
\hline
1&(1)(2)(3)(4)(5)(6)(7)(8)&$x_{1}^{8}$\\
2&(1,2,4,8,7,5)(3,6)&$x_{2}x_{6}$\\
4&(1,4,7)(2,8,5)(3)(6)&$x_{1}^{2}x_{3}^{2}$\\
5&(1,5,7,8,4,2)(3,6)&$x_{2}x_{6}$\\
7&(1,7,4)(2,5,8)(3)(6)&$x_{1}^{2}x_{3}^{2}$\\
8&(1,8)(2,7)(4,5)(3,6)&$x_{2}^{4}$\\
\hline
\end{tabular}
\end{quote}
Therefore the cycle index is given by:
\[A_{11}=\frac{1}{6}(x_{1}^{8}+2x_{2}x_{6}+2x_{1}^{2}x_{3}^{2}+x_{2}^{4})\]
and substituting $x_{i}=2$ for all $i$ we obtain
\[A_{11}=\frac{1}{6}(2^{8}+2(2)(2)+2(2)^{2}(2)^{2}+(2)^4)=52\]

Let us now determine $A_{12}$. Since in $A_{12}$ the multipliers are equal and we have that $4X_0=X_0$, we need to consider the orbits of the action $(\mathbb{Z}^\ast_9,Y^\ast_0 \cup Y_1)$, that is, $(\mathbb{Z}^\ast_9,\lbrace\lbrace 1,4,7\rbrace,\lbrace 2,8,5 \rbrace, 3,6 \rbrace)$. The sets $\lbrace 1,4,7\rbrace$ and $\lbrace 2,8,5 \rbrace$ are blocks, that is, each block must appear whole as a neighbour or not at all. Therefore each block may be seen as a point. The contents of $X_1$ do not influence whether or not $X_0$ is 4-invariant, therefore 3 and 6 are acted upon separately. Now the action of $\mathbb{Z}^\ast_9$ on $\lbrace\lbrace 1,4,7\rbrace, \lbrace2,8,5\rbrace\rbrace$ as two points is equivalent to the action of $\mathbb{Z}^\ast_9$ on $\lbrace1,2\rbrace \bmod3$. Therefore, this action may be expressed as $(\mathbb{Z}^\ast_{9},\lbrace1',2',1,2\rbrace)$ mod3 where 1', 2' represent $\lbrace1,4,7\rbrace$ and $\lbrace 2,8,5\rbrace$ respectively.

\begin{quote}
\centering
\begin{tabular}{|c|c|c|}
\hline
Action of&Action&Cycle Structure\\
\hline
1&(1)(2)(1')(2')&$x_{1}^{4}$\\
2&(1,2)(1',2')&$x_{2}^{2}$\\
4&(1)(2)(1')(2')&$x_{1}^{4}$\\
5&(1,2)(1',2')&$x_{2}^{2}$\\
7&(1)(2)(1')(2')&$x_{1}^{4}$\\
8&(1,2)(1',2')&$x_{2}^{2}$\\
\hline
\end{tabular}
\end{quote}

Therefore the cycle index is given by:

\[A_{12}=\frac{1}{6}(3x_{1}^{4}+3x_{2}^{2})=\frac{1}{2}(x_{1}^{4}+x_{2}^{2})\]
and substituting $x_{i}=2$ for all $i$ we obtain

\[A_{12}=\frac{1}{2}(2^{4}+2^{2})=10\]
Therefore we have $A_1(9)=A_{11}-A_{12}=52-10=42$.

Let us now consider $A_2(9)$. In this case we have the invariance condition $4X_{(0)}=X_{(0)}$ and no restriction on the multipliers, that is, two circulants may be equivalent under one multiplier or two distinct multipliers.

Here we need the action of $\mathbb{Z}^\ast_{9}\times \mathbb{Z}^\ast_{9}$ on $\lbrace Y_1 \cup Y^\ast_0 \rbrace$. The multiplier on $Y_1$ can be the same or different from that on $Y^\ast_0$. Therefore we have to consider the action of all $(i,j)\in \mathbb{Z}^\ast_{9} \times \mathbb{Z}^\ast_{9}$ on $\lbrace 3,6\rbrace \cup \lbrace F,H \rbrace$ where $F=\lbrace1,4,7\rbrace$ and $H=\lbrace2,8,5\rbrace$. For example $(2,5)$ has the action (3\hspace{2mm}6)(F\hspace{2mm}H), where 2 acts on $\lbrace3,6\rbrace$ and 5 acts on $\lbrace F,H\rbrace$. This gives the monomial $x_{1}^{2}x_{2}$. There are $6^2$ such actions since $|\mathbb{Z}^\ast_{9}|=6$.
We may however, determine $A_{2}(9)$ more simply, by finding the cycle index of $\mathbb{Z}^\ast_{9}$ on $\lbrace 3,6\rbrace$ and $\mathbb{Z}^\ast_{9}$ on $\lbrace F,H\rbrace$ and take the product. Both of these are equivalent to the action of $\mathbb{Z}^\ast_{9}$ on $\lbrace 1,2\rbrace$ mod 3, as such we may simply find the latter and square.

\begin{quote}
\centering
\begin{tabular}{|c|c|c|}
\hline
Action of&Action&Cycle Structure\\
\hline
1&(1)(2)&$x_{1}^{2}$\\
2&(1,2)&$x_{2}$\\
4&(1)(2)&$x_{1}^{2}$\\
5&(1,2)&$x_{2}$\\
7&(1)(2)&$x_{1}^{2}$\\
8&(1,2)&$x_{2}$\\
\hline
\end{tabular}
\end{quote}
Therefore the cycle index is:

\[\frac{1}{6}(3x_{1}^{2}+3x_{2})\]
Squaring and simplifying gives
\[A_2(9)=\frac{1}{4}(x_{1}^{2}+x_{2})^2\]
Substituting $x_{i}=2$ for all $i$ we obtain $A_2(9)=9$

Therefore combining our results we obtain:

\[A_1(9)+A_2(9)=42+9=51.\]
This means that 51 directed(mixed), non-isomorphic circulants on 9 vertices exist.

\vspace{3mm}
Summarizing the above in order to see the inclusion-exclusion more clearly, what we have essentially, is the set $A_{11}$ which counts the number of distinct circulants under the conditions $(R)$ and $(\neg R)$ and $m_0=m_1$. The set $A_2$ counts the set of distinct circulants under the condition $(\neg R)$ with no restriction on the multipliers (that is the multipliers could be the same or different). Now what we require is $|A_{11} \cup A_2|$. We know that $|A_{11} \cup A_2|=|A_{11}|+|A_2|-|A_{11} \cap A_2|$, where $|A_{11} \cap A_2|$ counts all those circulants with $(\neg R)$ and $m_1=m_0$. This is simply the set $A_{12}$ mentioned above.

\vspace{3mm}
The same procedure may be repeated for undirected circulants, however as previously stated we now need a slight modification on the connection set.
\begin{equation*}
\mathbb{Z}^\ast_{9}=\lbrace 1,2,4,5,7,8\rbrace
\end{equation*}
and
\begin{equation*}
\mathbb{Z}'_{9}=\lbrace 1,2,3,4,5,6,7,8\rbrace
\end{equation*}
that is the connection set is a subset $X$ of $\mathbb{Z}'_9$ which must have the property $X=-X$, and the multipliers $m_0$ and $m_1$ come from $\mathbb{Z}^\ast_9$.

Since the elements in the connecting sets are paired by inversion, as we did when we discussed Turner's trick, we partition $\mathbb{Z}'_9$ as
\[\mathbb{Z}'_9=\lbrace \lbrace1,8\rbrace,\lbrace2,7\rbrace,\lbrace3,6\rbrace,\lbrace4,5\rbrace\rbrace\]

One must note that any connection set we shall work with must have either both elements of a given pair or none. The multiplicative action must therefore be taken on these pairs, again as we did when discussing Turner's trick.

In this case we have
\begin{equation*}
\begin{split}
Y_0&=\lbrace\lbrace1,8\rbrace,\lbrace2,7\rbrace,\lbrace4,5\rbrace\rbrace\\
Y_1&=\lbrace{3,6}\rbrace
\end{split}
\end{equation*}
still partitioned into inverse pairs and
\[X_{(0)}=X \cap Y_0 \mbox{ and } X_{(1)}=X \cap Y_1.\]
We note here that $Y^\ast_0=\lbrace 1,2,4,5,7,8\rbrace$, that is the two separate blocks obtained previously in $Y^\ast_0$ are now merged together so that every element and its inverse are in the same set.

In order to determine $A_{11}$, again we require the orbits of the action $(\mathbb{Z}^\ast_{9},\mathbb{Z}'_{9})$ and therefore we will consider
\[(\mathbb{Z}^\ast_{9},\lbrace \lbrace1,8\rbrace,\lbrace2,7\rbrace,\lbrace3,6\rbrace,\lbrace4,5\rbrace\rbrace).\]
Again $\lbrace1,8\rbrace,\lbrace2,7\rbrace,\lbrace3,6\rbrace,\lbrace4,5\rbrace$ are considered as blocks, that is, we may consider $(\mathbb{Z}^\ast_{9},\lbrace K,L,M,N\rbrace)$ where $K=\lbrace1,8\rbrace$, $L=\lbrace2,7\rbrace$, $M=\lbrace3,6\rbrace$, $N=\lbrace4,5\rbrace$. The cycle index corresponding to this action is given by
\[A_{11}=\frac{1}{6}(2x_{1}^{4}+4x_{3}x_{1})\]
and substituting $x_{i}=2$ for all $i$, we obtain $A_{11}=8$.

To determine $A_{12}$ again we require the action $(\mathbb{Z}^\ast_{9},Y^\ast_0\cup Y_1)$. Therefore we have
\[(\mathbb{Z}^\ast_{9},\lbrace \lbrace1,2,4,5,7,8\rbrace,\lbrace3,6\rbrace \rbrace)\]
This action may be seen as $(\mathbb{Z}^\ast_{9},\lbrace M,O \rbrace)$ where $O=\lbrace 1,2,4,5,7,8\rbrace$. The cycle index corresponding to this action is
\[A_{12}=\frac{1}{6}(6)x_{1}^{2}=x_{1}^{2}\]
and substituting $x_{i}=2$ we obtain $A_{12}=4$.
Therefore
\[A_1(9)=A_{11}-A_{12}=4.\]
For $A_2(9)$ we need to multiply the cycle indices of the actions $(\mathbb{Z}^\ast_{9},\lbrace3,6\rbrace)=(\mathbb{Z}^\ast_{9},\lbrace M\rbrace)$ and $(\mathbb{Z}^\ast_{9},\lbrace 1,2,4,5,7,8\rbrace)=(\mathbb{Z}^\ast_{9},\lbrace O\rbrace)$. Both these cycle indices are equivalent to $x_{1}$. Therefore

\[A_2(9)=x_{1}^{2}.\]
Substituting, we obtain $A_2(9)=4$.

Therefore the number of non-isomorphic, undirected circulants on 9 vertices is

\[A_1(9)+A_2(9)=4+4=8.\]

\section{The Case $\MakeLowercase{n=p^3}$}
The main isomorphism theorem for circulants of order $p^k$ is the one we give next. This was proved by Klin and P\"{o}schel in \cite{KP80} and it was used for the enumeration of circulants of order $p^2$ in \cite{KLP96} and \cite{LP2000}. Whereas Theorem~\ref{thm:theorem2} for $p^2$ involved two multipliers and two subcases depending on non-invariance conditions on the layers, the isomorphism theorem for $p^k$ involves $k$ multipliers and $\binom{k}{2}$ cases coming from the non-invariance conditions $R_{ij}$, making it more difficult to apply in practice for enumeration purposes. It is this that makes the enumeration problem difficult, not only that there are multipliers for the separate layers of the connection sets, but that, depending on non-invariance conditions, some multipliers must be equal in certain cases. Moreover, the intersection between the conditions makes this case even more difficult. The two cases for $p^2$ involved three different enumeration problems, as we have seen, and for $p^3$, the three non-invariance relations $R_{00}, R_{01}, R_{10}$ below, will break up into five cases which will eventually give eleven enumeration subproblems, as we shall see in the next chapter.
\subsection{The Main Theorem}
\begin{Theorem}[\cite{KP80},\cite{LP2000}]
\label{thm:theorem3}
Let $n=p^k$ ($p$ an odd prime) and let $\Gamma$ and $\Gamma'$ be two $p^k$-circulants with the connection sets $X$ and $X'$, respectively. Then $\Gamma$ and $\Gamma'$ are isomorphic if and only if their respective layers are multiplicatively equivalent, i.e.
\begin{equation}\tag{$M_{k}$}
X'_{(i)}=m_{i}X_{(i)}, \hspace{5mm} i=0,1,...,k-1,
\end{equation}
for an  arbitrary set of multipliers $m_0, m_1,...m_{k-1}\in \mathbb{Z}^\ast_p$ which satisfy the following constraints: whenever the layer $X_{(i)}$ satisfies the non-invariance condition

\begin{equation}\tag{$R_{ij}$}
(1+p^{k-i-j-1})X_{(i)}\neq X_{(i)}
\end{equation}
for some $i \in \lbrace 0,1,...,k-2\rbrace$ and $j \in \lbrace 0,1,...,k-2-i\rbrace$, the successive multipliers $m_{i},...,m_{k-j-1}$ meet the system of congruences
\begin{equation}\tag{$E_{ij}$}
\begin{split}
m_{i+1}&\equiv m_i(\bmod p^{k-i-j-1}),\\
m_{i+2}&\equiv m_{i+1}(\bmod p^{k-i-j-2}),\\
&\vdots\\
m_{k-j-1}&\equiv m_{k-j-2}(\bmod p).
\end{split}
\end{equation}
\end{Theorem}
For the case when $k=3$, Theorem~\ref{thm:theorem3} translates to the following theorem:

\begin{Theorem}
\label{thm:theorem4}
Let $n=p^3$ ($p$ an odd prime) and let $\Gamma$ and $\Gamma'$ be two $p^3$-circulants with the connection sets $X$ and $X'$, respectively. Then $\Gamma$ and $\Gamma'$ are isomorphic if and only if their respective layers are multiplicatively equivalent, i.e.
\begin{equation}\tag{$M_{3}$}
X'_{(0)}=m_{0}X_{(0)},\hspace{3mm}X'_{(1)}=m_{1}X_{(1)}\hspace{3mm}X'_{(2)}=m_{2}X_{(2)},
\end{equation}
for an  arbitrary set of multipliers $m_0, m_1, m_2 \in \mathbb{Z}^\ast_{p^3}$.  Moreover, in the above, one must have
%\tag{$E_{00}$}

\begin{itemize}
\item[(i)] \hfill \makebox[0pt][r]{%
            \begin{minipage}[b]{\textwidth}
              \begin{equation}
                 m_1 \equiv m_0(\bmod p^2) \mbox{ and } m_2 \equiv m_1(\bmod p)\tag{$E_{00}$})
              \end{equation}
          \end{minipage}}
          \end{itemize}
whenever
\begin{equation}\tag{$R_{00}$}
(1+p^2)X_{(0)}\neq X_{(0)},
\end{equation}

\begin{itemize}
\item[(ii)] \hfill \makebox[0pt][r]{%
            \begin{minipage}[b]{\textwidth}
              \begin{equation}
                 m_1 \equiv m_0(\bmod p)\tag{$E_{01}$})
              \end{equation}
          \end{minipage}}
          \end{itemize}
whenever
\begin{equation}\tag{$R_{01}$}
(1+p)X_{(0)}\neq X_{(0)},
\end{equation}

\begin{itemize}
\item[(iii)] \hfill \makebox[0pt][r]{%
            \begin{minipage}[b]{\textwidth}
              \begin{equation}
                 m_2 \equiv m_1(\bmod p)\tag{$E_{10}$})
              \end{equation}
          \end{minipage}}
          \end{itemize}
whenever
\begin{equation}\tag{$R_{10}$}
(1+p)X_{(1)}\neq X_{(1)}.
\end{equation}
\end{Theorem}

\subsection{Representation of the Main Theorem}
It is possible to use Theorem~\ref{thm:theorem4} directly to count circulants of order $p^3$ and we shall show how to do this later in the next chapter. This involves considering the different subcases and their intersection. However, Liskovets and P\"{o}schel in \cite{LP2000} manage for $p^3$ and $p^4$, to partition the conditions of the theorem into five parts which make their use in enumeration much easier. Liskovets and P\"{o}schel take into consideration all combinations of non-invariance conditions $(R_{ij})$, together with the remaining invariance conditions

\begin{equation}\tag{$\neg R_{ij}$}
(1+p^{k-i-j-1})X_{(i)}=X_{(i)}
\end{equation}
and make use of a number of results, in order to obtain the subproblem list for counting circulants of order $p^k$, $k\leq 4$. For details of how this list has been generated from the main theorem using results from number theory and walks through a rectangular lattice, the reader is referred to \cite{LP2000}. The necessary information required for the enumeration of $p^3$ circulants is listed in Table~\ref{table:Table3}, which we therefore take to be a rewording of Theorem~\ref{thm:theorem4}. This has been obtained from Table 1 in \cite{LP2000}.

\begin{table}[ht]
\begin{adjustwidth}{-1cm}{}
\caption{The conditions for isomorphism of circulants of order $p^3$}
\begin{quote}
\centering
\begin{tabular}{|c|c|c|c|}
\hline
Subproblem&Non-Invariance&Invariance&Condition on\\
&Condition&Condition&Multipliers\\
\hline
$A_1$&$\emptyset$&$\neg R_{01}$, $\neg R_{10}$&no restriction\\
$A_2$&$R_{00}$&$\emptyset$&$m_2=m_1=m_0$\\
$A_3$&$R_{01}$&$\neg R_{00}$, $\neg R_{10}$&$m_1=m_0$\\
$A_4$&$R_{10}$&$\neg R_{01}$&$m_2=m_1$\\
$A_5$&$R_{01},R_{10}$&$\neg R_{00}$&$m_2=m_1$ and $m_1 \equiv m_0(\bmod p)$\\
\hline
\end{tabular}
\label{table:Table3}
\end{quote}
\end{adjustwidth}
\end{table}

The five subcases $A_1$ to $A_5$ shown in Table~\ref{table:Table3}, give conditions on the three multipliers $m_0$, $m_1$ and $m_2$ for two circulants of order $p^3$ to be isomorphic. The three multipliers must satisfy at least one of the conditions. For example $A_3$ means that if the non-invariance condition $R_{01}$ holds, together with the invariance conditions $\neg R_{00}$ and $\neg R_{10}$, then $m_1=m_0$ but $m_2$ can be independent.

In addition to this information, we also require the following observations\cite{LP2000}
\begin{equation*}
(R_{ij})\Rightarrow(R_{ij'}) \mbox{ whenever } j'\geq j
\end{equation*}
Therefore
\begin{equation*}
\neg(R_{ij'})\Rightarrow \neg(R_{ij})
\end{equation*}
As a result we have that
\begin{equation}
\neg(R_{01})\Rightarrow \neg(R_{00})
\end{equation}
In addition,
\begin{equation*}
(E_{ij})\Rightarrow(E_{i'j}) \mbox{ whenever } i'\geq i
\end{equation*}
Therefore
\begin{equation*}
\neg(E_{i'j})\Rightarrow \neg(E_{ij})
\end{equation*}
that is
\begin{equation}
\neg(E_{10})\Rightarrow \neg(E_{00})
\end{equation}

In the next chapter we shall use Theorem~\ref{thm:theorem4} as represented in Table~\ref{table:Table3} and also as stated in the theorem itself in order to enumerate circulants of order $p^3$ for $p=3$ and $p=5$.

\chapter{Enumeration of Circulants of Order $\MakeLowercase{p^3}$}\label{chp:Chp4}
\section{Using Table ~\ref{table:Table3}}
In this chapter, we will first count $p^3$ circulants for $p=3$ and $p=5$ by making use of Table~\ref{table:Table3}, that is, by dividing the enumeration problem into the five subproblems corresponding to the conditions for isomorphism $A_1,A_2,...A_5$. In order to determine $|A_1\cup A_2\cup A_3\cup A_4\cup A_5|$ we then have to consider the following:
\begin{equation*}
\begin{split}
|A_1|+&|A_2|+|A_3|-|\mbox{pairwise intersections}|\\
&+|\mbox{triplewise intersections}|-|\mbox{4-wise intersections}|+|\mbox{all 5 intersections}|
\end{split}
\end{equation*}

As stated in the previous chapter, it is easier to count orbits under invariance conditions, in this case $\neg (R_{ij})$, rather than under the non-invariance conditions $(R_{ij})$. Therefore, when the subproblem in question includes one or more non-invariance conditions, these are changed to invariance conditions and then the result is subtracted from the total amount. The procedure for counting would therefore be as described previously, with the difference that here we do not have just one non-invariance condition and one invariance condition. In order to count under a given non-invariance relation $R_{ij}$, we first
\begin{itemize}
\item[(i)]Determine the count under the action assuming the invariance relation $\neg (R_{ij})$,
\item[(ii)] Determine the count under the action without any (non)-invariance relations,
\item[(iii)]Subtract the result of (i) from (ii).
\end{itemize}
This procedure is often complicated by having both invariance and non-invariance conditions. For example to count the number of non-isomorphic circulants in case $A_3$ of Table~\ref{table:Table3}, we
\begin{itemize}
\item[(i)]First count under the conditions $m_1=m_0$, $\neg (R_{00}), \neg (R_{10})$.
\item[(ii)]Then count under the conditions $m_1=m_0$, $\neg (R_{01}),\neg (R_{00}),\neg (R_{10})$.
\item[(iii)]Subtract the result of (ii) from (i).
\end{itemize}

Having counted the number of non-isomorphic circulants under each of the five isomorphism conditions $A_1$ to $A_5$, we then need to consider their intersections, as we described above. It however transpires that these intersections are empty. This can be seen by considering the invariance and non-invariance relations. For example, if conditions $A_2$ and $A_3$ are both to hold, then the connection sets must satisfy both $R_{00}$ (from $A_2$) and $\neg (R_{00})$ (from $A_3$), and this is impossible. Therefore there are no cases corresponding to both conditions $A_2$ and $A_3$. Similarly if conditions $A_1$ and $A_2$ are both to hold, then the connection sets must satisfy $\neg (R_{01})$ and $\neg (R_{10})$ (from $A_1$) and $R_{00}$ (from $A_2$). However since $\neg(R_{01})\Rightarrow\neg(R_{00})$, then it is not possible to have both $\neg(R_{01})$ and $R_{00}$. Similarly for the other intersections. When we later do enumeration using the actual statement of Theorem~\ref{thm:theorem4}, we shall have to deal with cases whose intersection is not empty. This is the main reason why it is easier to do enumeration using the formulation of Theorem~\ref{thm:theorem4} as in Table~\ref{table:Table3}.

In order to determine cycle indices and generating functions efficiently, the computer package GAP was used. The reader is referred to Appendix A and Appendix B for the GAP programmes and a full explanation of the commands used.

\subsection{$\MakeLowercase{n=27}$ Directed}
Let us consider the case when $p=3$ and count all non-isomorphic directed circulants on $3^3$ vertices. In this case we have
%\flushleft
\begin{equation*}
\begin{split}
\mathbb{Z}^\ast_{27}&=\lbrace 1,2,4,5,7,8,10,11,13,14,16,17,19,20,22,23,25,26 \rbrace\\
\mathbb{Z}'_{27}&=\lbrace 1,2,3,4,...,26 \rbrace
\end{split}
\end{equation*}
that is, the connection set is a subset $X$ of $\mathbb{Z}'_{27}$ and the multipliers $m_0,m_1,m_2$ come from $\mathbb{Z}^\ast_{27}$.

Let $Y_0$, $Y_1$, $Y_2$, be the three layers of $\mathbb{Z}'_{27}$ where $Y_0$ contains all elements of $\mathbb{Z}'_{27}$ which are relatively prime to 27, $Y_1$ contains those elements which are divisible by 3 and not by 9 and $Y_2$ contains those divisible by 9. Therefore we have $\mathbb{Z}'_{27}=Y_0 \dot{\cup}Y_1 \dot{\cup}Y_2$ where:
\begin{equation*}
\begin{split}
Y_0&=\lbrace1,2,4,5,7,8,10,11,13,14,16,17,19,20,22,23,25,26\rbrace\\
Y_1&=\lbrace3,6,12,15,21,24\rbrace\\
Y_2&=\lbrace 9,18\rbrace\\
\end{split}
\end{equation*}
Now by theorem~\ref{thm:theorem4}, the non-invariance conditions in this case are:
\begin{equation*}
\begin{split}
R_{00}&: \hspace{3mm} 10X_{(0)}\neq X_{(0)}\\
R_{01}&: \hspace{3mm} 4X_{(0)}\neq X_{(0)}\\
R_{10}&: \hspace{3mm} 4X_{(1)}\neq X_{(1)},
\end{split}
\end{equation*}
where $X_{(0)}=X \cap \mathbb{Z}^\ast_{27}$ that is, $X_{(0)}=X \cap Y_0$ and $X_{(1)}=X \cap 3\mathbb{Z}^\ast_{9}$, that is, $X_{(1)}=X \cap Y_1$. Recall that $X_{(2)}=X\cap Y_2$.

As described in the $p^2$ case, when we enumerate under an invariance condition, such as $10X_{(0)}=X_{(0)}$, we must take $X_{(0)}$ from whole subsets of $Y_0$ which are invariant under $10Y_0=Y_0$. These subsets partition $Y_0$, therefore under the condition $10X_{(0)}=X_{(0)}$, the set $X_{(0)}$ must be a union of these parts. Therefore the multiplicative action is taken on these parts or blocks.

We shall denote the partitioned set corresponding to the invariance condition $10Y_{0}=Y_{0}$ by $Y^\ast_0$, that corresponding to $4Y_{0}=Y_{0}$ by $Y^{\ast\ast}_0$, and that corresponding to the invariance condition $4Y_{1}=Y_{1}$ by $Y^\ast_1$.
We have that
\begin{equation*}
\begin{split}
Y^\ast_0&=\lbrace\lbrace 1,10,19\rbrace,\lbrace 2,20,11\rbrace,\lbrace 4,13,22\rbrace,\lbrace 5,23,14\rbrace,\lbrace 7,16,25\rbrace,\lbrace 8,26,17\rbrace\rbrace\\
Y^{\ast\ast}_0&=\lbrace\lbrace 1,4,16,10,13,25,19,22,7\rbrace,\lbrace2,8,5,20,26,23,11,17,14\rbrace\rbrace\\
Y^\ast_1&=\lbrace\lbrace3,12,21\rbrace,\lbrace6,24,15\rbrace\rbrace
\end{split}
\end{equation*}
Therefore, under the non-invariance condition $4X_{(0)}=X_{(0)}$, the set $X_{(0)}$ must either contain all of the set $\lbrace 1,4,16,10,13,25,19,22,7\rbrace$ or none of it, similarly all of the set $\lbrace2,8,5,20,26,23,11,17,14\rbrace$ or none of it. Therefore the multiplicative action is on the sets $\lbrace 1,4,16,10,13,25,19,22,7\rbrace$ and $\lbrace2,8,5,20,26,23,11,17,14\rbrace$ not on the single elements of $\mathbb{Z}'_{27}$. Under the multiplicative action by elements of $\mathbb{Z}^\ast_{27}$, these sets are invariant, as such the action is well defined.

The sets defined above, are identified in GAP as follows:
\begin{spverbatim}
zstar27:=PrimeResidues(27);
zprime27:=[1..26];
y0:=zstar27;
y1:=[3,6,12,15,21,24];
y2:=[9,18];
ystar0:= List([[1,10,19],[8,17,26],[2,11,20],[7,16,25],[4,13,22],
[5,14,23]],x->AsSet(x));
ystarstar0:=List([[1,4,16,10,13,25,19,22,7],[2,8,5,20,26,23,11,17,14]],
x->AsSet(x));
ystar1:=List([[3,12,21],[6,24,15]], x->AsSet(x));
\end{spverbatim}
\vspace{5mm}
Having defined all the necessary sets, we shall now proceed to count the number of non-isomorphic directed circulants for $n=27$ under each of the five isomorphism conditions $A_1$ to $A_5$.

Recall that since
\begin{equation*}
\neg(R_{ij'})\Rightarrow \neg(R_{ij})
\end{equation*}

We have
\begin{equation*}
\neg(R_{01})\Rightarrow \neg(R_{00})
\end{equation*}

Consider the subproblem $A_1$. In this case we have that, when the invariance conditions $\neg R_{01}$ and $\neg R_{10}$ hold, there is no restriction on the multipliers.  Therefore we have blocks arising from
\begin{equation*}
\begin{split}
4X_{(0)}&=X_{(0)} \mbox{ and }\\
4X_{(1)}&=X_{(1)}
\end{split}
\end{equation*}
Since under the invariance condition $4X_{(0)}=X_{(0)}$, we must take $X_{(0)}$ from whole subsets of $Y_0$ which are invariant under $4Y_0=Y_0$. Therefore the set $X_{(0)}$ must be a union of the blocks in $Y^{\ast\ast}_0$. Consequently, instead of $Y_0$ we shall make use of $Y^{\ast\ast}_0$. Similarly, under the invariance condition $4X_{(1)}=X_{(1)}$, we must take $X_{(1)}$ from whole subsets of $Y_1$ which are invariant under $4Y_{1}=Y_{1}$. These subsets partition $Y_{1}$ as $Y^\ast_1$. Therefore under the condition $4X_{(1)}=X_{(1)}$, the set $X_{(1)}$ must be a union of these parts. We shall therefore use $Y^\ast_1$ instead of $Y_1$.

Since we have no restriction on the multipliers here, the cycle index of our action is the product of the cycle indices $I_{(\mathbb{Z}^\ast_{27},Y^{\ast\ast}_0)}$, $I_{(\mathbb{Z}^\ast_{27},Y^{\ast}_1)}$ and $I_{(\mathbb{Z}^\ast_{27},Y_2)}$.

If we consider the cycle index $I_{(\mathbb{Z}^\ast_{27},Y^{\ast\ast}_0)}$ for example, this is coded in GAP as follows:

\begin{spverbatim}
l1:= List(zstar27,c->PermList(List(ystarstar0,x->Position(ystarstar0,
AsSet(\mod(c*x,27))))));
c1:=List(l1,g->CycleIndex(g,[1..Length(ystarstar0)]));
\end{spverbatim}
The generating function may then be obtained by using the following commands:
\begin{spverbatim}
g1:=List(c1, f->Value(f,[x_1,x_2],[1+t^9,1+t^18]));
allones:=[];
for i in [1..Length(g1)] do
allones[i]:=1;
od;
G1:=(1/Length(g1))*(g1*allones);
\end{spverbatim}
\vspace{5mm}
\verb!l1! is a list of permutations of elements of $Y^{\ast\ast}_0$, generated by elements of $\mathbb{Z}^\ast_{27}$. \verb!c1! gives the cycle structures of the elements in \verb!l1! and \verb!g1! gives the generating function for each cycle structure in \verb!c1!. The result of \verb!g1! is still a list. In order to obtain one generating function for this action, we first create a list of ones (identified by \verb!allones! in the code above), of the same length as the list \verb!g1!. We then carry out componentwise multiplication of this list with the list \verb!g1!, and divide by the length of \verb!g1!. This is given by \verb!G1! above. The cycle indices $I_{(\mathbb{Z}^\ast_{27},Y^{\ast}_1)}$ and $I_{(\mathbb{Z}^\ast_{27},Y_2)}$ may be obtained in a similar manner (refer to Appendix B). The three generating functions obtained in these three actions are then multiplied together to obtain
\begin{equation*}
\begin{split}
t^{26}+&t^{25}+t^{24}+t^{23}+t^{22}+t^{21}+t^{20}+t^{19}+t^{18}+t^{17}+t^{16}+\\
&t^{15}+t^{14}+t^{13}+t^{12}+t^{11}+t^{10}+t^9+t^8+t^7+t^6+t^5+t^4+t^3+t^2+t+1
\end{split}
\end{equation*}
Substituting $t=1$, we obtain 27 distinct circulants under $A_1$.

Let us now consider $A_2$. Here we have the condition that when $m_0=m_1=m_2$ then the non-invariance condition $R_{00}$ must hold. Since in this case we need to consider the non-invariance condition $R_{00}$, we will use the counting procedure described at the beginning of this chapter. Therefore this isomorphism condition will be split into the following two problems:
\begin{equation*}
\begin{split}
A_{21}:& \mbox{ The result of our action } (\mathbb{Z}^\ast_{27},\mathbb{Z}'_{27})\\
A_{22}:& \mbox{ The result of an action with } \neg R_{00} \mbox{ that is with } 10X_{(0)}=X_{(0)}.
\end{split}
\end{equation*}
Again since we have the condition $10X_{(0)}=X_{(0)}$ in $A_{22}$, we must take $X_{(0)}$ from whole subsets of $Y^\ast_0$. Since $m_0=m_1=m_2$, the action in $A_{22}$ is therefore $(\mathbb{Z}^\ast_{27},Y^\ast_0 \cup Y_1 \cup Y_2)$. The required result is then given by $A_{21}-A_{22}$.

Once again using GAP, the generating function for $A_{21}-A_{22}$ is:
\begin{equation*}
\begin{split}
&t^{25}+17t^{24}+141t^{23}+828t^{22}+3642t^{21}+12785t^{20}+36514t^{19}+86790t^{18}+\\
&173529t^{17}+295097t^{16}+429153t^{15}+536547t^{14}+577720t^{13}+536547t^{12}+\\
&429153t^{11}+295097t^{10}+173529t^9+86790t^8+36514t^7+12785t^6+3642t^5+\\
&828t^4+141t^3+17t^2+t
\end{split}
\end{equation*}
Substituting $t=1$ we obtain 3727808 distinct, directed circulants.

Let us consider $A_3$. Here we have the conditions $R_{01}$, $\neg R_{00}$, $\neg R_{10}$ for $m_1=m_0$. Therefore in this case our isomorphism problem will again be divided into two problems, namely $A_{31}$ and $A_{32}$, where\\
$A_{31}$: is the result of our action with $\neg R_{00}$ and $\neg R_{10}$, that is, with blocks arising from

\begin{equation*}
\begin{split}
10X_{(0)}&=X_{(0)} \mbox{ and }\\
4X_{(1)}&=X_{(1)}
\end{split}
\end{equation*}
and therefore we shall need to use $Y^{\ast}_0$ instead of $Y_0$ and $Y^\ast_1$ instead of $Y_1$, and\\
$A_{32}$: is the result of our action with $\neg R_{00}$, $\neg R_{10}$ and $\neg R_{01}$, that is, with blocks arising from
\begin{equation*}
\begin{split}
10X_{(0)}&=X_{(0)} \mbox{ and }\\
4X_{(1)}&=X_{(1)} \mbox{ and }\\
4X_{(0)}&=X_{(0)}
\end{split}
\end{equation*}
In this case however, we know that $\neg(R_{01})\Rightarrow \neg(R_{00})$, therefore the first equation is redundant. Therefore for $A_{32}$ we shall use $Y^{\ast\ast}_0$ instead of $Y_0$ and $Y^\ast_1$ instead of $Y_1$.
Since $m_1=m_0$ and $m_2$ is independent, the cycle index of our action here, is the product of the cycle indices $I_{(\mathbb{Z}^\ast_{27},Y_0\cup Y_1)}$ and $I_{(\mathbb{Z}^\ast_{27},Y_2)}$, blocked as required. This means we have:\\
\begin{equation*}
\begin{split}
A_{31}:&I_{(\mathbb{Z}^\ast_{27},Y^\ast_0\cup Y^\ast_1)}\times I_{(\mathbb{Z}^\ast_{27},Y_2)}\\
A_{32}:&I_{(\mathbb{Z}^\ast_{27},Y^{\ast\ast}_0\cup Y^\ast_1)}\times I_{(\mathbb{Z}^\ast_{27},Y_2)}\\
\end{split}
\end{equation*}
The generating function for $A_{31}-A_{32}$ is:
\begin{equation*}
\begin{split}
t^{23}+t^{22}&+t^{21}+5t^{20}+5t^{19}+5t^{18}+9t^{17}+9t^{16}+9t^{15}+12t^{14}+12t^{13}+\\
&12t^{12}+9t^{11}+9t^{10}+9t^9+5t^8+5t^7+5t^6+t^5+t^4+t^3
\end{split}
\end{equation*}
This gives 126 distinct directed circulants under our isomorphism problem $A_3$.

We shall now consider $A_4$. Here we have the conditions $R_{10}$ and $\neg R_{01}$ when $m_2=m_1$. Therefore we will now consider $A_{41}$ and $A_{42}$ as follows:\\
$A_{41}$: The result of an action with $\neg R_{01}$, that is with blocks arising from $4X_{(0)}=X_{(0)}$. Therefore in this action $X_{(0)}$ must be a union of parts in $Y^{\ast\ast}_0$, as such we will use $Y^{\ast\ast}_0$ instead of $Y_0$.\\
$A_{42}$: The result of an action with $\neg R_{01}$ and$\neg R_{10}$. This means the set $X_{(0)}$ must be a union of the
blocks in $Y^{\ast\ast}_0$ and $X_{(1)}$ a union of blocks in $Y^\ast_1$.
The required result will then be $A_{41}-A_{42}$.
Since $m_2=m_1$ while $m_0$ is independent, we require
\[I_{(\mathbb{Z}^\ast_{27},Y_1\cup Y_2)}\times I_{(\mathbb{Z}^\ast_{27},Y_0)},\]
blocked as required. Therefore we have:
\begin{equation*}
\begin{split}
A_{41}:&I_{(\mathbb{Z}^\ast_{27},Y_1\cup Y_2)}\times I_{(\mathbb{Z}^\ast_{27},Y^{\ast\ast}_0)}\\
A_{42}:&I_{(\mathbb{Z}^\ast_{27},Y^{\ast}_1\cup Y_2)}\times I_{(\mathbb{Z}^\ast_{27},Y^{\ast\ast}_0)}.\\
\end{split}
\end{equation*}

The generating function for $A_{41}-A_{42}$ is:
\begin{equation*}
\begin{split}
&t^{25}+5t^{24}+9t^{23}+12t^{22}+9t^{21}+5t^{20}+t^{19}+t^{16}+5t^{15}+9t^{14}+\\
&12t^{13}+9t^{12}+5t^{11}+t^{10}+t^7+5t^6+9t^5+12t^4+9t^3+5t^2+t
\end{split}
\end{equation*}
and for $t=1$ we have 126 distinct directed circulants.

Once again, $A_5$ will be divided into the problems $A_{51}$ and $A_{52}$. Although $A_{51}$ is determined in a manner similar to the previous cases, one should be cautious when determining $A_{52}$, since this time we have 2 non-invariance conditions. This means that we have to consider the following:\\
$A_{51}$: The result of an action with $\neg R_{00}$, that is, we shall use $Y^\ast_0$ instead of $Y_0$,\\
$A_{52}$: The result of an action with $\neg(R_{01} \mbox {and } R_{10} )$ and $\neg R_{00}$.\\
Now for $A_{52}$, by de Morgan's laws, we have that:
\begin{equation*}
\begin{split}
\neg(R_{01} \mbox { and } R_{10})\mbox{ and } \neg R_{00}=&(\neg R_{01} \mbox { or } \neg R_{10}) \mbox {and } \neg R_{00}\\
=&(\neg R_{01}\mbox { and } \neg R_{00}) \mbox { or } (\neg R_{10}\mbox { and } \neg R_{00})\\
|\neg(R_{01} \mbox { and } R_{10})\mbox{ and } \neg R_{00}|=&|(\neg R_{01}\mbox { and } \neg R_{00})|+|(\neg R_{10}\mbox { and } \neg R_{00})|-\\
&|(\neg R_{01}\mbox { and } \neg R_{10} \mbox{ and } \neg R_{00})|\\
\end{split}
\end{equation*}
Now $\neg(R_{01})\Rightarrow \neg(R_{00})$, therefore for $A_{52}$ we have:
\begin{equation*}
\begin{split}
|\neg(R_{01} \mbox { and } R_{10})\mbox{ and } \neg R_{00}|=&|\neg (R_{01})|+|(\neg (R_{10})\mbox { and } \neg (R_{00}))|-\\
&|(\neg (R_{01})\mbox { and } \neg (R_{10}))|.\\
\end{split}
\end{equation*}
Therefore we shall split $A_{52}$ into 3 enumeration subproblems, with the first problem enumerating under the condition $\neg (R_{01})$, the second under $\neg (R_{10})\mbox { and } \neg (R_{00})$ and the last under the invariance conditions $\neg (R_{01})\mbox { and } \neg (R_{10})$.

Moreover, we now have 2 restrictions on the multipliers, an equality with $m_2=m_1$ and a congruence, $m_1\equiv m_0\bmod 3$. Therefore in this case, we need to define a group $G$ which will act on $\lbrace Y_1\cup Y_2\cup Y_3 \rbrace$ (blocked as required according to the given invariance condition), such that:
\begin{enumerate}
\item[(1)] The same multiplier acts on all $\lbrace Y_1\cup Y_2\cup Y_3 \rbrace$\\
\item[(2)] Two different multipliers act:
\begin{itemize}
\item $a$ on $Y_1 \cup Y_2$
\item $a'$ on $Y_0$
\end{itemize}
with $a'\equiv a\bmod3$.
\end{enumerate}
Now, since $a'\equiv a\bmod3$ implies the possibility that $a' \equiv a$, the second possibility includes the first. Therefore we will construct $G$ as follows:
$G$ will contain all ordered pairs $(a,a')$ such that $a, a' \in \mathbb{Z}^\ast_{27}$ and $a'\equiv a\bmod3$. Then $(a,a')$ will act as follows:
\[(a,a')(y)=\left\{
\begin{array}{ll}
&ay \mbox{ if } y \in Y_1 \cup Y_2\\
&a'y \mbox{ if } y \in Y_0
\end{array}\right.\]
%\vspace{3mm}
once again, with $Y_0$, $Y_1$ and $Y_2$ blocked as required. This would give all actions as in (1) and (2) above. One may verify that $G$ is in fact a group since if $(a,a'),(b,b')\in G$ then $(ab,a'b')\in G$. This group is coded in GAP as follows:

\begin{spverbatim}
allpairs:= function(x,y) return [x,y]; end;
S1:=ListX(zstar27, zstar27,allpairs);
ismod3:= function(a) return \mod(a[1],3) = \mod(a[2],3); end;
grp := Filtered(S1, ismod3);
\end{spverbatim}
\vspace{5mm}
Here \verb!S1! is a list of all pairs $(a,a')$, of elements of $\mathbb{Z}^\ast_{27}$ and the function \verb!ismod3! ensures that $a'\equiv a\bmod3$. In order to obtain the required group $G$ (identified by \verb!grp! in the code above), the pairs in \verb!S1! are filtered so that only those which satisfy the condition \verb!ismod3! remain.

Let us start off by determining the generating function for $A_{51}$. Here we need to consider all pairs $(a, a')\in G$ where $a$ acts on $\lbrace Y_1 \cup Y_2\rbrace$ and $a'$ acts on $Y^\ast_0$. In this case, we have following generating function:
\begin{equation*}
\begin{split}
&t^{26}+2t^{25}+6t^{24}+11t^{23}+18t^{22}+20t^{21}+29t^{20}+38t^{19}+47t^{18}+64t^{17}+86t^{16}+\\
&91t^{15}+109t^{14}+124t^{13}+109t^{12}+91t^{11}+86t^{10}+64t^9+47t^8+38t^7+29t^6+20t^5+\\
&18t^4+11t^3+6t^2+2t+1
\end{split}
\end{equation*}
Substituting $t=1$ this gives 1168 distinct circulants.

Now let us consider $A_{52}$.
In order to determine $|\neg R_{01}|$ we need to consider all pairs $(a, a')\in G$ where $a$ acts on $\lbrace Y_1 \cup Y_2\rbrace$ and $a'$ acts on $Y^{\ast\ast}_0$. This action produces the following generating function

\begin{equation*}
\begin{split}
&t^{26}+2t^{25}+6t^{24}+10t^{23}+14t^{22}+10t^{21}+6t^{20}+2t^{19}+t^{18}+t^{17}+4t^{16}+10t^{15}+\\
&20t^{14}+26t^{13}+20t^{12}+10t^{11}+4t^{10}+t^9+t^8+2t^7+6t^6+10t^5+14t^4+10t^3+\\
&6t^2+2t+1
\end{split}
\end{equation*}
If we substitute $t=1$, this generating function gives 200 distinct directed circulants.

In order to determine $|(\neg R_{10} \mbox{ and } \neg R_{00})|$, we will now require every pair $(a, a')\in G$ with $a$ acting on $\lbrace Y^\ast_1 \cup Y_2\rbrace$ and $a'$ acting on $Y^{\ast}_0$. This returns the following generating function:
\begin{equation*}
\begin{split}
&t^{26}+t^{25}+t^{24}+2t^{23}+4t^{22}+2t^{21}+6t^{20}+10t^{19}+6t^{18}+10t^{17}+20t^{16}+10t^{15}+\\
&14t^{14}+26t^{13}+14t^{12}+10t^{11}+20t^{10}+10t^9+6t^8+10t^7+6t^6+2t^5+4t^4+2t^3+\\
&t^2+t+1
\end{split}
\end{equation*}
and on substituting for $t$ we obtain another 200 distinct circulants.

Similarly, we can obtain $|(\neg R_{01} \mbox{ and } \neg R_{10})|$ by considering all the pairs $(a, a')\in G$, this time with $a$ acting on $\lbrace Y^\ast_1 \cup Y_2\rbrace$ and $a'$ acting on $Y^{\ast\ast}_0$. This gives the following generating function
\begin{equation*}
\begin{split}
&t^{26}+t^{25}+t^{24}+t^{23}+2t^{22}+t^{21}+t^{20}+t^{19}+t^{18}+t^{17}+2t^{16}+t^{15}+2t^{14}+4t^{13}+\\
&2t^{12}+t^{11}+2t^{10}+t^9+t^8+t^7+t^6+t^5+2t^4+t^3+t^2+t+1.
\end{split}
\end{equation*}
Therefore when $t=1$ we have 36 distinct circulants.

Combining our results we have:
\[A_{52}=|(\neg R_{01})|+|(\neg R_{10}\mbox { and } \neg R_{00})|-|(\neg R_{01}\mbox { and } \neg R_{10})|=200+200-36=364\]
Therefore our final result for $A_5$ is given by $A_{51}-A_{52}=1168-364=804$. This means the last isomorphism problem $A_5$ counts 804 circulants.

Adding all the generating functions of the individual problems $A_1,...A_5$, returns the following generating function

\begin{equation*}
\begin{split}
&t^{26}+3t^{25}+23t^{24}+152t^{23}+844t^{22}+3662t^{21}+12814t^{20}+36548t^{19}+86837t^{18}+\\
&173593t^{17}+295172t^{16}+429240t^{15}+536646t^{14}+577821t^{13}+536646t^{12}+\\
&429240t^{11}+295172t^{10}+173593t^9+86837t^8+36548t^7+12814t^6+3662t^5+\\
&844t^4+152t^3+23t^2+3t+1\\
\end{split}
\end{equation*}
giving a total of 3728891 non-isomorphic directed circulants on 27 vertices.

\subsection{$\MakeLowercase{n=27}$ Undirected}
Let us consider the case when $p=3$ and count all non-isomorphic undirected circulants on $3^3$ vertices. In this case we have
%\flushleft
\begin{equation*}
\begin{split}
\mathbb{Z}^\ast_{27}&=\lbrace 1,2,4,5,7,8,10,11,13,14,16,17,19,20,22,23,25,26 \rbrace\\
\mathbb{Z}'_{27}&=\lbrace 1,2,3,4,...26 \rbrace
\end{split}
\end{equation*}
that is, the connection set is a subset $X$ of $\mathbb{Z}'_{27}$ which must have the property $X=-X$, and the multipliers $m_0,m_1,m_2$ come from $\mathbb{Z}^\ast_{27}$.
Since the elements in the connecting sets are paired by inversion, as discussed in the previous chapter for the $p^2$ case, we partition $\mathbb{Z}'_{27}$ as
\begin{equation*}
\begin{split}
\mathbb{Z}'_{27}=\lbrace \lbrace1,26\rbrace,&\lbrace2,25\rbrace,\lbrace3,24\rbrace,\lbrace4,23\rbrace,\lbrace5,22\rbrace,\lbrace6,21\rbrace,\\
&\lbrace7,20\rbrace,\lbrace8,19\rbrace,\lbrace9,18\rbrace,\lbrace10,17\rbrace,\lbrace11,16\rbrace,\lbrace12,15\rbrace,\lbrace13,14\rbrace\rbrace
\end{split}
\end{equation*}
Again note that any connection set we shall work with must have either both elements of a given pair or none. The multiplicative action must therefore be taken on these pairs, again as we did when discussing Turner's trick and the $p^2$ case.

Let $Y_0$, $Y_1$, $Y_2$, be the three layers of $\mathbb{Z}'_{27}$, with elements in $Y_0$ having no factor of 3, elements of $Y_1$ having a factor of 3 but not 9 and elements of $Y_2$ with a factor of 9. Then $\mathbb{Z}'_{27}=Y_0 \dot{\cup}Y_1 \dot{\cup}Y_2$ where:
\begin{equation*}
\begin{split}
Y_0&=\lbrace\lbrace1,26\rbrace,\lbrace2,25\rbrace,\lbrace4,23\rbrace,\lbrace5,22\rbrace,\lbrace7,20\rbrace,\lbrace8,19\rbrace,\lbrace10,17\rbrace,\lbrace11,16\rbrace,\lbrace13,14\rbrace\rbrace\\
Y_1&=\lbrace \lbrace3,24\rbrace,\lbrace6,21\rbrace,\lbrace12,15\rbrace\rbrace\\
Y_2&=\lbrace \lbrace9,18\rbrace \rbrace\\
\end{split}
\end{equation*}
still partitioned into inverse pairs.
Now by theorem~\ref{thm:theorem4}, the non-invariance conditions in this case are:
\begin{equation*}
\begin{split}
R_{00}&: \hspace{3mm} 10X_{(0)}\neq X_{(0)}\\
R_{01}&: \hspace{3mm} 4X_{(0)}\neq X_{(0)}\\
R_{10}&: \hspace{3mm} 4X_{(1)}\neq X_{(1)}
\end{split}
\end{equation*}
where $X_{(0)}=X \cap \mathbb{Z}^\ast_{27}$ that is, $X_{(0)}=X \cap Y_0$ and $X_{(1)}=X \cap 3\mathbb{Z}^\ast_{9}$, that is, $X_{(1)}=X \cap Y_1$. Recall that $X_{(2)}=X\cap Y_2$ and that throughout, $X,X_{(0)},X_{(1)},X_{(2)}$ are partitioned into inverse pairs and it is the action on these pairs that we are dealing with.

Recall that when we enumerate under an invariance condition $\neg(R_{ij})$, such as $4X_{(1)}=X_{(1)}$, then we must take $X_{(1)}$ from whole subsets of $Y_1$ which are invariant under $4Y_1=Y_1$. Once again we shall denote the partitioned set corresponding to the invariance condition $10Y_{0}=Y_{0}$ by $Y^\ast_0$, that corresponding to $4Y_{0}=Y_{0}$ by $Y^{\ast\ast}_0$, and that corresponding to the invariance condition $4Y_{1}=Y_{1}$ by $Y^\ast_1$. However, we must now make sure that each part or block used in these sets for the directed case, is fused so that additive inverses are together in the same block. This results in the following partitioned sets:

\begin{equation*}
Y^\ast_0=\lbrace\lbrace1,10,19,8,26,17\rbrace, \lbrace2,20,11,7,16,25\rbrace, \lbrace 4,13,22,5,23,14\rbrace\rbrace
\end{equation*}
Therefore under the condition $10X_{(0)}= X_{(0)}$, $X_{(0)}$ is a union of these parts, therefore the multiplicative action is taken on these blocks. This means that the sets $\lbrace1,10,19,8,26,17\rbrace$, $\lbrace2,20,11,7,16,25\rbrace$ and $\lbrace 4,13,22,5,23,14\rbrace$, must be fixed, that is, every set must appear whole as a neighbour or not at all.
\begin{equation*}
Y^{\ast\ast}_0=\lbrace\lbrace 1,4,16,10,13,25,19,22,7,2,8,5,20,26,23,11,17,14\rbrace\rbrace
\end{equation*}
This means that under the invariance condition $4X_{(0)}= X_{(0)}$ the multiplicative action takes place on only one block not on 18 different points.
\begin{equation*}
Y^\ast_1=\lbrace3,12,21,6,24,15\rbrace
\end{equation*}
$Y^\ast_1$ gives that under the invariance condition $4X_{(1)}= X_{(1)}$, the set $X_{(1)}$ is simply the block $\lbrace3,12,21,6,24,15\rbrace$, therefore we may again consider the multiplicative action as being on just one point.

Therefore we can now consider the isomorphism problems $A_1,...A_5$ using the actions mentioned in the directed case for $n=27$, however this time using the sets listed here.

The cycle index for the action in $A_{1}$ was given by
\[I_{(\mathbb{Z}^\ast_{27},Y^{\ast\ast}_0)}\times I_{(\mathbb{Z}^\ast_{27},Y^{\ast}_1)}\times I_{(\mathbb{Z}^\ast_{27},Y_2)}.\]
Using the software package GAP, the following generating function has been obtained for $A_1$:
\[t^{26}+t^{24}+t^{20}+t^{18}+t^8+t^6+t^2+1\]
Substituting $t=1$, we obtain 8 distinct circulants.
\vspace{3mm}

As previously mentioned, $A_{2}$ is given by $A_{21}-A_{22}$ where for $A_{21}$ we have the action
\[(\mathbb{Z}^\ast_{27},\mathbb{Z}'_{27})\]
and for $A_{22}$ we have
\[(\mathbb{Z}^\ast_{27},Y^\ast_0 \cup Y_1 \cup Y_2)\]

Once again using GAP, the generating function for $A_{21}$ is
\begin{equation*}
\begin{split}
&t^{26}+3t^{24}+10t^{22}+34t^{20}+83t^{18}+147t^{16}+194t^{14}+194t^{12}+147t^{10}+83t^8+\\
&34t^6+10t^4+3t^2+1
\end{split}
\end{equation*}
Substituting $t=1$, we obtain 944 undirected circulants.
The generating function for $A_{22}$ is
\[t^{26}+2t^{24}+2t^{22}+3t^{20}+5t^{18}+6t^{16}+5t^{14}+5t^{12}+6t^{10}+5t^8+3t^6+2t^4+2t^2+1\]
and substituting $t=1$ we obtain 48 undirected circulants.
Therefore the subproblem $A_2$ gives $944-48=896$ distinct undirected circulants.

Once again, for $A_3$, we need to determine $A_{31}-A_{32}$ where
\[A_{31}:I_{(\mathbb{Z}^\ast_{27},Y^\ast_0\cup Y^\ast_1)}\times I_{(\mathbb{Z}^\ast_{27},Y_2)}\]
and
\[A_{32}:I_{(\mathbb{Z}^\ast_{27},Y^{\ast\ast}_0\cup Y^\ast_1)}\times I_{(\mathbb{Z}^\ast_{27},Y_2)}\]
These actions correspond to the following generating functions:

\begin{equation*}
A_{31}: t^{26}+t^{24}+2t^{20}+2t^{18}+2t^{14}+2t^{12}+2t^8+2t^6+t^2+1
\end{equation*}
for which, when $t=1$, we obtain 16 circulants, and
\begin{equation*}
A_{32}:t^{26}+t^{24}+t^{20}+t^{18}+t^8+t^6+t^2+1
\end{equation*}
which results in 8 circulants. Therefore the number of distinct, undirected circulants in this case is $16-8=8$.

$A_4$ will again be given by $A_{41}-A_{42}$, where:
\begin{equation*}
\begin{split}
A_{41}:&I_{(\mathbb{Z}^\ast_{27},Y_1\cup Y_2)}\times I_{(\mathbb{Z}^\ast_{27},Y^{\ast\ast}_0)}\\
A_{42}:&I_{(\mathbb{Z}^\ast_{27},Y^{\ast}_1\cup Y_2)}\times I_{(\mathbb{Z}^\ast_{27},Y^{\ast\ast}_0)}
\end{split}
\end{equation*}

$A_{41}$ returns the generating function:
\[t^{26}+2t^{24}+2t^{22}+2t^{20}+t^{18}+t^8+2t^6+2t^4+2t^2+1\]
and substituting for $t$, we obtain 16 circulants.\\
$A_{42}$ returns the generating function:
\[t^{26}+t^{24}+t^{20}+t^{18}+t^8+t^6+t^2+1.\]
Substituting for $t=1$, we obtain 8 circulants. Therefore, for the case $A_4$, we have $16-8=8$, undirected circulants which are distinct.

For $A_5$ we will once again proceed as stated for the directed case.
In $A_{51}$, we need to consider all pairs $(a, a')\in G$ where $a$ acts on $\lbrace Y_1 \cup Y_2\rbrace$ and $a'$ acts on $Y^\ast_0$. In doing so, we obtain the following generating function
\[t^{26}+2t^{24}+2t^{22}+3t^{20}+3t^{18}+2t^{16}+3t^{14}+3t^{12}+2t^{10}+3t^8+3t^6+2t^4+2t^2+1\]
and when substituting $t=1$ we obtain 32 circulants.

Now let us consider $A_{52}$.
In order to determine $|\neg R_{01}|$ we need to consider all pairs $(a, a')\in G$ where $a$ acts on $\lbrace Y_1 \cup Y_2\rbrace$ and $a'$ acts on $Y^{\ast\ast}_0$. This action produces the following generating function

\[t^{26}+2t^{24}+2t^{22}+2t^{20}+t^{18}+t^8+2t^6+2t^4+2t^2+1\]
as such $|\neg R_{01}|=16$.\\

In order to determine $|(\neg R_{10} \mbox{ and } \neg R_{00})|$, we again require every pair $(a, a')\in G$ with $a$ acting on $\lbrace Y^\ast_1 \cup Y_2\rbrace$ and $a'$ acting on $Y^{\ast}_0$. This returns the following generating function:
\[t^{26}+t^{24}+2t^{20}+2t^{18}+2t^{14}+2t^{12}+2t^8+2t^6+t^2+1\]
and on substituting for $t$ we obtain 16 circulants.

Similarly, as done in the directed case for $n=27$ we obtain $|(\neg R_{01} \mbox{ and } \neg R_{10})|$ by considering all the pairs $(a, a')\in G$, this time with $a$ acting on $\lbrace Y^\ast_1 \cup Y_2\rbrace$ and $a'$ acting on $Y^{\ast\ast}_0$. This gives the following generating function
\[t^{26}+t^{24}+t^{20}+t^{18}+t^8+t^6+t^2+1.\]
Therefore when $t=1$ we have $|(\neg R_{01} \mbox{ and } \neg R_{10})|=8$.

Combining our results we have:
\[A_{52}=|(\neg R_{01})|+|(\neg R_{10}\mbox { and } \neg R_{00})|-|(\neg R_{01}\mbox { and } \neg R_{10})|=16+16-8=24\]
Therefore our final result for $A_5$ is given by $A_{51}-A_{52}=32-24=8$.

Combining all the results from our isomorphism problems $A_1,...A_5$ gives
\[A_1+A_2+A_3+A_4+A_5=8+896+8+8+8=928.\]
This means we have 928 non-isomorphic, undirected circulants on 27 vertices.

The following generating function has been obtained for this problem:
\begin{equation*}
\begin{split}
&t^{26}+3t^{24}+10t^{22}+34t^{20}+81t^{18}+143t^{16}+192t^{14}+192t^{12}+143t^{10}+81t^8+\\
&34t^6+10t^4+3t^2+1
\end{split}
\end{equation*}

\subsection{$\MakeLowercase{n=125}$ Directed}
The method used above, to count the number of non-isomorphic circulants on 27 vertices, may of course be applied to any other prime, not necessarily only for the case when $p=3$. As such in this thesis we also present results for the case $5^3$, that is, we will count the number of non-isomorphic circulants on 125 vertices. In the case of directed circulants, we will require the following sets:
%\flushleft
\begin{equation*}
\begin{split}
\mathbb{Z}'_{125}&=\lbrace 1,2,3,4,...124 \rbrace\\
\mathbb{Z}^\ast_{125}&=\mbox{ All numbers in } \mathbb{Z}'_{125} \mbox{ relatively prime to }125
\end{split}
\end{equation*}
In the case when $p=5$, the three layers which arise for $p^3$ are $Y_0$, which contains all those elements relatively prime to 125, $Y_1$, with elements which are divisible by 5, but not $5^2$, and $Y_2$, which contains elements divisible by $p^2=5^2$. The sets $Y_0, Y_1, Y_2$, are therefore as follows:
\begin{equation*}
\begin{split}
Y_0=&\mathbb{Z}^\ast_{125}\\
Y_1=&\lbrace 5,10,15,20,30,35,40,45,55,60,65,70,80,85,90,95,105,110,115,120\rbrace\\
Y_2=&\lbrace 25,50,75,100 \rbrace\\
\end{split}
\end{equation*}

By theorem~\ref{thm:theorem4}, the non-invariance conditions in this case are:
\begin{equation*}
\begin{split}
R_{00}&: \hspace{3mm} 26X_{(0)}\neq X_{(0)}\\
R_{01}&: \hspace{3mm} 6X_{(0)}\neq X_{(0)}\\
R_{10}&: \hspace{3mm} 6X_{(1)}\neq X_{(1)}
\end{split}
\end{equation*}

$Y^\ast_0$ is made up of the blocks $a_1,a_2,...a_{20}$ given below:
\renewcommand{\tabcolsep}{4.6pt}
\begin{quote}
\centering
\small
\begin{tabular}{@{} *{21}{l} @{}}
$a_1=\lbrace 1,26,51,76,101\rbrace$,&$a_6=\lbrace 7,57,107,32,82\rbrace,$\\
$a_2=\lbrace 2,52,102,27,77\rbrace$,&$a_7=\lbrace 8,83,33,108,58\rbrace$,\\
$a_3=\lbrace 3,78,28,103,53\rbrace$,&$a_8=\lbrace 9,109,84,59,34\rbrace$,\\
$a_4=\lbrace 4,104,79,54,29\rbrace$,&$a_9=\lbrace 11,36,61,86,111\rbrace$,\\
$a_5=\lbrace 6,31,56,81,106\rbrace$,&$a_{10}=\lbrace 12,62,112,37,87\rbrace$,\\
\end{tabular}
\end{quote}

\begin{quote}
\centering
\small
\begin{tabular}{@{} *{21}{l} @{}}
$a_{11}=\lbrace 13,88,38,113,63\rbrace$,&$a_{16}=\lbrace 19,119,94,69,44\rbrace$,\\
$a_{12}=\lbrace 14,114,89,64,39\rbrace$,&$a_{17}=\lbrace 21,46,71,96,121\rbrace$,\\
$a_{13}=\lbrace 16,41,66,91,116\rbrace$,&$a_{18}=\lbrace 22,72,122,47,97\rbrace$,\\
$a_{14}=\lbrace 17,67,117,42,92\rbrace$,&$a_{19}=\lbrace 23,98,48,123,73\rbrace$,\\
$a_{15}=\lbrace 18,93,43,118,68\rbrace$,&$a_{20}=\lbrace 24,124,99,74,49\rbrace.$\\
\end{tabular}
\end{quote}
that is
\begin{equation*}
Y^\ast_0=\lbrace a_1,a_2,...a_{20} \rbrace
\end{equation*}
The sets $Y^{\ast\ast}_0$ and $Y^\ast_1$ are given as follows:
\begin{equation*}
\begin{split}
Y^{\ast\ast}_0=&\lbrace \lbrace a_1 \cup a_5 \cup a_9 \cup a_{13} \cup a_{17}\rbrace, \lbrace a_2 \cup a_6 \cup a_{10} \cup a_{14} \cup a_{18}\rbrace,\\
&\lbrace a_3 \cup a_7 \cup a_{11} \cup a_{15} \cup a_{19}\rbrace,\lbrace a_4 \cup a_8 \cup a_{12} \cup a_{16} \cup a_{20}\rbrace\rbrace\\
Y^\ast_1=&\lbrace\lbrace 5,30,55,80,105\rbrace, \lbrace 10,60,110,35,85 \rbrace, \lbrace 15,90,40,115,65\rbrace,\\
&\lbrace 20,120,95,70,45\rbrace\rbrace
\end{split}
\end{equation*}

We will now consider the isomorphism problems $A_1$ to $A_5$. Unlike the case $n=27$, we shall not present the generating functions in this case, but we will give the final result in each case.

In the first problem, $A_1$, we have the following product of cycle indices:
\[I_{(\mathbb{Z}^\ast_{125},Y^{\ast\ast}_0)}\times I_{(\mathbb{Z}^\ast_{125},Y^{\ast}_1)}\times I_{(\mathbb{Z}^\ast_{125},Y_2)}\]
Again the software package GAP was used in order to determine the number of circulants under this isomorphism problem. This came to 216 circulants.

$A_{2}$ is given by $A_{21}-A_{22}$ where for $A_{21}$ we have the action

\[(\mathbb{Z}^\ast_{125},\mathbb{Z}'_{125})\]
and for $A_{22}$ we have
\[(\mathbb{Z}^\ast_{125},Y^\ast_0 \cup Y_1 \cup Y_2).\]
The number of distinct circulants obtained for $A_2$ from $A_{21}-A_{22}$ is\\
212,676,479,325,586,539,710,725,813,950,176,256.

Similarly, for $A_3$, the result we require is $A_{31}-A_{32}$ where

\[A_{31}:I_{(\mathbb{Z}^\ast_{125},Y^\ast_0\cup Y^\ast_1)}\times I_{(\mathbb{Z}^\ast_{125},Y_2)}\]
and
\[A_{32}:I_{(\mathbb{Z}^\ast_{125},Y^{\ast\ast}_0\cup Y^\ast_1)}\times I_{(\mathbb{Z}^\ast_{125},Y_2)}\]
The result $A_{31}-A_{32}$ gives 5034348 circulants.

The number of distinct, directed circulants given by subproblem $A_4$, is determined by working out $A_{41}-A_{42}$ where
\[A_{41}:I_{(\mathbb{Z}^\ast_{125},Y_1\cup Y_2)}\times I_{(\mathbb{Z}^\ast_{125},Y^{\ast\ast}_0)}\]
and
\vspace{3mm}
\[A_{42}:I_{(\mathbb{Z}^\ast_{125},Y^{\ast}_1\cup Y_2)}\times I_{(\mathbb{Z}^\ast_{125},Y^{\ast\ast}_0)}.\]
The result obtained from $A_{41}-A_{42}$ is again 5034348.

$A_5$, will be determined in a manner similar to that used for the directed case when $n=27$, however the group $G$ effecting the action in this case, is constructed as follows: $G$ will contain all ordered pairs $(a,a')$ such that $a, a' \in \mathbb{Z}^\ast_{125}$ and $a'\equiv a\bmod5$.

The number of distinct circulants obtained in this case, by splitting this isomorphism problem into a number of other subproblems as described for the case $n=27$ is 175916533428.

The following list gives the number of distinct, directed circulants for each subproblem:

\begin{quote}
\centering
\begin{tabular}{c|l}
Sub-Problem&Number of distinct circulants\\
\hline
$A_1$&216\\
$A_2$&212,676,479,325,586,539,710,725,813,950,176,256\\
$A_3$&5034348\\
$A_4$&5034348\\
$A_5$&175,916,533,428\\
\end {tabular}
\end{quote}

Therefore the total number of non-isomorphic, directed circulants on 125 vertices is 212,676,479,325,586,539,710,725,989,876,778,596.
\subsection{$\MakeLowercase{n=125}$ Undirected}
We will now enumerate the number of non-isomorphic undirected circulants on 125 vertices. In this case we will require the following sets:
%\flushleft
\begin{equation*}
\begin{split}
\mathbb{Z}'_{125}&=\lbrace 1,2,3,4,...124 \rbrace\\
\mathbb{Z}^\ast_{125}&=\mbox{ All numbers in } \mathbb{Z}'_{125} \mbox{ relatively prime to }125
\end{split}
\end{equation*}
where the elements in $\mathbb{Z}'_{125}$ are paired by inversion.

The sets $Y_0$, $Y_1$ and $Y_2$ are as follows:

\begin{equation*}
\begin{split}
Y_0=&\mathbb{Z}^\ast_{125}\\
Y_1=&\lbrace \lbrace 5,120 \rbrace,\lbrace 10,115 \rbrace,\lbrace 15,110 \rbrace,\lbrace 20,105 \rbrace,\lbrace 30,95 \rbrace,\\
&\lbrace 35,90 \rbrace,\lbrace 40,85 \rbrace,\lbrace 45,80 \rbrace,\lbrace 55,70 \rbrace,\lbrace 60,65 \rbrace\rbrace\\
Y_2=&\lbrace \lbrace 25,100 \rbrace,\lbrace 50,75 \rbrace \rbrace.
\end{split}
\end{equation*}

Again the non-invariance conditions in this case are:
\begin{equation*}
\begin{split}
R_{00}&: \hspace{3mm} 26X_{(0)}\neq X_{(0)}\\
R_{01}&: \hspace{3mm} 6X_{(0)}\neq X_{(0)}\\
R_{10}&: \hspace{3mm} 6X_{(1)}\neq X_{(1)}
\end{split}
\end{equation*}
Pairing the sets $a_1,a_2,...a_{20}$ obtained in the directed case by inversion, we have:
\begin{equation*}
\begin{split}
Y^\ast_0=&\lbrace\lbrace a_1,a_{20}\rbrace,\lbrace a_2,a_{19}\rbrace,\lbrace a_3,a_{18}\rbrace,\lbrace a_4,a_{17}\rbrace,\\
&\lbrace a_5,a_{16}\rbrace,\lbrace a_6,a_{15}\rbrace,\lbrace a_7,a_{14}\rbrace,\lbrace a_8,a_{13}\rbrace,\\
&\lbrace a_9,a_{12}\rbrace,\lbrace a_{10},a_{11}\rbrace\rbrace\\
\end{split}
\end{equation*}

Similarly each block used in the sets $Y^{\ast\ast}$ and $Y^\ast_1$ for the directed case, is fused so that additive inverses are together in the same block. We therefore have the following sets:
\begin{equation*}
\begin{split}
Y^{\ast\ast}_0=&\lbrace \lbrace a_1 \cup a_4 \cup a_5 \cup a_8 \cup a_9 \cup a_{12} \cup a_{13} \cup a_{16} \cup a_{17} \cup a_{20} \rbrace,\\
&\lbrace a_2 \cup a_3 \cup a_6 \cup a_7 \cup a_{10} \cup a_{11} \cup a_{14} \cup a_{15} \cup a_{18} \cup a_{19}\rbrace\rbrace\\
Y^\ast_1=&\lbrace\lbrace 5,20,30,45,55,70,80,95,105,120\rbrace, \lbrace 10,15,35,40,60,65,85,90,110,115\rbrace\rbrace.
\end{split}
\end{equation*}
Therefore we may now determine $A_1, A_2, A_3, A_4, A_5$.

In the case of $A_1$, we have the following:
\[I_{(\mathbb{Z}^\ast_{125},Y^{\ast\ast}_0)}\times I_{(\mathbb{Z}^\ast_{125},Y^{\ast}_1)}\times I_{(\mathbb{Z}^\ast_{125},Y_2)}\]
This produces the generating function
\begin{equation*}
\begin{split}
&t^{124}+t^{122}+t^{120}+t^{114}+t^{112}+t^{110}+t^{104}+t^{102}+t^{100}+t^{74}+t^{72}+t^{70}+\\
&t^{64}+t^{62}+t^{60}+t^{54}+t^{52}+t^{50}+t^{24}+t^{22}+t^{20}+t^{14}+t^{12}+t^{10}+t^4+t^2+1
\end{split}
\end{equation*}
which gives 27 distinct, undirected circulants on 125 vertices.

\vspace{3mm}
$A_{2}$ is given by $A_{21}-A_{22}$ where for $A_{21}$ we have the action
\[(\mathbb{Z}^\ast_{125},\mathbb{Z}'_{125})\]
and for $A_{22}$ we have
\[(\mathbb{Z}^\ast_{125},Y^\ast_0 \cup Y_1 \cup Y_2).\]
The number of distinct circulants obtained from $A_{21}-A_{22}$ is 92233720411413504.

Similarly, for $A_3$, the result we require is $A_{31}-A_{32}$ where
\[A_{31}:I_{(\mathbb{Z}^\ast_{125},Y^\ast_0\cup Y^\ast_1)}\times I_{(\mathbb{Z}^\ast_{125},Y_2)}\]
and
\[A_{32}:I_{(\mathbb{Z}^\ast_{125},Y^{\ast\ast}_0\cup Y^\ast_1)}\times I_{(\mathbb{Z}^\ast_{125},Y_2)}\]
The generating function obtained for $A_{31}-A_{32}$ is

\begin{equation*}
\begin{split}
&t^{114}+t^{112}+t^{110}+7t^{104}+7t^{102}+7t^{100}+22t^{94}+22t^{92}+22t^{90}+51t^{84}+51t^{82}+\\
&51t^{80}+79t^{74}+79t^{72}+79t^{70}+94t^{64}+94t^{62}+94t^{60}+79t^{54}+79t^{52}+79t^{50}+\\
&51t^{44}+51t^{42}+51t^{40}+22t^{34}+22t^{32}+22t^{30}+7t^{24}+7t^{22}+7t^{20}+t^{14}+t^{12}+t^{10}
\end{split}
\end{equation*}
Substituting $t=1$ we obtain 1242 circulants.

The number of distinct, undirected circulants, given by subproblem $A_4$, is determined by working out $A_{41}-A_{42}$ where

\[A_{41}:I_{(\mathbb{Z}^\ast_{125},Y_1\cup Y_2)}\times I_{(\mathbb{Z}^\ast_{125},Y^{\ast\ast}_0)}\]
and
\vspace{3mm}
\[A_{42}:I_{(\mathbb{Z}^\ast_{125},Y^{\ast}_1\cup Y_2)}\times I_{(\mathbb{Z}^\ast_{125},Y^{\ast\ast}_0)}.\]
$A_{41}-A_{42}$ gives the following generating function

\begin{equation*}
\begin{split}
&t^{122}+7t^{120}+22t^{118}+51t^{116}+79t^{114}+94t^{112}+79t^{110}+51t^{108}+22t^{106}+\\
&7t^{104}+t^{102}+t^{72}+7t^{70}+22t^{68}+51t^{66}+79t^{64}+94t^{62}+79t^{60}+\\
&51t^58+22t^{56}+7t^{54}+t^{52}+t^{22}+7t^{20}+22t^{18}+51t^{16}+79t^{14}+\\
&94t^{12}+79t^{10}+51t^8+22t^6+7t^4+t^2
\end{split}
\end{equation*}
and substituting for $t$, again gives 1242 distinct circulants.

$A_5$, is determined as described for the directed case when $n=125$, however this time the sets $Y^\ast_0, Y^{\ast\ast}_0$ and $Y^\ast_1$ defined in this subsection are used.
The number of distinct circulants obtained in this case is 83268.

Therefore adding up our results from $A_1,...,A_5$, we obtain 92233720411499283 non-isomorphic, undirected circulants on 125 vertices. The generating function in this case is given by
\begin{equation*}
\begin{split}
\small
&t^{124}+3t^{122}+45t^{120}+774t^{118}+11207t^{116}+129485t^{114}+1229657t^{112}+9835988t^{110}\\
&+67622641t^{108}+405731843t^{106}+2150382085t^{104}+10165426468t^{102}\\
&+43203077195t^{100}+166165624857t^{98}+581579739591t^{96}+1861054998416t^{94}\\
&+5466849215583t^{92}+14792650391699t^{90}+36981626382405t^{88}+85641660162366t^{86}\\
&+184129570236171t^{84}+368259138698205t^{82}+686301123812811t^{80}\\
&+1193567168903172t^{78}+1939546652290065t^{76}+2948110907190899t^{74}\\
&+4195388602819760t^{72}+5593851464926268t^{70}+6992314336461413t^{68}\\
&+8197885767564289t^{66}+9017674350331611t^{64}+9308567065105337t^{62}\\
&+9017674350331611t^{60}+8197885767564289t^{58}+6992314336461413t^{56}\\
&+5593851464926268t^{54}+4195388602819760t^{52}+2948110907190899t^{50}\\
&+1939546652290065t^{48}+1193567168903172t^{46}+686301123812811t^{44}\\
&+368259138698205t^{42}+184129570236171t^{40}+85641660162366t^{38}\\
&+36981626382405t^{36}+14792650391699t^{34}+5466849215583t^{32}+1861054998416t^{30}\\
&+581579739591t^{28}+166165624857t^{26}+43203077195t^{24}+10165426468t^{22}\\
&+2150382085t^{20}+405731843t^{18}+67622641t^{16}+9835988t^{14}+1229657t^{12}\\
&+129485t^{10}+11207t^8+774t^6+45t^4+3t^2+1\\
\end{split}
\end{equation*}

\subsection{Discussion of Results}
Tables ~\ref{table:Table5} and ~\ref{table:Table6} sum up our enumeration results obtained in this chapter by using the isomorphism problems $A_1,...A_5$ given in Table~\ref{table:Table3}.
\begin{table}[h!]
\caption{The number of non-isomorphic circulants on 27 vertices}
%\begin{quote}
\centering
\begin{tabular}{|c|c|}
\hline
&Number of Circulants\\
\hline
Undirected&928\\
Directed&3728891\\
\hline
\end{tabular}
\label{table:Table5}
%\end{quote}
\end{table}
\begin{table}[h!]
\caption{The number of non-isomorphic circulants on 125 vertices}
\centering
\begin{tabular}{|c|c|}
\hline
&Number of Circulants\\
\hline
Undirected&92233720411499283\\
Directed&212,676,479,325,586,539,710,725,989,876,778,596\\
\hline
\end{tabular}
\label{table:Table6}
\end{table}

The value we have obtained for the undirected case when $n=27$, that of 928, agrees with the figure B.D. McKay obtained (see \cite{KLP2003}) by using exhaustive computer analysis. Our generating function for this case, has also matched Matan Ziv-Av's generating function which he also obtained by brute force. The value (3728891) obtained for the directed case when $n=27$, also agrees with Matan Ziv-Av's result.

Present methods of the above-mentioned authors, do not go up to $n=125$. However, the generating functions we have obtained may be useful for checking our result, using exhaustive search since these methods may be feasible for $n=125$ and some specific values of the degree.

\section{Using Theorem ~\ref{thm:theorem4} directly: $n=27$ undirected}

In this section, we will describe another method which may be used to enumerate non-isomorphic undirected circulants on 27 vertices. This method makes direct use of Theorem ~\ref{thm:theorem4}. Rearranging this theorem for the case when $p=3$, we have the following:

\begin{Corollary}
\label{thm:theorem8}
Let $n=27$ and let $\Gamma$ and $\Gamma'$ be two 27-circulants with the connection sets $X$ and $X'$, respectively. Then $\Gamma$ and $\Gamma'$ are isomorphic if and only if their respective layers are multiplicatively equivalent, i.e.
\begin{equation}\tag{$M_{3}$}
X'_{(0)}=m_{0}X_{(0)},\hspace{3mm}X'_{(1)}=m_{1}X_{(1)}\hspace{3mm}X'_{(2)}=m_{2}X_{(2)},
\end{equation}

for an  arbitrary set of multipliers $m_0, m_1, m_2 \in \mathbb{Z}^\ast_{p^3}$.  Moreover, in the above, one must have

\begin{equation}\tag{$E_{00}$}
m_1 \equiv m_0(\bmod 9) \mbox{ and } m_2 \equiv m_1(\bmod 3)
\end{equation}
whenever
\begin{equation}\tag{$R_{00}$}
10X_{(0)}\neq X_{(0)}
\end{equation}

\begin{equation}\tag{$E_{01}$}
m_1 \equiv m_0(\bmod 3)
\end{equation}
whenever
\begin{equation}\tag{$R_{01}$}
4X_{(0)}\neq X_{(0)}
\end{equation}

\begin{equation}\tag{$E_{10}$}
m_2 \equiv m_1(\bmod 3)
\end{equation}
whenever
\begin{equation}\tag{$R_{10}$}
4X_{(1)}\neq X_{(1)}.
\end{equation}
\end{Corollary}

We shall therefore consider the following subproblems:
\\
$|A_1|$: Counts all non-isomorphic circulants with multipliers satisfying condition $E_{00}$, whenever $R_{00}$ holds,\\
$|A_2|$: Counts all non-isomorphic circulants with multipliers satisfying condition $E_{01}$, whenever $R_{01}$ holds,\\
$|A_3|$: Counts all non-isomorphic circulants with multipliers satisfying condition $E_{10}$, whenever $R_{10}$ holds.\\
In addition to these, we also need to consider case B, where $|B|$ is the number of non-isomorphic circulants under independent multipliers, whenever none of $R_{00}$, $R_{01}$ and $R_{10}$ hold.

The result we are after is then given by
\[|A_1\cup A_2\cup A_3\cup B|.\]

Now since B is disjoint from the rest, this may be calculated separately and added to the rest. Therefore, in order to enumerate the total number of non-isomorphic undirected circulants on 27 vertices we simply have the following inclusion-exclusion problem.
\begin{equation}\label{eq:erl2}
|A_1|+|A_2|+|A_3|-|A_1\cap A_2|-|A_2\cap A_3|-|A_1\cap A_3|+|A_1\cap A_2\cap A_3|+|B|
\end{equation}

Recall that it is simpler to carry out the calculations using the invariance conditions, rather than the non-invariance conditions. Therefore each of the subproblems listed in \eqref{eq:erl2}, will be split further into other problems. In doing so however, we must be careful when dealing with the intersections. For example, in the case $A_2\cap A_3$, we have to consider the non-invariance conditions
\[R_{01} \mbox{ and } R_{10}.\]
Therefore the corresponding invariance condition in this case would be
\[\neg(R_{01} \mbox{ and }R_{10}),\]
which, by de Morgan's law is equivalent to
\[\neg(R_{01}) \mbox{ or } \neg(R_{10})\]

So let us for a moment stop to discuss this question: suppose $(G_1,X)$ and $(G_2,X)$ are two permutation groups, then how many elements are non-equivalent under at least one of these actions? We say that $x,y \in X$, are equivalent under $G_1$, $x\sim_1 y$, if there exists $\sigma \in G_1$ such that $\sigma(x)=y$. Similarly $x \sim_2 y$. Let $\sim$ be defined on $X$ as follows:

\[x \sim y \mbox{ iff } \left\{
\begin{array}{ll}
&\mbox{ either } x \sim_1 y\\
&\mbox{ or } x\sim_2 y
\end{array} \right.\]

Now $\sim$ need not be an equivalence relation. We may have $x \sim y$, $y\sim z$ but $x\nsim z$. Consider the following example:

\begin{equation*}
\begin{split}
G_1&=\langle (1\hspace{2mm}2) (3\hspace{2mm}4)(5\hspace{2mm}6)(7\hspace{2mm}8)(9\hspace{2mm}10)\rangle\\
G_2&=\langle (1\hspace{2mm}2\hspace{2mm}3) (4\hspace{2mm}5\hspace{2mm}6)(7\hspace{2mm}8\hspace{2mm}9)(10)\rangle\\
\mbox{with } X&=\lbrace 1,2,3,...,10 \rbrace.\\
\end{split}
\end{equation*}
We start off by choosing an element of $X$, and remove those equivalent to it under $\sim$. We repeat this procedure, until we obtain a set of representatives. For example if we start off with the element, 2, in $X$, we have:
\[{\bf1},\underline{2},{\bf{3}},4,5,6,7,8,9,10.\]
Suppose we choose $4 \in X$ next. Then we have
\[{\bf{1}},\underline{2},{\bf{3}},\underline{4},{\bf{5}},{\bf{6}},7,8,9,10.\]
If we pick 7 next we have:
\[{\bf{1}},\underline{2},{\bf{3}},\underline{4},{\bf{5}},{\bf{6}},\underline{7},{\bf{8}},{\bf{9}},10.\]
Therefore, $\lbrace2,4,7,10\rbrace$, is a set of non-equivalent elements, such that any other element is equivalent to one of them.

Now if we pick 9, 6, 2 in that order and repeat the above procedure, we obtain the set of representatives, $\lbrace 2,6,9 \rbrace$. This means that these three elements are non-equivalent under $\sim$, and any other element of $X$ is equivalent to one of them. This shows that equivalence under the ``or" relation $\sim$ is not well defined here. In order to be able to determine how many elements are non-equivalent under $\sim$, we must have that $\sim$ is an equivalence relation.

If $\sim$ is an equivalence relation, then $\sim_1$, $\sim_2$ are said to be \emph{compatible}. When a class $C_1$ of $\sim_1$, intersects a class $C_2$ of $\sim_2$, as shown in Figure~\ref{intersections}, then $\sim_1$, $\sim_2$ are not compatible since $a \sim b$ and $b \sim c$ but $a\nsim c$.

\begin{figure}[!h]
\centering
\includegraphics[width=0.4\linewidth]{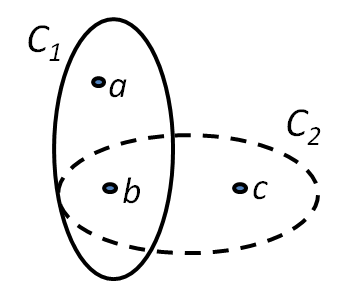}
\caption{Class Intersections}
\label{intersections}
\end{figure}
Therefore, for $\sim_1$, $\sim_2$ to be compatible, the classes must be such, that every class of $\sim_1$ is a union of classes of $\sim_2$, or together with other classes of $\sim_1$, forms part of a partition of classes of $\sim_2$. This is illustrated in Figure~\ref{partition}.
\begin{figure}[!h]
\centering
\includegraphics[width=0.6\linewidth]{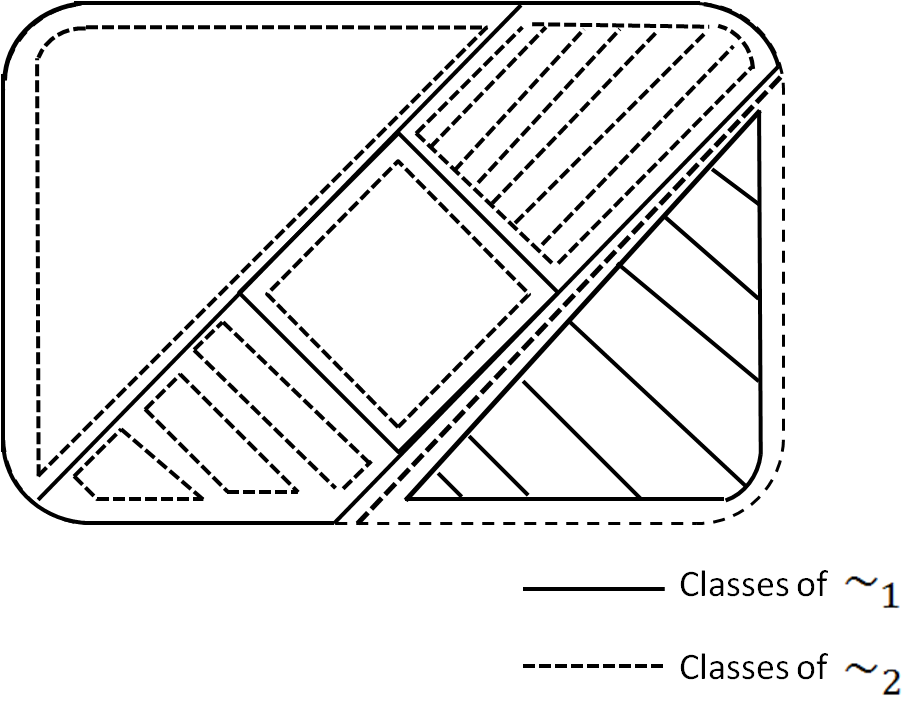}
\caption{Classes of $\sim_1$ and $\sim_2$ }
\label{partition}
\end{figure}
In this example, the relation $\sim$ has 20 classes. Using Figure~\ref{partition}, we may count these classes in terms of the number of classes of $\sim_1$ and $\sim_2$. The equivalence relation $\sim_1$, has 12 classes, while $\sim_2$ has 13 classes. Now $20=12+13-5$. This means that 5 classes are counted twice. These 5 over-counted classes are shown in Figure~\ref{overcountedclasses}.
\begin{figure}[!h]
\centering
\includegraphics[width=0.9\linewidth]{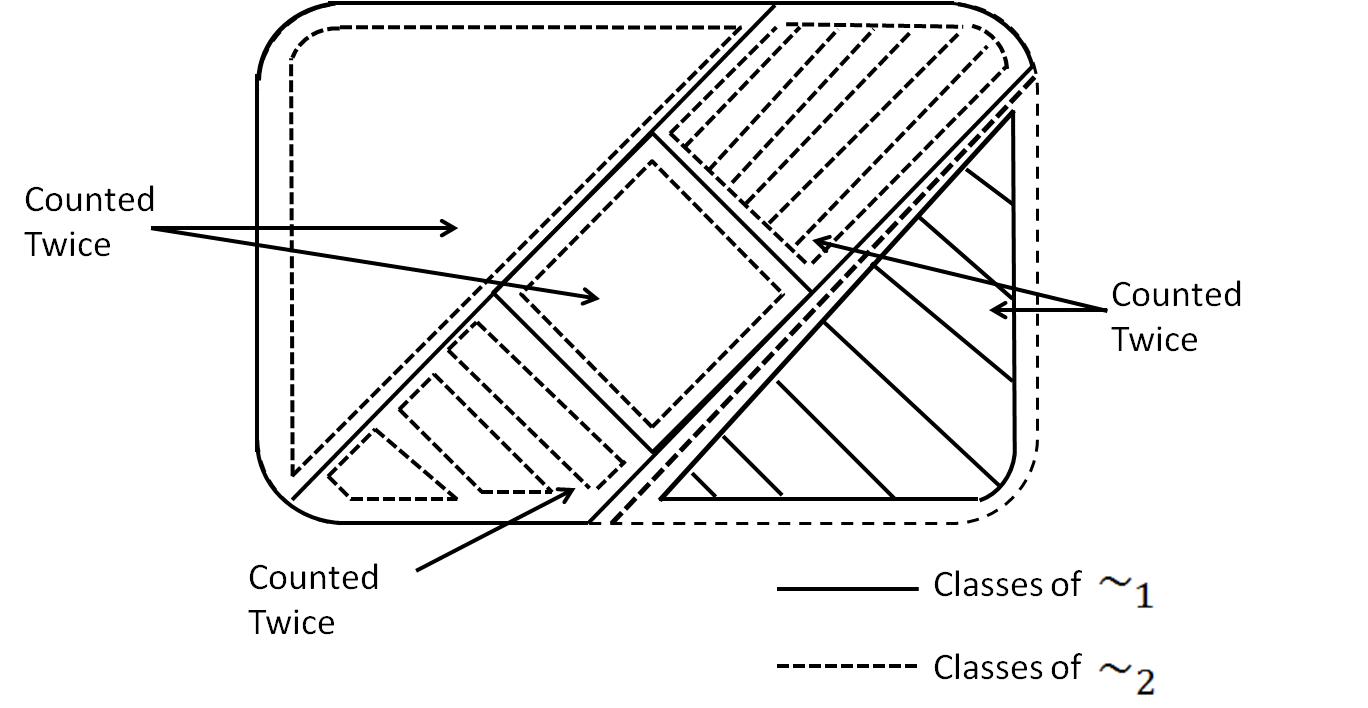}
\caption{Over-counted Classes }
\label{overcountedclasses}
\end{figure}

Now, in order to find the number of non-equivalent elements under $\sim$, we must be able to count and remove the number of over-counted classes. One must note that
\begin{equation*}
\begin{split}
\mbox{ Number of orbits }(G_1,X)+&\mbox{ Number of orbits }(G_2,X)-\\
&\mbox{ Number of orbits }(G_1\cap G_2,X)
\end{split}
\end{equation*}
is a wrong answer, because the number of orbits of $(G_1\cap G_2,X)$ is generally larger than the number of orbits $(G_1,X)$ or the number of orbits $(G_2,X)$, therefore it cannot enumerate the over-counted classes. Also, $(G_1 \cup G_2,X)$ does not work since $G_1 \cup G_2$ is not a group in general. In order to obtain the over-counted classes as orbits of some group, we need to consider the orbits of $(\langle G_1 \cup G_2 \rangle,X)$, where $\langle G_1 \cup G_2 \rangle$ is the group generated by $G_1 \cup G_2$. In order to see this, suppose the classes of $\sim_1$, $\sim_2$ in Figure~\ref{partition} are respectively the orbits of $(G_1,X)$ and $(G_2,X)$. Then we may obtain the five over-counted classes as orbits of $(\langle G_1 \cup G_2 \rangle,X)$, for suppose $\sigma \in G_1$ and $\tau \in G_2$. Then $\sigma \tau$ cannot take us beyond any of the boundaries of the over-counted orbits, however, one of $\sigma \in G_1$ or $\tau \in G_2$, would be able to move us anywhere within any of these orbits, even if partitioned by orbits of the other group.

Now in our case, we have the groups $G_1$, $G_2$ and $G_3$ (subgroups of $\mathbb{Z}^\ast_{27}\times \mathbb{Z}^\ast_{27}\times \mathbb{Z}^\ast_{27}$ under multiplication modulo 27). These are defined as follows

\begin{equation*}
\begin{split}
G_1&=\lbrace (a,b,c): a\equiv b\bmod 9 \mbox{ and } b\equiv c\bmod 3\rbrace\\
G_2&=\lbrace (a,b,c): a\equiv b\bmod 3\rbrace\\
G_3&=\lbrace (a,b,c): b\equiv c\bmod 3\rbrace\\
\end{split}
\end{equation*}
These groups are abelian, therefore $\langle G_1 \cup G_2 \rangle$ for example, gives all products of an element from $G_1$ with an element from $G_2$. This also includes the groups themselves, since each contains the identity. ($G_1$, $G_2$ and $G_3$ are respectively identified by \verb!grp1!, \verb!grp2! and \verb!grp3! in our program).

What follows, gives an overview of how we go about determining \eqref{eq:erl2} in GAP. The reader is referred to Appendix B for the actual program.

\vspace{3mm}
We start off by defining the sets $Y_0, Y_1, Y_2, Y^\ast_0, Y^{\ast\ast}_0, Y^\ast_1$ as defined in section 4.1.1. These are identified in GAP by \verb!uy0, uy1, uy2, uystar0, uystarstar0!, and \verb!uystar1! respectively. We then define the set \verb!alltriplezstar27! of all possible combinations of multipliers, that is, this set contains all triples $\lbrack a,b,c\rbrack$ for $a,b,c \in \mathbb{Z}^\ast_{27}$. Now suppose that the opposite extreme of \'{A}d\'{a}m's conjecture holds, that is, isomorphism can be through any independent multipliers from $\mathbb{Z}^\ast_{27}$ acting on the three layers, without any invariance or non-invariance condition. Then we could find the whole cycle index of $\mathbb{Z}^\ast_{27}$ separately on $Y_0, Y_1$ and $Y_2$ and multiply the three cycle indices. Alternatively we can use \verb!alltriplezstar27! to get the lists of cycle indices \verb!c1, c2, c3! on the three separate layers. In the latter case, we make use of the first component, \verb!c[1]!, of \verb!alltriplezstar27! (which is $m_0$) for the list \verb!c1!, the second component, \verb!c[2]! for \verb!c2! and the third, \verb!c[3]! for \verb!c3!. We then carry out three ``inner products" of \verb!c1, c2! and \verb!c3!. This makes sure that we combine actions of multipliers as allowed by the restrictions (this will be necessary for the other cases). The result is a list of individual cycle indices of permutations, and as such this is then transformed into one cycle index by taking a final inner product with the vector $\lbrack 1,1,...1\rbrack$. Since \verb!alltriplezstar27! has all possible combinations of cycle indices, this has the same effect as multiplying the three cycle indices.

\vspace{3mm}
Now suppose that we have a restriction as in $A_1$. Then we first have to filter \verb!alltriplezstar27! in order to get only those triples of multipliers which satisfy these restrictions. Therefore we define a function, identified by \verb!cond1! in GAP, which satisfies the conditions on the multipliers (in this case $E_{00}$). This function is defined in GAP as follows
\begin{spverbatim}
cond1:= function(a) return (\mod(a[2],9) = \mod(a[1],9)) and
(\mod(a[3],3) = \mod(a[2],3)); end;
\end{spverbatim}

The set \verb!alltriplezstar27! is then filtered according to this condition, by using the command
\[\verb!grp1 := Filtered(alltripleszstar27, cond1);!\]
The resulting group, grp1, will then be used as above instead of the group \verb!alltriplezstar27!.

Moreover, in the case of $A_1$ say, we have the condition $R_{00}$. Therefore, in order to enumerate, we first obtain the lists of cycle indices when there is no invariance condition, that is, we act with grp1 on $Y_0, Y_1, Y_2$. This will be given by A1.1. We shall then consider the invariance condition $\neg{R_{00}}$, that is, we now act with grp1 on $Y^\ast_0,Y_1$ and $Y_2$. This will be given by A1.2. $A_1$ will then be obtained by determining $A1.1-A1.2$. A similar procedure is adopted for $A_2$ and $A_3$.

\vspace{3mm}
Let us now consider the intersections. The intersection $|A_1\cap A_2|$, will be denoted by $A_{12}$ and the group effecting the action by grp12. Similarly $|A_1\cap A_3|$ is denoted by $A_{13}$ and the group by grp13, $|A_2\cap A_3|$ by $A_{23}$ and grp23 and $|A_1\cap A_2\cap A_3|$ by $A_{123}$ and grp123. As described previously, the group \verb!grp12!, is that generated by $\mbox{ grp1 }\cup \mbox{ grp2 }$, $\langle \mbox{ grp1 } \cup \mbox{ grp2 }\rangle$. This group is created in GAP as follows.

Suppose we have the lists of triples \verb!grp1! and \verb!grp2!, obtained by filtering \verb!alltipleszstar27! according to the conditions on the multipliers. Define a list, \verb!grp12!, to be an empty list. We then use two \verb!for! loops, to fill this list as follows
\begin{verbatim}
grp12:=[];
for a in grp1 do
for b in grp2 do
c:=[1,1,1];
c[1]:=\mod(a[1]*b[1],27);
c[2]:=\mod(a[2]*b[2],27);
c[3]:=\mod(a[3]*b[3],27);
Add(grp12,c);
od;
od;
\end{verbatim}
where the list \verb!c! is a triple to be included in the group we require. This list has as its first component, \verb!c[1]!, the result modulo 27, of the first component of a triple in \verb!grp1!, \verb!a[1]!, multiplied by the first component of a triple in \verb!grp2!, \verb!b[1]!. The second and third components of the list \verb!c! are obtained by taking the second and third components of the triples in \verb!grp1! and \verb!grp2! respectively.

In order to include all the lists \verb!c! in our empty list \verb!grp12! we use the command \verb!Add!, which takes our empty list as its first argument, and all lists \verb!c! to be added, as its second argument. This command would however create duplicates. In order to eliminate these duplicates we use the command \verb!DuplicateFreeList!. The resulting list of lists, gives the group $\langle \mbox{ grp1 } \cup \mbox{ grp2 }\rangle$.

%When considering the intersections, we may be restricted with more than one non-invariance condition. For example in the case $A_{12}$, the non-invariance conditions, $R_{00}$ and $R_{01}$ must hold. Therefore we again consider the case when there is no invariance condition (given by $A12.1$ in our program), and subtract from it the case $\neg(R_{00} \mbox{ and } R_{01}$) (given by A12.2). In the latter case, we may simplify by using de Morgan's rules as follows:

%\begin{equation*}
%\begin{split}
%|\neg(R_{00} \mbox{ and } R_{01})|&=|(\neg R_{00} \mbox{ or } \neg R_{01})|\\
%&=|\neg R_{00}|+|\neg R_{01}|-|\neg R_{00} \cap \neg R_{01}|\\
%\end{split}
%\end{equation*}
%but $\neg R_{01}\Rightarrow \neg R_{00}$, therefore

%\begin{equation*}
%|\neg(R_{00} \mbox{ and } R_{01})|=|\neg R_{00}|
%\end{equation*}

The groups required for $A_{13}$, $A_{23}$ and $A_{123}$ need not be determined in the same way as \verb!grp12!. If we consider \verb!grp13! for example, since \verb!grp1!$\leq$\verb!grp3!, we have that \verb!grp13!$=$\verb!grp3! immediately. One must keep in mind however, that the sets being acted on are different in this case. Similarly \verb!grp123!=\verb!grp23!. Moreover, \verb!grp23!, is in fact equal to the set of all triples, \verb!alltripleszstar27!.

This method gave the same generating function and final result, that of 928 non-isomorphic undirected circulants, as the method described earlier, which made use of Table~\ref{table:Table3}.

In the next chapter, we verify this result again by using a different technique, namely one involving Schur rings.

\chapter{The Structural Method}
\section{The General Method Described}

The structural approach is based on the lattice $\mathcal{L}(n)$ of all Schur rings over $\mathbb{Z}_n$. This, together with information of their automorphism groups, suffices to carry out the enumeration. Therefore, in order to be able to use this method, determination of all Schur rings over $\mathbb{Z}_n$ and the description of the lattice are essential. In addition, the following conditions need to be satisfied:
\begin{itemize}
\item Every $\mathcal{S}$-ring is a transitivity module of a suitable overgroup $(G,\mathbb{Z}_n)$ of the regular group $(\mathbb{Z}_n,\mathbb{Z}_n)$, i.e. each $\mathcal{S}$-ring in the lattice is Schurian;
\item All $\mathcal{S}$-rings from the lattice are pairwise non-isomorphic.
\end{itemize}

Muzychuk \cite{Muzy94} proved that the second condition is satisfied for all values of $n$. Owing to the fact that a $p$-group is Schurian if and only if it is cyclic \cite{KLP96}, we have that the first condition is satisfied for $n=p^{m}, m\geq 1$ where $p$ is prime. This property also holds in the case of square-free $n$ and may in fact be valid for all values of $n$ \cite{KLP96}.
With these assumptions, one may use the enumeration scheme which follows. This scheme has already been described in \cite{KLP96}, but will be reproduced in this thesis for completeness purposes.
\begin{itemize}
\item We will first need to construct the lattice $\mathcal{L}(n)$ of all Schur rings as a sequence $\mathcal{L}(n)=(\mathfrak{S}_{1},\mathfrak{S}_{2},...\mathfrak{S}_{s})$ such that $\mathfrak{S}_{j}\subseteq\mathfrak{S}_{i}$ implies $j\leq i$;
\item The number of $r$-element basis sets of the $\mathcal{S}$-ring $\mathfrak{S}_{i}$, different from the basis set $T_{0}=\lbrace 0 \rbrace$ is given by

\begin{equation}
\tilde{d}_{ir}:=\vert \lbrace T_{(x)}\in \mathfrak{S}_{i}\vert \hspace{0.7mm} x\neq 0 \hspace{0.7mm} \mbox{ and }\hspace{0.7mm} \vert T_{(x)}\vert=r \rbrace\vert
\end{equation}

\item For undirected circulants, we will require the number of $r$-element symmetrized (that is closed under taking of inverses) basis sets of $\mathfrak{S}_{i}$, different from $T_{0}$. Let this be
\begin{equation}
d_{ir}:=\vert\lbrace T_{(x)}^{sym} \vert \hspace{0.7mm} x\neq 0 \hspace{0.7mm} \mbox { and } \hspace{0.7mm} \vert T_{(x)}^{sym}\vert=r \rbrace\vert
\end{equation}
\item Enumeration of all labelled directed and undirected circulant graphs which belong to the Schur ring $\mathfrak{S}_{i}$ may then be carried out by making use of generating functions $\tilde{f}_{i}(t)$ and $f_i(t)$ respectively, given by:

\begin{equation}\label{eq:erl}
\begin{split}
\tilde{f}_{i}(t)&:=\sum_{r=0}^{n-1}\tilde{f}_{ir}t^{r}:=\prod_{r=1}^{n-1}(1+t^{r})^{\tilde{d}_{ir}}\\
f_i(t)&:=\sum_{r=0}^{n-1}f_{ir}t^{r}:=\prod_{r=1}^{n-1}(1+t^{r})^{d_{ir}}
\end{split}
\end{equation}
\end{itemize}

Substituting $t=1$ in the generating functions, would give us the number of all labelled directed and undirected circulant graphs in $\mathfrak{S}_{i}$. In addition, the graph corresponding to $T \in \mathfrak{S}_{i}$ is of valency $r$ if $T$ has $r$ elements.

\begin{Lemma}[\cite{KLP96}]
Let $G_{i}= \mbox{ Aut }(\mathfrak{S}_{i})$, let $N(G_{i})=N_{S_{n}}(G_{i})$ be the normalizer of the group $G_{i}$ in $S_{n}$, and let $\Gamma$ be a circulant graph belonging to $\mathfrak{S}_{i}$. Then

\begin{itemize}
\item[\text{(a)}] $\mbox{ Aut }(\Gamma)=G_{i}\Longleftrightarrow \Gamma$ generates $\mathfrak{S}_{i}$.
\item[(b)] If $\mbox{ Aut }(\Gamma)=G_{i}$ then there are exactly $\lbrack N(G_{i}):G_{i} \rbrack$ (that is, equal to the number of cosets of $G_i$ in $N(G_{i})$) distinct circulant graphs which are isomorphic to $\Gamma$.
\end{itemize}
\end{Lemma}

Let the generating function for the number of pairwise non-isomorphic undirected circulant graphs with automorphism group $G_{i}$, be given by
\begin{equation*}
g_i(t)=\sum_{r=0}^{n-1}g_{ir}t^{r}
\end{equation*}
whereas the generating function for the number of pairwise non-isomorphic directed circulant graphs with automorphism group $G_{i}$ is
\begin{equation*}
\tilde{g}_{i}(t)=\sum_{r=0}^{n-1}\tilde{g}_{ir}t^{r}
\end{equation*}
Moreover,
\begin{equation*}
g(t)=g(n,t) \mbox{ and } \tilde{g}(t)=\tilde{g}(n,t)
\end{equation*}
will denote the generating functions for the number of pairwise non-isomorphic undirected and directed circulant graphs, respectively, with $n$ vertices. $g(1)$ and $\tilde{g}(1)$ give the numbers of all non-isomorphic undirected and directed circulant graphs respectively, with $n$ vertices. Here if $G$ is a subpermutation group of $S_n$, then $N(G)$ denotes the normaliser of $G$ in $S_n$ \cite{KLP96}.

\begin{Theorem}[\cite{KLP96}]
\label{thm:Theorem5}
\begin{equation}\label{eq:erl1}
\begin{split}
g_{i}(t)&=\frac{\vert G_{i} \vert}{\vert N(G_i)\vert} \left( f_{i}(t)-\sum_{\mathfrak{S}_{j} \subseteq \mathfrak{S}_{i}}\frac{\vert N(G_{j})\vert}{\vert G_{j} \vert}g_{j}(t)\right),\\
\tilde{g}_{i}(t)&=\frac{\vert G_{i} \vert}{\vert N(G_i)\vert}\left(\tilde{f}_{i}(t)-\sum_{\mathfrak{S}_{j}\subseteq\mathfrak{S}_{i}}\frac{\vert N(G_{j})\vert}{\vert G_{j} \vert}\tilde{g}_{j}(t)\right),\\
g(t)&=\sum_{i=1}^{s}g_{i}(t), \hspace{3mm} \tilde{g}(t)=\sum_{i=1}^{s}\tilde{g}_{i}(t).
\end{split}
\end{equation}
\end{Theorem}
For a proof of this theorem, the reader is referred to \cite{KLP96}. This proof is based on the principle of inclusion and exclusion.

\section{Simple Examples}
\subsection{$\MakeLowercase{n=6}$}

Let us consider the example with $n=6$. The following is a list of all Schur rings, $(\mathfrak{S}_{1},\mathfrak{S}_{2},...\mathfrak{S}_{6})$, over $\mathbb{Z}_6$ as given in \cite{KLP96}:
\begin{equation*}
\begin{split}
\mathfrak{S}_{1}&=\langle\underline{0},\underline{1,2,3,4,5}\rangle,\\
\mathfrak{S}_{2}&=\langle\underline{0},\underline{1,2,4,5},\underline{3}\rangle,\\
\mathfrak{S}_{3}&=\langle\underline{0},\underline{1,3,5},\underline{2,4}\rangle,\\
\mathfrak{S}_{4}&=\langle\underline{0},\underline{1,5},\underline{2,4},\underline{3}\rangle,\\
\mathfrak{S}_{5}&=\langle\underline{0},\underline{1,3,5},\underline{2},\underline{4}\rangle,\\
\mathfrak{S}_{6}&=\langle\underline{0},\underline{1},\underline{2},\underline{3},\underline{4},\underline{5}\rangle.
\end{split}
\end{equation*}
One must note that for $\mathfrak{S}_{j} \subseteq\mathfrak{S}_{i}$, the automorphism groups $G_{i}$ and $G_{i}$ satisfy $G_{j}\geq G_{i}$. In the above list of $\mathcal{S}$-rings, $\mathfrak{S}_{1}$ is the coarsest, having the largest automorphism group, whereas $\mathfrak{S}_{6}$ is the finest with the smallest automorphism group. This means $\mathfrak{S}_{6}$ contains all the other Schur rings. In figure~\ref{Fig3} we have the lattice of all $\mathcal{S}$-rings.

\begin{figure}[!h]
\centering
\includegraphics[width=0.3\linewidth]{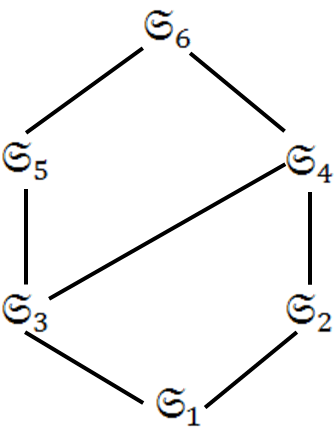}
\caption{Lattice of all $\mathcal{S}$-rings over $\mathbb{Z}_6$}
\label{Fig3}
\end{figure}

Using \eqref{eq:erl} we obtain the following:
\begin{quote}
\centering
\begin{tabular}{ll}
$f_{1}(t)=1+t^{5}$,&$\tilde{f}_{1}(t)=1+t^{5}$,\\
$f_{2}(t)=(1+t)(1+t^4)$,&$\tilde{f}_{2}(t)=(1+t)(1+t^{4})$,\\
$f_{3}(t)=(1+t^{2})(1+t^{3})$,&$\tilde{f}_{3}(t)=(1+t^{2})(1+t^{3})$,\\
$f_{4}(t)=(1+t^{2})^{2}(1+t)$,&$\tilde{f}_{4}(t)=(1+t^{2})^{2}(1+t)$,\\
$f_{5}(t)=(1+t^{2})(1+t^{3})$,&$\tilde{f}_{5}(t)=(1+t)^{2}(1+t^{3})$,\\
$f_{6}(t)=(1+t)(1+t^{2})^{2}$,&$\tilde{f}_{6}(t)=(1+t)^{5}$,\\
\end{tabular}
\end{quote}

The following is a list of automorphism groups and their corresponding normalizers.

\begin{quote}
\centering
\begin{tabular}{l l}
Automorphism Group&Normalizer\\
$G_{1}=S_{6}$,&$N(G_{1})=G_{1}$,\\
$G_{2}=S_{3}\wr S_{2}$,&$N(G_{2})=G_{2}$,\\
$G_{3}=S_{2}\wr S_{3}$,&$N(G_{3})=G_{3}$,\\
$G_{4}=D_{6}$,&$N(G_{4})=G_{4}$,\\
$G_{5}=S_{2}\wr\mathbb{Z}_{3}$,&$\lbrack N(G_{5}):G_{5}\rbrack=2$,\\
$G_{6}=\mathbb{Z}_{6}$,&$\lbrack N(G_{6}):G_{6}\rbrack=2$.
\end {tabular}
\end{quote}
These groups and their respective normalizers, may be determined using the software package COCO. Some of these may even be determined visually since for example $G_{1}=$Aut$(\mathfrak{S}_1)=\bigcap_{T_{i}\in \mathfrak{S}_{i}}$Aut$(Cay(\mathbb{Z}_6,T_{i}))\hspace{0.7mm}$. Cayley graphs for $\mathfrak{S}_{1}, \mathfrak{S}_{2}$ and $\mathfrak{S}_{4}$ are shown in figure~\ref{Fig4}. From these one may note that the automorphism group of $\mathfrak{S}_{1}$ is given by the symmetric group $S_{6}$, $G_{2}$ is given by the wreath product of $S_{3}$ with $S_{2}$, while $G_{4}$ is given by the dihedral group $D_{6}$.

\begin{figure}[h!]
\centering
\includegraphics[width=0.75\linewidth]{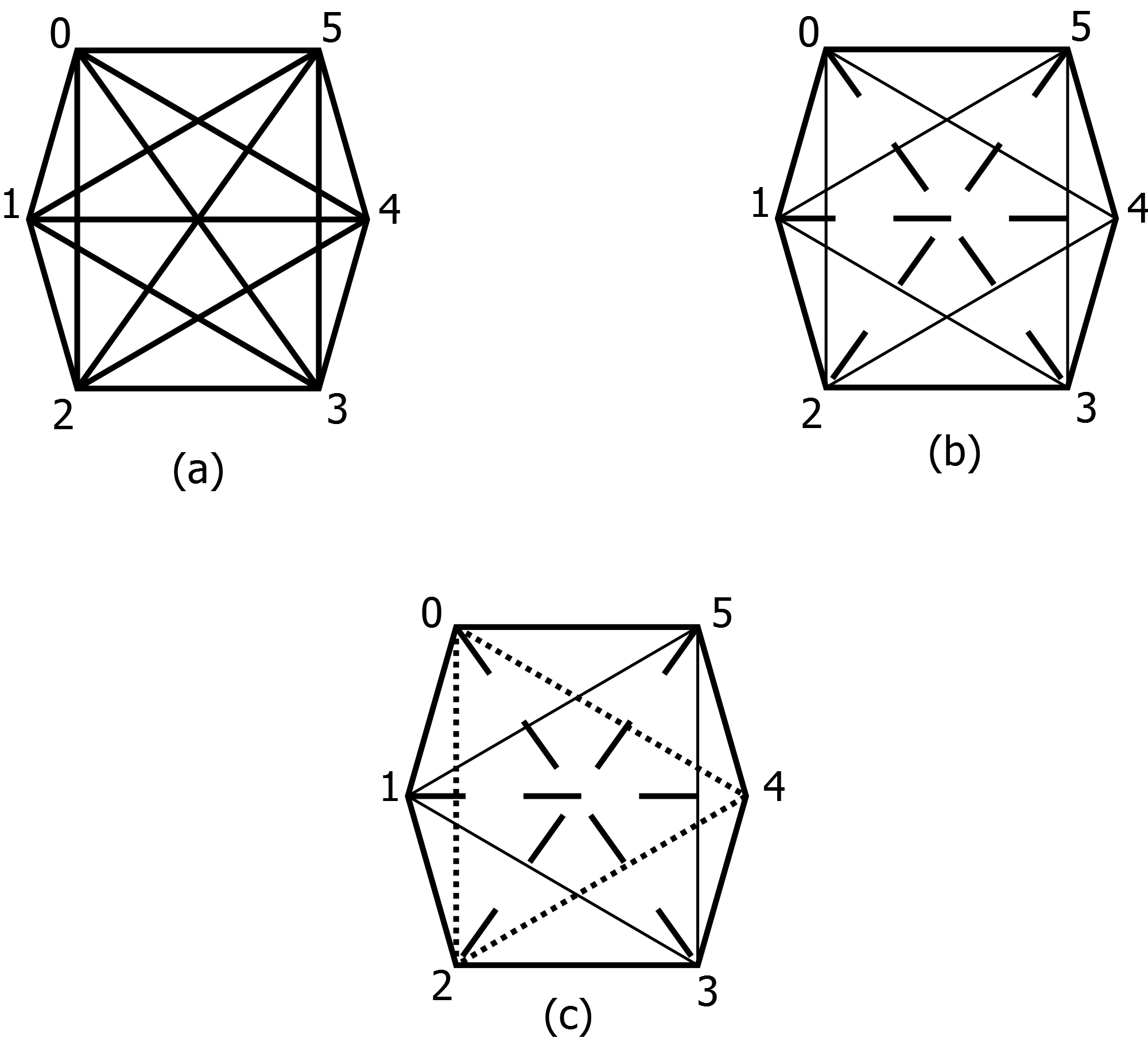}
\caption{(a) Cay$(\mathbb{Z}_6, T_{i})\hspace{0.7mm}, T_{i} \in \mathfrak{S}_{1}$, (b) Cay$(\mathbb{Z}_6, T_{i}) \hspace{0.7mm}, T_{i} \in \mathfrak{S}_{2}$,\\
\newline
(c) Cay$(\mathbb{Z}_6, T_{i}) \hspace{0.7mm}, T_{i} \in \mathfrak{S}_{4}$}
\label{Fig4}
\end{figure}

We will now determine $g_{i}(t)$ for $i=1,2,...,6$, using \eqref{eq:erl1} and figure~\ref{Fig3}. If we consider $\mathfrak{S}_{1}$, this $\mathcal{S}$-ring contains no other Schur rings according to the lattice. Therefore the summation in our equation will reduce to zero and we have $g_{1}(t)=\frac{\vert G_{1} \vert}{\vert N(G_1)\vert} \left( f_{1}(t)\right)$. As stated above, $G_{1}=S_{6}$, therefore $\vert G_{1} \vert=6!$ also. Since $N(G_{1})=G_{1}$, then $\vert N(G_{1})\vert =6!$. Therefore $g_{1}(t)=f_{1}(t)=1+t^{5}$.

Let us now consider $\mathfrak{S}_{2}$. According to the lattice, this $\mathcal{S}$-ring contains $\mathfrak{S}_{1}$. Therefore the summation in the equation will now contain a term in $g_{1}$.

\[g_{2}(t)=\frac{\vert G_{2} \vert}{\vert N(G_2)\vert} \left( f_{2}(t)-\frac{\vert N(G_{1}) \vert}{\vert G_1\vert}g_{1}(t)\right)\]

Since $N(G_{1})=G_{1}$ and $N(G_{2})=G_{2}$, the ratios $\frac{\vert N(G_{1}) \vert}{\vert G_1\vert}$ and $\frac{\vert G_{2} \vert}{\vert N(G_2)\vert}$ are both equal to one. Therefore we obtain
\[g_{2}(t)=f_{2}(t)-g_{1}(t)=(1+t)(1+t^{4})-(1+t^{5})=t+t^4\]
Like $\mathfrak{S}_{2}$, $\mathfrak{S}_{3}$ only contains $\mathfrak{S}_{1}$. Therefore $g_{3}$ may be obtained in a similar fashion to $g_{2}$. In doing so we obtain

\[g_{3}(t)=f_{3}(t)-g_{1}(t)=t^{2}+t^{3}\]
The $\mathcal{S}$-ring $\mathfrak{S}_{4}$, contains $\mathfrak{S}_{1}$, $\mathfrak{S}_{2}$,and $\mathfrak{S}_{3}$, as such the summation in our equation will now contain terms in $g_{1}, g_{2}$ and $g_{3}$. In fact we have
\begin{equation*}
\begin{split}
g_{4}(t)&=\frac{\vert G_{4} \vert}{\vert N(G_4)\vert} \left( f_{4}(t)-\frac{\vert N(G_{1}) \vert}{\vert G_1\vert}g_{1}(t)+\frac{\vert N(G_{2}) \vert}{\vert G_2\vert}g_{2}(t)+\frac{\vert N(G_{3}) \vert}{\vert G_3\vert}g_{3}(t)\right)
\end{split}
\end{equation*}
Now $\vert N(G_{1})\vert=\vert G_{1}\vert, \vert N(G_{2})\vert=\vert G_{2}\vert, \vert N(G_{3})\vert=\vert G_{3}\vert, \vert N(G_{4})\vert=\vert G_{4}\vert$. Therefore

\begin{equation*}
\begin{split}
g_{4}(t)&=f_{4}(t)-(g_{1}+g_{2}+g_{3})\\
g_{4}(t)&=(1+t^{2})^{2}(1+t)-(1+t^{5}+t+t^{4}+t^{2}+t^{3})\\
&=t^{2}+t^{3}
\end{split}
\end{equation*}

The same method may be repeated for $g_{5}$ and $g_6$. These give:

\begin{equation*}
\begin{split}
g_{5}(t)&=\frac{1}{2}(f_{5}(t)-g_{1}(t)-g_{3}(t))=0\\
g_{6}(t)&=\frac{1}{2}(f_{6}(t)-g_{1}(t)-g_{2}(t)-g_{3}(t)-g_{4}(t)-2g_{5}(t))=0
\end{split}
\end{equation*}

Therefore we have
\begin{equation*}
\begin{split}
g(t)&=g_{1}+g_{2}+g_{3}+g_{4}+g_{5}+g_{6}\\
&=1+t^5+t+t^4+t^2+t^3+t^2+t^3+0+0\\
&=1+t+2t^2+2t^3+t^4+t^5
\end{split}
\end{equation*}

The generating function for the number of pairwise non-isomorphic directed circulant graphs $\tilde{g}(t)$ may be determined using the same method described above. This is given by \cite{KLP96}:

\[\tilde{g}(t)=1+3t+6t^2+6t^3+3t^4+t^5.\]
Substituting $t=1$ in both generating functions, we obtain the number of all non-isomorphic undirected and directed circulant graphs with 6 vertices. This gives $g(1)=8$ and $\tilde g(1)=20$.

\subsection{$\MakeLowercase{n=13}$}

Now let us consider the case when $n=p$, where $p$ is a prime number. The list of Schur rings may be found as follows. Consider the multiplicative group $\mathbb{Z}^{\ast}_{p}$ and let $H=H_{d}$ be a subgroup of order $d$. Let the set of distinct cosets of the subgroup $H$ in $\mathbb{Z}^{\ast}_{p}$ be $H_1, H_2,...H_r$ where $r=\frac{p-1}{d}$. Then the basic sets of the Schur ring $\mathfrak{S}_{d}$ over $\mathbb{Z}_{p}$ are $\langle \underline{0}, \underline{H},\underline{H_1},...\underline{H_r}\rangle$. Each subgroup gives a Schur ring which is not isomorphic to any other $\mathcal{S}$-ring, thus there are as many $\mathcal{S}$-rings as there are subgroups $H$. Since $\mathbb{Z}^{\ast}_{p}$ is cyclic, of order $p-1$, there is exactly one subgroup for every divisor $d$ of $p-1$ \cite{KLP96}.

\theoremstyle{definition}
\begin{Theorem}[\cite{KLP96}]
\label{thm:theorem6}
\textnormal{ }\\
\begin{itemize}
\item[(i)] \itemEq{\mbox{ Each $\mathcal{S}$-ring over } \mathbb{Z}_{p} \mbox{ coincides with some } \mathfrak{S}_{d} \hspace{0.7mm}(d\vert(p-1)).\notag}
\item[(ii)] \itemEq {\mbox{ Aut }(\mathfrak{S}_{d})=\left\{
\begin{array}{ll}
\mathbb{Z}_{p} \rtimes H_{d} &\mbox{ if } d <p-1\\
S_{p} & \mbox { if } d=p-1.
\end{array} \right.\notag}
\item[(iii)] \itemEq {\mbox{ N((Aut }(\mathfrak{S}_{d}))=\left\{
\begin{array}{ll}
\mathbb{Z}_{p} \rtimes \mathbb{Z}^{\ast}_{p} &\mbox{ if } d<p-1\\
S_{p} &\mbox{ if } d=p-1.
\end{array} \right.\notag}
\item[(iv)] \itemEq {\lbrack \mbox {N(Aut }(\mathfrak{S}_{d})):\mbox{ Aut }(\mathfrak{S}_{d})\rbrack=\frac{p-1}{d}.\notag}
\end{itemize}
\end{Theorem}

From this theorem it can be proved that Adam's conjecture holds for all prime numbers $n=p$ \cite{FIK90}.

Let us suppose $n=13$. The divisors of $p-1$, i.e. of 12 are 1, 2, 3, 4, 6 and 12. Therefore we have 6 Schur rings. The subgroup $H=\lbrace 1 \rbrace$, returns the $\mathcal{S}$-ring $\mathfrak{S}_{1}$ given by
\[\mathfrak{S}_{1}=\langle \underline{0}, \underline{1},\underline{2},\underline{3},\underline{4},\underline{5},\underline{6},\underline{7},\underline{8},\underline{9},\underline{10},\underline{11},\underline{12}\rangle \]

The only subgroup $H$ of order 2, is $\lbrace 1,a\rbrace$ where $a^2=1\bmod 13$. This gives $a=12$ since $144=1+143=1+11(13)$. Therefore the basic subsets of the $\mathcal{S}$-ring $\mathfrak{S}_{2}$, given by the cosets of the subgroup $\lbrace 1, 12 \rbrace$ in $\mathbb{Z}^\ast_{13}$ are
\[\mathfrak{S}_{2}=\langle \underline{0}, \underline{1,12},\underline{2,11},\underline{3,10},\underline{4,9},\underline{5,8},\underline{6,7}\rangle. \]

The subgroup $H$ of order 3, is $\lbrace 1,a,a^{2}\rbrace$ where $a^3=1\bmod 13$. This gives $a=3$ since $27=1+26=1+2(13)$. Therefore $\mathfrak{S}_{3}$ is given by:
\[\mathfrak{S}_{3}=\langle \underline{0}, \underline{1,3,9},\underline{2,6,5},\underline{4,12,10},\underline{7,8,9}\rangle. \]

Repeating the above procedure one can obtain the remaining Schur rings. These are
\[\mathfrak{S}_{4}=\langle \underline{0}, \underline{1,5,8,12},\underline{2,3,10,11},\underline{4,7,6,9}\rangle, \]
\[\mathfrak{S}_{6}=\langle \underline{0}, \underline{1,3,4,9,10,12},\underline{2,5,6,7,8,11}\rangle, \]
\[\mathfrak{S}_{12}=\langle \underline{0}, \underline{1,2,3,4,5,6,7,8,9,10,11,12}\rangle. \]

\vspace{5mm}
Using the method described previously in $n=6$, we can obtain the generating functions $f_{i}(t)$ and $\tilde{f}_{i}(t)$ for undirected and directed circulants respectively and evaluate $f_{i}(1)$ and $\tilde{f}_{i}(1)$ accordingly. These are as follows:

\begin{quote}
\centering
\begin{tabular}{ll}
$f_{1}(t)=(1+t^{2})^{6}=2^6$,&$\tilde{f}_{1}(t)=(1+t)^{12}=2^{12}$,\\
$f_{2}(t)=(1+t^{2})^{6}=2^6$,&$\tilde{f}_{2}(t)=(1+t^{2})^{6}=2^6$,\\
$f_{3}(t)=(1+t^{6})^{2}=2^2$,&$\tilde{f}_{3}(t)=(1+t^{3})^4=2^4$,\\
$f_{4}(t)=(1+t^{4})^{3}=2^3$,&$\tilde{f}_{4}(t)=(1+t^{4})^{3}=2^3$,\\
$f_{6}(t)=(1+t^{6})^{2}=2^2$,&$\tilde{f}_{6}(t)=(1+t^{6})^{2}=2^2$,\\
$f_{12}(t)=(1+t^{12})=2^1$,&$\tilde{f}_{12}(t)=(1+t^{12})=2^1$,\\
\end{tabular}
\end{quote}

Figure~\ref{Fig5} shows the lattice of Schur rings. Note here, that unlike the case $n=6$, the Schur ring with the largest automorphism group, has the largest index whereas that with the smallest automorphism group, has the smallest index.
\vspace{5mm}
\begin{figure}[h!]
\centering
\includegraphics[width=0.3\linewidth]{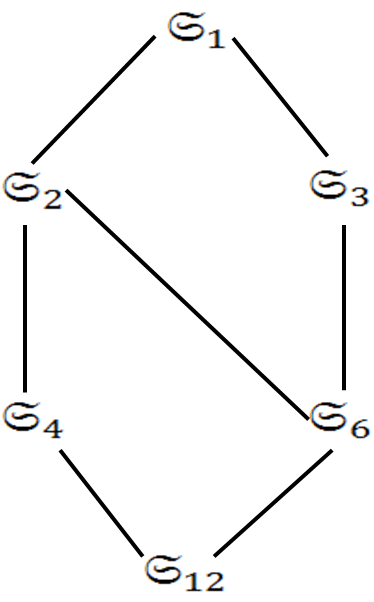}
\caption{Lattice of all $\mathcal{S}$-rings over $\mathbb{Z}_{13}$}
\label{Fig5}
\end{figure}

Using the computer package GAP and/or Theorem~\ref{thm:theorem6}, one can obtain the following sizes for the automorphism groups and their normalizers:

\begin{quote}
\centering
\begin{tabular}{|c|c|c|c|}
\hline
Schur Ring&Size of Automorphism Group&Size of Normalizer&$\frac{\vert G_{i}\vert}{\vert N(G_{i})\vert}$\\
\hline
$\mathfrak{S}_{1}$&13&156&$\frac{1}{12}$\\
$\mathfrak{S}_{2}$&26&156&$\frac{1}{6}$\\
$\mathfrak{S}_{3}$&39&156&$\frac{1}{4}$\\
$\mathfrak{S}_{4}$&52&156&$\frac{1}{3}$\\
$\mathfrak{S}_{6}$&78&156&$\frac{1}{2}$\\
$\mathfrak{S}_{12}$&6227020800&6227020800&$\frac{1}{1}$\\
\hline
\end{tabular}
\end{quote}

\vspace{3mm}
Since
\[\mbox{ Aut }\mathfrak{S}=\bigcap^{r-1}_{i=0}\mbox { Aut }\Gamma_i\]
the sizes of the automorphism groups of the Schur rings and their normalizers, are obtained in GAP as follows. We first create the group generated by $a=(1,...,13)$ and determine the connecting sets $T_1,...,T_{13}$, by taking into consideration the simple quantities in that particular Schur ring. For example for $\mathfrak{S}_4$, we have: $T_1:=[a,a^5,a^8,a^{12}]$, $T_2:=[a^2,a^{3},a^{10},a^{11}]$ and $T_3:=[a^4,a^7,a^6,a^9]$.

The Cayley graphs are then created and the automorphism group for each Cayley graph is determined. The intersection of all automorphism groups then gives the automorphism group of the Schur ring being considered, as such we simply have to determine the size of the intersection of all automorphism groups. The size of the normalizer is then found using the appropriate command. For example for $\mathfrak{S}_4$ we have:

\begin{spverbatim}
a:=(1,2,3,4,5,6,7,8,9,10,11,12,13);;
grp:=Group([a]);;

T1:=[a,a^5,a^8,a^{12}];;
T2:=[a^2,a^3,a^{10},a^{11}];;
T3:=[a^4,a^7,a^6,a^9];;

g1:=CayleyGraph(grp,T1);;
g2:=CayleyGraph(grp,T2);;
g3:=CayleyGraph(grp,T3);;

grp1:=AutomorphismGroup(g1);;
grp2:=AutomorphismGroup(g2);;
grp3:=AutomorphismGroup(g3);;

grpall:=Intersection(grp1,grp2,grp3);;
Size(grpall);
Size(Normalizer(SymmetricGroup(13),grpall));
\end{spverbatim}

\vspace{5mm}
For other Schur rings that have connection sets which are not closed under inversion, the command \verb!CayleyGraph(grp,Ti)! should be replaced by \[\verb!CayleyGraph(grp,Ti,false)!.\]
By default, \verb!CayleyGraph(grp,Ti,true)! (or \verb!CayleyGraph(grp,Ti)!), returns an undirected Cayley graph, therefore to specify a directed Cayley Graph, the argument ``false" needs to be included.

\vspace{5mm}
Now using \eqref{eq:erl1} and the lattice, we can obtain the functions $g_{i}(t)$ and $\tilde{g}_{i}(t)$.

\begin{equation*}
\begin{split}
g_{12}(1)&=\frac{1}{1}(f_{12}(1))=\frac{1}{1}\cdot 2=2\\
g_{6}(1)&=\frac{1}{2}(f_{6}(1)-\frac{1}{1}g_{12}(1))\\
&=\frac{1}{2}(2^2-2)=1\\
g_{4}(1)&=\frac{1}{3}(f_{4}(1)-\frac{1}{1}g_{12}(1))\\
&=\frac{1}{3}(2^3-2)=2\\
g_{3}(1)&=\frac{1}{4}(f_{3}(1)-\frac{1}{1}g_{12}(1)-\frac{2}{1}g_{6}(1))\\
&=\frac{1}{2}(2^2-1(2)-2(1))=0\\
g_{2}(1)&=\frac{1}{6}(f_{2}(1)-\frac{1}{1}g_{12}(1)-\frac{2}{1}g_{6}(1)-\frac{3}{1}g_{4}(1))\\
&=\frac{1}{6}(2^6-1(2)-2(1)-3(2))=9\\
g_{1}(1)&=\frac{1}{12}(f_{1}(1)-\frac{1}{1}g_{12}(1)-\frac{2}{1}g_{6}(1)-\frac{3}{1}g_{4}(1)-\frac{4}{1}g_{3}(1)-\frac{6}{1}g_{2}(1))\\
&=\frac{1}{12}(2^6-1(2)-2(1)-3(2)-4(0)-6(9))=0
\end{split}
\end{equation*}

\vspace{3mm}
Therefore the number of non-isomorphic undirected circulants on 13 vertices is 14.

\section{The Structural approach for $\MakeLowercase{p^{2}}$ circulants}

\begin{Definition}[\cite{KLP96}]
A Schur ring $\mathfrak{S}=\langle \underline{T_{0}},\underline{T_{1}},...,\underline{T_{l}} \rangle$, $l \in \mathbb{Z}^+$, over $\mathbb{Z}_{n}$ is called \emph{wreath decomposable} if there exists a non-trivial proper subgroup $K \leq \mathbb{Z}_{n}$ such that for every basic set $T_{i}$, either $T_{i}\subseteq K$ or $T_{i}=\bigcup_{x\in T_{i}}(K+x)$. The $\mathcal{S}$-ring $\mathfrak{S}$ is called wreath indecomposable, if it is not wreath decomposable.
\end{Definition}

\begin{Definition}[\cite{KLP96}]
Let $n=p^{2}, K=\mathbb{Z}_{p}$. Let
\[\mathfrak{S}_{1}=\langle \underline{Q_{0}},\underline{Q_{1}},...,\underline{Q_{d}}\rangle \mbox { and } \mathfrak{S}_{2}=\langle \underline{R_{0}},\underline{R_{1}},...,\underline{R_{k}}\rangle\]
be $\mathcal{S}$-rings over $\mathbb{Z}_{p}$. Let
\begin{equation*}
\begin{split}
T_{1,i}&:=\lbrace px_{2}|x_{2}\in R_{i}\rbrace, \hspace{0.7mm} 0\leq i \leq k,\\
T_{2,i}&:=\lbrace x_{1}+px_{2}|x_{1}\in Q_{i}, \hspace{0.7mm} x_{2}\in\mathbb{Z}_{p}\rbrace, 1\leq i \leq d.
\end{split}
\end{equation*}
These sets are the basic sets of an $\mathcal{S}$-ring $\mathfrak{S}$ over $\mathbb{Z}_{p^2}$

\[\mathfrak{S}:=\langle \underline{T_1,_{0}},\underline{T_1,_{1}},...,\underline{T_1,_{k}},\underline{T_2,_{1}},...,\underline{T_2,_{d}}\rangle\]
which is called the \emph{wreath composition} \index{wreath composition} of $\mathfrak{S}_{1}$ and $\mathfrak{S}_{2}$ and is denoted by $\mathfrak{S}=\mathfrak{S}_{1}\lbrack \mathfrak{S}_{2}\rbrack$.
\end{Definition}

For the case $n=p^{2}$, we shall require the fact that a wreath decomposable $\mathcal{S}$-ring over some group $H$, may be represented as the wreath composition of suitable $\mathcal{S}$-rings $\mathfrak{S}_{1}$ and $\mathfrak{S}_{2}$ over groups of smaller order. Moreover, $\mbox{Aut}(\mathfrak{S}_{1}\lbrack\mathfrak{S}_{2}\rbrack)$ is the wreath product of the automorphism groups of the $\mathcal{S}$-rings $\mathfrak{S}_{1}$ and $\mathfrak{S}_{2}$.

\theoremstyle{definition}
\begin{Theorem}[\cite{KLP96}]
\label{thm:theorem13}
Let $\mathfrak{S}$ be a non-trivial $S$-ring over $\mathbb{Z}_{p^2}$. Then we have
\begin{itemize}
\item[(a)] Either
\hspace{30mm} \item[(i)]$\mathfrak{S}$ is a wreath composition $\mathfrak{S}=\mathfrak{S}_{1}\lbrack\mathfrak{S}_{2}\rbrack$ of $\mathcal{S}$-rings $\mathfrak{S}_{1}$, $\mathfrak{S}_{2}$ over $\mathbb{Z}_{p}$,\\
or
\hspace{30mm} \item[(ii)]$\mathfrak{S}=\langle\underline{0},\underline{H},\underline{y_{1}H},...\underline{y_{l}H}\rangle$, where $H\leq\mathbb{Z}^\ast_{p^2}, (1+p)\notin H, y_{i} \in \mathbb{Z}^\ast_{p^{2}}$ for $1\leq i\leq l$;
\item[(b)] \itemEq{\mbox{ Aut }(\mathfrak{S})=\left\{
\begin{array}{ll}
\mbox{Aut}(\mathfrak{S}_{1})\wr \mbox{ Aut }(\mathfrak{S}_{2}) &\mbox{  in case (i) }\\
\mathbb{Z}_{p^{2}} \rtimes H & \mbox{ in case (ii) };
\end{array} \right.\notag}
\item[(c)] \itemEq {\lbrack N(\mbox{Aut}(\mathfrak{S})): \mbox {Aut}(\mathfrak{S})\rbrack=\left\{
\begin{array}{ll}
\prod_{i=1}^{2}\lbrack N(\mbox{Aut}(\mathfrak{S}_{i}):\mbox{Aut}(\mathfrak{S}_{i})\rbrack &\mbox{ in case (i) }\\
\frac{(p-1)p}{|H|} & \mbox{ in case (ii) }.
\end{array} \right.\notag}
\end{itemize}
\end{Theorem}

\begin{Example}
Let us suppose $p=3$, i.e. $n=9$. Let $\mathfrak{S}_1$ and $\mathfrak{S}_2$ be two $\mathcal{S}$-rings of $\mathbb{Z}_{3}$, where
\begin{equation*}
\begin{split}
\mathfrak{S}_{1}&=\lbrace \lbrace0\rbrace,\lbrace 1,2 \rbrace \rbrace=\lbrace Q_{0},Q_{1}\rbrace\\
\mathfrak{S}_{2}&=\lbrace \lbrace0\rbrace,\lbrace 1\rbrace,\lbrace 2 \rbrace \rbrace=\lbrace R_{0},R_{1},R_{2} \rbrace
\end{split}
\end{equation*}
\end{Example}
Then by the definition above we have
\begin{equation*}
\begin{split}
T_{1,0}&=\lbrace 3x_{2}|x_{2}\in R_{0}\rbrace=\lbrace 0 \rbrace\\
T_{1,1}&=\lbrace 3x_{2}|x_{2}\in R_{1}\rbrace=\lbrace 3 \rbrace\\
T_{1,2}&=\lbrace 3x_{2}|x_{2}\in R_{2}\rbrace=\lbrace 6 \rbrace\\
\end{split}
\end{equation*}
and
\begin{equation*}
\begin{split}
T_{2,1}&=\lbrace x_{1}+px_{2}|x_{1} \in Q_{i}, \hspace{2mm} x_{2} \in \mathbb{Z}_{3}\rbrace\\
&=\lbrace 1+3\cdot0, 1+3\cdot1, 1+3\cdot2, 2+3\cdot0, 2+3\cdot1, 2+3\cdot2\rbrace\\
&=\lbrace1,4,7,2,5,8\rbrace.
\end{split}
\end{equation*}
Therefore, the wreath composition of $\mathfrak{S}_{1}$ and $\mathfrak{S}_{2}$ is
\[\mathfrak{S}_{1}\lbrack\mathfrak{S}_{2}\rbrack=\lbrace\lbrace 0\rbrace, \lbrace 3\rbrace, \lbrace 6\rbrace, \lbrace 1,2,4,5,7,8\rbrace\rbrace\]

Furthermore, by the definition above, if we take $K=3\mathbb{Z}_{3}$, or in general $K=p\mathbb{Z}_{p}$, one may note that $\mathfrak{S}$ is wreath decomposable, since $3\mathbb{Z}_{3}=\lbrace 0,3,6\rbrace$ is a subgroup of $\mathbb{Z}_{9}$ such that for every basic set $T_{i}$ of $\mathfrak{S}$, either $T_{i}\subseteq K$ or $T_{i}=\bigcup_{x\in T_{i}}(K+x)$.

In order to determine $\mathfrak{S}_{2}\lbrack\mathfrak{S}_{1}\rbrack$, let
\begin{equation*}
\begin{split}
\mathfrak{S}_{1}&=\lbrace \lbrace0\rbrace,\lbrace 1,2 \rbrace \rbrace=\lbrace R_{0},R_{1}\rbrace\\
\mathfrak{S}_{2}&=\lbrace \lbrace0\rbrace,\lbrace 1\rbrace,\lbrace 2 \rbrace \rbrace=\lbrace Q_{0},Q_{1},Q_{2} \rbrace
\end{split}
\end{equation*}
Thus we obtain
\begin{equation*}
\begin{split}
T_{1,0}&=\lbrace 3x_{2}|x_{2}\in R_{0}\rbrace=\lbrace 0 \rbrace\\
T_{1,1}&=\lbrace 3x_{2}|x_{2}\in R_{1}\rbrace=\lbrace 3,6 \rbrace\\
T_{2,1}&=\lbrace x_{1}+px_{2}|x_{1} \in Q_{i}, \hspace{2mm} x_{2} \in \mathbb{Z}_{3}\rbrace\\
&=\lbrace 1+3\cdot0, 1+3\cdot1, 1+3\cdot2\rbrace\\
&=\lbrace1,4,7\rbrace\\
T_{2,2}&=\lbrace 2+3\cdot0, 2+3\cdot1, 2+3\cdot2\rbrace\\
&=\lbrace 2,5,8 \rbrace
\end{split}
\end{equation*}
Therefore we have
\[\mathfrak{S}_{2}\lbrack\mathfrak{S}_{1}\rbrack=\lbrace\lbrace 0\rbrace, \lbrace 3,6\rbrace, \lbrace 1,4,7\rbrace, \lbrace 2,5,8\rbrace\rbrace\]

Repeating the above procedure for $\mathfrak{S}_{1}\lbrack\mathfrak{S}_{1}\rbrack$ and $\mathfrak{S}_{2}\lbrack\mathfrak{S}_{2}\rbrack$ we obtain

\[\mathfrak{S}_{1}\lbrack\mathfrak{S}_{1}\rbrack=\lbrace \lbrace0 \rbrace,\lbrace3,6 \rbrace,\lbrace1,4,7,2,5,8 \rbrace\rbrace.\]
and
\[\mathfrak{S}_{2}\lbrack\mathfrak{S}_{2}\rbrack=\lbrace \lbrace0 \rbrace,\lbrace3 \rbrace,\lbrace6 \rbrace,\lbrace1,4,7\rbrace, \lbrace2,5,8 \rbrace\rbrace.\]

The following is a list of all Schur rings over $\mathbb{Z}_{9}$
\begin{equation*}
\begin{split}
\mathfrak{S}_{1}&=\langle \underline{0},\underline{1,2,3,4,5,6,7,8}\rangle,\\
\mathfrak{S}_{2}&=\langle \underline{0},\underline{1,2,4,5,7,8},\underline{3,6}\rangle,\\
\mathfrak{S}_{3}&=\langle \underline{0},\underline{1,4,7},\underline{2,5,8},\underline{3,6}\rangle,\\
\mathfrak{S}_{4}&=\langle \underline{0},\underline{1,2,4,5,7,8},\underline{3},\underline{6}\rangle,\\
\mathfrak{S}_{5}&=\langle \underline{0},\underline{1,4,7},\underline{2,5,8},\underline{3},\underline{6}\rangle,\\
\mathfrak{S}_{6}&=\langle \underline{0},\underline{1,8},\underline{2,7},\underline{3,6},\underline{4,5}\rangle,\\
\mathfrak{S}_{7}&=\langle \underline{0},\underline{1},\underline{2},\underline{3},\underline{4},\underline{5},\underline{6},\underline{7},\underline{8}\rangle,\\
\end{split}
\end{equation*}

Together with the four $\mathcal{S}$-rings obtained by taking the wreath products, we have the two trivial $\mathcal{S}$-rings $\mathfrak{S}_{1}$ and $\mathfrak{S}_{7}$ as well as another $\mathcal{S}$-ring obtained by using the second part of (a) in Theorem~\ref{thm:theorem13}, given as $\mathfrak{S}_{6}$ above. In order to obtain this Schur ring, cosets of the subgroup $H=\lbrace 1,8 \rbrace$ in $\mathbb{Z}^\ast_{p^2}$ are determined.

\section{The Undirected Case for $\MakeLowercase{n=27}$ using the Structural Method}

A list of Schur rings over $\mathbb{Z}_{27}$ may be determined using the package COCO. For the purpose of this thesis, we have been provided with a list of symmetric Schur rings for the undirected case for $n=27$ by M. Klin, and details of how these are generated, will not be given in this thesis. (By ``symmetric Schur-Ring", we mean one in which every basic set $T$ satisfies $-T=T$. This is sufficient for our purpose of enumerating undirected circulants.) The following is a list of all Schur rings required to determine the number of non-isomorphic, undirected circulants on 27 vertices.

\begin{equation*}
\begin{split}
\mathfrak{S}_{1}&=\langle \underline{0},\underline{1,26},\underline{2,25},\underline{3,24},\underline{4,23},\underline{5,22},\underline{6,21},\underline{7,20},\underline{8,19},\underline{9,18},\underline{10,17},\underline{11,16},\underline{12,15},\underline{13,14}\rangle,\\
\mathfrak{S}_{2}&=\langle \underline{0},\underline{1,26,8,19,10,17},\underline{2,25,7,20,11,16},\underline{3,24},\underline{4,23,5,22,13,14},\underline{6,21},\underline{9,18},\underline{12,15}\rangle,\\
\mathfrak{S}_{3}&=\langle \underline{0},\underline{1,26,2,25,4,23,5,22,7,20,8,19,10,17,11,16,13,14},\underline{3,24},\underline{6,21},\underline{9,18},\underline{12,15}\rangle,\\
\mathfrak{S}_{4}&=\langle \underline{0},\underline{1,26,8,19,10,17},\underline{2,25,7,20,11,16},\underline{3,24,6,21,12,15},\underline{4,23,5,22,13,14},\underline{9,18}\rangle,\\
\mathfrak{S}_{5}&=\langle \underline{0},\underline{1,26,2,25,4,23,5,22,7,20,8,19,10,17,11,16,13,14},\underline{3,24,6,21,12,15},\underline{9,18}\rangle,\\
\mathfrak{S}_{6}&=\langle \underline{0},\underline{1,26,2,25,4,23,5,22,7,20,8,19,10,17,11,16,13,14},\underline{3,24,6,21,9,18,12,15}\rangle,\\
\mathfrak{S}_{7}&=\langle \underline{0},\underline{1,26,2,25,3,24,4,23,5,22,6,21,7,20,8,19,10,17,11,16,12,15,13,14},\underline{9,18}\rangle,\\
\mathfrak{S}_{8}&=\langle \underline{0},\underline{1,2,3,4,5,6,7,8,9,10,11,12,13,14,15,16,17,18,19,20,21,22,23,24,25,26}\rangle\\
\end{split}
\end{equation*}

In this list, we can observe that $\mathfrak{S}_{1}$ is the finest with the smallest automorphism group, while $\mathfrak{S}_{8}$ has the largest automorphism group. Therefore $\mathfrak{S}_{1}$ contains all the other Schur rings. We may now construct the lattice of Schur rings. This is given in Figure~\ref{Fig7}.
\begin{figure}[h!]
\centering
\includegraphics[width=0.3\linewidth]{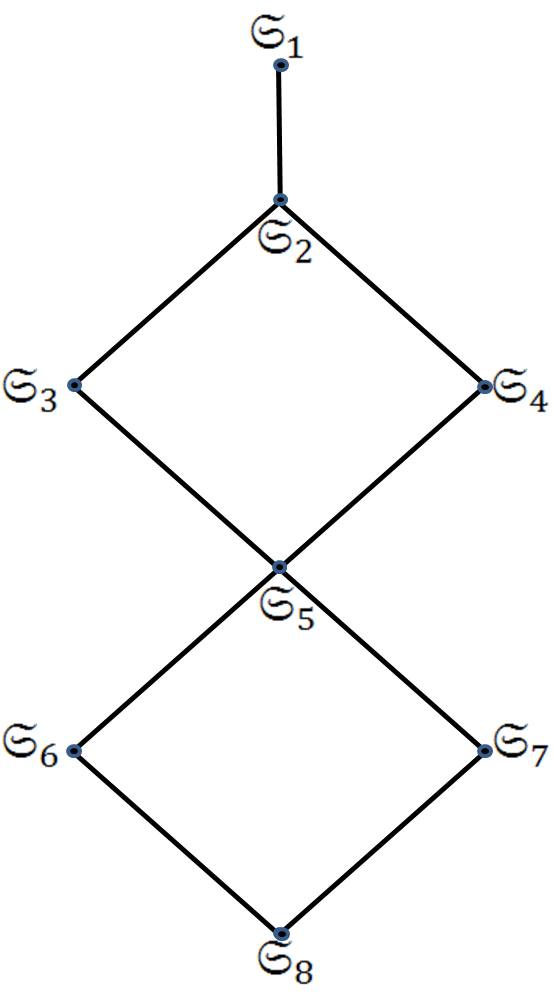}
\caption{Lattice of all $\mathcal{S}$-rings for $n=27$ Undirected}
\label{Fig7}
\end{figure}
\\
\vspace{5mm}
Using \eqref{eq:erl} we can obtain the generating functions $f_{i}(t)$. These are as follows
\begin{equation*}
\begin{split}
f_1(t)&=(1+t^2)^{13}\\
f_2(t)&=(1+t^6)^3(1+t^2)^4\\
f_3(t)&=(1+t^{18})(1+t^2)^4\\
f_4(t)&=(1+t^6)^4(1+t^2)\\
f_5(t)&=(1+t^{18})(1+t^6)(1+t^2)\\
f_6(t)&=(1+t^{18})(1+t^8)\\
f_7(t)&=(1+t^{24})(1+t^2)\\
f_8(t)&=(1+t^{26})
\end{split}
\end{equation*}
Table ~\ref{table:Table4} gives a list of the sizes of the automorphism groups and their normalizers. These may be obtained by using GAP as described for the case $n=13$.
\begin{table}[h!]
\caption{Sizes of Automorphism Groups and their Normalizers for the Case n=27}
\begin{adjustwidth}{-1.2cm}{}
\begin{quote}
\centering
\begin{tabular}{|c|c|c|c|}
\hline
$G_i$&$|G_i|$&$|N(G_i)|$&$\frac{|G_i|}{|N(G_i)|}$\\
\hline
$G_1$&54&486&$\frac{1}{9}$\\
$G_2$&486&4374&$\frac{1}{9}$\\
$G_3$&34992&104976&$\frac{1}{3}$\\
$G_4$&181398528&544195584&$\frac{1}{3}$\\
$G_5$&13060694016&13060694016&1\\
$G_6$&286708355039232000&286708355039232000&1\\
$G_7$&3656994324480&3656994324480&1\\
$G_8$&10888869450418352160768000000&10888869450418352160768000000&1\\
\hline
\end{tabular}
\label{table:Table4}
\end{quote}
 \end{adjustwidth}
\end{table}
\vspace{8mm}
\\
We may now determine $g_i=g_{i}(t)$ for $i=1,2,...,8$, using \eqref{eq:erl1} and figure~\ref{Fig7}
\begin{equation*}
\begin{split}
g_8&=f_8=1+t^{26}\\
g_7&=f_7-g_8=t^{24}+t^2\\
g_6&=f_6-g_8=t^{18}+t^8\\
g_5&=f_5-(g_8+g_7+g_6)=t^{20}+t^6\\
g_4&=\frac{1}{3}(f_4-g_8-g_7-g_6-g_5)=t^{20}+t^{18}+2t^{14}+2t^{12}+t^8+t^6\\
g_3&=\frac{1}{3}(f_3-g_8-g_7-g_6-g_5)=t^{24}+2t^{22}+t^{20}+t^6+2t^4+t^2\\
g_2&=\frac{1}{9}(f_2-g_8-g_7-g_6-g_5-3g_4-3g_3)=t^{18}+2t^{16}+t^{14}+t^{12}+2t^{10}+t^8\\
g_1&=\frac{1}{9}(f_1-g_8-g_7-g_6-g_5-3g_4-3g_3-9g_2)\\
&=t^{24}+8t^{22}+31t^{20}+78t^{18}+141t^{16}+189t^{14}+189t^{12}+141t^{10}+78t^8+31t^6+\\
&8t^4+t^2\\
\end{split}
\end{equation*}
Therefore
\begin{equation*}
\begin{split}
g(t)=g_1+g_2+...+g_8=t^{26}+&3t^{24}+10t^{22}+34t^{20}+81t^{18}+143t^{16}+192t^{14}+\\
&192t^{12}+143t^{10}+81t^8+34t^6+10t^4+3t^2+1\\
\end{split}
\end{equation*}
One may observe that this generating function is the same as that obtained in the method described in \Chapref{chp:Chp4}, which made use of the isomorphism theorem. Once again we have 928 non-isomorphic, undirected circulants on 27 vertices.

\chapter{Conclusion}
In his paper \cite{Turner67}, Turner proved that \'{A}d\'{a}m's conjecture holds for prime-ordered circulants (Cayley graphs on cyclic groups), that is, he showed that two cyclic graphs, $\Gamma(\mathbb{Z}_p,X)$ and $\Gamma(\mathbb{Z}_p,X’)$ are isomorphic if and only if $X'=mX$ for some integer $m \in \mathbb{Z}^\ast_p$. Therefore, the isomorphism classes of circulants on $\mathbb{Z}_p$, correspond to the equivalence classes of subsets of $\mathbb{Z}_p$, under this multiplicative action. Moreover, this relation may be seen as the action of the group $\mathbb{Z}^\ast_p$ on the connecting sets, where the orbits of the group give the equivalence classes. As a result, Turner went on to use Polya's Enumeration theorem to enumerate undirected circulants of prime order. His techniques however, have also been shown to work equally well for directed circulants.

When multiplying by the integer $m\in \mathbb{Z}^\ast_p$, we obtain an automorphism of the group $\mathbb{Z}_p$, where the group operation is addition modulo $p$. By considering the resulting action of the group automorphism on the vertices of the group, we see that the automorphism maps one Cayley graph $\Gamma$ on $\mathbb{Z}_p$, onto another Cayley graph $\Gamma'$ on $\mathbb{Z}_p$. In fact, if $\Gamma(\mathbb{Z}_p,X)$ is a Cayley graph on the group $\mathbb{Z}_p$, and $\alpha \in Aut(\mathbb{Z}_p)$, then $\alpha$ is an isomorphism from the graph $\Gamma$ to the graph $\Gamma'=\Gamma(\mathbb{Z}_p,\alpha(X))$ \cite{Mishna98}.
In terms of this result, Turner proved that two prime ordered circulants are isomorphic if and only if there exists some automorphism of the cyclic group in question which simultaneously acts as a graph isomorphism. This has given rise to a whole area of research on isomorphism of Cayley graphs which led to the definitions which follow.

\section{Cayley-Isomorphic Graphs}
If $G$ is a finite group and for every undirected Cayley graph $Cay(G,T)$  isomorphic to the undirected Cayley graph $Cay(G,S)$ , there exists a group automorphism mapping $S$ into $T$, then $Cay(G,S)$ is said to be a CI graph (or is said to have the CI property). When every undirected Cayley graph of $G$ is a CI graph, $G$ is called a CI-group. When the graphs are directed, the CI-group is referred to as a DCI-group.

If $G$ is a DCI-group, the isomorphism classes of the directed Cayley graphs on $G$ correspond to the orbits of $Aut(G)$ acting on the Cayley subsets of $G$. In the case of undirected Cayley graphs, the action is taken on the inverse pairs $\lbrace x, -x \rbrace$ chosen from the Cayley subsets of $G$. Determining which graphs (or groups) have the CI property, has become one of the foremost problems in the study of Cayley graphs. This is the so called CI problem.

Therefore, Turner's result is essentially that all circulant graphs of order $n=p$, $p$ prime, have the CI property, or equivalently, that every cyclic group of prime order is a CI group. In 1970, Elspas and Turner proved that not all cyclic groups are CI groups and that for $p\geq 5$, $\mathbb{Z}_{p^2}$ is not CI \cite{ET70}. In 1979, Alspach and Parsons \cite{AP79} generalised Babai's result \cite{Babai77}, who proved that $\mathbb{Z}_{2p}$ is CI. They proved that in fact $\mathbb{Z}_{pq}$ is CI, where $p$ and $q$ are distinct primes.  In 1997, Muzychuk \cite{Muzy97} solved the CI problem for the cyclic groups in general. He showed that the cyclic groups which are CI groups, are those of order $n=8,9,18$ or $2^em$, where $e \in \lbrace0,1,2\rbrace$ and $m$ is odd and square free.  He also showed that the cyclic group $\mathbb{Z}_n$ is a DCI-group, if and only if $n=2^em$, where $e$ and $m$ are as defined previously.

The CI problem has been extended to other families of graphs. Since Turner showed that $\mathbb{Z}_p$ is a CI group, some researchers questioned whether or not all groups of the form $\mathbb{Z}^t_p$ are CI-groups, where $\mathbb{Z}^t_p$, denotes the direct product of $\mathbb{Z}_p$ with itself $t$ times and $p$ is a prime. In \cite{AM2002}, we find that Godsil, Dobson and Morris, proved that $\mathbb{Z}^2_p$, $\mathbb{Z}^3_p$ and $\mathbb{Z}^4_p$, respectively, are CI-groups. However, Nowitz showed that $\mathbb{Z}^6_2$ is not a CI group.

\section{Summary of Results}
In this research project, we have focused on the enumeration of circulant graphs of prime power order, in particular the $p^3$ case. We have shown how enumeration of circulants of order $p^2$ and order $p^3$, may be achieved using two different methods: the multiplier method, which involves the use of an isomorphism criterion and the structural method which makes use of Schur rings.

We have seen how in the former method, the invariance conditions imposed on the multipliers in the isomorphism criterion, make enumeration difficult and that in order to be able to count the number of non-isomorphic directed/undirected circulants on $p^3$ vertices, we required five separate cases and eleven enumeration sub-problems. In addition, unlike enumeration of directed circulants, enumeration of undirected circulants required the connecting sets to be paired by inversion. This method of splitting the isomorphism theorem into five subproblems, was applied for the cases $p=3$ and $p=5$, in order to obtain the number of directed and undirected circulants on 27 and 125 vertices respectively. We have also produced some generating functions when taking this approach. The following results have been obtained by using the software package GAP.

\begin{table}[h!]
\caption{The number of non-isomorphic circulants on 27 vertices}
%\begin{quote}
\centering
\begin{tabular}{|c|c|}
\hline
&Number of Circulants\\
\hline
Undirected&928\\
Directed&3728891\\
\hline
\end{tabular}
%\end{quote}
\end{table}
The generating function for the undirected case when $n=27$ is:
\begin{equation*}
\begin{split}
&t^{26}+3t^{24}+10t^{22}+34t^{20}+81t^{18}+143t^{16}+192t^{14}+192t^{12}+143t^{10}+81t^8+34t^6+\\
&10t^4+3t^2+1
\end{split}
\end{equation*}
and the generating function for the directed case when $n=27$ is:
\begin{equation*}
\begin{split}
&t^{26}+3t^{25}+23t^{24}+152t^{23}+844t^{22}+3662t^{21}+12814t^{20}+36548t^{19}+86837t^{18}\\
&+173593t^{17}+295172t^{16}+429240t^{15}+536646t^{14}+577821t^{13}+536646t^{12}\\
&+429240t^{11}+295172t^{10}+173593t^9+86837t^8+36548t^7+12814t^6+3662t^5\\
&+844t^4+152t^3+23t^2+3t+1\\
\end{split}
\end{equation*}
\begin{table}[h!]
\caption{The number of non-isomorphic circulants on 125 vertices}
\centering
\begin{tabular}{|c|c|}
\hline
&Number of Circulants\\
\hline
Undirected&92233720411499283\\
Directed&212,676,479,325,586,539,710,725,989,876,778,596\\
\hline
\end{tabular}
\end{table}
\\
The generating function for the undirected case when $n=125$ is given by
\begin{equation*}
\begin{split}
\small
&t^{124}+3t^{122}+45t^{120}+774t^{118}+11207t^{116}+129485t^{114}+1229657t^{112}+9835988t^{110}\\
&+67622641t^{108}+405731843t^{106}+2150382085t^{104}+10165426468t^{102}\\
&+43203077195t^{100}+166165624857t^{98}+581579739591t^{96}+1861054998416t^{94}\\
&+5466849215583t^{92}+14792650391699t^{90}+36981626382405t^{88}+85641660162366t^{86}\\
&+184129570236171t^{84}+368259138698205t^{82}+686301123812811t^{80}\\
&+1193567168903172t^{78}+1939546652290065t^{76}+2948110907190899t^{74}\\
&+4195388602819760t^{72}+5593851464926268t^{70}+6992314336461413t^{68}\\
&+8197885767564289t^{66}+9017674350331611t^{64}+9308567065105337t^{62}\\
&+9017674350331611t^{60}+8197885767564289t^{58}+6992314336461413t^{56}\\
&+5593851464926268t^{54}+4195388602819760t^{52}+2948110907190899t^{50}\\
&+1939546652290065t^{48}+1193567168903172t^{46}+686301123812811t^{44}\\
&+368259138698205t^{42}+184129570236171t^{40}+85641660162366t^{38}\\
&+36981626382405t^{36}+14792650391699t^{34}+5466849215583t^{32}+1861054998416t^{30}\\
&+581579739591t^{28}+166165624857t^{26}+43203077195t^{24}+10165426468t^{22}\\
&+2150382085t^{20}+405731843t^{18}+67622641t^{16}+9835988t^{14}+1229657t^{12}\\
&+129485t^{10}+11207t^8+774t^6+45t^4+3t^2+1\\
\end{split}
\end{equation*}

The results for both directed and undirected circulants on 27 vertices, have matched Matan Ziv-Av's results which he obtained by using the structural method and by brute force respectively. Comparison of the generating function for the undirected case when $n=27$, also gave the same result.

In this research, we also showed how, for undirected circulants on 27 vertices, enumeration can be done directly from Theorem ~\ref{thm:theorem4}, by using inclusion-exclusion. The final result and generating function obtained using this method, have also matched those which we obtained by splitting the isomorphism theorem into five subproblems.

The structural method was used to count the number of non-isomorphic undirected circulants on 27 vertices. In order to make use of this method, we required a full list of Schur rings over $\mathbb{Z}_{27}$, as well as the sizes of their automorphism groups and the normalizers in $S_{27}$ of these automorphism groups. A list of symmetric Schur rings for the undirected case for n = 27 was provided by M. Klin, and the sizes of the automorphism groups and their normalizers have been obtained by using the computer package GAP. Once again, the final result and generating function matched those obtained using the multiplier method. Although more algebraically profound, this method becomes more demanding as the value of $p$ increases, particularly because obtaining the lattice of all Schur rings may prove to be a difficult task and computer packages such as GAP do not have sufficient memory to compute the sizes of their automorphism groups and the normalizers.

Although no results have been published as yet for directed and undirected circulants on 125 vertices, the generating function we have obtained for the undirected case, may be useful for checking our results, using exhaustive search. Moreover, although focusing on the number of non-isomorphic directed and undirected circulant graphs on $p^3$ vertices in our work, enumeration of other types of circulant graphs, such as circulant tournaments and self-complementary circulants, may also produce interesting results.

\backmatter % book mode only
\appendix
\chapter{Appendix A: An explanation of some key GAP commands}
\lhead{\emph{APPENDIX A.} \emph{AN EXPLANATION OF SOME KEY GAP COMMANDS}}
The following is an explanation of the commands used in GAP for the programs given in Appendix B.

With reference to the GAP program for the number of non-isomorphic graphs for n = 27 (Directed)
using Table 3.1, we start off by defining all the sets we require. The following table gives a list of identifiers used in GAP, for the sets mentioned in \Chapref{chp:Chp4}.
\begin{quote}
\centering
\begin{tabular}{|c|c|}
\hline
Identifier in GAP&Notation in Thesis\\
\hline
zstar27&$\mathbb{Z}^\ast_{27}$\\
zprime27&$\mathbb{Z}'_{27}$\\
y0&$Y_0$\\
y1&$Y_1$\\
y2&$Y_2$\\
ystar0&$Y^\ast_0$\\
ystarstar0&$Y^{\ast\ast}_0$\\
ystar1&$Y^\ast_1$\\
\hline
\end{tabular}
\end{quote}
\vspace{5mm}
In order to assign an identifier to its meaning, the command := is used. For example
\[\mbox{y1}:=\lbrack3,6,12,15,21,24\rbrack,\]
where the collection of objects $\lbrack3,6,12,15,21,24\rbrack$, is called a list.

Now suppose that we have the list $\mbox{l}:=\lbrack 4,6,8\rbrack$. Then

\[\verb!List(l,x->x^2)!\]
returns the list $\lbrack 16,36,64\rbrack$. This means that the command \verb!List! creates a new list, in which each \verb!x! in \verb!l! is substituted by the value of the function
\[\verb!x->x^2!\]

As discussed in the thesis, the sets ystar0, ystarstar0 and ystar1 are partitioned into blocks. Now suppose a block \verb![3,12,21]! is mapped into \verb![15,6,24]!. We must let GAP know that this is a mapping into \verb![6,15,24]!. To achieve this, we use the command \verb!AsSet! when defining the sets and when determining the list of permutations. This command puts the elements in the list in ascending order so that \verb![15,6,24]! becomes \verb![6,15,24]!. For example the set \verb!ystar1! is defined in GAP as follows

\[\verb!ystar1:=List([[3,12,21],[6,24,15]], x->AsSet(x))!;\]

To define the indeterminates required for the cycle index, the command
\[\verb!x_i:=Indeterminate(Rationals,i)!;\]
is used. Moreover, to obtain the generating function in terms of one variable \verb!t!, we first need to define \verb!t!.  One has to be careful that this variable is defined in such a way that it does not take up the place of one of the \verb!x_i! (by default it would become the name for \verb!x_1!). This may be achieved by the command
\[\verb!t:=Indeterminate(Rationals,"t", [x_1]);!\]

When enumerating using the isomorphism theorem, what we have is essentially a list which is multiplied by $m \in \mathbb{Z}^\ast_{27}$. The resulting list is a permutation of the original list. We shall therefore need to express this permutation in the usual way GAP expresses permutations.

Suppose \verb!l:=[4,6,8]! as before. The command \verb!Position(l,8)! gives 3, because 8 is in the third position in \verb!l!. Therefore the command
\[\verb!List(l, x -> Position(l,x))!;\]
gives \verb![1,2,3]!, since each \verb!x! is in its own position in \verb!l!. However, if \verb!w! were a permutation of \verb!l!, say \verb!w=[6,4,8]!, then
\[\verb!List(l, x -> Position(w,x))!;\]
gives \verb![2,1,3]! since the 4 of \verb!l! is in position 2 of \verb!w!, the 6 of \verb!l! is in position 1 of \verb!w! and the 8 of \verb!l! is in position 3 of \verb!w!. Therefore the command
\[\verb!x->Position(zprime27,\mod(c*x,27))!;\]
replaces \verb!x! by the position of \verb!\mod(c*x,27)! in \verb!zprime27!. Note here that \verb!c! is a free variable and the command \verb!\mod(a,b)!, gives the value of \verb!a mod b!. The command
\[\verb!List(zprime27,x->Position(zprime27,\mod(c*x,27)))!;\]
replaces \verb!zprime27! by the altered positions of its elements. The variable \verb!c! is still a free variable here.

We saw previously that the command \verb!List(l, x -> Position(w,x))! returns \verb![2,1,3]!. Now \verb![2,1,3]! is a permutation of \verb![1,2,3]!. It is the permutation  \verb!(12)(3)!. In order to obtain this permutation we make use of the command \verb!PermList! as follows
\[\verb!PermList([2,1,3])!;\]
Therefore the following command
\[\verb!PermList(List(zprime27,x->Position(zprime27,\mod(c*x,27)));!\]
will replace the list of altered positions by a permutation. Once again \verb!c! is still a free variable here.

In order to define the range of the variable \verb!c!, we require the following command
\begin{verbatim}
l1:=List(zstar27,c->PermList(List(zprime27,x-> Position(zprime27,
\mod(c*x,27)))));
\end{verbatim}
where every element \verb!c! of the list \verb!zstar27!, is replaced by the permutation which it generates on \verb!zprime27!. The cycle structures of the permutations in \verb!l1! are then obtained by the command
\[\verb!c1:=List(l1,g->CycleIndex(g,[1..Length(zprime27)]));!\]
where every permutation \verb!g! of the list \verb!l1!, is replaced by its cycle structure. One must note that when given a list \verb![1,2,4,5,7,8]! say, every permutation in GAP is written in terms of 1,2,3,4,5,6. This means that the number of objects being permuted is important not the objects themselves. Therefore to avoid counting the size of the list in each case, the command \verb![1..Length(list)]! is used.

Now the result of \verb!c1! is still a list. It is the list of cycle structures of the permutations in \verb!l1!. To obtain the cycle index from this list, we will require an inner product of the list \verb!c1! with a list of 1's. To obtain a list of ones which is of the same length as the list \verb!c1!, the following commands are used:

\begin{verbatim}
                    allones:=[];
                    for in [1..Length(c1)] do
                    allones[i]:=1;
                    od;
\end{verbatim}
We begin by defining a list \verb!allones! to be an empty list. We then use a \verb!for! loop to fill this list with 1's, with its length being equal to the length of the list of cycle structures \verb!c1!. The \verb!for! loop should be terminated with the keyword \verb!od!. The cycle index is then obtained by the command
\[\verb!cycind1:=(1/Length(c1))*(c1*allones)!;\]
which multiplies the two lists \verb!c1! and \verb!allones! and divides the result by the length of the list \verb!c1!. We could have simply added the elements of the list \verb!c1!, however later on, we shall need to multiply two or more cycle indices in this fashion (componentwise), where none of them is the \verb!allones! list.

In order to obtain the generating function for this particular case in terms of the variable \verb!t!, we require the following command
\begin{verbatim}
g1:=List(c1, f->Value(f,[x_1,x_2,x_3,x_6,x_9,x_18],[1+t,1+t^2,1+t^3,
1+t^6,1+t^9,1+t^18]));
\end{verbatim}
in which each cycle structure \verb!f! in the list \verb!c1!, composed of one or more of the indeterminates \verb!x_1,x_2,x_3,x_6,x_9,x_18! is replaced by a polynomial in terms of \verb!t!. The result is again a list of the same length as \verb!c1!.  Each one of the indeterminates \verb!x_i! is substituted by one of the values \verb!1+t,1+t^2,1+t^3,1+t^6,1+t^9,1+t^18!. The value to be substituted depends on the position of the indeterminate. For example the indeterminate in the fourth position (\verb!x_6!), is substituted by the value in the fourth position (\verb!1+t^6!). Moreover, the power of \verb!t! depends on the degree of the generating sets. In this case since \verb!zprime27! consists of singletons, the indeterminate \verb!x_i! is replaced by \verb!1+t^i!. If the set being acted upon consists of blocks each containing 6 elements say, the substitution would be \verb!1+t^6i!.

Let us consider the isomorphism problem $A_3$, for the directed case when $n=27$. As mentioned in \Chapref{chp:Chp4}, the actions in this case are
\begin{equation*}
\begin{split}
A_{31}:&(\mathbb{Z}^\ast_{27},Y^\ast_0 \cup Y^\ast_1) \times (\mathbb{Z}^\ast_{27}, Y_2)\\
A_{32}:&(\mathbb{Z}^\ast_{27},Y^{\ast\ast}_0 \cup Y^\ast_1) \times (\mathbb{Z}^\ast_{27}, Y_2)
\end{split}
\end{equation*}

The lists of permutations required in $A_{31}$ are therefore given by
\begin{verbatim}
l1:= List(zstar27,c->PermList(List(ystar0,x->Position(ystar0,
AsSet(\mod(c*x,27))))));
l2:= List(zstar27,c->PermList(List(ystar1,x->Position(ystar1,
AsSet(\mod(c*x,27))))));
l3:= List(zstar27,c->PermList(List(y2,x->Position(y2,(\mod(c*x,27))))));
\end{verbatim}
For the first action of $A_{31}$ that is $(\mathbb{Z}^\ast_{27},Y^\ast_0 \cup Y^\ast_1)$, we first obtain the list of cycle structures \verb!c1! and \verb!c2! from the list of permutations \verb!l1! and \verb!l2! respectively using the following commands
\begin{verbatim}
                    c1:=List(l1,g->CycleIndex(g,[1..Length(ystar0)]));
                    c2:=List(l2,g->CycleIndex(g,[1..Length(ystar1)]));
\end{verbatim}
The corresponding lists of generating functions \verb!g1! and \verb!g2! are then determined using the method described above. Since both \verb!ystar0! and \verb!ystar1! have generating sets of degree 3, the indeterminate \verb!x_i! is substituted by \verb!1+t^3i!.
\begin{verbatim}
g1:=List(c1, f->Value(f,[x_1,x_2,x_3,x_6,x_9,x_18],[1+t^3,1+t^6,1+t^9,
1+t^18,1+t^27,1+t^54]));
g2:=List(c2, f->Value(f,[x_1,x_2,x_3,x_6],[1+t^3,1+t^6,1+t^9,1+t^18]));
\end{verbatim}
In order to obtain these two lists as one generating function, we may simply multiply \verb!g1! and \verb!g2! together and divide the result by the length of one of the lists. This is achieved by the following command in GAP

\[\verb!G1G2:=(1/Length(g1))*(g1*g2);!\]
For the second action of $A_{31}$ that is $(\mathbb{Z}^\ast_{27}, Y_2)$, we again determine the list of permutations \verb!l3! and list of cycle structures \verb!c3! as follows
\begin{verbatim}
l3:= List(zstar27,c->PermList(List(y2,x->Position(y2,(\mod(c*x,27))))));
c3:=List(l3,g->CycleIndex(g,[1..Length(y2)]));
\end{verbatim}
The corresponding list of generating functions \verb!g3! is obtained by using the command
\[\verb!g3:=List(c3, f->Value(f,[x_1,x_2],[1+t,1+t^2]));!\]
where each \verb!x_i! is substituted by \verb!1+t^i! since the generating sets in \verb!y2! are singletons. One generating function may then be obtained from this list by carrying out an inner product of \verb!g3! with the \verb!allones! vector as described above:

\begin{verbatim}
                    allones:=[];
                    for i in [1..Length(g3)] do
                    allones[i]:=1;
                    od;
                    G3:=(1/Length(g3))*(g3*allones);
\end{verbatim}

The final result for $A_{31}$ is then obtained by multiplying the result of the generating function \verb!G1G2! with the result obtained for \verb!G3! as follows
\[\verb!genfn3_1:=G1G2*G3;!\]
The actions in $A_{32}$ are dealt with in a similar manner, however the sets receiving the action and the values for the substitutions of the indeterminates in the generating functions are changed accordingly. If we denote the resulting generating function in $A_{32}$ by \verb!genfn3_2!, the final generating function for the isomorphism problem $A_3$ is then given by the command
\[\verb!genfn3 := genfn3_1 - genfn3_2;!\]
Substituting $t=1$ in the generating function \verb!genfn3!, gives the number of distinct circulants under this isomorphism problem. This may be achieved through the command
\[\verb!v3:=Value(genfn3,[t],[1]);!\]
The number of distinct circulants under the isomorphism problems $A_1$, $A_2$, $A_4$ and $A_5$ is obtained using the same commands described for $A_3$, however in $A_5$, we also need to determine the group effecting the action, as such we require some more commands. As described in \Chapref{chp:Chp4}, the group effecting the action, contains all ordered pairs $(a,a')$, such that $a, a' \in \mathbb{Z}^\ast_{27}$ and $a'\equiv a\bmod3$.

A list of ordered pairs is obtained by first assigning a function \verb!allpairs! to a list \verb![x,y]! by using the command
\[\verb!allpairs:= function(x,y) return [x,y]; end;!\]
Now to construct a list \verb!S1! of all possible pairs having elements in $\mathbb{Z}^\ast_{27}$, we require the command
\[\verb!S1:=ListX(zstar27, zstar27,allpairs);!\]
The command \verb!ListX!, returns a list constructed from the arguments, \verb!zstar27! (both arguments are the same in this case) under the function \verb!allpairs!.

Now this list must be filtered according to the condition on the multipliers. Since $a'\equiv a\bmod3$, we will define another function \verb!ismod3! as follows
\[\verb!ismod3:= function(a) return \mod(a[1],3) = \mod(a[2],3); end;!,\]
where \verb!a[1]! represents the first component of each pair in \verb!S1!, and \verb!a[2]! the second component. Finally we need to filter out those pairs which have the condition \verb!ismod3!, from the list of all possible pairs, \verb!S1!. This is obtained by using the command
\[\verb!grp := Filtered(S1, ismod3);!\]
The resulting group, \verb!grp!, is the group we require to effect the actions in the isomorphism problem $A_5$.

For the undirected case when $n=27$, the following identifiers are used to represent the sets mentioned in \Chapref{chp:Chp4}.
\begin{quote}
\centering
\begin{tabular}{|c|c|}
\hline
Identifier in GAP&Notation in Thesis\\
\hline
zstar27&$\mathbb{Z}^\ast_{27}$\\
uzprime27&$\mathbb{Z}'_{27}$\\
uy0&$Y_0$\\
uy1&$Y_1$\\
uy2&$Y_2$\\
uystar0&$Y^\ast_0$\\
uystarstar0&$Y^{\ast\ast}_0$\\
uystar1&$Y^\ast_1$\\
\hline
\end{tabular}
\end{quote}
\vspace{5mm}
The same commands used for the directed case when $n=27$ are used for the undirected case, however recall that the set on which the action is taking place will now be made up of inverse pairs of
elements. Therefore when defining these sets, the \verb!AsSet! command must be used as described previously. For example
\[\verb!uy1:=List([[3,24],[6,21],[12,15]], x->AsSet(x));!\]

For the directed case with $n=125$, the following identifiers are used to represent the sets mentioned in \Chapref{chp:Chp4}.
\begin{quote}
\centering
\begin{tabular}{|c|c|}
\hline
Identifier in GAP&Notation in Thesis\\
\hline
zstar125&$\mathbb{Z}^\ast_{125}$\\
zprime125&$\mathbb{Z}'_{125}$\\
y0&$Y_0$\\
y1&$Y_1$\\
y2&$Y_2$\\
ystar0&$Y^\ast_0$\\
ystarstar0&$Y^{\ast\ast}_0$\\
ystar1&$Y^\ast_1$\\
\hline
\end{tabular}
\end{quote}
\vspace{5mm}
\verb!zprime125! is the set $\lbrace 1,...124 \rbrace $ while \verb!zstar125! contains all elements in \verb!zprime125! which are relatively prime to 125. This may be obtained by using the command
\[\verb!zstar125:=PrimeResidues(125);!\]
The sets \verb! y0, y1,y2! are obtained as follows: \verb!y0! contains all elements in \verb!zprime125! which do not have a factor of 5. In other words this set is equivalent to the set \verb!zstar125!. Now \verb!y1! is the set containing elements of \verb!zprime125! which have a factor of 5 but not 25. In order to obtain this set we first filter out all those elements which have a factor of 5, by using the command

\[\verb!b:=Filtered(zprime125,x->\mod(x,5)=0);!\]
We then need to eliminate those elements which have a factor of 25. To do so we use the command

\[\verb!y1:=Filtered(b,x->\mod(x,25)<>0);!\]
This command chooses all those \verb!x! in \verb!b! such that $x\bmod 25 \neq 0$.

In \verb!y2! we need all those elements in \verb!zprime125! which have a factor of 25. To achieve this we have the command
\[\verb!y2:=Filtered(zprime125,x->\mod(x,25)=0);!\]
which takes all \verb!x! in \verb!zprime125! such that $x\bmod 25=0$.

To determine the subsets forming \verb!ystar0!, we simply require the orbits of the action \verb!x -> 26x mod 125! on \verb!y0!. The permutation \verb!g1! which corresponds to this action is given by the command

\[\verb!g1:=PermList(List(y0,x->Position(y0,\mod(26*x,125))));!\]
To find the group \verb!G1! generated by the permutation \verb!g1!, we have the command
\[\verb!G1:=Group(g1);!\]
The orbits of \verb!G1! are then determined using
\[\verb! O1:=Orbits(G1,[1..Length(y0)]);!\]
Again note that the domain is not \verb!y0! but \verb![1..Length(y0)]!. The resulting orbits give the positions of the elements of \verb!y0! not the elements themselves. In order to change these orbits into a list of lists of elements of \verb!y0!, we use the command

\[\verb!p1:=List(O1,x->List(x,y->y0[y]));!\]
where the outer \verb!List! command takes every list \verb!x! of \verb!O1!, and transforms it according to the inner \verb!List! command. The inner  \verb!List! command takes every element \verb!y! of \verb!x! (which is a number representing the position) and changes \verb!y! into the element of \verb!y0! which is in the position given by \verb!y!. The elements in the lists making up \verb!p1! are then ordered using the \verb!AsSet! command to give the set \verb!ystar0! as follows

\[\verb!ystar0:=List(p1,x->AsSet(x));!\]
The sets \verb!ystarstar0! and \verb!ystar1! are obtained in a similar manner.

The number of non-isomorphic undirected circulants when $n=125$, is obtained using the commands described above, however the following identifiers are used in GAP to represent the sets mentioned in \Chapref{chp:Chp4}.
\begin{quote}
\centering
\begin{tabular}{|c|c|}
\hline
Identifier in GAP&Notation in Thesis\\
\hline
zstar125&$\mathbb{Z}^\ast_{125}$\\
uzprime125&$\mathbb{Z}'_{125}$\\
uy0&$Y_0$\\
uy1&$Y_1$\\
uy2&$Y_2$\\
uystar0&$Y^\ast_0$\\
uystarstar0&$Y^{\ast\ast}_0$\\
uystar1&$Y^\ast_1$\\
\hline
\end{tabular}
\end{quote}

\chapter{Appendix B: The full GAP programmes}
\lhead{\emph{APPENDIX B.} \emph{THE FULL GAP PROGRAMMES}}
{\bf GAP program for the number of non-isomorphic graphs for $n=27$ (directed) using Table~\ref{table:Table3}}

\begin{spverbatim}
zstar27:=PrimeResidues(27);
zprime27:=[1..26];
y0:=zstar27;
y1:=[3,6,12,15,21,24];
y2:=[9,18];

ystar0:= List([[1,10,19],[8,17,26],[2,11,20],[7,16,25],[4,13,22],[5,14,23]],
x->AsSet(x));
ystarstar0:=List([[1,4,16,10,13,25,19,22,7],[2,8,5,20,26,23,11,17,14]],
x->AsSet(x));
ystar1:=List([[3,12,21],[6,24,15]], x->AsSet(x));

x_1:=Indeterminate(Rationals,1);
x_2:=Indeterminate(Rationals,2);
x_3:=Indeterminate(Rationals,3);
x_4:=Indeterminate(Rationals,4);
x_5:=Indeterminate(Rationals,5);
x_6:=Indeterminate(Rationals,6);
x_7:=Indeterminate(Rationals,7);
x_8:=Indeterminate(Rationals,8);
x_9:=Indeterminate(Rationals,9);
x_12:=Indeterminate(Rationals,12);
x_15:=Indeterminate(Rationals,15);
x_18:=Indeterminate(Rationals,18);
t:=Indeterminate(Rationals,"t", [x_1]);
\end{spverbatim}

\begin{lstlisting}[mathescape]
$\mbox A_{1}$
\end{lstlisting}

\begin{spverbatim}
l1:= List(zstar27,c->PermList(List(ystarstar0,x->Position(ystarstar0,
AsSet(\mod(c*x,27))))));
c1:=List(l1,g->CycleIndex(g,[1..Length(ystarstar0)]));
g1:=List(c1, f->Value(f,[x_1,x_2],[1+t^9,1+t^18]));
allones:=[];
for i in [1..Length(g1)] do
 allones[i]:=1;
 od;
G1:=(1/Length(g1))*(g1*allones);

l2:= List(zstar27,c->PermList(List(ystar1,x->Position(ystar1,
AsSet(\mod(c*x,27))))));
c2:=List(l2,g->CycleIndex(g,[1..Length(ystar1)]));
g2:=List(c2, f->Value(f,[x_1,x_2],[1+t^3,1+t^6]));
allones:=[];
for i in [1..Length(g2)] do
 allones[i]:=1;
 od;
G2:=(1/Length(g2))*(g2*allones);

l3:= List(zstar27,c->PermList(List(y2,x->Position(y2, (\mod(c*x,27))))));
c3:=List(l3,g->CycleIndex(g,[1..Length(y2)]));
g3:=List(c3, f->Value(f,[x_1,x_2],[1+t,1+t^2]));
allones:=[];
for i in [1..Length(g3)] do
 allones[i]:=1;
 od;
G3:=(1/Length(g3))*(g3*allones);
genfn1:=G1*G2*G3;
v1:=Value(genfn1,[t],[1]);
\end{spverbatim}

\begin{lstlisting}[mathescape]
$\mbox A_{2}$
\end{lstlisting}

\begin{lstlisting}[mathescape]
$\mbox A_{21}:$
\end{lstlisting}

\begin{spverbatim}
l1:= List(zstar27,c->PermList(List(y0,x->Position(y0,(\mod(c*x,27))))));
c1:=List(l1,g->CycleIndex(g,[1..Length(y0)]));
g1:=List(c1, f->Value(f,[x_1,x_2,x_3,x_6,x_9,x_18],[1+t,1+t^2,1+t^3,
1+t^6,1+t^9,1+t^18]));

l2:= List(zstar27,c->PermList(List(y1,x->Position(y1, (\mod(c*x,27))))));
c2:=List(l2,g->CycleIndex(g,[1..Length(y1)]));
g2:=List(c2, f->Value(f,[x_1,x_2,x_3,x_6],[1+t,1+t^2,1+t^3,1+t^6]));

l3:= List(zstar27,c->PermList(List(y2,x->Position(y2,(\mod(c*x,27))))));
c3:=List(l3,g->CycleIndex(g,[1..Length(y2)]));
g3:=List(c3, f->Value(f,[x_1,x_2],[1+t,1+t^2]));

cyc:=[];
for i in [1..Length(g1)] do
   cyc[i]:=g1[i]*g2[i]*g3[i];
  od;
allones:=[];
for i in [1..Length(g1)] do
   allones[i]:=1;
  od;
genfn2_1:=(1/Length(g1))*(cyc*allones);
\end{spverbatim}

\begin{lstlisting}[mathescape]
$\mbox A_{22}:$
\end{lstlisting}

\begin{spverbatim}
l1:= List(zstar27,c->PermList(List(ystar0,x->Position(ystar0,
AsSet(\mod(c*x,27))))));
c1:=List(l1,g->CycleIndex(g,[1..Length(ystar0)]));
g1:=List(c1, f->Value(f,[x_1,x_2,x_3,x_6,x_9,x_18],[1+t^3,1+t^6,1+t^9,
1+t^18,1+t^27,1+t^54]));

l2:= List(zstar27,c->PermList(List(y1,x->Position(y1,(\mod(c*x,27))))));
c2:=List(l2,g->CycleIndex(g,[1..Length(y1)]));
g2:=List(c2, f->Value(f,[x_1,x_2,x_3,x_6],[1+t,1+t^2,1+t^3,1+t^6]));

l3:= List(zstar27,c->PermList(List(y2,x->Position(y2,(\mod(c*x,27))))));
c3:=List(l3,g->CycleIndex(g,[1..Length(y2)]));
g3:=List(c3, f->Value(f,[x_1,x_2],[1+t,1+t^2]));
cyc:=[];
for i in [1..Length(g1)] do
   cyc[i]:=g1[i]*g2[i]*g3[i];
  od;
allones:=[];
for i in [1..Length(g1)] do
   allones[i]:=1;
  od;

genfn2_2:=(1/Length(g1))*(cyc*allones);
genfn2:= genfn2_1 - genfn2_2;
v2:=Value(genfn2,[t],[1]);
\end{spverbatim}

\begin{lstlisting}[mathescape]
$\mbox A_{3}$
\end{lstlisting}

\begin{lstlisting}[mathescape]
$\mbox A_{31}:$
\end{lstlisting}

\begin{spverbatim}
l1:= List(zstar27,c->PermList(List(ystar0,x->Position(ystar0,
AsSet(\mod(c*x,27))))));
c1:=List(l1,g->CycleIndex(g,[1..Length(ystar0)]));
g1:=List(c1, f->Value(f,[x_1,x_2,x_3,x_6,x_9,x_18],[1+t^3,1+t^6,1+t^9,
1+t^18,1+t^27,1+t^54]));

l2:= List(zstar27,c->PermList(List(ystar1,x->Position(ystar1,
AsSet(\mod(c*x,27))))));
c2:=List(l2,g->CycleIndex(g,[1..Length(ystar1)]));
g2:=List(c2, f->Value(f,[x_1,x_2,x_3,x_6],[1+t^3,1+t^6,1+t^9,1+t^18]));
G1G2:=(1/Length(g1))*(g1*g2);

l3:= List(zstar27,c->PermList(List(y2,x->Position(y2,(\mod(c*x,27))))));
c3:=List(l3,g->CycleIndex(g,[1..Length(y2)]));
g3:=List(c3, f->Value(f,[x_1,x_2],[1+t,1+t^2]));
allones:=[];
for i in [1..Length(g3)] do
   allones[i]:=1;
  od;
G3:=(1/Length(g3))*(g3*allones);

genfn3_1:=G1G2*G3;
\end{spverbatim}

\begin{lstlisting}[mathescape]
$\mbox A_{32}:$
\end{lstlisting}

\begin{spverbatim}
l1:= List(zstar27,c->PermList(List(ystarstar0,x->Position(ystarstar0,
AsSet(\mod(c*x,27))))));
c1:=List(l1,g->CycleIndex(g,[1..Length(ystarstar0)]));
g1:=List(c1, f->Value(f,[x_1,x_2,x_3,x_6,x_9],[1+t^9,1+t^18,1+t^27,
1+t^54,1+t^81]));

l2:= List(zstar27,c->PermList(List(ystar1,x->Position(ystar1,
AsSet(\mod(c*x,27))))));
c2:=List(l2,g->CycleIndex(g,[1..Length(ystar1)]));
g2:=List(c2, f->Value(f,[x_1,x_2,x_3,x_6],[1+t^3,1+t^6,1+t^9,1+t^18]));
G1G2:=(1/Length(g1))*(g1*g2);

l3:= List(zstar27,c->PermList(List(y2,x->Position(y2,(\mod(c*x,27))))));
c3:=List(l3,g->CycleIndex(g,[1..Length(y2)]));
g3:=List(c3, f->Value(f,[x_1,x_2],[1+t,1+t^2]));
allones:=[];
for i in [1..Length(g3)] do
   allones[i]:=1;
  od;
G3:=(1/Length(g3))*(g3*allones);
genfn3_2:=G1G2*G3;

genfn3 := genfn3_1 - genfn3_2;
v3:=Value(genfn3,[t],[1]);
\end{spverbatim}

\begin{lstlisting}[mathescape]
$\mbox A_{4}$
\end{lstlisting}

\begin{lstlisting}[mathescape]
$\mbox A_{41}:$
\end{lstlisting}

\begin{spverbatim}
l1:= List(zstar27,c->PermList(List(y1,x->Position(y1,(\mod(c*x,27))))));
c1:=List(l1,g->CycleIndex(g,[1..Length(y1)]));
g1:=List(c1, f->Value(f,[x_1,x_2,x_3,x_6],[1+t,1+t^2,1+t^3,1+t^6]));

l2:= List(zstar27,c->PermList(List(y2,x->Position(y2,(\mod(c*x,27))))));
c2:=List(l2,g->CycleIndex(g,[1..Length(y2)]));
g2:=List(c2, f->Value(f,[x_1,x_2],[1+t,1+t^2]));
G1G2:=(1/Length(g1))*(g1*g2);

l3:= List(zstar27,c->PermList(List(ystarstar0,x->Position(ystarstar0,
AsSet(\mod(c*x,27))))));
c3:=List(l3,g->CycleIndex(g,[1..Length(ystarstar0)]));
g3:=List(c3, f->Value(f,[x_1,x_2,x_3,x_6,x_9],[1+t^9,1+t^18,1+t^27,
1+t^54,1+t^81]));
allones:=[];
for i in [1..Length(g3)] do
   allones[i]:=1;
  od;
G3:=(1/Length(g3))*(g3*allones);

genfn4_1:=G1G2*G3;
\end{spverbatim}

\begin{lstlisting}[mathescape]
$\mbox A_{42}:$
\end{lstlisting}

\begin{spverbatim}
l1:= List(zstar27,c->PermList(List(ystar1,x->Position(ystar1,
AsSet(\mod(c*x,27))))));
c1:=List(l1,g->CycleIndex(g,[1..Length(ystar1)]));
g1:=List(c1, f->Value(f,[x_1,x_2,x_3,x_6],[1+t^3,1+t^6,1+t^9,1+t^18]));

l2:= List(zstar27,c->PermList(List(y2,x->Position(y2,(\mod(c*x,27))))));
c2:=List(l2,g->CycleIndex(g,[1..Length(y2)]));
g2:=List(c2, f->Value(f,[x_1,x_2],[1+t,1+t^2]));
G1G2:=(1/Length(g1))*(g1*g2);

l3:= List(zstar27,c->PermList(List(ystarstar0,x->Position(ystarstar0,
AsSet(\mod(c*x,27))))));
c3:=List(l3,g->CycleIndex(g,[1..Length(ystarstar0)]));
g3:=List(c3, f->Value(f,[x_1,x_2,x_3,x_6,x_9],[1+t^9,1+t^18,1+t^27,
1+t^54,1+t^81]));
allones:=[];
for i in [1..Length(g3)] do
   allones[i]:=1;
  od;
G3:=(1/Length(g3))*(g3*allones);
genfn4_2:=G1G2*G3;

genfn4:=genfn4_1 - genfn4_2;
v4:=Value(genfn4,[t],[1]);
\end{spverbatim}

\begin{lstlisting}[mathescape]
$\mbox A_{5}$
\end{lstlisting}

\begin{spverbatim}
allpairs:= function(x,y) return [x,y]; end;
S1:=ListX(zstar27, zstar27,allpairs);
ismod3:= function(a) return \mod(a[1],3) = \mod(a[2],3); end;
grp := Filtered(S1, ismod3);
\end{spverbatim}

\begin{lstlisting}[mathescape]
$\mbox A_{51}:$
\end{lstlisting}

\begin{spverbatim}
list1:=List(grp,c->PermList(List(y1,x->Position(y1,(\mod(c[1]*x,27))))));;
c1:=List(list1, g -> CycleIndex(g,[1..Length(y1)]));
g1:=List(c1, f->Value(f,[x_1,x_2,x_3,x_6],[1+t,1+t^2,1+t^3,1+t^6]));

list2:=List(grp,c->PermList(List(y2,x->Position(y2,(\mod(c[1]*x,27))))));;
c2:=List(list2, g -> CycleIndex(g,[1..Length(y2)]));
g2:=List(c2, f->Value(f,[x_1,x_2],[1+t,1+t^2]));

list3:=List(grp,c->PermList(List(ystar0,x->Position(ystar0,
AsSet(\mod(c[2]*x,27))))));;
c3:=List(list3, g -> CycleIndex(g,[1..Length(ystar0)]));
g3:=List(c3, f->Value(f,[x_1,x_2,x_3,x_6,x_9,x_18],[1+t^3,1+t^6,1+t^9,
1+t^18,1+t^27,1+t^54]));

cyc:=[];
for i in [1..Length(g1)] do
   cyc[i]:=g1[i]*g2[i]*g3[i];
  od;	
allones:=[];
for i in [1..Length(g1)] do
   allones[i]:=1;
  od;
genfn5_1:=(1/Length(g1))*(cyc*allones);
\end{spverbatim}

\begin{lstlisting}[mathescape]
$\mbox A_{52}:$
\end{lstlisting}

\begin{lstlisting}[mathescape]
$\neg \mbox R_{01}:$
\end{lstlisting}

\begin{spverbatim}
list1:=List(grp,c->PermList(List(y1,x->Position(y1,(\mod(c[1]*x,27))))));;
c1:=List(list1, g -> CycleIndex(g,[1..Length(y1)]));
g1:=List(c1, f->Value(f,[x_1,x_2,x_3,x_6],[1+t,1+t^2,1+t^3,1+t^6]));

list2:=List(grp,c->PermList(List(y2,x->Position(y2,(\mod(c[1]*x,27))))));;
c2:=List(list2, g -> CycleIndex(g,[1..Length(y2)]));
g2:=List(c2, f->Value(f,[x_1,x_2],[1+t,1+t^2]));

list3:=List(grp,c->PermList(List(ystarstar0,x->Position(ystarstar0,
AsSet(\mod(c[2]*x,27))))));;
c3:=List(list3, g -> CycleIndex(g,[1..Length(ystarstar0)]));
g3:=List(c3, f->Value(f,[x_1,x_2,x_3,x_6,x_9],[1+t^9,1+t^18,1+t^27,
1+t^54,1+t^81]));

cyc:=[];
for i in [1..Length(g1)] do
  cyc[i]:=g1[i]*g2[i]*g3[i];
 od;

allones:=[];
for i in [1..Length(g1)] do
 allones[i]:=1;
od;
genfn5_21:=(1/Length(g1))*(cyc*allones);
v_51:=Value(genfn5_21,[t],[1]);
\end{spverbatim}

\begin{lstlisting}[mathescape]
$\neg \mbox R_{10} \mbox{ and } \neg \mbox R_{00}:$
\end{lstlisting}

\begin{spverbatim}
list1:=List(grp,c->PermList(List(ystar1,x->Position(ystar1,
AsSet(\mod(c[1]*x,27))))));;
c1:=List(list1, g -> CycleIndex(g,[1..Length(ystar1)]));
g1:=List(c1, f->Value(f,[x_1,x_2,x_3,x_6],[1+t^3,1+t^6,1+t^9,1+t^18]));

list2:=List(grp,c->PermList(List(y2,x->Position(y2,(\mod(c[1]*x,27))))));;
c2:=List(list2, g -> CycleIndex(g,[1..Length(y2)]));
g2:=List(c2, f->Value(f,[x_1,x_2],[1+t,1+t^2]));

list3:=List(grp,c->PermList(List(ystar0,x->Position(ystar0,
AsSet(\mod(c[2]*x,27))))));;
c3:=List(list3, g -> CycleIndex(g,[1..Length(ystar0)]));
g3:=List(c3, f->Value(f,[x_1,x_2,x_3,x_6,x_9,x_18],[1+t^3,1+t^6,1+t^9,
1+t^18,1+t^27,1+t^54]));

cyc:=[];
for i in [1..Length(g1)] do
 cyc[i]:=g1[i]*g2[i]*g3[i];
 od;

allones:=[];
for i in [1..Length(g1)] do
 allones[i]:=1;
 od;
genfn5_22:=(1/Length(g1))*(cyc*allones);
v_52:=Value(genfn5_22,[t],[1]);
\end{spverbatim}

\begin{lstlisting}[mathescape]
$\neg \mbox R_{01} \mbox{ and } \neg \mbox R_{10}:$
\end{lstlisting}

\begin{spverbatim}
list1:=List(grp,c->PermList(List(ystar1,x->Position(ystar1,
AsSet(\mod(c[1]*x,27))))));;
c1:=List(list1, g -> CycleIndex(g,[1..Length(ystar1)]));
g1:=List(c1, f->Value(f,[x_1,x_2,x_3,x_6],[1+t^3,1+t^6,1+t^9,1+t^18]));

list2:=List(grp,c->PermList(List(y2,x->Position(y2,(\mod(c[1]*x,27))))));;
c2:=List(list2, g -> CycleIndex(g,[1..Length(y2)]));
g2:=List(c2, f->Value(f,[x_1,x_2],[1+t,1+t^2]));

list3:=List(grp,c->PermList(List(ystarstar0,x->Position(ystarstar0,
AsSet(\mod(c[2]*x,27))))));;
c3:=List(list3, g -> CycleIndex(g,[1..Length(ystarstar0)]));
g3:=List(c3, f->Value(f,[x_1,x_2,x_3,x_6,x_9],[1+t^9,1+t^18,1+t^27,
1+t^54,1+t^81]));

cyc:=[];
for i in [1..Length(g1)] do
 cyc[i]:=g1[i]*g2[i]*g3[i];
od;
allones:=[];
for i in [1..Length(g1)] do
 allones[i]:=1;
 od;
genfn5_23:=(1/Length(g1))*(cyc*allones);
v_53:=Value(genfn5_23,[t],[1]);

genfn5_2:=(genfn5_21+genfn5_22)-genfn5_23;
v52:=Value(genfn5_2,[t],[1]);

genfn5:=genfn5_1-genfn5_2;
v5:=Value(genfn5,[t],[1]);


totgenfn:=genfn1+ genfn2 + genfn3 + genfn4 + genfn5;
V:=Value(totgenfn, [t], [1]);
\end{spverbatim}

\newpage

{\bf GAP program for the number of non-isomorphic graphs for $n=27$ (undirected) using Table~\ref{table:Table3}}

\begin{spverbatim}
zstar27:=PrimeResidues(27);
zprime27:=[1..26];
y0:=zstar27;
y1:=[3,6,12,15,21,24];
y2:=[9,18];
uzprime27:=List([[1,26],[2,25],[3,24],[4,23],[5,22],[6,21],[7,20],[8,19], [9,18],[10,17],[11,16],[12,15],[13,14]], x->AsSet(x));
uy0:=List([[1,26],[2,25],[4,23],[5,22],[7,20],[8,19],[10,17],[11,16],
[13,14]], x-> AsSet(x));
uy1:=List([[3,24],[6,21],[12,15]], x->AsSet(x));
uy2:=List([[9,18]], x->AsSet(x));
uystar0 := List([[1,10,19,8,17,26],[2,11,20,7,16,25],[4,13,22,5,14,23]],
x->AsSet(x));
uystarstar0:=List([[1,2,4,5,7,8,10,11,13,14,16,17,19,20,22,23,25,26]],
x-> AsSet(x));
uystar1:=List([[3,6,12,15,21,24]],x->AsSet(x));


x_1:=Indeterminate(Rationals,1);
x_2:=Indeterminate(Rationals,2);
x_3:=Indeterminate(Rationals,3);
x_4:=Indeterminate(Rationals,4);
x_5:=Indeterminate(Rationals,5);
x_6:=Indeterminate(Rationals,6);
x_7:=Indeterminate(Rationals,7);
x_8:=Indeterminate(Rationals,8);
x_9:=Indeterminate(Rationals,9);
t:=Indeterminate(Rationals,"t", [x_1]);
\end{spverbatim}

\begin{lstlisting}[mathescape]
$\mbox A_1$
\end{lstlisting}

\begin{spverbatim}
l1:= List(zstar27,c->PermList(List(uystarstar0,x->Position(uystarstar0,
AsSet(\mod(c*x,27))))));
c1:=List(l1,g->CycleIndex(g,[1..Length(uystarstar0)]));
g1:=List(c1, f->Value(f,[x_1],[1+t^18]));

l2:= List(zstar27,c->PermList(List(uystar1,x->Position(uystar1,
AsSet(\mod(c*x,27))))));
c2:=List(l2,g->CycleIndex(g,[1..Length(uystar1)]));
g2:=List(c2, f->Value(f,[x_1],[1+t^6]));

l3:= List(zstar27,c->PermList(List(uy2,x->Position(uy2,
AsSet(\mod(c*x,27))))));
c3:=List(l3,g->CycleIndex(g,[1..Length(uy2)]));
g3:=List(c3, f->Value(f,[x_1],[1+t^2]));

cyc:=[];
for i in [1..Length(g1)] do
   cyc[i]:=g1[i]*g2[i]*g3[i];
  od;
allones:=[];
for i in [1..Length(g1)] do
   allones[i]:=1;
  od;
genfn1:=(1/Length(g1))*(cyc*allones);
v1:=Value(genfn1,[t],[1]);
\end{spverbatim}

\begin{lstlisting}[mathescape]
$\mbox A_2$
\end{lstlisting}

\begin{lstlisting}[mathescape]
$\mbox A_{21}:$
\end{lstlisting}

\begin{spverbatim}
l1:= List(zstar27,c->PermList(List(uy0,x->Position(uy0,
AsSet(\mod(c*x,27))))));
c1:=List(l1,g->CycleIndex(g,[1..Length(uy0)]));
g1:=List(c1, f->Value(f,[x_1,x_3,x_9],[1+t^2,1+t^6,1+t^18]));

l2:= List(zstar27,c->PermList(List(uy1,x->Position(uy1,
AsSet(\mod(c*x,27))))));
c2:=List(l2,g->CycleIndex(g,[1..Length(uy1)]));
g2:=List(c2, f->Value(f,[x_1,x_3,x_9],[1+t^2,1+t^6,1+t^18]));

l3:= List(zstar27,c->PermList(List(uy2,x->Position(uy2,
AsSet(\mod(c*x,27))))));
c3:=List(l3,g->CycleIndex(g,[1..Length(uy2)]));
g3:=List(c3, f->Value(f,[x_1,x_3,x_9],[1+t^2,1+t^6,1+t^18]));

cyc:=[];
for i in [1..Length(g1)] do
   cyc[i]:=g1[i]*g2[i]*g3[i];
  od;
allones:=[];
for i in [1..Length(g1)] do
   allones[i]:=1;
  od;
genfn2_1:=(1/Length(g1))*(cyc*allones);
\end{spverbatim}

\begin{lstlisting}[mathescape]
$\mbox A_{22}:$
\end{lstlisting}

\begin{spverbatim}
l1:= List(zstar27,c->PermList(List(uystar0,x->Position(uystar0,
AsSet(\mod(c*x,27))))));
c1:=List(l1,g->CycleIndex(g,[1..Length(uystar0)]));
g1:=List(c1, f->Value(f,[x_1,x_3,x_9],[1+t^6,1+t^18,1+t^54]));

l2:= List(zstar27,c->PermList(List(uy1,x->Position(uy1,
AsSet(\mod(c*x,27))))));
c2:=List(l2,g->CycleIndex(g,[1..Length(uy1)]));
g2:=List(c2, f->Value(f,[x_1,x_3,x_9],[1+t^2,1+t^6,1+t^18]));

l3:= List(zstar27,c->PermList(List(uy2,x->Position(uy2,
AsSet(\mod(c*x,27))))));
c3:=List(l3,g->CycleIndex(g,[1..Length(uy2)]));
g3:=List(c3, f->Value(f,[x_1,x_3,x_9],[1+t^2,1+t^6,1+t^18]));
cyc:=[];
for i in [1..Length(g1)] do
   cyc[i]:=g1[i]*g2[i]*g3[i];
  od;
allones:=[];
for i in [1..Length(g1)] do
   allones[i]:=1;
  od;
genfn2_2:=(1/Length(g1))*(cyc*allones);
genfn2:= genfn2_1 - genfn2_2;
v2:=Value(genfn2,[t],[1]);
\end{spverbatim}

\begin{lstlisting}[mathescape]
$\mbox A_3$
\end{lstlisting}

\begin{lstlisting}[mathescape]
$\mbox A_{31}:$
\end{lstlisting}

\begin{spverbatim}
l1:= List(zstar27,c->PermList(List(uystar0,x->Position(uystar0,
AsSet(\mod(c*x,27))))));
c1:=List(l1,g->CycleIndex(g,[1..Length(uystar0)]));
g1:=List(c1, f->Value(f,[x_1,x_3,x_9],[1+t^6,1+t^18,1+t^54]));

l2:= List(zstar27,c->PermList(List(uystar1,x->Position(uystar1,
AsSet(\mod(c*x,27))))));
c2:=List(l2,g->CycleIndex(g,[1..Length(uystar1)]));
g2:=List(c2, f->Value(f,[x_1,x_3,x_9],[1+t^6,1+t^18,1+t^54]));

G1G2:=(1/Length(g1))*(g1*g2);

l3:= List(zstar27,c->PermList(List(uy2,x->Position(uy2,
AsSet(\mod(c*x,27))))));
c3:=List(l3,g->CycleIndex(g,[1..Length(uy2)]));
g3:=List(c3, f->Value(f,[x_1,x_3,x_9],[1+t^2,1+t^6,1+t^18]));

allones:=[];
for i in [1..Length(g3)] do
   allones[i]:=1;
  od;
G3:=(1/Length(g3))*(g3*allones);
genfn3_1:=G1G2*G3;
\end{spverbatim}

\begin{lstlisting}[mathescape]
$\mbox A_{32}:$
\end{lstlisting}

\begin{spverbatim}
l1:= List(zstar27,c->PermList(List(uystarstar0,x->Position(uystarstar0,
AsSet(\mod(c*x,27))))));
c1:=List(l1,g->CycleIndex(g,[1..Length(uystarstar0)]));
g1:=List(c1, f->Value(f,[x_1,x_3,x_9],[1+t^18,1+t^54,1+t^162]));

l2:= List(zstar27,c->PermList(List(uystar1,x->Position(uystar1,
AsSet(\mod(c*x,27))))));
c2:=List(l2,g->CycleIndex(g,[1..Length(uystar1)]));
g2:=List(c2, f->Value(f,[x_1,x_3,x_9],[1+t^6,1+t^18,1+t^54]));
G1G2:=(1/Length(g1))*(g1*g2);

l3:= List(zstar27,c->PermList(List(uy2,x->Position(uy2,
AsSet(\mod(c*x,27))))));
c3:=List(l3,g->CycleIndex(g,[1..Length(uy2)]));
g3:=List(c3, f->Value(f,[x_1,x_3,x_9],[1+t^2,1+t^6,1+t^18]));
allones:=[];
for i in [1..Length(g3)] do
   allones[i]:=1;
  od;
G3:=(1/Length(g3))*(g3*allones);

genfn3_2:=G1G2*G3;
genfn3 := genfn3_1 - genfn3_2;
v3:=Value(genfn3,[t],[1]);
\end{spverbatim}

\begin{lstlisting}[mathescape]
$\mbox A_{4}$
\end{lstlisting}

\begin{lstlisting}[mathescape]
$\mbox A_{41}:$
\end{lstlisting}

\begin{spverbatim}
l1:= List(zstar27,c->PermList(List(uy1,x->Position(uy1,
AsSet(\mod(c*x,27))))));
c1:=List(l1,g->CycleIndex(g,[1..Length(uy1)]));
g1:=List(c1, f->Value(f,[x_1,x_3,x_9],[1+t^2,1+t^6,1+t^18]));

l2:= List(zstar27,c->PermList(List(uy2,x->Position(uy2,
AsSet(\mod(c*x,27))))));
c2:=List(l2,g->CycleIndex(g,[1..Length(uy2)]));
g2:=List(c2, f->Value(f,[x_1,x_3,x_9],[1+t^2,1+t^6,1+t^18]));

G1G2:=(1/Length(g1))*(g1*g2);

l3:= List(zstar27,c->PermList(List(uystarstar0,
x->Position(uystarstar0,
AsSet(\mod(c*x,27))))));
c3:=List(l3,g->CycleIndex(g,[1..Length(uystarstar0)]));
g3:=List(c3, f->Value(f,[x_1,x_3,x_9],[1+t^18,1+t^54,1+t^162]));
allones:=[];
for i in [1..Length(g3)] do
   allones[i]:=1;
  od;
G3:=(1/Length(g3))*(g3*allones);
genfn4_1:=G1G2*G3;
\end{spverbatim}

\begin{lstlisting}[mathescape]
$\mbox A_{42}:$
\end{lstlisting}

\begin{spverbatim}
l1:= List(zstar27,c->PermList(List(uystar1,x->Position(uystar1,
AsSet(\mod(c*x,27))))));
c1:=List(l1,g->CycleIndex(g,[1..Length(uystar1)]));
g1:=List(c1, f->Value(f,[x_1,x_3,x_9],[1+t^6,1+t^18,1+t^54]));

l2:= List(zstar27,c->PermList(List(uy2,x->Position(uy2,
AsSet(\mod(c*x,27))))));
c2:=List(l2,g->CycleIndex(g,[1..Length(uy2)]));
g2:=List(c2, f->Value(f,[x_1,x_3,x_9],[1+t^2,1+t^6,1+t^18]));
G1G2:=(1/Length(g1))*(g1*g2);

l3:= List(zstar27,c->PermList(List(uystarstar0,
x->Position(uystarstar0,
AsSet(\mod(c*x,27))))));
c3:=List(l3,g->CycleIndex(g,[1..Length(uystarstar0)]));
g3:=List(c3, f->Value(f,[x_1,x_3,x_9],[1+t^18,1+t^54,1+t^162]));
allones:=[];
for i in [1..Length(g3)] do
   allones[i]:=1;
  od;
G3:=(1/Length(g3))*(g3*allones);
genfn4_2:=G1G2*G3;
genfn4:=genfn4_1 - genfn4_2;
v4:=Value(genfn4,[t],[1]);
\end{spverbatim}

\begin{lstlisting}[mathescape]
$\mbox A_{5}$
\end{lstlisting}

\begin{spverbatim}
allpairs:= function(x,y) return [x,y]; end;
S1:=ListX(zstar27, zstar27,allpairs);
ismod3:= function(a) return \mod(a[1],3) = \mod(a[2],3); end;
grp := Filtered(S1, ismod3);
\end{spverbatim}

\begin{lstlisting}[mathescape]
$\mbox A_{51}:$
\end{lstlisting}

\begin{spverbatim}
list1:=List(grp,c->PermList(List(uy1,x->Position(uy1,
AsSet(\mod(c[1]*x,27))))));
c1:=List(list1, g -> CycleIndex(g,[1..Length(uy1)]));
g1:=List(c1, f->Value(f,[x_1,x_3,x_9],[1+t^2,1+t^6,1+t^18]));

list2:=List(grp,c->PermList(List(uy2,x->Position(uy2,
AsSet(\mod(c[1]*x,27))))));
c2:=List(list2, g -> CycleIndex(g,[1..Length(uy2)]));
g2:=List(c2, f->Value(f,[x_1,x_3,x_9],[1+t^2,1+t^6,1+t^18]));

list3:=List(grp,c->PermList(List(uystar0,x->Position(uystar0,
AsSet(\mod(c[2]*x,27))))));;
c3:=List(list3, g -> CycleIndex(g,[1..Length(uystar0)]));;
g3:=List(c3, f->Value(f,[x_1,x_3,x_9],[1+t^6,1+t^18,1+t^54]));

cyc:=[];
for i in [1..Length(g1)] do
   cyc[i]:=g1[i]*g2[i]*g3[i];
  od;	
allones:=[];
for i in [1..Length(g1)] do
   allones[i]:=1;
  od;
genfn5_1:=(1/Length(g1))*(cyc*allones);
\end{spverbatim}

\begin{lstlisting}[mathescape]
$\mbox A_{52}:$
\end{lstlisting}

\begin{lstlisting}[mathescape]
$\neg \mbox R_{01}:$
\end{lstlisting}

\begin{spverbatim}
list1:=List(grp,c->PermList(List(uy1,x->Position(uy1,
AsSet(\mod(c[1]*x,27))))));;
c1:=List(list1, g -> CycleIndex(g,[1..Length(uy1)]));;
g1:=List(c1, f->Value(f,[x_1,x_3,x_9],[1+t^2,1+t^6,1+t^18]));

list2:=List(grp,c->PermList(List(uy2,x->Position(uy2,
AsSet(\mod(c[1]*x,27))))));;
c2:=List(list2, g -> CycleIndex(g,[1..Length(uy2)]));;
g2:=List(c2, f->Value(f,[x_1,x_3,x_9],[1+t^2,1+t^6,1+t^18]));

list3:=List(grp,c->PermList(List(uystarstar0,
x->Position(uystarstar0,
AsSet(\mod(c[2]*x,27))))));;
c3:=List(list3, g -> CycleIndex(g,[1..Length(uystarstar0)]));;
g3:=List(c3, f->Value(f,[x_1,x_3,x_9],[1+t^18,1+t^54,1+t^162]));

cyc:=[];
for i in [1..Length(g1)] do
  cyc[i]:=g1[i]*g2[i]*g3[i];
 od;

allones:=[];
for i in [1..Length(g1)] do
 allones[i]:=1;
od;
genfn5_21:=(1/Length(g1))*(cyc*allones);
\end{spverbatim}

\begin{lstlisting}[mathescape]
$\neg \mbox R_{10}$ and $\neg \mbox R_{00}:$
\end{lstlisting}

\begin{spverbatim}
list1:=List(grp,c->PermList(List(uystar1,x->Position(uystar1,
AsSet(\mod(c[1]*x,27))))));;
c1:=List(list1, g -> CycleIndex(g,[1..Length(uystar1)]));;
g1:=List(c1, f->Value(f,[x_1,x_3,x_9],[1+t^6,1+t^18,1+t^54]));

list2:=List(grp,c->PermList(List(uy2,x->Position(uy2,
AsSet(\mod(c[1]*x,27))))));;
c2:=List(list2, g -> CycleIndex(g,[1..Length(uy2)]));;
g2:=List(c2, f->Value(f,[x_1,x_3,x_9],[1+t^2,1+t^6,1+t^18]));

list3:=List(grp,c->PermList(List(uystar0,x->Position(uystar0,
AsSet(\mod(c[2]*x,27))))));;
c3:=List(list3, g -> CycleIndex(g,[1..Length(uystar0)]));;
g3:=List(c3, f->Value(f,[x_1,x_3,x_9],[1+t^6,1+t^18,1+t^54]));

cyc:=[];
for i in [1..Length(g1)] do
 cyc[i]:=g1[i]*g2[i]*g3[i];
 od;

allones:=[];
for i in [1..Length(g1)] do
 allones[i]:=1;
 od;
genfn5_22:=(1/Length(g1))*(cyc*allones);
\end{spverbatim}

\begin{lstlisting}[mathescape]
$\neg \mbox  R_{01}$ and $\neg \mbox R_{10}:$
\end{lstlisting}

\begin{spverbatim}
list1:=List(grp,c->PermList(List(uystar1,x->Position(uystar1,
AsSet(\mod(c[1]*x,27))))));;
c1:=List(list1, g -> CycleIndex(g,[1..Length(uystar1)]));;
g1:=List(c1, f->Value(f,[x_1,x_3,x_9],[1+t^6,1+t^18,1+t^54]));

list2:=List(grp,c->PermList(List(uy2,x->Position(uy2,
AsSet(\mod(c[1]*x,27))))));;
c2:=List(list2, g -> CycleIndex(g,[1..Length(uy2)]));;
g2:=List(c2, f->Value(f,[x_1,x_3,x_9],[1+t^2,1+t^6,1+t^18]));

list3:=List(grp,c->PermList(List(uystarstar0,
x->Position(uystarstar0,
AsSet(\mod(c[2]*x,27))))));;
c3:=List(list3, g -> CycleIndex(g,[1..Length(uystarstar0)]));;
g3:=List(c3, f->Value(f,[x_1,x_3,x_9],[1+t^18,1+t^54,1+t^162]));
cyc:=[];
for i in [1..Length(g1)] do
 cyc[i]:=g1[i]*g2[i]*g3[i];
od;

allones:=[];
for i in [1..Length(g1)] do
 allones[i]:=1;
 od;
genfn5_23:=(1/Length(g1))*(cyc*allones);
genfn5_2:=genfn5_21+genfn5_22-genfn5_23;
genfn5:=genfn5_1-genfn5_2;
v5:=Value(genfn5,[t],[1]);

totcycind := genfn1+ genfn2+ genfn3 + genfn4 + genfn5;
V:=Value(totcycind, [t], [1]);
\end{spverbatim}

\newpage

{\bf GAP program for the number of non-isomorphic graphs for $n=125$ (directed) using Table~\ref{table:Table3}}

\begin{spverbatim}
zstar125:=PrimeResidues(125);
zprime125:=[1..124];
y0:=zstar125;
b:=Filtered(zprime125,x->\mod(x,5)=0);
y1:=Filtered(b,x->\mod(x,25)<>0);
y2:=Filtered(zprime125,x->\mod(x,25)=0);
g1:=PermList(List(y0,x->Position(y0,\mod(26*x,125))));
#Find the group G generated by the permutation
g1;
G1:=Group(g1);
#To find the orbits of G1;
O1:=Orbits(G1,[1..Length(y0)]);
p1:=List(O1,x->List(x,y->y0[y]));
ystar0:=List(p1,x->AsSet(x));
g2:=PermList(List(y0,x->Position(y0,\mod(6*x,125))));
#Find the group G2 generated by the permutation
g2;
G2:=Group(g2);
#To find orbits of G2;
O2:=Orbits(G2,[1..Length(y0)]);
p2:=List(O2,x->List(x,y->y0[y]));
ystarstar0:=List(p2,x->AsSet(x));

g3:=PermList(List(y1,x->Position(y1,\mod(6*x,125))));
#Find the group G3 generated by the permutation
g3;
G3:=Group(g3);
O3:=Orbits(G3,[1..Length(y1)]);
p3:=List(O3,x->List(x,y->y1[y]));
ystar1:=List(p3,x->AsSet(x));

x_1:=Indeterminate(Rationals,1);
x_2:=Indeterminate(Rationals,2);
x_3:=Indeterminate(Rationals,3);
x_4:=Indeterminate(Rationals,4);
x_5:=Indeterminate(Rationals,5);
x_6:=Indeterminate(Rationals,6);
x_7:=Indeterminate(Rationals,7);
x_8:=Indeterminate(Rationals,8);
x_9:=Indeterminate(Rationals,9);
x_10:=Indeterminate(Rationals,10);
x_20:=Indeterminate(Rationals,20);
x_25:=Indeterminate(Rationals,25);
x_50:=Indeterminate(Rationals,50);
x_100:=Indeterminate(Rationals,100);
\end{spverbatim}

\begin{lstlisting}[mathescape]
$A_1$
\end{lstlisting}

\begin{spverbatim}
l1:= List(zstar125,c->PermList(List(ystarstar0,
x->Position(ystarstar0,
AsSet(\mod(c*x,125))))));
c1:=List(l1,g->CycleIndex(g,[1..Length(ystarstar0)]));
allones:=[];
for i in [1..Length(c1)] do
 allones[i]:=1;
 od;
c1:=(1/Length(c1))*(c1*allones);

l2:= List(zstar125,c->PermList(List(ystar1,x->Position(ystar1,
AsSet(\mod(c*x,125))))));
c2:=List(l2,g->CycleIndex(g,[1..Length(ystar1)]));
allones:=[];
for i in [1..Length(c2)] do
 allones[i]:=1;
 od;
c2:=(1/Length(c2))*(c2*allones);

l3:= List(zstar125,c->PermList(List(y2,
x->Position(y2,(\mod(c*x,125))))));
c3:=List(l3,g->CycleIndex(g,[1..Length(y2)]));
allones:=[];
for i in [1..Length(c3)] do
 allones[i]:=1; 	
 od;
c3:=(1/Length(c3))*(c3*allones);

cycleindex1:=c1*c2*c3;
v1:=Value(cycleindex1,[x_1,x_2,x_4],[2,2,2]);
\end{spverbatim}

\begin{lstlisting}[mathescape]
$A_2$
\end{lstlisting}

\begin{lstlisting}[mathescape]
$A_{21}$
\end{lstlisting}

\begin{spverbatim}
l1:= List(zstar125,c->PermList(List(y0,x->Position(y0,(\mod(c*x,125))))));
c1:=List(l1,g->CycleIndex(g,[1..Length(y0)]));

l2:= List(zstar125,c->PermList(List(y1,x->Position(y1,(\mod(c*x,125))))));
c2:=List(l2,g->CycleIndex(g,[1..Length(y1)]));

l3:= List(zstar125,c->PermList(List(y2,x->Position(y2,(\mod(c*x,125))))));
c3:=List(l3,g->CycleIndex(g,[1..Length(y2)]));

cyc:=[];
for i in [1..Length(c1)] do
   cyc[i]:=c1[i]*c2[i]*c3[i];
  od;
allones:=[];
for i in [1..Length(c1)] do
   allones[i]:=1;
  od;
cycleindex2_1:=(1/Length(c1))*(cyc*allones);
\end{spverbatim}

\begin{lstlisting}[mathescape]
$A_{22}$
\end{lstlisting}

\begin{spverbatim}
l1:= List(zstar125,c->PermList(List(ystar0,x->Position(ystar0,
AsSet(\mod(c*x,125))))));
c1:=List(l1,g->CycleIndex(g,[1..Length(ystar0)]));

l2:= List(zstar125,c->PermList(List(y1,x->Position(y1,(\mod(c*x,125))))));
c2:=List(l2,g->CycleIndex(g,[1..Length(y1)]));

l3:= List(zstar125,c->PermList(List(y2,x->Position(y2,(\mod(c*x,125))))));
c3:=List(l3,g->CycleIndex(g,[1..Length(y2)]));

cyc:=[];

for i in [1..Length(c1)] do
   cyc[i]:=c1[i]*c2[i]*c3[i];
  od;
allones:=[];
for i in [1..Length(c1)] do
   allones[i]:=1;
  od;

cycleindex2_2:=(1/Length(c1))*(cyc*allones);

cycleindex2:= cycleindex2_1 - cycleindex2_2;
v2:=Value(cycleindex2,[x_1,x_2,x_4,x_5,x_6,x_10,x_20,x_25,x_50,x_100],
[2,2,2,2,2,2,2,2,2,2]);
\end{spverbatim}

\begin{lstlisting}[mathescape]
$A_{3}$
\end{lstlisting}

\begin{lstlisting}[mathescape]
$A_{31}$
\end{lstlisting}

\begin{spverbatim}
l1:= List(zstar125,c->PermList(List(ystar0,x->Position(ystar0,
AsSet(\mod(c*x,125))))));
c1:=List(l1,g->CycleIndex(g,[1..Length(ystar0)]));

l2:= List(zstar125,c->PermList(List(ystar1,x->Position(ystar1,
AsSet(\mod(c*x,125))))));
c2:=List(l2,g->CycleIndex(g,[1..Length(ystar1)]));

c1c2:=(1/Length(c1))*(c1*c2);

l3:= List(zstar125,c->PermList(List(y2,x->Position(y2,(\mod(c*x,125))))));
c3:=List(l3,g->CycleIndex(g,[1..Length(y2)]));

allones:=[];
for i in [1..Length(c3)] do
   allones[i]:=1;
  od;
c3:=(1/Length(c3))*(c3*allones);
cycleindex3_1:=c1c2*c3;
\end{spverbatim}

\begin{lstlisting}[mathescape]
$A_{32}$
\end{lstlisting}

\begin{spverbatim}
l1:= List(zstar125,c->PermList(List(ystarstar0,x->Position(ystarstar0,
AsSet(\mod(c*x,125))))));
c1:=List(l1,g->CycleIndex(g,[1..Length(ystarstar0)]));

l2:= List(zstar125,c->PermList(List(ystar1,x->Position(ystar1,
AsSet(\mod(c*x,125))))));
c2:=List(l2,g->CycleIndex(g,[1..Length(ystar1)]));
c1c2:=(1/Length(c1))*(c1*c2);

l3:= List(zstar125,c->PermList(List(y2,x->Position(y2,(\mod(c*x,125))))));
c3:=List(l3,g->CycleIndex(g,[1..Length(y2)]));
allones:=[];
for i in [1..Length(c3)] do
   allones[i]:=1;
  od;
c3:=(1/Length(c3))*(c3*allones);

cycleindex3_2:=c1c2*c3;

cycleindex3 := cycleindex3_1 - cycleindex3_2;
v3:=Value(cycleindex3,[x_1,x_2,x_4,x_5,x_6,x_10,x_20,x_25,x_50,x_100],
[2,2,2,2,2,2,2,2,2,2]);
\end{spverbatim}

\begin{lstlisting}[mathescape]
$A_{4}$
\end{lstlisting}

\begin{lstlisting}[mathescape]
$A_{41}$
\end{lstlisting}

\begin{spverbatim}
l1:= List(zstar125,c->PermList(List(y1,x->Position(y1,(\mod(c*x,125))))));
c1:=List(l1,g->CycleIndex(g,[1..Length(y1)]));

l2:= List(zstar125,c->PermList(List(y2,x->Position(y2,(\mod(c*x,125))))));
c2:=List(l2,g->CycleIndex(g,[1..Length(y2)]));
c1c2:=(1/Length(c1))*(c1*c2);

l3:= List(zstar125,c->PermList(List(ystarstar0,x->Position(ystarstar0,
AsSet(\mod(c*x,125))))));
c3:=List(l3,g->CycleIndex(g,[1..Length(ystarstar0)]));

allones:=[];
for i in [1..Length(c3)] do
   allones[i]:=1;
  od;
c3:=(1/Length(c3))*(c3*allones);
cycleindex4_1:=c1c2*c3;
\end{spverbatim}

\begin{lstlisting}[mathescape]
$A_{42}$
\end{lstlisting}

\begin{spverbatim}
l1:= List(zstar125,c->PermList(List(ystar1,x->Position(ystar1,
AsSet(\mod(c*x,125))))));
c1:=List(l1,g->CycleIndex(g,[1..Length(ystar1)]));

l2:= List(zstar125,c->PermList(List(y2,x->Position(y2,(\mod(c*x,125))))));
c2:=List(l2,g->CycleIndex(g,[1..Length(y2)]));
c1c2:=(1/Length(c1))*(c1*c2);

l3:= List(zstar125,c->PermList(List(ystarstar0,x->Position(ystarstar0,
AsSet(\mod(c*x,125))))));
c3:=List(l3,g->CycleIndex(g,[1..Length(ystarstar0)]));

allones:=[];
for i in [1..Length(c3)] do
   allones[i]:=1;
  od;
c3:=(1/Length(c3))*(c3*allones);
cycleindex4_2:=c1c2*c3;

cycleindex4:=cycleindex4_1 - cycleindex4_2;
v4:=Value(cycleindex4,[x_1,x_2,x_4,x_5,x_6,x_10,x_20,x_25,x_50,x_100],
[2,2,2,2,2,2,2,2,2,2]);
\end{spverbatim}

\begin{lstlisting}[mathescape]
$A_{5}$
\end{lstlisting}

\begin{spverbatim}
allpairs:= function(x,y) return [x,y]; end;
g1:=ListX(zstar125, zstar125,allpairs);
ismod5:= function(a) return \mod(a[1],5) = \mod(a[2],5); end;
grp := Filtered(g1, ismod5);
\end{spverbatim}

\begin{lstlisting}[mathescape]
$A_{51}$
\end{lstlisting}

\begin{spverbatim}
list1:=List(grp,c->PermList(List(y1,x->Position(y1,(\mod(c[1]*x,125))))));
c1:=List(list1, g -> CycleIndex(g,[1..Length(y1)]));

list2:=List(grp,c->PermList(List(y2,x->Position(y2,(\mod(c[1]*x,125))))));
c2:=List(list2, g -> CycleIndex(g,[1..Length(y2)]));

list3:=List(grp,c->PermList(List(ystar0,x->Position(ystar0,
AsSet(\mod(c[2]*x,125))))));;
c3:=List(list3, g -> CycleIndex(g,[1..Length(ystar0)]));

cyc:=[];
for i in [1..Length(c1)] do
   cyc[i]:=c1[i]*c2[i]*c3[i];
  od;	
allones:=[];
for i in [1..Length(c1)] do
   allones[i]:=1;
  od;
cycleindex5_1:=(1/Length(c1))*(cyc*allones);
\end{spverbatim}

\begin{lstlisting}[mathescape]
$A_{52}$
\end{lstlisting}

\begin{lstlisting}[mathescape]
$\neg \mbox R_{01}:$
\end{lstlisting}

\begin{spverbatim}
list1:=List(grp,c->PermList(List(y1,x->Position(y1,(\mod(c[1]*x,125))))));
c1:=List(list1, g -> CycleIndex(g,[1..Length(y1)]));

list2:=List(grp,c->PermList(List(y2,x->Position(y2,(\mod(c[1]*x,125))))));
c2:=List(list2, g -> CycleIndex(g,[1..Length(y2)]));

list3:=List(grp,c->PermList(List(ystarstar0,x->Position(ystarstar0,
AsSet(\mod(c[2]*x,125))))));
c3:=List(list3, g -> CycleIndex(g,[1..Length(ystarstar0)]));

cyc:=[];
for i in [1..Length(c1)] do
  cyc[i]:=c1[i]*c2[i]*c3[i];
 od;

allones:=[];
for i in [1..Length(c1)] do
 allones[i]:=1;
od;
cycleindex5_21:=(1/Length(c1))*(cyc*allones);
\end{spverbatim}

\begin{lstlisting}[mathescape]
$\neg \mbox  R_{10}$ and $\neg \mbox R_{00}:$
\end{lstlisting}

\begin{spverbatim}
list1:=List(grp,c->PermList(List(ystar1,x->Position(ystar1,
AsSet(\mod(c[1]*x,125))))));
c1:=List(list1, g -> CycleIndex(g,[1..Length(ystar1)]));

list2:=List(grp,c->PermList(List(y2,
x->Position(y2,(\mod(c[1]*x,125))))));
c2:=List(list2, g -> CycleIndex(g,[1..Length(y2)]));

list3:=List(grp,c->PermList(List(ystar0,x->Position(ystar0,
AsSet(\mod(c[2]*x,125))))));
c3:=List(list3, g -> CycleIndex(g,[1..Length(ystar0)]));

cyc:=[];
for i in [1..Length(c1)] do
 cyc[i]:=c1[i]*c2[i]*c3[i];
 od;

allones:=[];
for i in [1..Length(c1)] do
 allones[i]:=1;
 od;
cycleindex5_22:=(1/Length(c1))*(cyc*allones);
\end{spverbatim}

\begin{lstlisting}[mathescape]
$\neg \mbox  R_{01}$ and $\neg \mbox R_{10}:$
\end{lstlisting}

\begin{spverbatim}
list1:=List(grp,c->PermList(List(ystar1,x->Position(ystar1,
AsSet(\mod(c[1]*x,125))))));
c1:=List(list1, g -> CycleIndex(g,[1..Length(ystar1)]));

list2:=List(grp,c->PermList(List(y2,x->Position(y2,(\mod(c[1]*x,125))))));
c2:=List(list2, g -> CycleIndex(g,[1..Length(y2)]));

list3:=List(grp,c->PermList(List(ystarstar0,x->Position(ystarstar0,
AsSet(\mod(c[2]*x,125))))));
c3:=List(list3, g -> CycleIndex(g,[1..Length(ystarstar0)]));
cyc:=[];
for i in [1..Length(c1)] do
 cyc[i]:=c1[i]*c2[i]*c3[i];
od;

allones:=[];
for i in [1..Length(c1)] do
 allones[i]:=1;
 od;
cycleindex5_23:=(1/Length(c1))*(cyc*allones);
cycleindex5_2:=cycleindex5_21+cycleindex5_22-cycleindex5_23;
cycleindex5:=cycleindex5_1-cycleindex5_2;
v5:=Value(cycleindex5,[x_1,x_2,x_4,x_5,x_6,x_10,x_20,x_25,x_50,x_100],
[2,2,2,2,2,2,2,2,2,2]);
\end{spverbatim}

\newpage

{\bf GAP program for the number of non-isomorphic graphs for $n=125$ (undirected) using Table~\ref{table:Table3}}

\begin{spverbatim}
zstar125:=PrimeResidues(125);
zprime125:=[1..124];
y0:=zstar125;
b:=Filtered(zprime125,x->\mod(x,5)=0);
y1:=Filtered(b,x->\mod(x,25)<>0);
y2:=Filtered(zprime125,x->\mod(x,25)=0);


pair:=function(x,y)
return [x,y];
end;
listofpairs:=ListX(y0,y0,pair);
inversepairs:=function(a) return \mod(a[1],125)=\mod(-a[2],125);
end;
b1:=Filtered(listofpairs,inversepairs);
b2:= function (a) return a[1]<a[2];
end;
uy0:=List(Filtered(b1,b2), x-> AsSet(x));

pair:=function(x,y)
return [x,y];
end;
listofpairs:=ListX(y1,y1,pair);
inversepairs:=function(a) return \mod(a[1],125)=\mod(-a[2],125);
end;
b4:=Filtered(listofpairs,inversepairs);
b5:= function (a) return a[1]<a[2];
end;
uy1:=List(Filtered(b4,b5),x->AsSet(x));

pair:=function(x,y)
return [x,y];
end;
listofpairs:=ListX(y2,y2,pair);
inversepairs:=function(a) return \mod(a[1],125)=\mod(-a[2],125);
end;
b7:=Filtered(listofpairs,inversepairs);
b8:= function (a) return a[1]<a[2];
end;
uy2:=List(Filtered(b7,b8),x->AsSet(x));

pair:=function(x,y)
return [x,y];
end;
listofpairs:=ListX(zprime125,zprime125,pair);
inversepairs:=function(a) return \mod(a[1],125)=\mod(-a[2],125);
end;
b10:=Filtered(listofpairs,inversepairs);
b11:= function (a) return a[1]<a[2];
end;
uzprime125:=List(Filtered(b10,b11),x->AsSet(x));

g1:=PermList(List(y0,x->Position(y0,\mod(26*x,125))));
#Find the group G generated by the permutation
g1;
G1:=Group(g1);
#To find the orbits of G1;
O1:=Orbits(G1,[1..Length(y0)]);
L1:=List(O1,x->List(x,y->y0[y]));

#arranging sets (x,-x)
uystar0:=List([[ 1, 26, 51, 76, 101, 24, 124, 99, 74, 49 ], [ 2, 52, 102, 27, 77, 23, 98, 48, 123, 73  ],
[ 3, 78, 28, 103, 53, 22, 72, 122, 47, 97], [ 4, 104, 79, 54, 29, 21, 46, 71, 96, 121],
[ 6, 31, 56, 81, 106, 19, 119, 94, 69, 44], [ 7, 57, 107, 32, 82, 18, 93, 43, 118, 68 ],
[ 8, 83, 33, 108, 58, 17, 67, 117, 42, 92  ], [ 9, 109, 84, 59, 34, 16, 41, 66, 91, 116  ],
[ 11, 36, 61, 86, 111, 14, 114, 89, 64, 39  ], [ 12, 62, 112, 37, 87, 13, 88, 38, 113, 63]],x-> AsSet(x));

g2:=PermList(List(y0,x->Position(y0,\mod(6*x,125))));
#Find the group G2 generated by the permutation
g2;
G2:=Group(g2);
#To find orbits of G2;
O2:=Orbits(G2,[1..Length(y0)]);
L2:=List(O2,x->List(x,y->y0[y]));

#arranging sets (x,-x)
uystarstar0:= List([[1, 6, 36, 91, 46, 26, 31, 61, 116, 71, 51, 56, 86, 16, 96, 76, 81, 111, 41, 121, 101, 106, 11, 66, 21, 4, 24, 19, 114, 59, 104, 124, 119, 89, 34, 79, 99, 94, 64, 9, 54, 74, 69, 39, 109, 29, 49, 44, 14, 84  ], [ 2, 12, 72, 57, 92, 52, 62, 122, 107, 17, 102, 112, 47, 32, 67, 27, 37, 97, 82, 117, 77, 87, 22, 7, 42, 3, 18, 108, 23, 13, 78, 93, 58, 98, 88, 28, 43, 8, 48, 38, 103, 118, 83, 123, 113, 53, 68, 33, 73, 63]],
x-> AsSet(x));

g3:=PermList(List(y1,x->Position(y1,\mod(6*x,125))));
#Find the group G3 generated by the permutation
g3;
G3:=Group(g3);
O3:=Orbits(G3,[1..Length(y1)]);
l3:=List(O3,x->List(x,y->y1[y]));
uystar1:=List([[ 5, 30, 55, 80, 105, 20, 120, 95, 70, 45  ], [ 10, 60, 110, 35, 85, 15, 90, 40, 115, 65]], x-> AsSet(x));
\end{spverbatim}

\begin{spverbatim}
x_1:=Indeterminate(Rationals,1);
x_2:=Indeterminate(Rationals,2);
x_3:=Indeterminate(Rationals,3);
x_4:=Indeterminate(Rationals,4);
x_5:=Indeterminate(Rationals,5);
x_6:=Indeterminate(Rationals,6);
x_7:=Indeterminate(Rationals,7);
x_8:=Indeterminate(Rationals,8);
x_9:=Indeterminate(Rationals,9);
x_10:=Indeterminate(Rationals,10);
x_25:=Indeterminate(Rationals,25);
x_50:=Indeterminate(Rationals,50);
t:=Indeterminate(Rationals,"t", [x_1]);
\end{spverbatim}

\begin{lstlisting}[mathescape]
$A_1$
\end{lstlisting}

\begin{spverbatim}
l1:= List(zstar125,c->PermList(List(uystarstar0,x->Position(uystarstar0,
AsSet(\mod(c*x,125))))));
c1:=List(l1,g->CycleIndex(g,[1..Length(uystarstar0)]));
h1:=List(c1, f->Value(f,[x_1,x_2],[1+t^50,1+t^100]));
allones:=[];
for i in [1..Length(h1)] do
 allones[i]:=1;
 od;
H1:=(1/Length(h1))*(h1*allones);

l2:= List(zstar125,c->PermList(List(uystar1,x->Position(uystar1,
AsSet(\mod(c*x,125))))));
c2:=List(l2,g->CycleIndex(g,[1..Length(uystar1)]));
h2:=List(c2, f->Value(f,[x_1,x_2],[1+t^10,1+t^20]));
allones:=[];
for i in [1..Length(h2)] do
 allones[i]:=1;
 od;
H2:=(1/Length(h2))*(h2*allones);

l3:= List(zstar125,c->PermList(List(uy2,x->Position(uy2,
AsSet(\mod(c*x,125))))));
c3:=List(l3,g->CycleIndex(g,[1..Length(uy2)]));
h3:=List(c3, f->Value(f,[x_1,x_2],[1+t^2,1+t^4]));
allones:=[];
for i in [1..Length(h3)] do
 allones[i]:=1; 	
 od;
H3:=(1/Length(h3))*(h3*allones);
genfn1:=H1*H2*H3;
v1:=Value(genfn1,[t],[1]);
\end{spverbatim}

\begin{lstlisting}[mathescape]
$A_2$
\end{lstlisting}

\begin{lstlisting}[mathescape]
$A_{21}$
\end{lstlisting}
\begin{spverbatim}
l1:= List(zstar125,c->PermList(List(uy0,x->Position(uy0,
AsSet(\mod(c*x,125))))));
c1:=List(l1,g->CycleIndex(g,[1..Length(uy0)]));
h1:=List(c1, f->Value(f, [x_1,x_2,x_5,x_10,x_25,x_50],[1+t^2,1+t^4,
1+t^10,1+t^20, 1+t^50,1+t^100]));
l2:= List(zstar125,c->PermList(List(uy1,x->Position(uy1,
AsSet(\mod(c*x,125))))));
c2:=List(l2,g->CycleIndex(g,[1..Length(uy1)]));
h2:=List(c2, f->Value(f,[x_1,x_2,x_5,x_10,x_25,x_50],[1+t^2,1+t^4,
1+t^10,1+t^20, 1+t^50,1+t^100]));

l3:= List(zstar125,c->PermList(List(uy2,x->Position(uy2,
AsSet(\mod(c*x,125))))));
c3:=List(l3,g->CycleIndex(g,[1..Length(uy2)]));
h3:=List(c3, f->Value(f,[x_1,x_2,x_5,x_10,x_25,x_50],[1+t^2,1+t^4,
1+t^10,1+t^20, 1+t^50,1+t^100]));

cyc:=[];
for i in [1..Length(h1)] do
   cyc[i]:=h1[i]*h2[i]*h3[i];
  od;
allones:=[];
for i in [1..Length(h1)] do
   allones[i]:=1;
  od;
genfn2_1:=(1/Length(h1))*(cyc*allones);
\end{spverbatim}

\begin{lstlisting}[mathescape]
$A_{22}$
\end{lstlisting}

\begin{spverbatim}
l1:= List(zstar125,c->PermList(List(uystar0,x->Position(uystar0,
AsSet(\mod(c*x,125))))));
c1:=List(l1,g->CycleIndex(g,[1..Length(uystar0)]));
h1:=List(c1, f->Value(f, [x_1,x_2,x_5,x_10,x_25,x_50],[1+t^10,1+t^20,
1+t^50,1+t^100, 1+t^250,1+t^500]));

l2:= List(zstar125,c->PermList(List(uy1,x->Position(uy1,
AsSet(\mod(c*x,125))))));
c2:=List(l2,g->CycleIndex(g,[1..Length(uy1)]));
h2:=List(c2, f->Value(f,[x_1,x_2,x_5,x_10,x_25,x_50],[1+t^2,1+t^4, 1+t^10,
1+t^20,1+t^50,1+t^100]));

l3:= List(zstar125,c->PermList(List(uy2,x->Position(uy2,
AsSet(\mod(c*x,125))))));
c3:=List(l3,g->CycleIndex(g,[1..Length(uy2)]));
h3:=List(c3, f->Value(f,[x_1,x_2,x_5,x_10,x_25,x_50],[1+t^2,1+t^4,
1+t^10,1+t^20, 1+t^50,1+t^100]));

cyc:=[];
for i in [1..Length(h1)] do
   cyc[i]:=h1[i]*h2[i]*h3[i];
  od;
allones:=[];
for i in [1..Length(h1)] do
   allones[i]:=1;
  od;

genfn2_2:=(1/Length(h1))*(cyc*allones);

genfn2:= genfn2_1 - genfn2_2;
v2:=Value(genfn2,[t],[1]);
\end{spverbatim}

\begin{lstlisting}[mathescape]
$A_3$
\end{lstlisting}

\begin{lstlisting}[mathescape]
$A_{31}$
\end{lstlisting}

\begin{spverbatim}
l1:= List(zstar125,c->PermList(List(uystar0,x->Position(uystar0,
AsSet(\mod(c*x,125))))));
c1:=List(l1,g->CycleIndex(g,[1..Length(uystar0)]));
h1:=List(c1, f->Value(f, [x_1,x_2,x_5,x_10,x_25,x_50],[1+t^10,1+t^20,
1+t^50,1+t^100, 1+t^250,1+t^500]));

l2:= List(zstar125,c->PermList(List(uystar1,x->Position(uystar1,
AsSet(\mod(c*x,125))))));
c2:=List(l2,g->CycleIndex(g,[1..Length(uystar1)]));
h2:=List(c2, f->Value(f, [x_1,x_2,x_5,x_10,x_25,x_50],[1+t^10,1+t^20,
1+t^50,1+t^100, 1+t^250,1+t^500]));
H1H2:=(1/Length(h1))*(h1*h2);

l3:= List(zstar125,c->PermList(List(uy2,x->Position(uy2,
AsSet(\mod(c*x,125))))));
c3:=List(l3,g->CycleIndex(g,[1..Length(uy2)]));
h3:=List(c3, f->Value(f,[x_1,x_2,x_5,x_10,x_25,x_50],[1+t^2,1+t^4,
1+t^10,1+t^20, 1+t^50,1+t^100]));
allones:=[];
for i in [1..Length(h3)] do
   allones[i]:=1;
  od;
H3:=(1/Length(h3))*(h3*allones);
genfn3_1:=H1H2*H3;
\end{spverbatim}

\begin{lstlisting}[mathescape]
$A_{32}$
\end{lstlisting}

\begin{spverbatim}
l1:= List(zstar125,c->PermList(List(uystarstar0,
x->Position(uystarstar0,AsSet(\mod(c*x,125))))));
c1:=List(l1,g->CycleIndex(g,[1..Length(uystarstar0)]));
h1:=List(c1, f->Value(f,[x_1,x_2],[1+t^50,1+t^100]));

l2:= List(zstar125,c->PermList(List(uystar1,x->Position(uystar1,
AsSet(\mod(c*x,125))))));
c2:=List(l2,g->CycleIndex(g,[1..Length(uystar1)]));
h2:=List(c2, f->Value(f, [x_1,x_2,x_5,x_10,x_25,x_50],[1+t^10,1+t^20,
1+t^50,1+t^100, 1+t^250,1+t^500]));
H1H2:=(1/Length(h1))*(h1*h2);

l3:= List(zstar125,c->PermList(List(uy2,x->Position(uy2,
AsSet(\mod(c*x,125))))));
c3:=List(l3,g->CycleIndex(g,[1..Length(uy2)]));
h3:=List(c3, f->Value(f,[x_1,x_2,x_5,x_10,x_25,x_50],[1+t^2,1+t^4,
1+t^10,1+t^20, 1+t^50,1+t^100]));

allones:=[];
for i in [1..Length(h3)] do
   allones[i]:=1;
  od;
H3:=(1/Length(h3))*(h3*allones);

genfn3_2:=H1H2*H3;

genfn3:= genfn3_1 - genfn3_2;
v3:=Value(genfn3,[t],[1]);
\end{spverbatim}

\begin{lstlisting}[mathescape]
$A_4$
\end{lstlisting}

\begin{lstlisting}[mathescape]
$A_{41}$
\end{lstlisting}

\begin{spverbatim}
l1:= List(zstar125,c->PermList(List(uy1,x->Position(uy1,
AsSet(\mod(c*x,125))))));
c1:=List(l1,g->CycleIndex(g,[1..Length(uy1)]));
h1:=List(c1, f->Value(f,[x_1,x_2,x_5,x_10,x_25,x_50],[1+t^2,1+t^4,
1+t^10,1+t^20, 1+t^50,1+t^100]));

l2:= List(zstar125,c->PermList(List(uy2,x->Position(uy2,
AsSet(\mod(c*x,125))))));
c2:=List(l2,g->CycleIndex(g,[1..Length(uy2)]));
h2:=List(c2, f->Value(f,[x_1,x_2,x_5,x_10,x_25,x_50],[1+t^2,1+t^4,
1+t^10,1+t^20, 1+t^50,1+t^100]));
H1H2:=(1/Length(h1))*(h1*h2);

l3:= List(zstar125,c->PermList(List(uystarstar0,
x->Position(uystarstar0,AsSet(\mod(c*x,125))))));
c3:=List(l3,g->CycleIndex(g,[1..Length(uystarstar0)]));
h3:=List(c3, f->Value(f,[x_1,x_2],[1+t^50,1+t^100]));

allones:=[];
for i in [1..Length(h3)] do
   allones[i]:=1;
  od;
H3:=(1/Length(h3))*(h3*allones);
genfn4_1:=H1H2*H3;
\end{spverbatim}

\begin{lstlisting}[mathescape]
$A_{42}$
\end{lstlisting}

\begin{spverbatim}
l1:= List(zstar125,c->PermList(List(uystar1,x->Position(uystar1,
AsSet(\mod(c*x,125))))));
c1:=List(l1,g->CycleIndex(g,[1..Length(uystar1)]));
h1:=List(c1, f->Value(f, [x_1,x_2,x_5,x_10,x_25,x_50],[1+t^10,1+t^20,
1+t^50,1+t^100, 1+t^250,1+t^500]));

l2:= List(zstar125,c->PermList(List(uy2,x->Position(uy2,
AsSet(\mod(c*x,125))))));
c2:=List(l2,g->CycleIndex(g,[1..Length(uy2)]));
h2:=List(c2, f->Value(f,[x_1,x_2,x_5,x_10,x_25,x_50],[1+t^2,1+t^4,
1+t^10,1+t^20, 1+t^50,1+t^100]));

H1H2:=(1/Length(h1))*(h1*h2);

l3:= List(zstar125,c->PermList(List(uystarstar0,x->Position(uystarstar0,
AsSet(\mod(c*x,125))))));
c3:=List(l3,g->CycleIndex(g,[1..Length(uystarstar0)]));
h3:=List(c3, f->Value(f,[x_1,x_2],[1+t^50,1+t^100]));

allones:=[];
for i in [1..Length(h3)] do
   allones[i]:=1;
  od;
H3:=(1/Length(h3))*(h3*allones);
genfn4_2:=H1H2*H3;

genfn4:=genfn4_1 - genfn4_2;
v4:=Value(genfn4,[t],[1]);
\end{spverbatim}

\begin{lstlisting}[mathescape]
$A_5$
\end{lstlisting}

\begin{spverbatim}
allpairs:= function(x,y) return [x,y]; end;
S1:=ListX(zstar125, zstar125,allpairs);
ismod5:= function(a) return \mod(a[1],5) = \mod(a[2],5); end;
grp := Filtered(S1, ismod5);
\end{spverbatim}

\begin{lstlisting}[mathescape]
$A_{51}$
\end{lstlisting}

\begin{spverbatim}
list1:=List(grp,c->PermList(List(uy1,x->Position(uy1,
AsSet(\mod(c[1]*x,125))))));
c1:=List(list1, g -> CycleIndex(g,[1..Length(uy1)]));
h1:=List(c1, f->Value(f,[x_1,x_2,x_5,x_10,x_25,x_50],[1+t^2,1+t^4,
1+t^10,1+t^20, 1+t^50,1+t^100]));

list2:=List(grp,c->PermList(List(uy2,x->Position(uy2,
AsSet(\mod(c[1]*x,125))))));
c2:=List(list2, g -> CycleIndex(g,[1..Length(uy2)]));
h2:=List(c2, f->Value(f,[x_1,x_2,x_5,x_10,x_25,x_50],[1+t^2,1+t^4,
1+t^10,1+t^20, 1+t^50,1+t^100]));

list3:=List(grp,c->PermList(List(uystar0,x->Position(uystar0,
AsSet(\mod(c[2]*x,125))))));;
c3:=List(list3, g -> CycleIndex(g,[1..Length(uystar0)]));
h3:=List(c3, f->Value(f, [x_1,x_2,x_5,x_10,x_25,x_50],[1+t^10,1+t^20,
1+t^50,1+t^100, 1+t^250,1+t^500]));

cyc:=[];
for i in [1..Length(h1)] do
   cyc[i]:=h1[i]*h2[i]*h3[i];
  od;	
allones:=[];
for i in [1..Length(h1)] do
   allones[i]:=1;
  od;
genfn5_1:=(1/Length(h1))*(cyc*allones);
\end{spverbatim}

\begin{lstlisting}[mathescape]
$A_{52}$
\end{lstlisting}

\begin{lstlisting}[mathescape]
$\neg \mbox R_{01}:$
\end{lstlisting}

\begin{spverbatim}
list1:=List(grp,c->PermList(List(uy1,x->Position(uy1,
AsSet(\mod(c[1]*x,125))))));
c1:=List(list1, g -> CycleIndex(g,[1..Length(uy1)]));
h1:=List(c1, f->Value(f,[x_1,x_2,x_5,x_10,x_25,x_50],[1+t^2,1+t^4,
1+t^10,1+t^20, 1+t^50,1+t^100]));

list2:=List(grp,c->PermList(List(uy2,x->Position(uy2,
AsSet(\mod(c[1]*x,125))))));
c2:=List(list2, g -> CycleIndex(g,[1..Length(uy2)]));
h2:=List(c2, f->Value(f,[x_1,x_2,x_5,x_10,x_25,x_50],[1+t^2,1+t^4,
1+t^10,1+t^20, 1+t^50,1+t^100]));

list3:=List(grp,c->PermList(List(uystarstar0,
x->Position(uystarstar0,AsSet(\mod(c[2]*x,125))))));
c3:=List(list3, g -> CycleIndex(g,[1..Length(uystarstar0)]));
h3:=List(c3, f->Value(f,[x_1,x_2],[1+t^50,1+t^100]));

cyc:=[];
for i in [1..Length(h1)] do
  cyc[i]:=h1[i]*h2[i]*h3[i];
 od;

allones:=[];
for i in [1..Length(h1)] do
 allones[i]:=1;
od;
genfn5_21:=(1/Length(h1))*(cyc*allones);
\end{spverbatim}

\begin{lstlisting}[mathescape]
$\neg \mbox  R_{10}$ and $\neg \mbox R_{00}:$
\end{lstlisting}

\begin{spverbatim}
list1:=List(grp,c->PermList(List(uystar1,x->Position(uystar1,
AsSet(\mod(c[1]*x,125))))));
c1:=List(list1, g -> CycleIndex(g,[1..Length(uystar1)]));
h1:=List(c1, f->Value(f, [x_1,x_2,x_5,x_10,x_25,x_50],[1+t^10,1+t^20,
1+t^50,1+t^100, 1+t^250,1+t^500]));
list2:=List(grp,c->PermList(List(uy2,x->Position(uy2,
AsSet(\mod(c[1]*x,125))))));
c2:=List(list2, g -> CycleIndex(g,[1..Length(uy2)]));
h2:=List(c2, f->Value(f,[x_1,x_2,x_5,x_10,x_25,x_50],[1+t^2,1+t^4,
1+t^10,1+t^20, 1+t^50,1+t^100]));
list3:=List(grp,c->PermList(List(uystar0,x->Position(uystar0,
AsSet(\mod(c[2]*x,125))))));
c3:=List(list3, g -> CycleIndex(g,[1..Length(uystar0)]));
h3:=List(c3, f->Value(f, [x_1,x_2,x_5,x_10,x_25,x_50],[1+t^10,1+t^20,
1+t^50,1+t^100, 1+t^250,1+t^500]));

cyc:=[];
for i in [1..Length(h1)] do
 cyc[i]:=h1[i]*h2[i]*h3[i];
 od;

allones:=[];
for i in [1..Length(h1)] do
 allones[i]:=1;
 od;
genfn5_22:=(1/Length(h1))*(cyc*allones);
\end{spverbatim}

\begin{lstlisting}[mathescape]
$\neg \mbox  R_{01}$ and $\neg \mbox R_{10}:$
\end{lstlisting}

\begin{spverbatim}
list1:=List(grp,c->PermList(List(uystar1,x->Position(uystar1,
AsSet(\mod(c[1]*x,125))))));
c1:=List(list1, g -> CycleIndex(g,[1..Length(uystar1)]));
h1:=List(c1, f->Value(f, [x_1,x_2,x_5,x_10,x_25,x_50],[1+t^10,1+t^20,
1+t^50,1+t^100, 1+t^250,1+t^500]));

list2:=List(grp,c->PermList(List(uy2,x->Position(uy2,
AsSet(\mod(c[1]*x,125))))));
c2:=List(list2, g -> CycleIndex(g,[1..Length(uy2)]));
h2:=List(c2, f->Value(f,[x_1,x_2,x_5,x_10,x_25,x_50],[1+t^2,1+t^4,
1+t^10,1+t^20, 1+t^50,1+t^100]));

list3:=List(grp,c->PermList(List(uystarstar0,
x->Position(uystarstar0,AsSet(\mod(c[2]*x,125))))));
c3:=List(list3, g -> CycleIndex(g,[1..Length(uystarstar0)]));
h3:=List(c3, f->Value(f,[x_1,x_2],[1+t^50,1+t^100]));

cyc:=[];
for i in [1..Length(h1)] do
 cyc[i]:=h1[i]*h2[i]*h3[i];
od;

allones:=[];
for i in [1..Length(h1)] do
 allones[i]:=1;
 od;
genfn5_23:=(1/Length(h1))*(cyc*allones);
genfn5_2:=genfn5_21+genfn5_22-genfn5_23;
genfn5:=genfn5_1-genfn5_2;
v5:=Value(genfn5,[t],[1]);
totcycind :=genfn1+ genfn2 + genfn3 + genfn4 + genfn5;
v:=Value(totcycind,[t],[1]);
\end{spverbatim}

\newpage

{\bf GAP program for the number of non-isomorphic graphs for $n=27$ (undirected) using Theorem~\ref{thm:theorem4} directly.}

\begin{spverbatim}
zstar27:=PrimeResidues(27);
zprime27:=[1..26];
y0:=zstar27;
y1:=[3,6,12,15,21,24];
y2:=[9,18];
uzprime27:=List([[1,26],[2,25],[3,24],[4,23],[5,22],[6,21],[7,20],[8,19], [9,18],[10,17],[11,16],[12,15],[13,14]], x->AsSet(x));
uy0:=List([[1,26],[2,25],[4,23],[5,22],[7,20],[8,19],[10,17],[11,16],[13,14]],
x-> AsSet(x));
uy1:=List([[3,24],[6,21],[12,15]], x->AsSet(x));
uy2:=List([[9,18]], x->AsSet(x));
uystar0 := List([[1,10,19,8,17,26],[2,11,20,7,16,25],[4,13,22,5,14,23]],
x->AsSet(x));
uystarstar0:=List([[1,2,4,5,7,8,10,11,13,14,16,17,19,20,22,23,25,26]],
x-> AsSet(x));
uystar1:=List([[3,6,12,15,21,24]],x->AsSet(x));
#All triples of elements from Z*_27 before the conditions on the multipliers for the three cases;
allpairsfn1:= function(x,y) return [x,y]; end;
allpairsfn2:= function(x,y) return [x[1],x[2],y]; end;
alltripleszstar27:=ListX(ListX(zstar27,zstar27,allpairsfn1),zstar27,
allpairsfn2);
x_1:=Indeterminate(Rationals,1);
x_3:=Indeterminate(Rationals,3);
x_9:=Indeterminate(Rationals,9);
t:=Indeterminate(Rationals,"t", [x_1]);
\end{spverbatim}

\begin{lstlisting}[mathescape]
${A_1}:$
\end{lstlisting}

\begin{spverbatim}
#The condition on the three multipliers;
cond1:= function(a) return (\mod(a[2],9) = \mod(a[1],9)) and
(\mod(a[3],3) = \mod(a[2],3)); end;
grp1 := Filtered(alltripleszstar27, cond1);
#grp1 will be the triple of group elements for A_1;
\end{spverbatim}

\begin{lstlisting}[mathescape]
$\mbox {A1.1}:$
\end{lstlisting}
\begin{spverbatim}
l1:= List(grp1,c->PermList(List(uy0,x->Position(uy0,
AsSet(\mod(c[1]*x,27))))));
c1:=List(l1,g->CycleIndex(g,[1..Length(uy0)]));
g1:=List(c1, f->Value(f,[x_1,x_3,x_9],[1+t^2,1+t^6,1+t^18]));

l2:= List(grp1,c->PermList(List(uy1,x->Position(uy1,
AsSet(\mod(c[2]*x,27))))));
c2:=List(l2,g->CycleIndex(g,[1..Length(uy1)]));
g2:=List(c2, f->Value(f,[x_1,x_3,x_9],[1+t^2,1+t^6,1+t^18]));
l3:= List(grp1,c->PermList(List(uy2,x->Position(uy2,
AsSet(\mod(c[3]*x,27))))));
c3:=List(l3,g->CycleIndex(g,[1..Length(uy2)]));
g3:=List(c3, f->Value(f,[x_1,x_3,x_9],[1+t^2,1+t^6,1+t^18]));
cyc:=[];
for i in [1..Length(c1)] do
   cyc[i]:=g1[i]*g2[i]*g3[i];
  od;
allones:=[];
for i in [1..Length(g1)] do
   allones[i]:=1;
  od;
genfn1_1:=(1/Length(g1))*(cyc*allones);
\end{spverbatim}

\begin{lstlisting}[mathescape]
$\mbox {A1.2}:$
\end{lstlisting}
\begin{spverbatim}
l1:= List(grp1,c->PermList(List(uystar0,x->Position(uystar0,
AsSet(\mod(c[1]*x,27))))));
c1:=List(l1,g->CycleIndex(g,[1..Length(uystar0)]));
g1:=List(c1, f->Value(f,[x_1,x_3,x_9],[1+t^6,1+t^18,1+t^54]));

l2:= List(grp1,c->PermList(List(uy1,x->Position(uy1,
AsSet(\mod(c[2]*x,27))))));
c2:=List(l2,g->CycleIndex(g,[1..Length(uy1)]));
g2:=List(c2, f->Value(f,[x_1,x_3,x_9],[1+t^2,1+t^6,1+t^18]));

l3:= List(grp1,c->PermList(List(uy2,x->Position(uy2,
AsSet(\mod(c[3]*x,27))))));
c3:=List(l3,g->CycleIndex(g,[1..Length(uy2)]));
g3:=List(c3, f->Value(f,[x_1,x_3,x_9],[1+t^2,1+t^6,1+t^18]));

cyc:=[];
for i in [1..Length(g1)] do
   cyc[i]:=g1[i]*g2[i]*g3[i];
  od;
allones:=[];
for i in [1..Length(g1)] do
   allones[i]:=1;
  od;
genfn1_2:=(1/Length(g1))*(cyc*allones);

genfn1:= genfn1_1 - genfn1_2;
\end{spverbatim}

\begin{lstlisting}[mathescape]
$A_2:$
\end{lstlisting}

\begin{spverbatim}
#The condition on the three multipliers;
cond2:= function(a) return (\mod(a[2],3) = \mod(a[1],3)); end;
grp2 := Filtered(alltripleszstar27, cond2);
#grp2 will be the triple of group elements for A_2;
\end{spverbatim}

\begin{lstlisting}[mathescape]
$\mbox {A2.1}:$
\end{lstlisting}
\begin{spverbatim}
l1:= List(grp2,c->PermList(List(uy0,x->Position(uy0,
AsSet(\mod(c[1]*x,27))))));
c1:=List(l1,g->CycleIndex(g,[1..Length(uy0)]));
g1:=List(c1, f->Value(f,[x_1,x_3,x_9],[1+t^2,1+t^6,1+t^18]));

l2:= List(grp2,c->PermList(List(uy1,x->Position(uy1,
AsSet(\mod(c[2]*x,27))))));
c2:=List(l2,g->CycleIndex(g,[1..Length(uy1)]));
g2:=List(c2, f->Value(f,[x_1,x_3,x_9],[1+t^2,1+t^6,1+t^18]));

l3:= List(grp2,c->PermList(List(uy2,x->Position(uy2,
AsSet(\mod(c[3]*x,27))))));
c3:=List(l3,g->CycleIndex(g,[1..Length(uy2)]));
g3:=List(c3, f->Value(f,[x_1,x_3,x_9],[1+t^2,1+t^6,1+t^18]));

cyc:=[];
for i in [1..Length(c1)] do
   cyc[i]:=g1[i]*g2[i]*g3[i];
  od;
allones:=[];
for i in [1..Length(g1)] do
   allones[i]:=1;
  od;
genfn2_1:=(1/Length(g1))*(cyc*allones);
\end{spverbatim}

\begin{lstlisting}[mathescape]
$\mbox {A2.2}:$
\end{lstlisting}
\begin{spverbatim}
l1:= List(grp2,c->PermList(List(uystarstar0,x->Position(uystarstar0,
AsSet(\mod(c[1]*x,27))))));
c1:=List(l1,g->CycleIndex(g,[1..Length(uystarstar0)]));
g1:=List(c1, f->Value(f,[x_1,x_3,x_9],[1+t^18,1+t^54,1+t^162]));

l2:= List(grp2,c->PermList(List(uy1,x->Position(uy1,
AsSet(\mod(c[2]*x,27))))));
c2:=List(l2,g->CycleIndex(g,[1..Length(uy1)]));
g2:=List(c2, f->Value(f,[x_1,x_3,x_9],[1+t^2,1+t^6,1+t^18]));

l3:= List(grp2,c->PermList(List(uy2,x->Position(uy2,
AsSet(\mod(c[3]*x,27))))));
c3:=List(l3,g->CycleIndex(g,[1..Length(uy2)]));
g3:=List(c3, f->Value(f,[x_1,x_3,x_9],[1+t^2,1+t^6,1+t^18]));

cyc:=[];
for i in [1..Length(c1)] do
   cyc[i]:=g1[i]*g2[i]*g3[i];
  od;
allones:=[];
for i in [1..Length(g1)] do
   allones[i]:=1;
  od;
genfn2_2:=(1/Length(g1))*(cyc*allones);
genfn2:= genfn2_1 - genfn2_2;
\end{spverbatim}

\begin{lstlisting}[mathescape]
$A_3:$
\end{lstlisting}

\begin{spverbatim}
#The condition on the three multipliers;
cond3:= function(a) return (\mod(a[3],3) = \mod(a[2],3)); end;
grp3 := Filtered(alltripleszstar27, cond3);
#grp3 will be the triple of group elements for A_3;
\end{spverbatim}

\begin{lstlisting}[mathescape]
$\mbox {A3.1}:$
\end{lstlisting}
\begin{spverbatim}
l1:= List(grp3,c->PermList(List(uy0,x->Position(uy0,
AsSet(\mod(c[1]*x,27))))));
c1:=List(l1,g->CycleIndex(g,[1..Length(uy0)]));
g1:=List(c1, f->Value(f,[x_1,x_3,x_9],[1+t^2,1+t^6,1+t^18]));

l2:= List(grp3,c->PermList(List(uy1,x->Position(uy1,
AsSet(\mod(c[2]*x,27))))));
c2:=List(l2,g->CycleIndex(g,[1..Length(uy1)]));
g2:=List(c2, f->Value(f,[x_1,x_3,x_9],[1+t^2,1+t^6,1+t^18]));

l3:= List(grp3,c->PermList(List(uy2,x->Position(uy2,
AsSet(\mod(c[3]*x,27))))));
c3:=List(l3,g->CycleIndex(g,[1..Length(uy2)]));
g3:=List(c3, f->Value(f,[x_1,x_3,x_9],[1+t^2,1+t^6,1+t^18]));

cyc:=[];
for i in [1..Length(c1)] do
   cyc[i]:=g1[i]*g2[i]*g3[i];
  od;
allones:=[];
for i in [1..Length(g1)] do
   allones[i]:=1;
  od;
genfn3_1:=(1/Length(g1))*(cyc*allones);
\end{spverbatim}

\begin{lstlisting}[mathescape]
$\mbox {A3.2}:$
\end{lstlisting}
\begin{spverbatim}
l1:= List(grp3,c->PermList(List(uy0,x->Position(uy0,
AsSet(\mod(c[1]*x,27))))));
c1:=List(l1,g->CycleIndex(g,[1..Length(uy0)]));
g1:=List(c1, f->Value(f,[x_1,x_3,x_9],[1+t^2,1+t^6,1+t^18]));

l2:= List(grp3,c->PermList(List(uystar1,x->Position(uystar1,
AsSet(\mod(c[2]*x,27))))));
c2:=List(l2,g->CycleIndex(g,[1..Length(uystar1)]));
g2:=List(c2, f->Value(f,[x_1,x_3,x_9],[1+t^6,1+t^18,1+t^54]));

l3:= List(grp3,c->PermList(List(uy2,x->Position(uy2,
AsSet(\mod(c[3]*x,27))))));
c3:=List(l3,g->CycleIndex(g,[1..Length(uy2)]));
g3:=List(c3, f->Value(f,[x_1,x_3,x_9],[1+t^2,1+t^6,1+t^18]));

cyc:=[];
for i in [1..Length(g1)] do
   cyc[i]:=g1[i]*g2[i]*g3[i];
  od;
allones:=[];
for i in [1..Length(g1)] do
   allones[i]:=1;
  od;
genfn3_2:=(1/Length(g1))*(cyc*allones);

genfn3:= genfn3_1 - genfn3_2;
G1:=genfn1+genfn2+genfn3;
\end{spverbatim}

\begin{lstlisting}[mathescape]
$A_{12}:$
\end{lstlisting}
\begin{spverbatim}
grp12:=[];
for a in grp1 do
for b in grp2 do
c:=[1,1,1];
c[1]:=\mod(a[1]*b[1],27);
c[2]:=\mod(a[2]*b[2],27);
c[3]:=\mod(a[3]*b[3],27);
Add(grp12,c);
od;
od;
grp12:=DuplicateFreeList(grp12);
\end{spverbatim}

\begin{lstlisting}[mathescape]
$\mbox {A12.1}:$
\end{lstlisting}
\begin{spverbatim}
l1:= List(grp12,c->PermList(List(uy0,x->Position(uy0,
AsSet(\mod(c[1]*x,27))))));
c1:=List(l1,g->CycleIndex(g,[1..Length(uy0)]));
g1:=List(c1, f->Value(f,[x_1,x_3,x_9],[1+t^2,1+t^6,1+t^18]));

l2:= List(grp12,c->PermList(List(uy1,x->Position(uy1,
AsSet(\mod(c[2]*x,27))))));
c2:=List(l2,g->CycleIndex(g,[1..Length(uy1)]));
g2:=List(c2, f->Value(f,[x_1,x_3,x_9],[1+t^2,1+t^6,1+t^18]));

l3:= List(grp12,c->PermList(List(uy2,x->Position(uy2,
AsSet(\mod(c[3]*x,27))))));
c3:=List(l3,g->CycleIndex(g,[1..Length(uy2)]));
g3:=List(c3, f->Value(f,[x_1,x_3,x_9],[1+t^2,1+t^6,1+t^18]));

cyc:=[];
for i in [1..Length(g1)] do
   cyc[i]:=g1[i]*g2[i]*g3[i];
  od;
allones:=[];
for i in [1..Length(g1)] do
   allones[i]:=1;
  od;
genfn1_1:=(1/Length(g1))*(cyc*allones);
\end{spverbatim}

\begin{lstlisting}[mathescape]
$\mbox {A12.2}:$
\end{lstlisting}

\begin{lstlisting}[mathescape]
$\neg(R_{00})$
\end{lstlisting}
\begin{spverbatim}
l1:= List(grp12,c->PermList(List(uystar0,x->Position(uystar0,
AsSet(\mod(c[1]*x,27))))));
c1:=List(l1,g->CycleIndex(g,[1..Length(uystar0)]));
g1:=List(c1, f->Value(f,[x_1,x_3,x_9],[1+t^6,1+t^18,1+t^54]));

l2:= List(grp12,c->PermList(List(uy1,x->Position(uy1,
AsSet(\mod(c[2]*x,27))))));
c2:=List(l2,g->CycleIndex(g,[1..Length(uy1)]));
g2:=List(c2, f->Value(f,[x_1,x_3,x_9],[1+t^2,1+t^6,1+t^18]));

l3:= List(grp12,c->PermList(List(uy2,x->Position(uy2,
AsSet(\mod(c[3]*x,27))))));
c3:=List(l3,g->CycleIndex(g,[1..Length(uy2)]));
g3:=List(c3, f->Value(f,[x_1,x_3,x_9],[1+t^2,1+t^6,1+t^18]));

cyc:=[];
for i in [1..Length(g1)] do
   cyc[i]:=g1[i]*g2[i]*g3[i];
  od;
allones:=[];
for i in [1..Length(g1)] do
   allones[i]:=1;
  od;
genfn1_2:=(1/Length(g1))*(cyc*allones);

generatingfn1:= genfn1_1- genfn1_2;
\end{spverbatim}

\begin{lstlisting}[mathescape]
$A_{13}:$
\end{lstlisting}
\begin{spverbatim}
grp13:=[];
for a in grp1 do
for b in grp3 do
c:=[1,1,1];
c[1]:=\mod(a[1]*b[1],27);
c[2]:=\mod(a[2]*b[2],27);
c[3]:=\mod(a[3]*b[3],27);
Add(grp13,c);
od;
od;
grp13:=DuplicateFreeList(grp13);
\end{spverbatim}

\begin{lstlisting}[mathescape]
$\mbox {A13.1}:$
\end{lstlisting}
\begin{spverbatim}
l1:= List(grp13,c->PermList(List(uy0,x->Position(uy0,
AsSet(\mod(c[1]*x,27))))));
c1:=List(l1,g->CycleIndex(g,[1..Length(uy0)]));
g1:=List(c1, f->Value(f,[x_1,x_3,x_9],[1+t^2,1+t^6,1+t^18]));

l2:= List(grp13,c->PermList(List(uy1,x->Position(uy1,
AsSet(\mod(c[2]*x,27))))));
c2:=List(l2,g->CycleIndex(g,[1..Length(uy1)]));
g2:=List(c2, f->Value(f,[x_1,x_3,x_9],[1+t^2,1+t^6,1+t^18]));

l3:= List(grp13,c->PermList(List(uy2,x->Position(uy2,
AsSet(\mod(c[3]*x,27))))));
c3:=List(l3,g->CycleIndex(g,[1..Length(uy2)]));
g3:=List(c3, f->Value(f,[x_1,x_3,x_9],[1+t^2,1+t^6,1+t^18]));

cyc:=[];
for i in [1..Length(g1)] do
   cyc[i]:=g1[i]*g2[i]*g3[i];
  od;
allones:=[];
for i in [1..Length(g1)] do
   allones[i]:=1;
  od;
genfn2_1:=(1/Length(g1))*(cyc*allones);
\end{spverbatim}

\begin{lstlisting}[mathescape]
$\mbox {A13.2}:$
\end{lstlisting}
\begin{lstlisting}[mathescape]
$\neg(R_{00})$
\end{lstlisting}
\begin{spverbatim}
l1:= List(grp13,c->PermList(List(uystar0,x->Position(uystar0,
AsSet(\mod(c[1]*x,27))))));
c1:=List(l1,g->CycleIndex(g,[1..Length(uystar0)]));
g1:=List(c1, f->Value(f,[x_1,x_3,x_9],[1+t^6,1+t^18,1+t^54]));

l2:= List(grp13,c->PermList(List(uy1,x->Position(uy1,
AsSet(\mod(c[2]*x,27))))));
c2:=List(l2,g->CycleIndex(g,[1..Length(uy1)]));
g2:=List(c2, f->Value(f,[x_1,x_3,x_9],[1+t^2,1+t^6,1+t^18]));

l3:= List(grp13,c->PermList(List(uy2,x->Position(uy2,
AsSet(\mod(c[3]*x,27))))));
c3:=List(l3,g->CycleIndex(g,[1..Length(uy2)]));
g3:=List(c3, f->Value(f,[x_1,x_3,x_9],[1+t^2,1+t^6,1+t^18]));

cyc:=[];
for i in [1..Length(g1)] do
   cyc[i]:=g1[i]*g2[i]*g3[i];
  od;
allones:=[];
for i in [1..Length(g1)] do
   allones[i]:=1;
  od;
genfn2_2:=(1/Length(g1))*(cyc*allones);
\end{spverbatim}

\begin{lstlisting}[mathescape]
$\neg(R_{10})$
\end{lstlisting}
\begin{spverbatim}
l1:= List(grp13,c->PermList(List(uy0,x->Position(uy0,
AsSet(\mod(c[1]*x,27))))));
c1:=List(l1,g->CycleIndex(g,[1..Length(uy0)]));
g1:=List(c1, f->Value(f,[x_1,x_3,x_9],[1+t^2,1+t^6,1+t^18]));

l2:= List(grp13,c->PermList(List(uystar1,x->Position(uystar1,
AsSet(\mod(c[2]*x,27))))));
c2:=List(l2,g->CycleIndex(g,[1..Length(uystar1)]));
g2:=List(c2, f->Value(f,[x_1,x_3,x_9],[1+t^6,1+t^18,1+t^54]));

l3:= List(grp13,c->PermList(List(uy2,x->Position(uy2,
AsSet(\mod(c[3]*x,27))))));
c3:=List(l3,g->CycleIndex(g,[1..Length(uy2)]));
g3:=List(c3, f->Value(f,[x_1,x_3,x_9],[1+t^2,1+t^6,1+t^18]));

cyc:=[];
for i in [1..Length(g1)] do
   cyc[i]:=g1[i]*g2[i]*g3[i];
  od;
allones:=[];
for i in [1..Length(g1)] do
   allones[i]:=1;
  od;
genfn2_3:=(1/Length(g1))*(cyc*allones);

generatingfn2:=genfn2_2+genfn2_3;
\end{spverbatim}

\begin{lstlisting}[mathescape]
$\neg(R_{00})\cap \neg(R_{10})$
\end{lstlisting}
\begin{spverbatim}
l1:= List(grp13,c->PermList(List(uystar0,x->Position(uystar0,
AsSet(\mod(c[1]*x,27))))));
c1:=List(l1,g->CycleIndex(g,[1..Length(uystar0)]));
g1:=List(c1, f->Value(f,[x_1,x_3,x_9],[1+t^6,1+t^18,1+t^54]));

l2:= List(grp13,c->PermList(List(uystar1,x->Position(uystar1,
AsSet(\mod(c[2]*x,27))))));
c2:=List(l2,g->CycleIndex(g,[1..Length(uystar1)]));
g2:=List(c2, f->Value(f,[x_1,x_3,x_9],[1+t^6,1+t^18,1+t^54]));

l3:= List(grp13,c->PermList(List(uy2,x->Position(uy2,
AsSet(\mod(c[3]*x,27))))));
c3:=List(l3,g->CycleIndex(g,[1..Length(uy2)]));
g3:=List(c3, f->Value(f,[x_1,x_3,x_9],[1+t^2,1+t^6,1+t^18]));

cyc:=[];
for i in [1..Length(g1)] do
   cyc[i]:=g1[i]*g2[i]*g3[i];
  od;
allones:=[];
for i in [1..Length(g1)] do
   allones[i]:=1;
  od;
genfn2_4:=(1/Length(g1))*(cyc*allones);
generatingfn3:=generatingfn2-genfn2_4;
generatingfn4:=genfn2_1-generatingfn3;
\end{spverbatim}

\begin{lstlisting}[mathescape]
$A_{23}:$
\end{lstlisting}
\begin{spverbatim}
grp23:=[];
for a in grp2 do
for b in grp3 do
c:=[1,1,1];
c[1]:=\mod(a[1]*b[1],27);
c[2]:=\mod(a[2]*b[2],27);
c[3]:=\mod(a[3]*b[3],27);
Add(grp23,c);
od;
od;
grp23:=DuplicateFreeList(grp23);
\end{spverbatim}

\begin{lstlisting}[mathescape]
$\mbox {A23.1}:$
\end{lstlisting}
\begin{spverbatim}
l1:= List(grp23,c->PermList(List(uy0,x->Position(uy0,
AsSet(\mod(c[1]*x,27))))));
c1:=List(l1,g->CycleIndex(g,[1..Length(uy0)]));
g1:=List(c1, f->Value(f,[x_1,x_3,x_9],[1+t^2,1+t^6,1+t^18]));

l2:= List(grp23,c->PermList(List(uy1,x->Position(uy1,
AsSet(\mod(c[2]*x,27))))));
c2:=List(l2,g->CycleIndex(g,[1..Length(uy1)]));
g2:=List(c2, f->Value(f,[x_1,x_3,x_9],[1+t^2,1+t^6,1+t^18]));

l3:= List(grp23,c->PermList(List(uy2,x->Position(uy2,
AsSet(\mod(c[3]*x,27))))));
c3:=List(l3,g->CycleIndex(g,[1..Length(uy2)]));
g3:=List(c3, f->Value(f,[x_1,x_3,x_9],[1+t^2,1+t^6,1+t^18]));

cyc:=[];
for i in [1..Length(g1)] do
   cyc[i]:=g1[i]*g2[i]*g3[i];
  od;
allones:=[];
for i in [1..Length(g1)] do
   allones[i]:=1;
  od;
genfn3_1:=(1/Length(g1))*(cyc*allones);
\end{spverbatim}

\begin{lstlisting}[mathescape]
$\mbox{A23.2}:$
\end{lstlisting}

\begin{lstlisting}[mathescape]
$\neg(R_{01})$
\end{lstlisting}
\begin{spverbatim}
l1:= List(grp23,c->PermList(List(uystarstar0,
x->Position(uystarstar0,AsSet(\mod(c[1]*x,27))))));
c1:=List(l1,g->CycleIndex(g,[1..Length(uystarstar0)]));
g1:=List(c1, f->Value(f,[x_1,x_3,x_9],[1+t^18,1+t^54,1+t^162]));

l2:= List(grp23,c->PermList(List(uy1,x->Position(uy1,
AsSet(\mod(c[2]*x,27))))));
c2:=List(l2,g->CycleIndex(g,[1..Length(uy1)]));
g2:=List(c2, f->Value(f,[x_1,x_3,x_9],[1+t^2,1+t^6,1+t^18]));

l3:= List(grp23,c->PermList(List(uy2,x->Position(uy2,
AsSet(\mod(c[3]*x,27))))));
c3:=List(l3,g->CycleIndex(g,[1..Length(uy2)]));
g3:=List(c3, f->Value(f,[x_1,x_3,x_9],[1+t^2,1+t^6,1+t^18]));


cyc:=[];
for i in [1..Length(g1)] do
   cyc[i]:=g1[i]*g2[i]*g3[i];
  od;
allones:=[];
for i in [1..Length(g1)] do
   allones[i]:=1;
  od;
genfn3_2:=(1/Length(g1))*(cyc*allones);
\end{spverbatim}

\begin{lstlisting}[mathescape]
$\neg(R_{10})$
\end{lstlisting}
\begin{spverbatim}
l1:= List(grp23,c->PermList(List(uy0,x->Position(uy0,
AsSet(\mod(c[1]*x,27))))));
c1:=List(l1,g->CycleIndex(g,[1..Length(uy0)]));
g1:=List(c1, f->Value(f,[x_1,x_3,x_9],[1+t^2,1+t^6,1+t^18]));

l2:= List(grp23,c->PermList(List(uystar1,x->Position(uystar1,
AsSet(\mod(c[2]*x,27))))));
c2:=List(l2,g->CycleIndex(g,[1..Length(uystar1)]));
g2:=List(c2, f->Value(f,[x_1,x_3,x_9],[1+t^6,1+t^18,1+t^54]));

l3:= List(grp23,c->PermList(List(uy2,x->Position(uy2,
AsSet(\mod(c[3]*x,27))))));
c3:=List(l3,g->CycleIndex(g,[1..Length(uy2)]));
g3:=List(c3, f->Value(f,[x_1,x_3,x_9],[1+t^2,1+t^6,1+t^18]));

cyc:=[];
for i in [1..Length(g1)] do
   cyc[i]:=g1[i]*g2[i]*g3[i];
  od;
allones:=[];
for i in [1..Length(g1)] do
   allones[i]:=1;
  od;
genfn3_3:=(1/Length(g1))*(cyc*allones);

generatingfn5:=genfn3_2+genfn3_3;
\end{spverbatim}

\begin{lstlisting}[mathescape]
$\neg(R_{01})\cap \neg(R_{10})$
\end{lstlisting}
\begin{spverbatim}
l1:= List(grp23,c->PermList(List(uystarstar0,
x->Position(uystarstar0,AsSet(\mod(c[1]*x,27))))));
c1:=List(l1,g->CycleIndex(g,[1..Length(uystarstar0)]));
g1:=List(c1, f->Value(f,[x_1,x_3,x_9],[1+t^18,1+t^54,1+t^162]));

l2:= List(grp23,c->PermList(List(uystar1,x->Position(uystar1,
AsSet(\mod(c[2]*x,27))))));
c2:=List(l2,g->CycleIndex(g,[1..Length(uystar1)]));
g2:=List(c2, f->Value(f,[x_1,x_3,x_9],[1+t^6,1+t^18,1+t^54]));

l3:= List(grp23,c->PermList(List(uy2,x->Position(uy2,
AsSet(\mod(c[3]*x,27))))));
c3:=List(l3,g->CycleIndex(g,[1..Length(uy2)]));
g3:=List(c3, f->Value(f,[x_1,x_3,x_9],[1+t^2,1+t^6,1+t^18]));
cyc:=[];
for i in [1..Length(g1)] do
   cyc[i]:=g1[i]*g2[i]*g3[i];
  od;
allones:=[];
for i in [1..Length(g1)] do
   allones[i]:=1;
  od;
genfn3_4:=(1/Length(g1))*(cyc*allones);
generatingfn6:=generatingfn5-genfn3_4;
generatingfn7:=genfn3_1- generatingfn6;
\end{spverbatim}

\begin{lstlisting}[mathescape]
$A_{123}:$
\end{lstlisting}

\begin{spverbatim}
grp123:=grp23;
\end{spverbatim}

\begin{lstlisting}[mathescape]
$\mbox{A123.1}:$
\end{lstlisting}
\begin{spverbatim}
l1:= List(grp123,c->PermList(List(uy0,x->Position(uy0,
AsSet(\mod(c[1]*x,27))))));
c1:=List(l1,g->CycleIndex(g,[1..Length(uy0)]));
g1:=List(c1, f->Value(f,[x_1,x_3,x_9],[1+t^2,1+t^6,1+t^18]));

l2:= List(grp123,c->PermList(List(uy1,x->Position(uy1,
AsSet(\mod(c[2]*x,27))))));
c2:=List(l2,g->CycleIndex(g,[1..Length(uy1)]));
g2:=List(c2, f->Value(f,[x_1,x_3,x_9],[1+t^2,1+t^6,1+t^18]));

l3:= List(grp123,c->PermList(List(uy2,x->Position(uy2,
AsSet(\mod(c[3]*x,27))))));
c3:=List(l3,g->CycleIndex(g,[1..Length(uy2)]));
g3:=List(c3, f->Value(f,[x_1,x_3,x_9],[1+t^2,1+t^6,1+t^18]));

cyc:=[];
for i in [1..Length(g1)] do
   cyc[i]:=g1[i]*g2[i]*g3[i];
  od;
allones:=[];
for i in [1..Length(g1)] do
   allones[i]:=1;
  od;
genfn4_1:=(1/Length(g1))*(cyc*allones);
\end{spverbatim}

\begin{lstlisting}[mathescape]
$\mbox{A123.2}:$
\end{lstlisting}

\begin{lstlisting}[mathescape]
$\neg(R_{00})$
\end{lstlisting}
\begin{spverbatim}
l1:= List(grp123,c->PermList(List(uystar0,x->Position(uystar0,
AsSet(\mod(c[1]*x,27))))));
c1:=List(l1,g->CycleIndex(g,[1..Length(uystar0)]));
g1:=List(c1, f->Value(f,[x_1,x_3,x_9],[1+t^6,1+t^18,1+t^54]));

l2:= List(grp123,c->PermList(List(uy1,x->Position(uy1,
AsSet(\mod(c[2]*x,27))))));
c2:=List(l2,g->CycleIndex(g,[1..Length(uy1)]));
g2:=List(c2, f->Value(f,[x_1,x_3,x_9],[1+t^2,1+t^6,1+t^18]));

l3:= List(grp123,c->PermList(List(uy2,x->Position(uy2,
AsSet(\mod(c[3]*x,27))))));
c3:=List(l3,g->CycleIndex(g,[1..Length(uy2)]));
g3:=List(c3, f->Value(f,[x_1,x_3,x_9],[1+t^2,1+t^6,1+t^18]));

cyc:=[];
for i in [1..Length(g1)] do
   cyc[i]:=g1[i]*g2[i]*g3[i];
  od;
allones:=[];
for i in [1..Length(g1)] do
   allones[i]:=1;
  od;
genfn4_2:=(1/Length(g1))*(cyc*allones);
\end{spverbatim}

\begin{lstlisting}[mathescape]
$\neg(R_{10})$
\end{lstlisting}
\begin{spverbatim}
l1:= List(grp123,c->PermList(List(uy0,x->Position(uy0,
AsSet(\mod(c[1]*x,27))))));
c1:=List(l1,g->CycleIndex(g,[1..Length(uy0)]));
g1:=List(c1, f->Value(f,[x_1,x_3,x_9],[1+t^2,1+t^6,1+t^18]));

l2:= List(grp123,c->PermList(List(uystar1,x->Position(uystar1,
AsSet(\mod(c[2]*x,27))))));
c2:=List(l2,g->CycleIndex(g,[1..Length(uystar1)]));
g2:=List(c2, f->Value(f,[x_1,x_3,x_9],[1+t^6,1+t^18,1+t^54]));

l3:= List(grp123,c->PermList(List(uy2,x->Position(uy2,
AsSet(\mod(c[3]*x,27))))));
c3:=List(l3,g->CycleIndex(g,[1..Length(uy2)]));
g3:=List(c3, f->Value(f,[x_1,x_3,x_9],[1+t^2,1+t^6,1+t^18]));


cyc:=[];
for i in [1..Length(g1)] do
   cyc[i]:=g1[i]*g2[i]*g3[i];
  od;
allones:=[];
for i in [1..Length(g1)] do
   allones[i]:=1;
  od;
genfn4_3:=(1/Length(g1))*(cyc*allones);
generatingfn8:=genfn4_2+genfn4_3;
\end{spverbatim}

\begin{lstlisting}[mathescape]
$\neg(R_{00})\cap \neg(R_{10})$
\end{lstlisting}
\begin{spverbatim}
l1:= List(grp123,c->PermList(List(uystar0,x->Position(uystar0,
AsSet(\mod(c[1]*x,27))))));
c1:=List(l1,g->CycleIndex(g,[1..Length(uystar0)]));
g1:=List(c1, f->Value(f,[x_1,x_3,x_9],[1+t^6,1+t^18,1+t^54]));

l2:= List(grp123,c->PermList(List(uystar1,x->Position(uystar1,
AsSet(\mod(c[2]*x,27))))));
c2:=List(l2,g->CycleIndex(g,[1..Length(uystar1)]));
g2:=List(c2, f->Value(f,[x_1,x_3,x_9],[1+t^6,1+t^18,1+t^54]));

l3:= List(grp123,c->PermList(List(uy2,x->Position(uy2,
AsSet(\mod(c[3]*x,27))))));
c3:=List(l3,g->CycleIndex(g,[1..Length(uy2)]));
g3:=List(c3, f->Value(f,[x_1,x_3,x_9],[1+t^2,1+t^6,1+t^18]));

cyc:=[];
for i in [1..Length(g1)] do
   cyc[i]:=g1[i]*g2[i]*g3[i];
  od;
allones:=[];
for i in [1..Length(g1)] do
   allones[i]:=1;
  od;
genfn4_4:=(1/Length(g1))*(cyc*allones);
generatingfn9:= generatingfn8-genfn4_4;
generatingfn10:=genfn4_1- generatingfn9;

G2:= generatingfn1+ generatingfn4+ generatingfn7;
G3:= generatingfn10;
\end{spverbatim}

\begin{lstlisting}[mathescape]
$\mbox {B}$
\end{lstlisting}
\begin{spverbatim}
#The condition on the three multipliers;
grpB :=alltripleszstar27;
#grpB will be the triple of group elements for B;
\end{spverbatim}

\begin{spverbatim}
l1:= List(grpB,c->PermList(List(uystarstar0,
x->Position(uystarstar0,AsSet(\mod(c[1]*x,27))))));
c1:=List(l1,g->CycleIndex(g,[1..Length(uystarstar0)]));
g1:=List(c1, f->Value(f,[x_1,x_3,x_9],[1+t^18,1+t^54,1+t^162]));

l2:= List(grpB,c->PermList(List(uystar1,x->Position(uystar1,
AsSet(\mod(c[2]*x,27))))));
c2:=List(l2,g->CycleIndex(g,[1..Length(uystar1)]));
g2:=List(c2, f->Value(f,[x_1,x_3,x_9],[1+t^6,1+t^18,1+t^54]));

l3:= List(grpB,c->PermList(List(uy2,x->Position(uy2,
AsSet(\mod(c[3]*x,27))))));c3:=List(l3,g->CycleIndex(g,[1..Length(uy2)]));
g3:=List(c3, f->Value(f,[x_1,x_3,x_9],[1+t^2,1+t^6,1+t^18]));

cyc:=[];
for i in [1..Length(g1)] do
   cyc[i]:=g1[i]*g2[i]*g3[i];
  od;
allones:=[];
for i in [1..Length(g1)] do
   allones[i]:=1;
  od;
genfnB:=(1/Length(g1))*(cyc*allones);
genfnB := (1+t^2)*(1+t^6)*(1+t^18);
G4:=genfnB;


G:=G1-G2+G3+G4;
V:=Value(G, [t], [1]);
\end{spverbatim}

\newpage
{\bf The following program describes how to obtain the sizes of the automorphism groups of the Schur rings and their normalisers, for the undirected case for n=27, in GAP.}

\begin{spverbatim}
LoadPackage("grape");;
a:=(1,2,3,4,5,6,7,8,9,10,11,12,13,14,15,16,17,18,19,20,21,22,23,24,25,
26,27);;
grp:=Group([a]);;
\end{spverbatim}

\begin{lstlisting}[mathescape]
$\mathfrak{S}_1$
\end{lstlisting}
\begin{spverbatim}
T1:=[a,a^26];;
T2:=[a^2,a^25];;
T3:=[a^3,a^24];;
T4:=[a^4,a^23];;
T5:=[a^5,a^22];;
T6:=[a^6,a^21];;
T7:=[a^7,a^20];;
T8:=[a^8,a^19];;
T9:=[a^9,a^18];;
T10:=[a^10,a^17];;
T11:=[a^11,a^16];;
T12:=[a^12,a^15];;
T13:=[a^13,a^14];;

g1:=CayleyGraph(grp,T1);;
g2:=CayleyGraph(grp,T2);;
g3:=CayleyGraph(grp,T3);;
g4:=CayleyGraph(grp,T4);;
g5:=CayleyGraph(grp,T5);;
g6:=CayleyGraph(grp,T6);;
g7:=CayleyGraph(grp,T7);;
g8:=CayleyGraph(grp,T8);;
g9:=CayleyGraph(grp,T9);;
g10:=CayleyGraph(grp,T10);;
g11:=CayleyGraph(grp,T11);;
g12:=CayleyGraph(grp,T12);;
g13:=CayleyGraph(grp,T13);;

grp1:=AutomorphismGroup(g1);;
grp2:=AutomorphismGroup(g2);;
grp3:=AutomorphismGroup(g3);;
grp4:=AutomorphismGroup(g4);;
grp5:=AutomorphismGroup(g5);;
grp6:=AutomorphismGroup(g6);;
grp7:=AutomorphismGroup(g7);;
grp8:=AutomorphismGroup(g8);;
grp9:=AutomorphismGroup(g9);;
grp10:=AutomorphismGroup(g10);;
grp11:=AutomorphismGroup(g11);;
grp12:=AutomorphismGroup(g12);;
grp13:=AutomorphismGroup(g13);;

grpall:=Intersection(grp1,grp2,grp3,grp4,grp5,grp6,grp7,grp8,grp9,grp10,
grp11,grp12,grp13);;
Size(grpall);
StructureDescription(grpall);
Size(Normalizer(SymmetricGroup(27),grpall));
\end{spverbatim}

\begin{lstlisting}[mathescape]
$\mathfrak{S}_2$
\end{lstlisting}
\begin{spverbatim}
T1:=[a,a^26,a^8,a^19,a^10,a^17];;
T2:=[a^2,a^25,a^7,a^20,a^11,a^16];;
T3:=[a^3,a^24];;
T4:=[a^4,a^23,a^5,a^22,a^13,a^14];;
T5:=[a^6,a^21];;
T6:=[a^9,a^18];;
T7:=[a^12,a^15];;

g1:=CayleyGraph(grp,T1);;
g2:=CayleyGraph(grp,T2);;
g3:=CayleyGraph(grp,T3);;
g4:=CayleyGraph(grp,T4);
g5:=CayleyGraph(grp,T5);;
g6:=CayleyGraph(grp,T6);;
g7:=CayleyGraph(grp,T7);;

grp1:=AutomorphismGroup(g1);;
grp2:=AutomorphismGroup(g2);;
grp3:=AutomorphismGroup(g3);;
grp4:=AutomorphismGroup(g4);;
grp5:=AutomorphismGroup(g5);;
grp6:=AutomorphismGroup(g6);;
grp7:=AutomorphismGroup(g7);;

grpall:=Intersection(grp1,grp2,grp3,grp4,grp5,grp6,grp7);;
Size(grpall);
StructureDescription(grpall);
Size(Normalizer(SymmetricGroup(27),grpall));
\end{spverbatim}

\begin{lstlisting}[mathescape]
$\mathfrak{S}_3$
\end{lstlisting}
\begin{spverbatim}
T1:=[a,a^26,a^2,a^25,a^4,a^23, a^5,a^22,a^7,a^20,a^8,a^19,a^10,a^17,a^11,
a^16,a^13,a^14];;
T2:=[a^3,a^24];;
T3:=[a^6,a^21];;
T4:=[a^9,a^18];;
T5:=[a^12,a^15];;

g1:=CayleyGraph(grp,T1);;
g2:=CayleyGraph(grp,T2);;
g3:=CayleyGraph(grp,T3);;
g4:=CayleyGraph(grp,T4);;
g5:=CayleyGraph(grp,T5);;

grp1:=AutomorphismGroup(g1);;
grp2:=AutomorphismGroup(g2);;
grp3:=AutomorphismGroup(g3);;
grp4:=AutomorphismGroup(g4);;
grp5:=AutomorphismGroup(g5);;

grpall:=Intersection(grp1,grp2,grp3,grp4,grp5);;

Size(grpall);
StructureDescription(grpall);
Size(Normalizer(SymmetricGroup(27),grpall));
\end{spverbatim}

\begin{lstlisting}[mathescape]
$\mathfrak{S}_4$
\end{lstlisting}
\begin{spverbatim}
T1:=[a,a^26,a^8,a^19,a^10,a^17];;
T2:=[a^2,a^25,a^7,a^20,a^11,a^16];;
T3:=[a^3,a^24, a^6,a^21, a^12,a^15];;
T4:=[a^4,a^23,a^5,a^22,a^13,a^14];;
T5:=[a^9,a^18];;

g1:=CayleyGraph(grp,T1);;
g2:=CayleyGraph(grp,T2);;
g3:=CayleyGraph(grp,T3);;
g4:=CayleyGraph(grp,T4);;
g5:=CayleyGraph(grp,T5);;

grp1:=AutomorphismGroup(g1);;
grp2:=AutomorphismGroup(g2);;
grp3:=AutomorphismGroup(g3);;
grp4:=AutomorphismGroup(g4);;
grp5:=AutomorphismGroup(g5);;

grpall:=Intersection(grp1,grp2,grp3,grp4,grp5);;
Size(grpall);
Size(Normalizer(SymmetricGroup(27),grpall));
\end{spverbatim}

\begin{lstlisting}[mathescape]
$\mathfrak{S}_5$
\end{lstlisting}
\begin{spverbatim}
T1:=[a,a^26,a^2,a^25,a^4,a^23,a^5,a^22,a^7,a^20,a^8,a^19,a^10,a^17,a^11,
a^16,a^13,a^14];;
T2:=[a^3,a^24, a^6,a^21, a^12,a^15];;
T3:=[a^9,a^18];;

g1:=CayleyGraph(grp,T1);;
g2:=CayleyGraph(grp,T2);;
g3:=CayleyGraph(grp,T3);;

grp1:=AutomorphismGroup(g1);;
grp2:=AutomorphismGroup(g2);;
grp3:=AutomorphismGroup(g3);;

grpall:=Intersection(grp1,grp2,grp3);;
Size(grpall);
Size(Normalizer(SymmetricGroup(27),grpall));
\end{spverbatim}

\begin{lstlisting}[mathescape]
$\mathfrak{S}_6$
\end{lstlisting}
\begin{spverbatim}
T1:=[a,a^26,a^2,a^25,a^4,a^23,a^5,a^22,a^7,a^20,a^8,a^19,a^10,a^17,a^11,
a^16,a^13,a^14];;
T2:=[a^3,a^24, a^6,a^21, a^9,a^18, a^12,a^15];;

g1:=CayleyGraph(grp,T1);;
g2:=CayleyGraph(grp,T2);;

grp1:=AutomorphismGroup(g1);;
grp2:=AutomorphismGroup(g2);;

grpall:=Intersection(grp1,grp2);;

Size(grpall);
Size(Normalizer(SymmetricGroup(27),grpall));
\end{spverbatim}

\begin{lstlisting}[mathescape]
$\mathfrak{S}_7$
\end{lstlisting}
\begin{spverbatim}
T1:=[a,a^26,a^2,a^25, a^3, a^24,a^4,a^23, a^5,a^22, a^6, a^21,a^7,a^20,
a^8,a^19,a^10,a^17,a^11,a^16, a^12, a^15,a^13,a^14];;
T2:=[a^9,a^18];;

g1:=CayleyGraph(grp,T1);;
g2:=CayleyGraph(grp,T2);;

gap> grp1:=AutomorphismGroup(g1);;
grp2:=AutomorphismGroup(g2);;

grpall:=Intersection(grp1,grp2);;

Size(grpall);
Size(Normalizer(SymmetricGroup(27),grpall));
\end{spverbatim}

\begin{lstlisting}[mathescape]
$\mathfrak{S}_8$
\end{lstlisting}
\begin{spverbatim}
T1:=[a,a^2,a^3,a^4, a^5, a^6,a^7,a^8, a^9,a^10, a^11, a^12,a^13,a^14,
a^15,a^16,a^17,a^18,a^19,a^20, a^21, a^22,a^23,a^24, a^25,a^26];;

g1:=CayleyGraph(grp,T1);;

grp1:=AutomorphismGroup(g1);;

Size(grp1);
Size(Normalizer(SymmetricGroup(27),grp1));
\end{spverbatim}

\newpage
{\bf GAP program for the generating function for the case $n=27$ (undirected) using the structural method.}

\begin{spverbatim}
x_1:=Indeterminate(Rationals,1);

t:=Indeterminate(Rationals,"t", [x_1]);
f1:=(1+t^2)^13;
f2:=(1+t^6)^3*(1+t^2)^4;
f3:=(1+t^18)*(1+t^2)^4;
f4:=(1+t^6)^4*(1+t^2);
f5:=(1+t^18)*(1+t^6) *(1+t^2);
f6:=(1+t^18)*(1+t^8);
f7:=(1+t^24)*(1+t^2);
f8:=1+t^26;

g8:=f8;
g7:=f7-f8;
g6:=f6-f8;
g5:=f5-(g8+g7+g6);
g4:=1/3*(f4-g8-g7-g6-g5);
g3:=1/3* (f3-g8-g7-g6-g5);
g2:=1/9*(f2-g8-g7-g6-g5-3*g4-3*g3);
g1:=1/9*(f1-g8-g7-g6-g5-(3*g4)-(3*g3)-(9*g2));

g:=g1+g2+g3+g4+g5+g6+g7+g8;
v:=Value(g,[t],[1]);
\end{spverbatim}

\appendix
\chapter{Appendix C: Complete List of Generating Functions}
\lhead{\emph{APPENDIX C.} \emph{Generating functions}}

\begin{changemargin}{-2.0cm}{0cm} 
\mathleft

\chapter{Generating Functions for Circulants of order 27, Undirected}

\begin{flalign*}
A_{1}[u;27](t)=t^{26}+t^{24}+t^{20}+t^{18}+t^8+t^6+t^2+1\\
\end{flalign*}

\begin{align*}
A_{21}[u;27](t)=t^{26}+3t^{24}+10t^{22}+34t^{20}+83t^{18}+147t^{16}\\
+194t^{14}+194t^{12}+147t^{10}+83t^8+34t^6+10t^4+3t^2+1
\end{align*}

\begin{align*}
A_{22}[u;27](t)=t^{26}+2t^{24}+2t^{22}+3t^{20}+5t^{18}+6t^{16}+5t^{14}+5t^{12}+\\
6t^{10}+5t^8+3t^6+2t^4+2t^2+1
\end{align*}

\begin{align*}
A_{2}[u;27](t)=t^{24}+8t^{22}+31t^{20}+78t^{18}+141t^{16}+189t^{14}+189t^{12}+\\
141t^{10}+78t^8+31t^6+8t^4+t^2
\end{align*}

\begin{align*}
A_{31}[u;27](t)=t^{26}+t^{24}+2t^{20}+2t^{18}+2t^{14}+2t^{12}+2t^8+2t^6+t^2+1
\end{align*}

\begin{align*}
A_{32}[u;27](t)=t^{26}+t^{24}+t^{20}+t^{18}+t^8+t^6+t^2+1
\end{align*}

\begin{align*}
A_{3}[u;27](t)=t^{20}+t^{18}+2t^{14}+2t^{12}+t^8+t^6
\end{align*}

\begin{align*}
A_{41}[u;27](t)=t^{26}+2t^{24}+2t^{22}+2t^{20}+t^{18}+t^8+2t^6+2t^4+2t^2+1
\end{align*}

\begin{align*}
A_{42}[u;27](t)=t^{26}+t^{24}+t^{20}+t^{18}+t^8+t^6+t^2+1
\end{align*}

\begin{align*}
A_{4}[u;27](t)&=t^{24}+2t^{22}+t^{20}+t^6+2t^4+t^2\\
\end{align*}

\begin{align*}
A_{51}[u;27](t)&=t^{26}+2t^{24}+2t^{22}+3t^{20}+3t^{18}+2t^{16}+3t^{14}+3t^{12}+\\
&2t^{10}+3t^8+3t^6+2t^4+2t^2+1
\end{align*}

\begin{align*}
A5_{21}[u;27](t)&=t^{26}+2t^{24}+2t^{22}+2t^{20}+t^{18}+t^8+2t^6+2t^4+2t^2+1\\
\end{align*}

\begin{align*}
A5_{22}[u;27](t)&=t^{26}+t^{24}+2t^{20}+2t^{18}+2t^{14}+2t^{12}+2t^8+2t^6+t^2+1
\end{align*}

\begin{align*}
A5_{23}[u;27](t)=t^{26}+t^{24}+t^{20}+t^{18}+t^8+t^6+t^2+1
\end{align*}

\begin{align*}
A_{52}[u;27](t)=t^{26}+2t^{24}+2t^{22}+3t^{20}+2t^{18}+2t^{14}+2t^{12}+2t^8+3t^6+2t^4+2t^2+1
\end{align*}

\begin{align*}
A_{5}[u;27](t)=t^{18}+2t^{16}+t^{14}+t^{12}+2t^{10}+t^8
\end{align*}

%\begin{multline*}
\begin{align*}
Total[u;27](t)&=A_{1}+A_{2}+A_{3}+A_{4}+A_{5}\\
&=t^{26}+3t^{24}+10t^{22}+34t^{20}+81t^{18}+143t^{16}+192t^{14}+\\
&192t^{12}+143t^{10}+81t^8+34t^6+10t^4+3t^2+1\\
\end{align*}
%\end{multline*}

\newpage
\chapter{Generating Functions for Circulants of Order 27, Directed}
\begin{align*}
A_{1}[d;27](t)&=t^{26}+t^{25}+t^{24}+t^{23}+t^{22}+t^{21}+t^{20}+t^{19}+t^{18}+t^{17}+\\
&t^{16}+t^{15}+t^{14}+t^{13}+t^{12}+t^{11}+t^{10}+t^9+t^8+t^7+t^6+t^5+t^4+t^3+t^2+t+1\\
\end{align*}

\begin{align*}
A_{21}[d;27](t)&=t^{26}+3t^{25}+23t^{24}+152t^{23}+850t^{22}+\\
&3680t^{21}+12850t^{20}+36606t^{19}+86919t^{18}+173701t^{17}+\\
&295311t^{16}+429388t^{15}+536810t^{14}+577996t^{13}+536810t^{12}+\\
&429388t^{11}+295311t^{10}+173701t^9+86919t^8+36606t^7+12850t^6+\\
&3680t^5+850t^4+152t^3+23t^2+3t+1\\
\end{align*}

\begin{align*}
A_{22}[d;27](t)&=t^{26}+2t^{25}+6t^{24}+11t^{23}+22t^{22}+38t^{21}+\\
&65t^{20}+92t^{19}+129t^{18}+172t^{17}+214t^{16}+235t^{15}+263t^{14}+\\
&276t^{13}+263t^{12}+235t^{11}+214t^{10}+172t^9+129t^8+92t^7+\\
&65t^6+38t^5+22t^4+11t^3+6t^2+2t+1\\
\end{align*}

\begin{align*}
A_{2}[d;27](t)&=t^{25}+17t^{24}+141t^{23}+828t^{22}+3642t^{21}+\\
&12785t^{20}+36514t^{19}+86790t^{18}+173529t^{17}+295097t^{16}+\\
&429153t^{15}+536547t^{14}+577720t^{13}+536547t^{12}+429153t^{11}+\\
&295097t^{10}+173529t^9+86790t^8+36514t^7+12785t^6+3642t^5+828t^4+\\
&141t^3+17t^2+t\\
\end{align*}

\begin{align*}
A_{31}[d;27](t)&=t^{26}+t^{25}+t^{24}+2t^{23}+2t^{22}+2t^{21}+6t^{20}+\\
&6t^{19}+6t^{18}+10t^{17}+10t^{16}+10t^{15}+14t^{14}+14t^{13}+14t^{12}+\\
&10t^{11}+10t^{10}+10t^9+6t^8+6t^7+6t^6+2t^5+2t^4+2t^3+t^2+t+1\\
\end{align*}

\begin{align*}
A_{32}[d;27](t)&=t^{26}+t^{25}+t^{24}+t^{23}+t^{22}+t^{21}+t^{20}+t^{19}+\\
&t^{18}+t^{17}+t^{16}+t^{15}+2t^{14}+2t^{13}+2t^{12}+t^{11}+t^{10}+\\
&t^9+t^8+t^7+t^6+t^5+t^4+t^3+t^2+t+1\\
\end{align*}

\begin{flalign*}
A_{3}[d;27](t)&=t^{23}+t^{22}+t^{21}+5t^{20}+5t^{19}+5t^{18}+9t^{17}+\\
&9t^{16}+9t^{15}+12t^{14}+12t^{13}+12t^{12}+9t^{11}+9t^{10}+9t^9+\\
&5t^8+5t^7+5t^6+t^5+t^4+t^3\\
\end{flalign*}

\begin{align*}
A_{41}[d;27](t)&=t^{26}+2t^{25}+6t^{24}+10t^{23}+14t^{22}+10t^{21}+\\
&6t^{20}+2t^{19}+t^{18}+t^{17}+2t^{16}+6t^{15}+10t^{14}+14t^{13}+\\
&10t^{12}+6t^{11}+2t^{10}+t^9+t^8+2t^7+6t^6+10t^5+14t^4+10t^3+6t^2+2t+1\\
\end{align*}

\begin{align*}
A_{42}[d;27](t)&=t^{26}+t^{25}+t^{24}+t^{23}+2t^{22}+t^{21}+t^{20}+t^{19}+\\
&t^{18}+t^{17}+t^{16}+t^{15}+t^{14}+2t^{13}+\\
&t^{12}+t^{11}+t^{10}+t^9+t^8+t^7+t^6+t^5+2t^4+t^3+t^2+t+1\\
\end{align*}

\begin{align*}
A_{4}[d;27](t)&=t^{25}+5t^{24}+9t^{23}+12t^{22}+9t^{21}+5t^{20}+t^{19}+\\
&t^{16}+5t^{15}+9t^{14}+12t^{13}+9t^{12}+\\
&5t^{11}+t^{10}+t^7+5t^6+9t^5+12t^4+9t^3+5t^2+t\\
\end{align*}

\begin{align*}
A_{51}[d;27](t)&=t^{26}+2t^{25}+6t^{24}+11t^{23}+18t^{22}+20t^{21}+\\
&29t^{20}+38t^{19}+47t^{18}+64t^{17}+86t^{16}+\\
&91t^{15}+109t^{14}+124t^{13}+109t^{12}+91t^{11}+86t^{10}+64t^9+\\
&47t^8+38t^7+29t^6+20t^5+18t^4+11t^3+6t^2+2t+1\\
\end{align*}

\begin{align*}
A5_{21}[d;27](t)&=t^{26}+2t^{25}+6t^{24}+10t^{23}+14t^{22}+10t^{21}+\\
&6t^{20}+2t^{19}+t^{18}+t^{17}+4t^{16}+10t^{15}+\\
&20t^{14}+26t^{13}+20t^{12}+10t^{11}+4t^{10}+t^9+t^8+2t^7+6t^6+\\
&10t^5+14t^4+10t^3+6t^2+2t+1\\
\end{align*}
\begin{align*}
A5_{22}[d;27](t)&=t^{26}+t^{25}+t^{24}+2t^{23}+4t^{22}+2t^{21}+6t^{20}+\\
&10t^{19}+6t^{18}+10t^{17}+20t^{16}+10t^{15}+\\
&14t^{14}+26t^{13}+14t^{12}+10t^{11}+20t^{10}+10t^9+6t^8+10t^7+\\
&6t^6+2t^5+4t^4+2t^3+t^2+t+1\\
\end{align*}

\begin{align*}
A5_{23}[d;27](t)&=t^{26}+t^{25}+t^{24}+t^{23}+2t^{22}+t^{21}+t^{20}+t^{19}+t^{18}+t^{17}+2t^{16}+t^{15}+\\
&2t^{14}+4t^{13}+2t^{12}+t^{11}+2t^{10}+t^9+t^8+t^7+t^6+t^5+2+t^4+t^3+t^2+t+1\\
\end{align*}

\begin{align*}
A_{52}[d;27](t)&=t^{26}+2t^{25}+6t^{24}+11t^{23}+16t^{22}+11t^{21}+11t^{20}+11t^{19}+\\
&6t^{18}+10t^{17}+22t^{16}+19t^{15}+32t^{14}+48t^{13}+32t^{12}+19t^{11}+22t^{10}+\\
&10t^9+6t^8+11t^7+11t^6+11t^5+16t^4+11t^3+6t^2+2t+1\\
\end{align*}

\begin{align*}
A_{5}[d;27](t)&=2t^{22}+9t^{21}+18t^{20}+27t^{19}+41t^{18}+54t^{17}+64t^{16}+72t^{15}+77t^{14}+76t^{13}+\\
&77t^{12}+72t^{11}+64t^{10}+54t^9+41t^8+27t^7+ 18t^6+9t^5+2t^4\\
\end{align*}

\begin{align*}
Total[d;27](t)&=A_{1}+A_{2}+A_{3}+A_{4}+A_{5}\\
&=t^{26}+3t^{25}+23t^{24}+152t^{23}+844t^{22}+3662t^{21}+12814t^{20}+36548t^{19}+\\
&86837t^{18}+173593t^{17}+295172t^{16}+429240t^{15}+536646t^{14}+577821t^{13}+\\
&536646t^{12}+429240t^{11}+295172t^{10}+173593t^9+86837t^8+36548t^7+12814t^6+\\
&3662t^5+844t^4+152t^3+23t^2+3t+1\\
\end{align*}

\newpage
\chapter{Generating Functions for Circulants of Order 125, Undirected}
\begin{align*}
A_{1}[u;125](t)&=t^{124}+t^{122}+t^{120}+t^{114}+t^{112}+t^{110}+t^{104}+t^{102}+\\
&t^{100}+t^{74}+t^{72}+t^{70}+t^{64}+t^{62}+t^{60}+t^{54}+t^{52}+t^{50}+\\
&t^{24}+t^{22}+t^{20}+t^{14}+t^{12}+t^{10}+t^4+t^2+1
\end{align*}

\begin{align*}
A_{21}[u;125](t)&=t^{124}+3t^{122}+45t^{120}+774t^{118}+\\
&11207t^{116}+129485t^{114}+1229667t^{112}+9836040t^{110}+\\
&67622817t^{108}+405732239t^{106}+2150382717t^{104}+10165427246t^{102}+\\
&43203078063t^{100}+166165626045t^{98}+581579741555t^{96}+\\
&1861055001312t^{94}+ 5466849219029t^{92}+14792650395167t^{90}+\\
&36981626386305t^{88}+85641660167910t^{86}+184129570243991t^{84}+\\
&368259138707289t^{82}+686301123821491t^{80}+\\
&1193567168911620t^{78}+1939546652300505t^{76}+2948110907204795t^{74}+\\
&4195388602835609t^{72}+5593851464940846t^{70}+6992314336474163t^{68}+\\
&8197885767577933t^{66}+9017674350348587t^{64}+9308567065124256t^{62}+\\
&9017674350348587t^{60}+8197885767577933t^{58}+6992314336474163t^{56}+\\
&5593851464940846t^{54}+4195388602835609t^{52}+2948110907204795t^{50}+\\
&1939546652300505t^{48}+1193567168911620t^{46}+686301123821491t^{44}+\\
&368259138707289t^{42}+184129570243991t^{40}+85641660167910t^{38}+\\
&36981626386305t^{36}+14792650395167t^{34}+ 5466849219029t^{32}+1861055001312t^{30}+\\
&581579741555t^{28}+166165626045t^{26}+43203078063t^{24}+10165427246t^{22}+\\
&2150382717t^{20}+405732239t^{18}+67622817t^{16}+9836040t^{14}+1229667t^{12}+\\
&129485t^{10}+11207t^8+774t^6+45t^4+3t^2+1\\
\end{align*}

\begin{align*}
A_{22}[u;125](t)&=t^{124}+2t^{122}+8t^{120}+22t^{118}+51t^{116}+\\
&81t^{114}+108t^{112}+146t^{110}+271t^{108}+517t^{106}+805t^{104}+\\
&980t^{102}+1093t^{100}+1485t^{98}+2455t^{96}+3642t^{94}+4324t^{92}+\\
&4357t^{90}+4875t^{88}+6930t^{86}+9826t^{84}+11394t^{82}+10901t^{80}+\\
&10560t^{78}+13050t^{76}+17450t^{74}+19872t^{72}+18308t^{70}+15954t^{68}+\\
&17094t^{66}+21374t^{64}+23792t^{62}+21374t^{60}+17094t^{58}+15954t^{56}+\\
&18308t^{54}+19872t^{52}+17450t^{50}+13050t^{48}+10560t^{46}+10901t^{44}+\\
&11394t^{42}+9826t^{40}+6930t^{38}+4875t^{36}+4357t^{34}+4324t^{32}+3642t^{30}+\\
&2455t^{28}+1485t^{26}+1093t^{24}+980t^{22}+805t^{20}+517t^{18}+271t^{16}+\\
&146t^{14}+108t^{12}+81t^{10}+51t^8+ 22t^6+8t^4+2t^2+1\\
\end{align*}

\begin{align*}
A_{2}[u;&125](t)=t^{122}+37t^{120}+752t^{118}+11156t^{116}+129404t^{114}+\\
&1229559t^{112}+9835894t^{110}+67622546t^{108}+405731722t^{106}+\\
&2150381912t^{104}+10165426266t^{102}+43203076970t^{100}+166165624560t^{98}+\\
&581579739100t^{96}+1861054997670t^{94}+5466849214705t^{92}+14792650390810t^{90}+\\
&36981626381430t^{88}+85641660160980t^{86}+184129570234165t^{84}+ 368259138695895t^{82}+\\
&686301123810590t^{80}+1193567168901060t^{78}+1939546652287455t^{76}+\\
&2948110907187345t^{74}+4195388602815737t^{72}+5593851464922538t^{70}+\\
&6992314336458209t^{68}+8197885767560839t^{66}+9017674350327213t^{64}+\\
&9308567065100464t^{62}+9017674350327213t^{60}+8197885767560839t^{58}+\\
&6992314336458209t^{56}+5593851464922538t^{54}+ 4195388602815737t^{52}+\\
&2948110907187345t^{50}+1939546652287455t^{48}+1193567168901060t^{46}+\\
&686301123810590t^{44}+368259138695895t^{42}+184129570234165t^{40}+85641660160980t^{38}+\\
&36981626381430t^{36}+14792650390810t^{34}+ 5466849214705t^{32}+1861054997670t^{30}+\\
&581579739100t^{28}+166165624560t^{26}+43203076970t^{24}+10165426266t^{22}+\\
&2150381912t^{20}+405731722t^{18}+67622546t^{16}+9835894t^{14}+1229559t^{12}+\\
&129404t^{10}+11156t^8+752t^6+37t^4+t^2
\end{align*}

\begin{align*}
A_{31}[u;125](t)&=t^{124}+t^{122}+t^{120}+2t^{114}+2t^{112}+2t^{110}+8t^{104}+8t^{102}+\\
&8t^{100}+22t^{94}+22t^{92}+22t^{90}+51t^{84}+51t^{82}+51t^{80}+80t^{74}+80t^{72}+\\
&80t^{70}+96t^{64}+96t^{62}+96t^{60}+80t^{54}+80t^{52}+80t^{50}+51t^{44}+51t^{42}+\\
&51t^{40}+22t^{34}+22t^{32}+22t^{30}+8t^{24}+8t^{22}+8t^{20}+2t^{14}+2t^{12}+2t^{10}+t^4+t^2+1
\end{align*}

\begin{align*}
A_{32}[u;125](t)&=t^{124}+t^{122}+t^{120}+t^{114}+t^{112}+t^{110}+t^{104}+t^{102}+t^{100}+t^{74}+\\
&t^{72}+t^{70}+2t^{64}+2t^{62}+2t^{60}+t^{54}+t^{52}+t^{50}+t^{24}+t^{22}+\\
&t^{20}+t^{14}+t^{12}+t^{10}+t^4+t^2+1
\end{align*}

\begin{align*}
A_{3}[u;125](t)&=t^{114}+t^{112}+t^{110}+7t^{104}+7t^{102}+7t^{100}+22t^{94}+22t^{92}+22t^{90}+\\
&51t^{84}+51t^{82}+51t^{80}+79t^{74}+79t^{72}+79t^{70}+94t^{64}+94t^{62}+94t^{60}+79t^{54}+79t^{52}+\\
&79t^{50}+51t^{44}+51t^{42}+51t^{40}+22t^{34}+22t^{32}+22t^{30}+7t^{24}+7t^{22}+7t^{20}+t^{14}+\\
&t^{12}+t^{10}
\end{align*}

\begin{align*}
A_{41}[u;125](t)&=t^{124}+2t^{122}+8t^{120}+22t^{118}+51t^{116}+80t^{114}+96t^{112}+80t^{110}+\\
&51t^{108}+22t^{106}+8t^{104}+2t^{102}+t^{100}+t^{74}+2t^{72}+8t^{70}+22t^{68}+51t^{66}+80t^{64}+\\
&96t^{62}+80t^{60}+51t^{58}+22t^{56}+8t^{54}+2t^{52}+t^{50}+t^{24}+2t^{22}+8t^{20}+22t^{18}+\\
&51t^{16}+80t^{14}+96t^{12}+80t^{10}+51t^8+22t^6+8t^4+2t^2+1
\end{align*}

\begin{align*}
A_{42}[u;125](t)&=t^{124}+t^{122}+t^{120}+t^{114}+2t^{112}+t^{110}+t^{104}+t^{102}+\\
&t^{100}+t^{74}+t^{72}+t^{70}+t^{64}+2t^{62}+t^{60}+t^{54}+t^{52}+t^{50}+t^{24}+t^{22}+\\
&t^{20}+t^{14}+2t^{12}+t^{10}+t^4+t^2+1
\end{align*}

\begin{align*}
A_{4}[u;125](t)&=t^{122}+7t^{120}+22t^{118}+51t^{116}+79t^{114}+94t^{112}+79t^{110}+\\
&51t^{108}+22t^{106}+7t^{104}+t^{102}+t^{72}+7t^{70}+22t^{68}+51t^{66}+79t^{64}+94t^{62}+\\
&79t^{60}+51t^{58}+22t^{56}+7t^{54}+t^{52}+t^{22}+7t^{20}+22t^{18}+51t^{16}+79t^{14}+\\
&94t^{12}+79t^{10}+51t^8+22t^6+7t^4+t^2
\end{align*}

\begin{align*}
A_{51}[u;&125](t)=t^{124}+2t^{122}+8t^{120}+22t^{118}+51t^{116}+81t^{114}+100t^{112}+94t^{110}+95t^{108}+\\
&121t^{106}+173t^{104}+208t^{102}+225t^{100}+297t^{98}+491t^{96}+746t^{94}+900t^{92}+889t^{90}+975t^{88}+\\
&1386t^{86}+2006t^{84}+2358t^{82}+2221t^{80}+2112t^{78}+2610t^{76}+3554t^{74}+4104t^{72}+3736t^{70}+\\
&3226t^{68}+3498t^{66}+4478t^{64}+5056t^{62}+4478t^{60}+3498t^{58}+3226t^{56}+3736t^{54}+4104t^{52}+\\
&3554t^{50}+2610t^{48}+2112t^{46}+2221t^{44}+2358t^{42}+2006t^{40}+1386t^{38}+975t^{36}+889t^{34}+\\
&900t^{32}+746t^{30}+491t^{28}+297t^{26}+225t^{24}+208t^{22}+173t^{20}+121t^{18}+95t^{16}+94t^{14}+\\
&100t^{12}+81t^{10}+51t^8+22t^6+8t^4+2t^2+1
\end{align*}

\begin{align*}
A5_{21}[u;125](t)&=t^{124}+2t^{122}+8t^{120}+22t^{118}+51t^{116}+80t^{114}+96t^{112}+80t^{110}+\\
&51t^{108}+22t^{106}+8t^{104}+2t^{102}+t^{100}+t^{74}+4t^{72}+14t^{70}+44t^{68}+99t^{66}+160t^{64}+\\
&188t^{62}+160t^{60}+99t^{58}+44t^{56}+14t^{54}+4t^{52}+t^{50}+t^{24}+2t^{22}+8t^{20}+22t^{18}+\\
&51t^{16}+80t^{14}+96t^{12}+80t^{10}+51t^8+22t^6+8t^4+2t^2+1
\end{align*}

\begin{align*}
A5_{22}[u;125](t)&=t^{124}+t^{122}+t^{120}+2t^{114}+4t^{112}+2t^{110}+8t^{104}+14t^{102}+8t^{100}+\\
&22t^{94}+44t^{92}+22t^{90}+51t^{84}+99t^{82}+51t^{80}+80t^{74}+160t^{72}+80t^{70}+96t^{64}+188t^{62}+\\
&96t^{60}+80t^{54}+160t^{52}+80t^{50}+51t^{44}+99t^{42}+51t^{40}+22t^{34}+44t^{32}+22t^{30}+8t^{24}+\\
&14t^{22}+8t^{20}+2t^{14}+4t^{12}+2t^{10}+t^4+t^2+1
\end{align*}

\begin{align*}
A5_{23}[u;125](t)&=t^{124}+t^{122}+t^{120}+t^{114}+2t^{112}+t^{110}+t^{104}+t^{102}+\\
&t^{100}+t^{74}+2t^{72}+t^{70}+2t^{64}+4t^{62}+2t^{60}+t^{54}+2t^{52}+t^{50}+t^{24}+t^{22}+\\
&t^{20}+t^{14}+2t^{12}+t^{10}+t^4+t^2+1\\
\end{align*}

\begin{align*}
A_{52}[u;&125](t)=t^{124}+2t^{122}+8t^{120}+22t^{118}+51t^{116}+81t^{114}+98t^{112}+\\
&81t^{110}+51t^{108}+22t^{106}+15t^{104}+15t^{102}+8t^{100}+22t^{94}+44t^{92}+22t^{90}+\\
&51t^{84}+99t^{82}+51t^{80}+80t^{74}+162t^{72}+93t^{70}+44t^{68}+99t^{66}+254t^{64}+\\
&372t^{62}+254t^{60}+99t^{58}+44t^{56}+93t^{54}+162t^{52}+80t^{50}+51t^{44}+99t^{42}+\\
&51t^{40}+22t^{34}+44t^{32}+22t^{30}+8t^{24}+15t^{22}+15t^{20}+22t^{18}+51t^{16}+\\
&81t^{14}+98t^{12}+81t^{10}+51t^8+22t^6+8t^4+2t^2+1
\end{align*}

\begin{align*}
A_{5}[u;&125](t)=2t^{112}+13t^{110}+44t^{108}+99t^{106}+158t^{104}+193t^{102}+217t^{100}+\\
&297t^{98}+491t^{96}+724t^{94}+856t^{92}+867t^{90}+975t^{88}+1386t^{86}+1955t^{84}+2259t^{82}+\\
&2170t^{80}+2112t^{78}+2610t^{76}+3474t^{74}+3942t^{72}+3643t^{70}+3182t^{68}+3399t^{66}+4224t^{64}+\\
&4684t^{62}+4224t^{60}+3399t^{58}+3182t^{56}+3643t^{54}+3942t^{52}+3474t^{50}+2610t^{48}+2112t^{46}+\\
&2170t^{44}+2259t^{42}+1955t^{40}+1386t^{38}+975t^{36}+867t^{34}+856t^{32}+724t^{30}+491t^{28}+\\
&297t^{26}+217t^{24}+193t^{22}+158t^{20}+99t^{18}+44t^{16}+13t^{14}+2t^{12}
\end{align*}

\begin{align*}
Total[u;125](t)&=A_{1}+A_{2}+A_{3}+A_{4}+A_{5}\\
&t^{124}+3t^{122}+45t^{120}+774t^{118}+11207t^{116}+129485t^{114}+\\
&1229657t^{112}+9835988t^{110}+67622641t^{108}+ 405731843t^{106}+\\
&2150382085t^{104}+10165426468t^{102}+43203077195t^{100}+\\
&166165624857t^{98}+581579739591t^{96}+1861054998416t^{94}+\\
&5466849215583t^{92}+14792650391699t^{90}+36981626382405t^{88}+\\
&85641660162366t^{86}+184129570236171t^{84}+368259138698205t^{82}+\\
&686301123812811t^{80}+1193567168903172t^{78}+1939546652290065t^{76}+\\
&2948110907190899t^{74}+4195388602819760t^{72}+5593851464926268t^{70}+\\
&6992314336461413t^{68}+8197885767564289t^{66}+\\
&9017674350331611t^{64}+9308567065105337t^{62}+9017674350331611t^{60}+\\
&8197885767564289t^{58}+6992314336461413t^{56}+5593851464926268t^{54}+\\
&4195388602819760t^{52}+2948110907190899t^{50}+1939546652290065t^{48}+\\
&1193567168903172t^{46}+686301123812811t^{44}+368259138698205t^{42}+\\
&184129570236171t^{40}+85641660162366t^{38}+36981626382405t^{36}+\\
&14792650391699t^{34}+5466849215583t^{32}+1861054998416t^{30}+\\
&581579739591t^{28}+166165624857t^{26}+43203077195t^{24}+\\
&10165426468t^{22}+2150382085t^{20}+405731843t^{18}+67622641t^{16}+\\
&9835988t^{14}+1229657t^{12}+129485t^{10}+11207t^8+\\
&774t^6+45t^4+3t^2+1
\end{align*}

\newpage
\chapter{Generating Functions for Circulants of Order 125, Directed}
\begin{align*}
A_{1}[d;&125](t)=t^{124}+t^{123}+2t^{122}+t^{121}+t^{120}+t^{119}+t^{118}+2t^{117}+t^{116}+t^{115}+\\
&2t^{114}+2t^{113}+4t^{112}+2t^{111}+2t^{110}+t^{109}+t^{108}+2t^{107}+t^{106}+\\
&t^{105}+t^{104}+t^{103}+2t^{102}+t^{101}+t^{100}+t^{99}+t^{98}+2t^{97}+t^{96}+t^{95}+t^{94}+\\
&t^{93}+2t^{92}+t^{91}+t^{90}+2t^{89}+2t^{88}+4t^{87}+2t^{86}+2t^{85}+t^{84}+t^{83}+\\
&2t^{82}+t^{81}+t^{80}+t^{79}+t^{78}+2t^{77}+t^{76}+t^{75}+2t^{74}+2t^{73}+4t^{72}+2t^{71}+\\
&2t^{70}+2t^{69}+2t^{68}+4t^{67}+2t^{66}+2t^{65}+4t^{64}+4t^{63}+8t^{62}+4t^{61}+\\
&4t^{60}+2t^{59}+2t^{58}+4t^{57}+2t^{56}+2t^{55}+2t^{54}+2t^{53}+4t^{52}+2t^{51}+\\
&2t^{50}+t^{49}+t^{48}+2t^{47}+t^{46}+t^{45}+t^{44}+t^{43}+2t^{42}+t^{41}+t^{40}+2t^{39}+\\
&2t^{38}+4t^{37}+2t^{36}+2t^{35}+t^{34}+t^{33}+2t^{32}+t^{31}+t^{30}+t^{29}+t^{28}+2t^{27}+\\
&t^{26}+t^{25}+t^{24}+t^{23}+2t^{22}+t^{21}+t^{20}+t^{19}+t^{18}+2t^{17}+t^{16}+t^{15}+2t^{14}+\\
&2t^{13}+4t^{12}+2t^{11}+2t^{10}+t^9+t^8+2t^7+t^6+t^5+t^4+t^3+2t^2+t+1
\end{align*}

\begin{align*}
A_{21}[d;&125](t)=t^{124}+3t^{123}+90t^{122}+3183t^{121}+94261t^{120}+\\
&2253202t^{119}+44660548t^{118}+752765650t^{117}+11009028511t^{116}+\\
&141893733677t^{115}+1631777415270t^{114}+16911146129477t^{113}+\\
&159246625098675t^{112}+1371970915997752t^{111}+10877769405955360t^{110}+\\
&79770308934046652t^{109}+543435229636808801t^{108}+3452412046875069875t^{107}+\\
&20522671612161479218t^{106}+114494904782475069427t^{105}+\\
&601098250109024238421t^{104}+2976867524344065707158t^{103}+\\
&13937152500343124859524t^{102}+61808241523238164274278t^{101}+\\
&260109683076982063079271t^{100}+1040438732307841647264001t^{99}+\\
&3961670557633787552369622t^{98}+14379396838818630702670729t^{97}+\\
&49814339048764832456390627t^{96}+164902639609703309837818576t^{95}+\\
&522191692097394744372570064t^{94}+1583419969585645757298166644t^{93}+\\
&4601814286608285714581287637t^{92}+12829300435392788916008371659t^{91}+\\
&34337245282963060081957805954t^{90}+88295773584762135475767267735t^{89}+\\
&218286773584550853415117681033t^{88}+519168542579472256033947258058t^{87}+\\
&1188622715905633892159206423188t^{86}+2621065476099602847256268985110t^{85}+\\
&5569764136711656142481362767183t^{84}+11411224084970222152069566117509t^{83}+\\
&22550752358393534437025056773982t^{82}+43003760311355111831098813088769t^{81}+\\
&79166013300449183486759966707923t^{80}+140739579200798547810856736509576t^{79}+\\
&241704929497023593576293669675592t^{78}+401127329803571069200032815963604t^{77}+\\
&643475091559895257811487192086457t^{76}+998042999154123255509591723712251t^{75}+\\
&1497064498731184884738438067393274t^{74}+2172211233453091791403216937366843t^{73}+\\
&3049450385424532709259897590332593t^{72}+4142649580199365187088270983859398t^{71}+\\
&5446817040632498674709277300717412t^{70}+6932312597168634673342988991544262t^{69}+\\
&8541599450082782011722336731949035t^{68}+10189978291326827658936491765174825t^{67}+\\
&11771181819291335403215232033465030t^{66}+13167762713105561632909744695314237t^{65}+\\
&14265076272531025106827725830461707t^{64}+14966637400688288631941128601570312t^{63}+\\
&15208034778118744904852527054387840t^{62}+14966637400688288631941128601570312t^{61}+\\
&14265076272531025106827725830461707t^{60}+13167762713105561632909744695314237t^{59}+\\
&11771181819291335403215232033465030t^{58}+10189978291326827658936491765174825t^{57}+\\
\end{align*}
%A_21 continued
\begin{align*}
&8541599450082782011722336731949035t^{56}+6932312597168634673342988991544262t^{55}+\\
&5446817040632498674709277300717412t^{54}+4142649580199365187088270983859398t^{53}+\\
&3049450385424532709259897590332593t^{52}+2172211233453091791403216937366843t^{51}+\\
&1497064498731184884738438067393274t^{50}+998042999154123255509591723712251t^{49}+\\
&643475091559895257811487192086457t^{48}+401127329803571069200032815963604t^{47}+\\
&241704929497023593576293669675592t^{46}+140739579200798547810856736509576t^{45}+\\
&79166013300449183486759966707923t^{44}+43003760311355111831098813088769t^{43}+\\
&22550752358393534437025056773982t^{42}+11411224084970222152069566117509t^{41}+\\
&5569764136711656142481362767183t^{40}+2621065476099602847256268985110t^{39}+\\
&1188622715905633892159206423188t^{38}+519168542579472256033947258058t^{37}+\\
&218286773584550853415117681033t^{36}+88295773584762135475767267735t^{35}+\\
&34337245282963060081957805954t^{34}+12829300435392788916008371659t^{33}+\\
&4601814286608285714581287637t^{32}+1583419969585645757298166644t^{31}+\\
&522191692097394744372570064t^{30}+164902639609703309837818576t^{29}+\\
&49814339048764832456390627t^{28}+14379396838818630702670729t^{27}+\\
&3961670557633787552369622t^{26}+1040438732307841647264001t^{25}+\\
&260109683076982063079271t^{24}+61808241523238164274278t^{23}+\\
&13937152500343124859524t^{22}+2976867524344065707158t^{21}+\\
&601098250109024238421t^{20}+114494904782475069427t^{19}+\\
&20522671612161479218t^{18}+3452412046875069875t^{17}+\\
&543435229636808801t^{16}+79770308934046652t^{15}+10877769405955360t^{14}+\\
&1371970915997752t^{13}+159246625098675t^{12}+16911146129477t^{11}+\\
&1631777415270t^{10}+141893733677t^9+11009028511t^8+752765650t^7+\\
&44660548t^6+2253202t^5+94261t^4+3183t^3+90t^2+3t+1
\end{align*}

\begin{align*}
A_{22}[d;&125](t)=t^{124}+2t^{123}+16t^{122}+102t^{121}+536t^{120}+\\
&2127t^{119}+6768t^{118}+17586t^{117}+38827t^{116}+76002t^{115}+\\
&140618t^{114}+259636t^{113}+483996t^{112}+879511t^{111}+1506588t^{110}+\\
&2430477t^{109}+3813087t^{108}+6025186t^{107}+9605478t^{106}+\\
&14990352t^{105}+22363341t^{104}+32126727t^{103}+45830880t^{102}+\\
&66260127t^{101}+95776861t^{100}+134520876t^{99}+181894438t^{98}+\\
&241482810t^{97}+323296983t^{96}+437175501t^{95}+582695804t^{94}+\\
&751017303t^{93}+943649236t^{92}+1186422728t^{91}+1511709474t^{90}+\\
&1920090810t^{89}+2374236579t^{88}+2851940514t^{87}+3400809384t^{86}+\\
&4105791435t^{85}+4981822098t^{84}+5929999778t^{83}+6852264216t^{82}+\\
&7805305478t^{81}+8968682293t^{80}+10411728126t^{79}+11949925688t^{78}+\\
&13334526060t^{77}+14585287978t^{76}+16008752876t^{75}+17798416164t^{74}+\\
&19707130012t^{73}+21283301536t^{72}+22438886362t^{71}+23596306968t^{70}+\\
&25139780062t^{69}+26846406914t^{68}+28090065316t^{67}+28621603364t^{66}+\\
&28906040812t^{65}+29522002718t^{64}+30378601066t^{63}+30797103792t^{62}+\\
&30378601066t^{61}+29522002718t^{60}+28906040812t^{59}+28621603364t^{58}+\\
&28090065316t^{57}+26846406914t^{56}+25139780062t^{55}+23596306968t^{54}+\\
&22438886362t^{53}+21283301536t^{52}+19707130012t^{51}+17798416164t^{50}+\\
&16008752876t^{49}+14585287978t^{48}+13334526060t^{47}+11949925688t^{46}+\\
&10411728126t^{45}+8968682293t^{44}+7805305478t^{43}+6852264216t^{42}+\\
&5929999778t^{41}+4981822098t^{40}+4105791435t^{39}+3400809384t^{38}+\\
&2851940514t^{37}+2374236579t^{36}+1920090810t^{35}+1511709474t^{34}+\\
&1186422728t^{33}+943649236t^{32}+751017303t^{31}+582695804t^{30}+\\
&437175501t^{29}+323296983t^{28}+241482810t^{27}+181894438t^{26}+\\
&134520876t^{25}+95776861t^{24}+66260127t^{23}+45830880t^{22}+\\
&32126727t^{21}+22363341t^{20}+14990352t^{19}+9605478t^{18}+\\
&6025186t^{17}+3813087t^{16}+2430477t^{15}+1506588t^{14}+879511t^{13}+\\
&483996t^{12}+259636t^{11}+140618t^{10}+76002t^9+38827t^8+17586t^7+\\
&6768t^6+2127t^5+536t^4+102t^3+16t^2+2t+1
\end{align*}

\begin{align*}
A_{2}[d;&125](t)=t^{123}+74t^{122}+3081t^{121}+93725t^{120}+2251075t^{119}+\\
&44653780t^{118}+752748064t^{117}+11008989684t^{116}+141893657675t^{115}+\\
&1631777274652t^{114}+16911145869841t^{113}+159246624614679t^{112}+\\
&1371970915118241t^{111}+10877769404448772t^{110}+79770308931616175t^{109}+\\
&543435229632995714t^{108}+3452412046869044689t^{107}+\\
&20522671612151873740t^{106}+114494904782460079075t^{105}+\\
&601098250109001875080t^{104}+2976867524344033580431t^{103}+\\
&13937152500343079028644t^{102}+61808241523238098014151t^{101}+\\
&260109683076981967302410t^{100}+1040438732307841512743125t^{99}+\\
&3961670557633787370475184t^{98}+14379396838818630461187919t^{97}+\\
&49814339048764832133093644t^{96}+164902639609703309400643075t^{95}+\\
&522191692097394743789874260t^{94}+1583419969585645756547149341t^{93}+\\
&4601814286608285713637638401t^{92}+12829300435392788914821948931t^{91}+\\
&34337245282963060080446096480t^{90}+88295773584762135473847176925t^{89}+\\
&218286773584550853412743444454t^{88}+519168542579472256031095317544t^{87}+\\
&1188622715905633892155805613804t^{86}+2621065476099602847252163193675t^{85}+\\
&5569764136711656142476380945085t^{84}+11411224084970222152063636117731t^{83}+\\
&22550752358393534437018204509766t^{82}+43003760311355111831091007783291t^{81}+\\
&79166013300449183486750998025630t^{80}+140739579200798547810846324781450t^{79}+\\
&241704929497023593576281719749904t^{78}+401127329803571069200019481437544t^{77}+\\
&643475091559895257811472606798479t^{76}+998042999154123255509575714959375t^{75}+\\
&1497064498731184884738420268977110t^{74}+2172211233453091791403197230236831t^{73}+\\
&3049450385424532709259876307031057t^{72}+4142649580199365187088248544973036t^{71}+\\
&5446817040632498674709253704410444t^{70}+6932312597168634673342963851764200t^{69}+\\
&8541599450082782011722309885542121t^{68}+10189978291326827658936463675109509t^{67}+\\
&11771181819291335403215203411861666t^{66}+13167762713105561632909715789273425t^{65}+\\
&14265076272531025106827696308458989t^{64}+14966637400688288631941098222969246t^{63}+\\
&15208034778118744904852496257284048t^{62}+14966637400688288631941098222969246t^{61}+\\
&14265076272531025106827696308458989t^{60}+13167762713105561632909715789273425t^{59}+\\
&11771181819291335403215203411861666t^{58}+10189978291326827658936463675109509t^{57}+\\
\end{align*}
%A_2 continued
\begin{align*}
A_{2}[d;&125](t)continued...\\
&8541599450082782011722309885542121t^{56}+6932312597168634673342963851764200t^{55}+\\
&5446817040632498674709253704410444t^{54}+4142649580199365187088248544973036t^{53}+\\
&3049450385424532709259876307031057t^{52}+2172211233453091791403197230236831t^{51}+\\
&1497064498731184884738420268977110t^{50}+998042999154123255509575714959375t^{49}+\\
&643475091559895257811472606798479t^{48}+401127329803571069200019481437544t^{47}+\\
&241704929497023593576281719749904t^{46}+140739579200798547810846324781450t^{45}+\\
&79166013300449183486750998025630t^{44}+43003760311355111831091007783291t^{43}+\\
&22550752358393534437018204509766t^{42}+11411224084970222152063636117731t^{41}+\\
&5569764136711656142476380945085t^{40}+2621065476099602847252163193675t^{39}+\\
&1188622715905633892155805613804t^{38}+519168542579472256031095317544t^{37}+\\
&218286773584550853412743444454t^{36}+88295773584762135473847176925t^{35}+\\
&34337245282963060080446096480t^{34}+12829300435392788914821948931t^{33}+\\
&4601814286608285713637638401t^{32}+1583419969585645756547149341t^{31}+\\
&522191692097394743789874260t^{30}+164902639609703309400643075t^{29}+\\
&49814339048764832133093644t^{28}+14379396838818630461187919t^{27}+\\
&3961670557633787370475184t^{26}+1040438732307841512743125t^{25}+\\
&260109683076981967302410t^{24}+61808241523238098014151t^{23}+13937152500343079028644t^{22}+\\
&2976867524344033580431t^{21}+601098250109001875080t^{20}+114494904782460079075t^{19}+\\
&20522671612151873740t^{18}+3452412046869044689t^{17}+543435229632995714t^{16}+\\
&79770308931616175t^{15}+10877769404448772t^{14}+1371970915118241t^{13}+\\
&159246624614679t^{12}+16911145869841t^{11}+1631777274652t^{10}+141893657675t^{9}+\\
&11008989684t^8+752748064t^7+44653780t^6+2251075t^5+93725t^4+3081t^3+74t^2+t
\end{align*}

\begin{align*}
A_{31}[d;&125](t)=t^{124}+t^{123}+2t^{122}+t^{121}+t^{120}+2t^{119}+2t^{118}+\\
&4t^{117}+2t^{116}+2t^{115}+16t^{114}+16t^{113}+32t^{112}+16t^{111}+16t^{110}+\\
&102t^{109}+102t^{108}+204t^{107}+102t^{106}+102t^{105}+536t^{104}+536t^{103}+\\
&1072t^{102}+536t^{101}+536t^{100}+2126t^{99}+2126t^{98}+4252t^{97}+2126t^{96}+\\
&2126t^{95}+6744t^{94}+6744t^{93}+13488t^{92}+6744t^{91}+6744t^{90}+17310t^{89}+\\
&17310t^{88}+34620t^{87}+17310t^{86}+17310t^{85}+36803t^{84}+36803t^{83}+\\
&73606t^{82}+36803t^{81}+36803t^{80}+65376t^{79}+65376t^{78}+130752t^{77}+\\
&65376t^{76}+65376t^{75}+98104t^{74}+98104t^{73}+196208t^{72}+98104t^{71}+\\
&98104t^{70}+124812t^{69}+124812t^{68}+249624t^{67}+124812t^{66}+124812t^{65}+\\
&135264t^{64}+135264t^{63}+270528t^{62}+135264t^{61}+135264t^{60}+124812t^{59}+\\
&124812t^{58}+249624t^{57}+124812t^{56}+124812t^{55}+98104t^{54}+98104t^{53}+\\
&196208t^{52}+98104t^{51}+98104t^{50}+65376t^{49}+65376t^{48}+130752t^{47}+\\
&65376t^{46}+65376t^{45}+36803t^{44}+36803t^{43}+73606t^{42}+36803t^{41}+\\
&36803t^{40}+17310t^{39}+17310t^{38}+34620t^{37}+17310t^{36}+17310t^{35}+\\
&6744t^{34}+6744t^{33}+13488t^{32}+6744t^{31}+6744t^{30}+2126t^{29}+2126t^{28}+\\
&4252t^{27}+2126t^{26}+2126t^{25}+536t^{24}+536t^{23}+1072t^{22}+536t^{21}+\\
&536t^{20}+102t^{19}+102t^{18}+204t^{17}+102t^{16}+102t^{15}+16t^{14}+16t^{13}+\\
&32t^{12}+16t^{11}+16t^{10}+2t^9+2t^8+4t^7+2t^6+2t^5+t^4+t^3+2t^2+t+1\\
\end{align*}

\begin{align*}
A_{32}[d;&125](t)=t^{124}+t^{123}+2t^{122}+t^{121}+t^{120}+t^{119}+t^{118}+2t^{117}+t^{116}+\\
&t^{115}+2t^{114}+2t^{113}+4t^{112}+2t^{111}+2t^{110}+t^{109}+t^{108}+2t^{107}+t^{106}+t^{105}+\\
&t^{104}+t^{103}+2t^{102}+t^{101}+t^{100}+t^{99}+t^{98}+2t^{97}+t^{96}+t^{95}+4t^{94}+4t^{93}+8t^{92}+\\
&4t^{91}+4t^{90}+6t^{89}+6t^{88}+12t^{87}+6t^{86}+6t^{85}+4t^{84}+4t^{83}+8t^{82}+4t^{81}+\\
&4t^{80}+t^{79}+t^{78}+2t^{77}+t^{76}+t^{75}+2t^{74}+2t^{73}+4t^{72}+2t^{71}+2t^{70}+6t^{69}+\\
&6t^{68}+12t^{67}+6t^{66}+6t^{65}+10t^{64}+10t^{63}+20t^{62}+10t^{61}+10t^{60}+6t^{59}+\\
&6t^{58}+12t^{57}+6t^{56}+6t^{55}+2t^{54}+2t^{53}+4t^{52}+2t^{51}+2t^{50}+t^{49}+t^{48}+2t^{47}+\\
&t^{46}+t^{45}+4t^{44}+4t^{43}+8t^{42}+4t^{41}+4t^{40}+6t^{39}+6t^{38}+12t^{37}+6t^{36}+6t^{35}+\\
&4t^{34}+4t^{33}+8t^{32}+4t^{31}+4t^{30}+t^{29}+t^{28}+2t^{27}+t^{26}+t^{25}+t^{24}+t^{23}+2t^{22}+t^{21}+\\
&t^{20}+t^{19}+t^{18}+2t^{17}+t^{16}+t^{15}+2t^{14}+2t^{13}+4t^{12}+2t^{11}+2t^{10}+t^9+t^8+2t^7+t^6+\\
&t^5+t^4+t^3+2t^2+t+1
\end{align*}

\begin{align*}
A_{3}&[d;125](t)=t^{119}+t^{118}+2t^{117}+t^{116}+t^{115}+14t^{114}+14t^{113}+28t^{112}+14t^{111}+\\
&14t^{110}+101t^{109}+101t^{108}+202t^{107}+101t^{106}+101t^{105}+535t^{104}+535t^{103}+1070t^{102}+\\
&535t^{101}+535t^{100}+2125t^{99}+2125t^{98}+4250t^{97}+2125t^{96}+2125t^{95}+6740t^{94}+6740t^{93}+\\
&13480t^{92}+6740t^{91}+6740t^{90}+17304t^{89}+17304t^{88}+34608t^{87}+17304t^{86}+17304t^{85}+\\
&36799t^{84}+36799t^{83}+73598t^{82}+36799t^{81}+36799t^{80}+65375t^{79}+65375t^{78}+130750t^{77}+\\
&65375t^{76}+65375t^{75}+98102t^{74}+98102t^{73}+196204t^{72}+98102t^{71}+98102t^{70}+124806t^{69}+\\
&124806t^{68}+249612t^{67}+124806t^{66}+124806t^{65}+135254t^{64}+135254t^{63}+270508t^{62}+\\
&135254t^{61}+135254t^{60}+124806t^{59}+124806t^{58}+249612t^{57}+124806t^{56}+\\
&124806t^{55}+98102t^{54}+98102t^{53}+196204t^{52}+98102t^{51}+98102t^{50}+65375t^{49}+65375t^{48}+\\
&130750t^{47}+65375t^{46}+65375t^{45}+36799t^{44}+36799t^{43}+73598t^{42}+36799t^{41}+36799t^{40}+\\
&17304t^{39}+17304t^{38}+34608t^{37}+17304t^{36}+17304t^{35}+6740t^{34}+6740t^{33}+13480t^{32}+\\
&6740t^{31}+6740t^{30}+2125t^{29}+2125t^{28}+4250t^{27}+2125t^{26}+2125t^{25}+535t^{24}+535t^{23}+\\
&1070t^{22}+535t^{21}+535t^{20}+101t^{19}+101t^{18}+202t^{17}+101t^{16}+101t^{15}+14t^{14}+14t^{13}+\\
&28t^{12}+14t^{11}+14t^{10}+t^9+t^8+2t^7+t^6+t^5
\end{align*}

\begin{align*}
A_{41}[d;&125](t)=t^{124}+2t^{123}+16t^{122}+102t^{121}+536t^{120}+2126t^{119}+6744t^{118}+\\
&17310t^{117}+36803t^{116}+65376t^{115}+98104t^{114}+124812t^{113}+135264t^{112}+124812t^{111}+\\
&98104t^{110}+65376t^{109}+36803t^{108}+17310t^{107}+6744t^{106}+2126t^{105}+536t^{104}+\\
&102t^{103}+16t^{102}+2t^{101}+t^{100}+t^{99}+2t^{98}+16t^{97}+102t^{96}+536t^{95}+2126t^{94}+\\
&6744t^{93}+17310t^{92}+36803t^{91}+65376t^{90}+98104t^{89}+124812t^{88}+135264t^{87}+\\
&124812t^{86}+98104t^{85}+65376t^{84}+36803t^{83}+17310t^{82}+6744t^{81}+2126t^{80}+536t^{79}+\\
&102t^{78}+16t^{77}+2t^{76}+t^{75}+2t^{74}+4t^{73}+32t^{72}+204t^{71}+1072t^{70}+4252t^{69}+\\
&13488t^{68}+34620t^{67}+73606t^{66}+130752t^{65}+196208t^{64}+249624t^{63}+270528t^{62}+\\
&249624t^{61}+196208t^{60}+130752t^{59}+73606t^{58}+34620t^{57}+13488t^{56}+4252t^{55}+1072t^{54}+\\
&204t^{53}+32t^{52}+4t^{51}+2t^{50}+t^{49}+2t^{48}+16t^{47}+102t^{46}+536t^{45}+2126t^{44}+\\
&6744t^{43}+17310t^{42}+36803t^{41}+65376t^{40}+98104t^{39}+124812t^{38}+135264t^{37}+124812t^{36}+\\
&98104t^{35}+65376t^{34}+36803t^{33}+17310t^{32}+6744t^{31}+2126t^{30}+536t^{29}+102t^{28}+16t^{27}+\\
&2t^{26}+t^{25}+t^{24}+2t^{23}+16t^{22}+102t^{21}+536t^{20}+2126t^{19}+6744t^{18}+\\
&17310t^{17}+36803t^{16}+65376t^{15}+98104t^{14}+124812t^{13}+135264t^{12}+124812t^{11}+\\
&98104t^{10}+65376t^9+36803t^8+17310t^7+6744t^6+2126t^5+536t^4+102t^3+16t^2+2t+1\\
\end{align*}

\begin{align*}
A_{42}[d;&125](t)=t^{124}+t^{123}+2t^{122}+t^{121}+t^{120}+t^{119}+4t^{118}+6t^{117}+4t^{116}+\\
&t^{115}+2t^{114}+6t^{113}+10t^{112}+6t^{111}+2t^{110}+t^{109}+4t^{108}+6t^{107}+4t^{106}+t^{105}+\\
&t^{104}+t^{103}+2t^{102}+t^{101}+t^{100}+t^{99}+t^{98}+2t^{97}+t^{96}+t^{95}+t^{94}+4t^{93}+6t^{92}+\\
&4t^{91}+t^{90}+2t^{89}+6t^{88}+10t^{87}+6t^{86}+2t^{85}+t^{84}+4t^{83}+6t^{82}+4t^{81}+t^{80}+t^{79}+\\
&t^{78}+2t^{77}+t^{76}+t^{75}+2t^{74}+2t^{73}+4t^{72}+2t^{71}+2t^{70}+2t^{69}+8t^{68}+12t^{67}+8t^{66}+\\
&2t^{65}+4t^{64}+12t^{63}+20t^{62}+12t^{61}+4t^{60}+2t^{59}+8t^{58}+12t^{57}+8t^{56}+2t^{55}+2t^{54}+\\
&2t^{53}+4t^{52}+2t^{51}+2t^{50}+t^{49}+t^{48}+2t^{47}+t^{46}+t^{45}+t^{44}+4t^{43}+6t^{42}+4t^{41}+\\
&t^{40}+2t^{39}+6t^{38}+10t^{37}+6t^{36}+2t^{35}+t^{34}+4t^{33}+6t^{32}+4t^{31}+t^{30}+t^{29}+t^{28}+\\
&2t^{27}+t^{26}+t^{25}+t^{24}+t^{23}+2t^{22}+t^{21}+t^{20}+t^{19}+4t^{18}+6t^{17}+4t^{16}+t^{15}+2t^{14}+\\
&6t^{13}+10t^{12}+6t^{11}+2t^{10}+t^9+4t^8+6t^7+4t^6+t^5+t^4+t^3+2t^2+t+1
\end{align*}

\begin{align*}
A_{4}[d;&125](t)=t^{123}+14t^{122}+101t^{121}+535t^{120}+2125t^{119}+6740t^{118}+17304t^{117}+\\
&36799t^{116}+65375t^{115}+98102t^{114}+124806t^{113}+135254t^{112}+124806t^{111}+98102t^{110}+\\
&65375t^{109}+36799t^{108}+17304t^{107}+6740t^{106}+2125t^{105}+535t^{104}+101t^{103}+14t^{102}+\\
&t^{101}+t^{98}+14t^{97}+101t^{96}+535t^{95}+2125t^{94}+6740t^{93}+17304t^{92}+36799t^{91}+\\
&65375t^{90}+98102t^{89}+124806t^{88}+135254t^{87}+124806t^{86}+98102t^{85}+65375t^{84}+\\
&36799t^{83}+17304t^{82}+6740t^{81}+2125t^{80}+535t^{79}+101t^{78}+14t^{77}+t^{76}+2t^{73}+\\
&28t^{72}+202t^{71}+1070t^{70}+4250t^{69}+13480t^{68}+34608t^{67}+73598t^{66}+130750t^{65}+\\
&196204t^{64}+249612t^{63}+270508t^{62}+249612t^{61}+196204t^{60}+130750t^{59}+73598t^{58}+\\
&34608t^{57}+13480t^{56}+4250t^{55}+1070t^{54}+202t^{53}+28t^{52}+2t^{51}+t^{48}+14t^{47}+\\
&101t^{46}+535t^{45}+2125t^{44}+6740t^{43}+17304t^{42}+36799t^{41}+65375t^{40}+98102t^{39}+\\
&124806t^{38}+135254t^{37}+124806t^{36}+98102t^{35}+65375t^{34}+36799t^{33}+17304t^{32}+\\
&6740t^{31}+2125t^{30}+535t^{29}+101t^{28}+14t^{27}+t^{26}+t^{23}+14t^{22}+101t^{21}+\\
&535t^{20}+2125t^{19}+6740t^{18}+17304t^{17}+36799t^{16}+65375t^{15}+98102t^{14}+124806t^{13}+\\
&135254t^{12}+124806t^{11}+98102t^{10}+65375t^{9}+36799t^{8}+17304t^{7}+6740t^6+2125t^5+\\
&535t^4+101t^3+14t^2+t
\end{align*}

\begin{align*}
A_{51}[d;&125](t)=t^{124}+2t^{123}+16t^{122}+102t^{121}+536t^{120}+2127t^{119}+6752t^{118}+17370t^{117}+\\
&37211t^{116}+67502t^{115}+106618t^{114}+151820t^{113}+205076t^{112}+275795t^{111}+379812t^{110}+\\
&538477t^{109}+792383t^{108}+1219370t^{107}+1926814t^{106}+2999852t^{105}+4473525t^{104}+\\
&6427127t^{103}+9168744t^{102}+13253727t^{101}+19155801t^{100}+26905876t^{99}+36385694t^{98}+\\
&48306810t^{97}+64666523t^{96}+87438501t^{95}+116551356t^{94}+150246539t^{93}+188817556t^{92}+\\
&237423764t^{91}+302556490t^{90}+384345810t^{89}+475302087t^{88}+570903850t^{87}+\\
&680616648t^{86}+821485935t^{85}+996603062t^{84}+1186235314t^{83}+1370684784t^{82}+\\
&1561200314t^{81}+1793772701t^{80}+2082399626t^{79}+2390194664t^{78}+2667219060t^{77}+\\
&2917266802t^{76}+3201802876t^{75}+3559761716t^{74}+3941739812t^{73}+4257131144t^{72}+\\
&4488091562t^{71}+4719342432t^{70}+5028066062t^{69}+5369713098t^{68}+5618695220t^{67}+\\
&5724896620t^{66}+5781621812t^{65}+5904979518t^{64}+6076751970t^{63}+6160718928t^{62}+\\
&6076751970t^{61}+5904979518t^{60}+5781621812t^{59}+5724896620t^{58}+5618695220t^{57}+\\
&5369713098t^{56}+5028066062t^{55}+4719342432t^{54}+4488091562t^{53}+4257131144t^{52}+\\
&3941739812t^{51}+3559761716t^{50}+3201802876t^{49}+2917266802t^{48}+2667219060t^{47}+\\
&2390194664t^{46}+2082399626t^{45}+1793772701t^{44}+1561200314t^{43}+1370684784t^{42}+\\
&1186235314t^{41}+996603062t^{40}+821485935t^{39}+680616648t^{38}+570903850t^{37}+\\
&475302087t^{36}+384345810t^{35}+302556490t^{34}+237423764t^{33}+188817556t^{32}+\\
&150246539t^{31}+116551356t^{30}+87438501t^{29}+64666523t^{28}+48306810t^{27}+36385694t^{26}+\\
&26905876t^{25}+19155801t^{24}+13253727t^{23}+9168744t^{22}+6427127t^{21}+4473525t^{20}+\\
&2999852t^{19}+1926814t^{18}+1219370t^{17}+792383t^{16}+538477t^{15}+379812t^{14}+\\
&275795t^{13}+205076t^{12}+151820t^{11}+106618t^{10}+67502t^9+37211t^8+17370t^7+6752t^6+\\
&2127t^5+536t^4+102t^3+16t^2+2t+1
\end{align*}

\begin{align*}
A5_{21}[d;&125](t)=t^{124}+2t^{123}+16t^{122}+102t^{121}+536t^{120}+2126t^{119}+\\
&6744t^{118}+17310t^{117}+36803t^{116}+65376t^{115}+98104t^{114}+124812t^{113}+\\
&135264t^{112}+124812t^{111}+98104t^{110}+65376t^{109}+36803t^{108}+17310t^{107}+\\
&6744t^{106}+2126t^{105}+536t^{104}+102t^{103}+16t^{102}+2t^{101}+t^{100}+t^{99}+\\
&8t^{98}+60t^{97}+408t^{96}+2126t^{95}+8504t^{94}+26932t^{93}+69240t^{92}+\\
&147107t^{91}+261504t^{90}+392256t^{89}+499248t^{88}+540860t^{87}+\\
&499248t^{86}+392256t^{85}+261504t^{84}+147107t^{83}+69240t^{82}+26932t^{81}+\\
&8504t^{80}+2126t^{79}+408t^{78}+60t^{77}+8t^{76}+t^{75}+2t^{74}+12t^{73}+92t^{72}+\\
&612t^{71}+3196t^{70}+12756t^{69}+40420t^{68}+103860t^{67}+220710t^{66}+392256t^{65}+\\
&588464t^{64}+748872t^{63}+811384t^{62}+748872t^{61}+588464t^{60}+392256t^{59}+\\
&220710t^{58}+103860t^{57}+40420t^{56}+12756t^{55}+3196t^{54}+612t^{53}+92t^{52}+\\
&12t^{51}+2t^{50}+t^{49}+8t^{48}+60t^{47}+408t^{46}+2126t^{45}+8504t^{44}+\\
&26932t^{43}+69240t^{42}+147107t^{41}+261504t^{40}+392256t^{39}+499248t^{38}+\\
&540860t^{37}+499248t^{36}+392256t^{35}+261504t^{34}+147107t^{33}+69240t^{32}+\\
&26932t^{31}+8504t^{30}+2126t^{29}+408t^{28}+60t^{27}+8t^{26}+t^{25}+t^{24}+\\
&2t^{23}+16t^{22}+102t^{21}+536t^{20}+2126t^{19}+6744t^{18}+17310t^{17}+\\
&36803t^{16}+65376t^{15}+98104t^{14}+124812t^{13}+135264t^{12}+124812t^{11}+\\
&98104t^{10}+65376t^9+36803t^8+17310t^7+6744t^6+2126t^5+536t^4+102t^3+16t^2+2t+1\\
\end{align*}

\begin{align*}
A5_{22}[d;&125](t)=t^{124}+t^{123}+2t^{122}+t^{121}+t^{120}+2t^{119}+\\
&8t^{118}+12t^{117}+8t^{116}+2t^{115}+16t^{114}+60t^{113}+92t^{112}+60t^{111}+\\
&16t^{110}+102t^{109}+408t^{108}+612t^{107}+408t^{106}+102t^{105}+536t^{104}+\\
&2126t^{103}+3196t^{102}+2126t^{101}+536t^{100}+2126t^{99}+8504t^{98}+12756t^{97}+\\
&8504t^{96}+2126t^{95}+6744t^{94}+26932t^{93}+40420t^{92}+26932t^{91}+6744t^{90}+\\
&17310t^{89}+69240t^{88}+103860t^{87}+69240t^{86}+17310t^{85}+36803t^{84}+\\
&147107t^{83}+220710t^{82}+147107t^{81}+36803t^{80}+65376t^{79}+\\
&261504t^{78}+392256t^{77}+261504t^{76}+65376t^{75}+98104t^{74}+\\
&392256t^{73}+588464t^{72}+392256t^{71}+98104t^{70}+124812t^{69}+499248t^{68}+\\
&748872t^{67}+499248t^{66}+124812t^{65}+135264t^{64}+540860t^{63}+811384t^{62}+\\
&540860t^{61}+135264t^{60}+124812t^{59}+499248t^{58}+748872t^{57}+499248t^{56}+\\
&124812t^{55}+98104t^{54}+392256t^{53}+588464t^{52}+392256t^{51}+98104t^{50}+\\
&65376t^{49}+261504t^{48}+392256t^{47}+261504t^{46}+65376t^{45}+36803t^{44}+\\
&147107t^{43}+220710t^{42}+147107t^{41}+36803t^{40}+17310t^{39}+69240t^{38}+\\
&103860t^{37}+69240t^{36}+17310t^{35}+6744t^{34}+26932t^{33}+40420t^{32}+\\
&26932t^{31}+6744t^{30}+2126t^{29}+8504t^{28}+12756t^{27}+8504t^{26}+2126t^{25}+\\
&536t^{24}+2126t^{23}+3196t^{22}+2126t^{21}+536t^{20}+102t^{19}+408t^{18}+\\
&612t^{17}+408t^{16}+102t^{15}+16t^{14}+60t^{13}+92t^{12}+60t^{11}+16t^{10}+2t^9+\\
&8t^8+12t^7+8t^6+2t^5+t^4+t^3+2t^2+t+1\\
\end{align*}

\begin{align*}
A5_{23}[d;&125](t)=t^{124}+t^{123}+2t^{122}+t^{121}+t^{120}+t^{119}+4t^{118}+\\
&6t^{117}+4t^{116}+t^{115}+2t^{114}+6t^{113}+10t^{112}+6t^{111}+2t^{110}+t^{109}+\\
&4t^{108}+6t^{107}+4t^{106}+t^{105}+t^{104}+t^{103}+2t^{102}+t^{101}+t^{100}+\\
&t^{99}+4t^{98}+6t^{97}+4t^{96}+t^{95}+4t^{94}+16t^{93}+24t^{92}+16t^{91}+4t^{90}+\\
&6t^{89}+24t^{88}+36t^{87}+24t^{86}+6t^{85}+4t^{84}+16t^{83}+24t^{82}+16t^{81}+\\
&4t^{80}+t^{79}+4t^{78}+6t^{77}+4t^{76}+t^{75}+2t^{74}+6t^{73}+10t^{72}+6t^{71}+\\
&2t^{70}+6t^{69}+24t^{68}+36t^{67}+24t^{66}+6t^{65}+10t^{64}+36t^{63}+56t^{62}+\\
&36t^{61}+10t^{60}+6t^{59}+24t^{58}+36t^{57}+24t^{56}+6t^{55}+2t^{54}+6t^{53}+\\
&10t^{52}+6t^{51}+2t^{50}+t^{49}+4t^{48}+6t^{47}+4t^{46}+t^{45}+4t^{44}+16t^{43}+\\
&24t^{42}+16t^{41}+4t^{40}+6t^{39}+24t^{38}+36t^{37}+24t^{36}+6t^{35}+4t^{34}+\\
&16t^{33}+24t^{32}+16t^{31}+4t^{30}+t^{29}+4t^{28}+6t^{27}+4t^{26}+t^{25}+t^{24}+\\
&t^{23}+2t^{22}+t^{21}+t^{20}+t^{19}+4t^{18}+6t^{17}+4t^{16}+t^{15}+2t^{14}+6t^{13}+\\
&10t^{12}+6t^{11}+2t^{10}+t^9+4t^8+6t^7+4t^6+t^5+t^4+t^3+2t^2+t+1
\end{align*}

\begin{align*}
A_{52}[d;&125](t)=t^{124}+2t^{123}+16t^{122}+102t^{121}+536t^{120}+2127t^{119}+\\
&6748t^{118}+17316t^{117}+36807t^{116}+65377t^{115}+98118t^{114}+124866t^{113}+\\
&135346t^{112}+124866t^{111}+98118t^{110}+65477t^{109}+37207t^{108}+17916t^{107}+\\
&7148t^{106}+2227t^{105}+1071t^{104}+2227t^{103}+3210t^{102}+2127t^{101}+536t^{100}+\\
&2126t^{99}+8508t^{98}+12810t^{97}+8908t^{96}+4251t^{95}+15244t^{94}+53848t^{93}+\\
&109636t^{92}+174023t^{91}+268244t^{90}+409560t^{89}+568464t^{88}+644684t^{87}+\\
&568464t^{86}+409560t^{85}+298303t^{84}+294198t^{83}+289926t^{82}+174023t^{81}+\\
&45303t^{80}+67501t^{79}+261908t^{78}+392310t^{77}+261508t^{76}+65376t^{75}+\\
&98104t^{74}+392262t^{73}+588546t^{72}+392862t^{71}+101298t^{70}+137562t^{69}+\\
&539644t^{68}+852696t^{67}+719934t^{66}+517062t^{65}+723718t^{64}+1289696t^{63}+\\
&1622712t^{62}+1289696t^{61}+723718t^{60}+517062t^{59}+719934t^{58}+852696t^{57}+\\
&539644t^{56}+137562t^{55}+101298t^{54}+392862t^{53}+588546t^{52}+392262t^{51}+\\
&98104t^{50}+65376t^{49}+261508t^{48}+392310t^{47}+261908t^{46}+67501t^{45}+\\
&45303t^{44}+174023t^{43}+289926t^{42}+294198t^{41}+298303t^{40}+409560t^{39}+\\
&568464t^{38}+644684t^{37}+568464t^{36}+409560t^{35}+268244t^{34}+174023t^{33}+\\
&109636t^{32}+53848t^{31}+15244t^{30}+4251t^{29}+8908t^{28}+12810t^{27}+8508t^{26}+\\
&2126t^{25}+536t^{24}+2127t^{23}+3210t^{22}+2227t^{21}+1071t^{20}+2227t^{19}+\\
&7148t^{18}+17916t^{17}+37207t^{16}+65477t^{15}+98118t^{14}+124866t^{13}+\\
&135346t^{12}+124866t^{11}+98118t^{10}+65377t^9+36807t^8+17316t^7+6748t^6+\\
&2127t^5+536t^4+102t^3+16t^2+2t+1
\end{align*}

\begin{align*}
A_{5}[d;&125](t)=4t^{118}+54t^{117}+404t^{116}+2125t^{115}+8500t^{114}+\\
&26954t^{113}+69730t^{112}+150929t^{111}+281694t^{110}+473000t^{109}+\\
&755176t^{108}+1201454t^{107}+1919666t^{106}+2997625t^{105}+4472454t^{104}+\\
&6424900t^{103}+9165534t^{102}+13251600t^{101}+19155265t^{100}+26903750t^{99}+\\
&36377186t^{98}+48294000t^{97}+64657615t^{96}+87434250t^{95}+116536112t^{94}+\\
&150192691t^{93}+188707920t^{92}+237249741t^{91}+302288246t^{90}+\\
&383936250t^{89}+474733623t^{88}+570259166t^{87}+680048184t^{86}+821076375t^{85}+\\
&996304759t^{84}+1185941116t^{83}+1370394858t^{82}+1561026291t^{81}+\\
&1793727398t^{80}+2082332125t^{79}+2389932756t^{78}+2666826750t^{77}+\\
&2917005294t^{76}+3201737500t^{75}+3559663612t^{74}+3941347550t^{73}+\\
&4256542598t^{72}+4487698700t^{71}+4719241134t^{70}+5027928500t^{69}+\\
&5369173454t^{68}+5617842524t^{67}+5724176686t^{66}+5781104750t^{65}+\\
&5904255800t^{64}+6075462274t^{63}+6159096216t^{62}+6075462274t^{61}+\\
&5904255800t^{60}+5781104750t^{59}+5724176686t^{58}+5617842524t^{57}+\\
&5369173454t^{56}+5027928500t^{55}+4719241134t^{54}+4487698700t^{53}+\\
&4256542598t^{52}+3941347550t^{51}+3559663612t^{50}+3201737500t^{49}+\\
&2917005294t^{48}+2666826750t^{47}+2389932756t^{46}+2082332125t^{45}+\\
&1793727398t^{44}+1561026291t^{43}+1370394858t^{42}+1185941116t^{41}+\\
&996304759t^{40}+821076375t^{39}+680048184t^{38}+570259166t^{37}+\\
&474733623t^{36}+383936250t^{35}+302288246t^{34}+237249741t^{33}+\\
&188707920t^{32}+150192691t^{31}+116536112t^{30}+87434250t^{29}+\\
&64657615t^{28}+48294000t^{27}+36377186t^{26}+26903750t^{25}+\\
&19155265t^{24}+13251600t^{23}+9165534t^{22}+6424900t^{21}+\\
&4472454t^{20}+2997625t^{19}+1919666t^{18}+1201454t^{17}+755176t^{16}+\\
&473000t^{15}+281694t^{14}+150929t^{13}+69730t^{12}+26954t^{11}+\\
&8500t^{10}+2125t^9+404t^8+54t^7+4t^6
\end{align*}

\begin{align*}
Total&[d;125](t)=A_{1}+A_{2}+A_{3}+A_{4}+A_{5}\\
&t^{124}+3t^{123}+90t^{122}+3183t^{121}+94261t^{120}+2253202t^{119}+\\
&44660526t^{118}+752765426t^{117}+11009026889t^{116}+141893725177t^{115}+\\
&1631777381270t^{114}+16911146021617t^{113}+159246624819695t^{112}+\\
&1371970915393992t^{111}+10877769404828584t^{110}+79770308932154652t^{109}+\\
&543435229633787791t^{108}+3452412046870263651t^{107}+\\
&20522671612153800248t^{106}+114494904782463078927t^{105}+\\
&601098250109006348605t^{104}+2976867524344040005968t^{103}+\\
&13937152500343088195264t^{102}+61808241523238111266288t^{101}+\\
&260109683076981986458211t^{100}+1040438732307841539649001t^{99}+\\
&3961670557633787406854497t^{98}+14379396838818630509486185t^{97}+\\
&49814339048764832197753486t^{96}+164902639609703309488079986t^{95}+\\
&522191692097394743906419238t^{94}+1583419969585645756697355513t^{93}+\\
&4601814286608285713826377107t^{92}+12829300435392788915059242212t^{91}+\\
&34337245282963060080748456842t^{90}+88295773584762135474231228583t^{89}+\\
&218286773584550853413218320189t^{88}+519168542579472256031665746576t^{87}+\\
&1188622715905633892156485804100t^{86}+2621065476099602847252984385458t^{85}+\\
&5569764136711656142477377352019t^{84}+11411224084970222152064822132446t^{83}+\\
&22550752358393534437019574995528t^{82}+43003760311355111831092568853122t^{81}+\\
&79166013300449183486752791791953t^{80}+140739579200798547810848407179486t^{79}+\\
&241704929497023593576284109748137t^{78}+401127329803571069200022148395060t^{77}+\\
&643475091559895257811475523869150t^{76}+998042999154123255509578916762251t^{75}+\\
&1497064498731184884738423828738826t^{74}+2172211233453091791403201171682487t^{73}+\\
&3049450385424532709259880563769891t^{72}+4142649580199365187088253032770042t^{71}+\\
&5446817040632498674709258423750752t^{70}+6932312597168634673342968879821758t^{69}+\\
&8541599450082782011722315254853863t^{68}+10189978291326827658936469293236257t^{67}+\\
&11771181819291335403215209136236758t^{66}+13167762713105561632909721570633733t^{65}+\\
&14265076272531025106827702213046251t^{64}+14966637400688288631941104298816390t^{63}+\\
&15208034778118744904852502416921288t^{62}+14966637400688288631941104298816390t^{61}+\\
&14265076272531025106827702213046251t^{60}+13167762713105561632909721570633733t^{59}+\\
\end{align*}

\begin{align*}
Total&[d;125](t) continued...\\
&11771181819291335403215209136236758t^{58}+10189978291326827658936469293236257t^{57}+\\
&8541599450082782011722315254853863t^{56}+6932312597168634673342968879821758t^{55}+\\
&5446817040632498674709258423750752t^{54}+4142649580199365187088253032770042t^{53}+\\
&3049450385424532709259880563769891t^{52}+2172211233453091791403201171682487t^{51}+\\
&1497064498731184884738423828738826t^{50}+998042999154123255509578916762251t^{49}+\\
&643475091559895257811475523869150t^{48}+401127329803571069200022148395060t^{47}+\\
&241704929497023593576284109748137t^{46}+140739579200798547810848407179486t^{45}+\\
&79166013300449183486752791791953t^{44}+43003760311355111831092568853122t^{43}+\\
&22550752358393534437019574995528t^{42}+11411224084970222152064822132446t^{41}+\\
&5569764136711656142477377352019t^{40}+2621065476099602847252984385458t^{39}+\\
&1188622715905633892156485804100t^{38}+519168542579472256031665746576t^{37}+\\
&218286773584550853413218320189t^{36}+88295773584762135474231228583t^{35}+\\
&34337245282963060080748456842t^{34}+12829300435392788915059242212t^{33}+\\
&4601814286608285713826377107t^{32}+1583419969585645756697355513t^{31}+\\
&522191692097394743906419238t^{30}+164902639609703309488079986t^{29}+\\
&49814339048764832197753486t^{28}+14379396838818630509486185t^{27}+\\
&3961670557633787406854497t^{26}+1040438732307841539649001t^{25}+\\
&260109683076981986458211t^{24}+61808241523238111266288t^{23}+\\
&13937152500343088195264t^{22}+2976867524344040005968t^{21}+601098250109006348605t^{20}+\\
&114494904782463078927t^{19}+20522671612153800248t^{18}+3452412046870263651t^{17}+\\
&543435229633787791t^{16}+79770308932154652t^{15}+10877769404828584t^{14}+\\
&1371970915393992t^{13}+159246624819695t^{12}+16911146021617t^{11}+\\
&1631777381270t^{10}+141893725177t^9+11009026889t^8+752765426t^7+44660526t^6+\\
&2253202t^5+94261t^4+3183t^3+90t^2+3t+1
\end{align*}

\end{changemargin}

\nocite{Mishna98, KP78, KRRT91, Parks2012, FIK90, LS2003}
\printindex
\bibliography{bibfile}

\end{document}